\documentclass[a4paper,11pt]{article}

\usepackage{graphicx,enumitem}
\usepackage{color}
\usepackage{latexsym}
\usepackage{amsmath, amssymb, mathrsfs}
\usepackage{algorithmic}
\usepackage{algorithm}
\usepackage[colorlinks]{hyperref}

\usepackage{libertine}
\usepackage{soul,xcolor}
\usepackage[colorlinks]{hyperref}
\usepackage{tikz}
\usetikzlibrary{arrows.meta}
\usetikzlibrary{calc}
\usepackage{subcaption} 
\usepackage{amsfonts,amsmath,amssymb}
\usepackage{amsthm}
\usepackage{geometry}
\usepackage{graphicx} 
\usepackage{xcolor}
\usepackage{hyperref}
\usepackage{footnote}
\usepackage{stackrel}
\usepackage{comment}
\allowdisplaybreaks
\usepackage[tikz]{mdframed}

\newtheorem{theorem}{Theorem}[section]
\newtheorem{definition}[theorem]{Definition}   
\newtheorem{remark}[theorem]{Remark}
\newtheorem{proposition}[theorem]{Proposition}
\newtheorem{lemma}[theorem]{Lemma}
\newtheorem{corollary}[theorem]{Corollary}
\newtheorem{problem}[theorem]{Problem}

\usepackage{amsmath}
\usepackage{mathtools}
\usepackage{mathrsfs}
\usepackage{amssymb}
\usepackage{multicol}
\usepackage[colorlinks]{hyperref}
\usepackage{caption,subcaption}
\usepackage{chngcntr}
\usepackage{apptools}

\usepackage{bm}
\newcommand{\vect}[1]{\boldsymbol{\mathbf{#1}}}
\newcommand{\vertiii}[1]{{\left\vert\kern-0.25ex\left\vert\kern-0.25ex\left\vert #1 
    \right\vert\kern-0.25ex\right\vert\kern-0.25ex\right\vert}}
\DeclarePairedDelimiter\abs{\lvert}{\rvert}%
\newcommand{\dive}{\operatorname{div}}
\newcommand{\ttau}{\vect{\tau}}
\newcommand{\uu}{\vect{u}}
\newcommand{\vv}{\vect{v}}
\newcommand{\ww}{\vect{w}}
\newcommand{\uud}{\vect{u}_{D}}
\newcommand{\uun}{\vect{u}_{N}}
\newcommand{\pd}{p_{D}}
\newcommand{\pn}{p_{N}}
\newcommand{\ff}{\vect{f}}
\newcommand{\ggb}{g_{\Sigma}}

\newcommand{\cv}{\overline{\vv}}
\newcommand{\cw}{\overline{\ww}}
\newcommand{\cvarphi}{\overline{\vect{\varphi}}}
\newcommand{\Mt}{M_{t}}

\newcommand{\Vgamma}{V_{\Gamma}}

\usepackage{stackengine} 
\newcommand\ocirc[1]{\ensurestackMath{\stackon[0.5pt]{#1}{\mkern1mu\circ}}}
\newcommand{\oV}{\ocirc{V}}
\newcommand{\oVzero}{\ocirc{V}_{0}}
\newcommand{\oQ}{\ocirc{Q}} 
 
\newcommand{\olambda}{\ocirc{\lambda}} 
 
\newcommand{\oVperp}{\ocirc{V}_{\perp}}
\newcommand{\un}{\vect{u}_\text{N}}
\newcommand{\ud}{\vect{u}_\text{D}}

\newcommand{\bphi}{\vect{\varphi}}
\newcommand{\bpsi}{\vect{\psi}}
\newcommand{\cbphi}{\overline{\vect{\varphi}}}
\newcommand{\cbpsi}{\overline{\vect{\psi}}}
\newcommand{\obphi}{\ocirc{\vect{\varphi}}}

\newcommand{\yt}{\vect{y}^{t}}
\newcommand{\cyt}{\overline{\vect{y}}^{t}}
\newcommand{\ut}{\uu^t}
\newcommand{\ur}{\uu_{r}}
\newcommand{\ui}{\uu_{i}}
\newcommand{\pr}{p_{r}}
\newcommand{\pim}{p_{i}}

\newcommand{\vr}{\vv_{r}}
\newcommand{\vi}{\vv_{i}}

\newcommand{\aaa}{a}
\newcommand{\aat}{a^{t}} 
\newcommand{\taat}{\tilde{a}^{t}} 
\newcommand{\LL}{\vect{\mathsf{L}}}

\newcommand{\sfTheta}{\vect{\Theta}}

\newcommand{\HH}{\vect{\mathsf{H}}}
\newcommand{\HHg}{\vect{\mathsf{V}}}
\newcommand{\NN}{\vect{N}}
 
\newcommand{\VV}{\vect{\theta}}
\newcommand{\Vn}{\theta_n} 
\newcommand{\nn}{\vect{n}}

\newcommand{\dn}[1]{\partial_{\nn}{#1}}   
\newcommand{\dnn}[1]{\partial_{\nn\nn}^2{#1}}

\newcommand{\intO}[1]{\int_{\Omega}{#1}{\, {d} x}}
\newcommand{\intdO}[1]{\int_{\partial\Omega}{#1}{\, {d} \sigma}}
\newcommand{\intOt}[1]{\int_{\Omega_{t}}{#1}{\, {d} x_{t}}} 

\newcommand{\intS}[1]{\int_{\Sigma}{#1}{\, {d} \sigma}} 

\newcommand{\intSt}[1]{\int_{\Sigma_{t}}{#1}{\, {d} \sigma_{t}}}
\newcommand{\intG}[1]{\int_{\Gamma}{#1}{\, {d} \sigma}} 

\begin{document}

\title{Numerical solution to a free boundary problem for the Stokes equation using the coupled complex boundary method in shape optimization settings\footnote{This work is partially supported by JSPS KAKENHI Grant Numbers JP20KK0058, JP18H01135, JP21H04431, and JP20H01823, and JST CREST Grant Number JPMJCR2014.}}
\author{Julius Fergy T. Rabago$^{1}$\ \ and Hirofumi Notsu$^{2}$}

\date{%
	{\footnotesize
	 Faculty of Mathematics and Physics\\%
	 Institute of Science and Engineering\\%
         Kanazawa University, Kanazawa 920-1192, Japan\\\vspace{-2pt}
        \texttt{$^{1}$rabagojft@se.kanazawa-u.ac.jp,\ jfrabago@gmail.com; $^{2}$notsu@se.kanazawa-u.ac.jp}}\\[2ex]
    \today
}

\maketitle

\begin{abstract}
A new reformulation of a free boundary problem for the Stokes equations governing a viscous flow with overdetermined condition on the free boundary is proposed.
The idea of the method is to transform the governing equations to a boundary value problem with a complex Robin boundary condition coupling the two boundary conditions on the free boundary.
The proposed formulation give rise to a new cost functional that apparently has not been exploited yet in the literature, specifically, and at least, in the context of free surface problems. 
The shape derivatives of the cost function constructed by the imaginary part of the solution in the whole domain in order to identify the free boundary is explicitly determined.
Using the computed shape gradient information, a domain variation method from a preconditioned steepest descent algorithm is applied to solve the shape optimization problem.
Numerical results illustrating the applicability of the method is then provided both in two and three spatial dimensions.
For validation and evaluation of the method, the numerical results are compared with the ones obtained via the classical tracking Dirichlet data.
\medskip

\end{abstract}

\clearpage
\hrule
\tableofcontents
\bigskip
\medskip
\hrule
\newpage
\section{Introduction}
\label{sec:Introduction}

We are interested in a free boundary problem for fluid flows governed by the Stokes equations which arises in many applications such as in magnetic shaping processes where the shape of the fluid is determined by the Lorentz force.
In such setting, the model is described by the Stokes flow equations and a pressure balance equation on the free boundary neglecting surface tension effects \cite{BouchonPeichlSayehTouzani2017}.
Depending on the application, two model problems can be considered in this context.
First, an interior case where the fluid is confined in a mould and has an internal unknown boundary and, second, an exterior case where a part of the fluid boundary adheres to a solid and the remaining part is free and is in contact with the ambient air.
Here, we pay attention to the second case for $d$-dimensional geometries, $d \in \{2,3\}$.

\textit{Main Problem}. In this work, we are particularly interested in the free surface problem analog of the Bernoulli free boundary problem or Alt-Caffarelli problem (see \cite{FlucherRumpf1997,AltCaffarelli1981}) where the Laplace operator is replaced by the Stokes operator. 
More precisely, we are interested in the problem formulated as follows (see \cite{Kasumba2014} and also \cite{BouchonPeichlSayehTouzani2017}).
Given a simply connected bounded $A \subset \mathbb{R}^{d}$ with boundary $\Gamma:=\partial A$ domain $\Omega$.
The fluid is considered levitating around $A$, which is influenced by a gravity-like force $\ff = (f_{1},f_{2},\ldots,f_{d})^{\top}$, and occupies then the domain $\Omega = B\setminus \overline{A}$, where $B$ is a larger bounded (simply connected) domain with boundary $\Sigma:=\partial B$ containing $\overline{A}$.
The incompressible viscous flow occupying $\Omega$, the velocity field $\uu = (u_{1},u_{2},\ldots,u_{d})^{\top}$, and the pressure $p$ is then supposed to satisfy the following overdetermined system of Stokes equations in non-dimensional form:
\begin{equation}
\label{eq:FSP}
\left\{\arraycolsep=1.4pt\def\arraystretch{1.1}
\begin{array}{rcll}
	- \alpha \Delta \uu + \nabla p	&=& \ff		&\quad\text{in $\Omega$},\\
	\nabla \cdot \uu				&=& \vect{0} 	&\quad\text{in $\Omega$},\\
	\uu	 					&=& \vect{g}	&\quad\text{on $\Gamma$},\\
	-p\nn + \alpha \dn{\uu}		&=& \vect{0}	&\quad\text{on $\Sigma$},\\
	\uu \cdot \nn				&=& 0		&\quad\text{on $\Sigma$},
\end{array}
\right.
\end{equation}
where $\alpha := {Re}^{-1}$, and $Re > 0$ is the Reynolds number, and $\dn{} := \nn \cdot \nabla$ denotes the derivative in the (outward unit) normal vector $\nn$.
Here, the boundary data $\vect{g}$ is a prescribed velocity which in some sense expresses the motion of $\Gamma$ while the boundary condition imposed on the free surface $\Sigma$ indicate zero ambient pressure and negligible surface tension effects\footnote{The zero-surface tension assumption is a typical setup in the literature which not only simplifies the discussion, but also allows one to ignore technical difficulties resulting from higher derivative terms.}.
Additional assumptions concerning on the density of external forces $\ff$ and on $\vect{g}$ will be given in next section.
Meanwhile, the slip boundary condition $\uu \cdot \nn=0$ on $\Sigma$ means that inflow or outflow of the fluid cannot happen on the free boundary while tangential velocities on the said boundary can be non-zero.
Such condition is appropriate for problems that involve free boundaries and situations where the usual no-slip condition $\uu = \vect{0}$ is not valid (e.g., flows past chemically reacting walls see \cite{BeaversJoseph1967,Liakos2001}).
Other realistic situations where no-slip and slip boundary conditions appear naturally are as follows.
The no-slip condition has long been firmly established for moderate pressures and velocities, not only by direct observations (see, e.g., \cite{GiraultRaviart1986} or \cite{Galdi1994}), but also through comparisons between numerical simulations and results from actual experiments of a large class of complex flow problems.
Early experimental investigations, however, showed that low-temperature slip past solid surfaces.
This happens, in particular, for sufficiently large Knudsen numbers where velocity slip occurs at the wall surface.
Such a physical phenomenon has also been observed in hydraulic fracturing and biological fluids \cite{Liakos2001} which are examples of nonlinear fluid flows.
A slip boundary condition also applies to free surfaces in free boundary problems.
A well-known example where the condition arises is the so-called coating problem; see, e.g., \cite{Babuska1963,SaitoScriven1981}.

Clearly, motivated by numerous applications such as ship hydrodynamics \cite{AlessandriniDelhommeau1994} and thin film manufacturing processes \cite{Volkovetal2009}, numerical solutions of flows that are partially bounded by a freely moving boundary are of great practical importance.
In such problems, not only the flow variables, but also the flow domain has to be determined -- hence the word `free boundary' is used for the unknown boundary.
Due to the complexity of simultaneously resolving these unknowns, a numerical solution which can be obtained through different methods \cite{Weietal1996} has to be carried out iteratively \cite{TiihonenKarkkainen1999}.
 
\textit{Known approaches}. In this investigation, we shall resolve the free boundary problem via a novel shape optimization method.
Because free surface problems (FSPs) such as \eqref{eq:FSP} consist of overdetermined boundary conditions on the unknown part of the boundary, they can be reformulated into an equivalent shape optimization problem; see, e.g., \cite{Kasumba2014,BouchonPeichlSayehTouzani2017}.\footnote{The same technique -- but in the context of optimal shape design problems -- of finding the boundary that minimizes a norm of the residual of one of the free surface conditions, subject to the boundary value problem with the remaining free surface conditions imposed, has also been used for potential free surface flows in \cite{VanBrummelenRavenKoren2001,VanBrummelenSegal2003}.}
Such approach can be carried out in several ways.
A typical strategy is to choose one of the boundary conditions on the unknown boundary to obtain a well-posed state equation and then track the remaining boundary data in an appropriate norm -- typically in $L^2$ (see, e.g., \cite{EpplerHarbrecht2009,EpplerHarbrecht2010,HIKKP2009,HKKP2003,IKP2006} for the exterior case of the Bernoulli problem and \cite{Kasumba2014} for free surface problems).
Another one is to utilize the so-called Kohn-Vogelius cost functional which consists of two auxiliary problems each posed with one of the boundary condition on the free boundary (see \cite{BenAbdaetal2013,Bacani2013,EpplerHarbrecht2012a} for the Bernoulli problem and for free surface problems \cite{Kasumba2014,BouchonPeichlSayehTouzani2017}).
For the latter formulation, one considers the minimization problem
\begin{equation}\label{eq:KV_method}
	J_{KV}(\Omega):=\frac12 \intO{\left| \nabla\left(\un - \ud\right) \right|^{2}} \ \to\ \inf,
\end{equation} 
where the state variables $\uun:=\uun(\Omega)$ and $\uud=\uud(\Omega)$ respectively satisfy the following well-posed systems of partial differential equations (PDEs):
\begin{equation}
\label{eq:state_ud}
\left\{\arraycolsep=1.4pt\def\arraystretch{1.1}
\begin{array}{rcll}
	- \alpha \Delta \uud + \nabla \pd	&=& \ff			&\quad\text{in $\Omega$},\\
	\nabla \cdot \uud			&=& \vect{0} 		&\quad\text{in $\Omega$},\\
	\uud	 					&=& \vect{g}		&\quad\text{on $\Gamma$},\\
	\uud \cdot \nn				&=& \vect{0}		&\quad\text{on $\Sigma$},\\
	\alpha \dn{\uud} \cdot \ttau		&=& \vect{0}		&\quad\text{on $\Sigma$},
\end{array}
\right.
\end{equation}
%
%
%
\begin{equation}
\label{eq:state_un}
\left\{\arraycolsep=1.4pt\def\arraystretch{1.1}
\begin{array}{rcll}
	- \alpha \Delta \uun + \nabla \pn	&=& \ff			&\quad\text{in $\Omega$},\\
	\nabla \cdot \uun			&=& \vect{0} 		&\quad\text{in $\Omega$},\\
	\uun	 					&=& \vect{g}		&\quad\text{on $\Gamma$},\\
	-\pn\nn + \dn{\uun}			&=& \vect{0}		&\quad\text{on $\Sigma$}.
\end{array}
\right.
\end{equation}
In \eqref{eq:state_ud}, $\ttau$ denotes the tangential vector to a domain $\Omega$.
Meanwhile, for tracking-boundary-data cost functional approach, the following minimization problems can be considered:
\begin{align}
	J_{D}(\Omega) &:=\frac12 \intS{\left(\un \cdot \nn \right)^{2}} \ \to\ \inf,\label{eq:tracking_Dirichlet_method}\\
	J_{N}(\Omega) &:=\frac12 \intS{\abs{-\pd\nn + \dn{\uud}}^{2}} \ \to\ \inf,\label{eq:tracking_Neumann_method}
\end{align} 
where, of course, $\ud$ and $\un$ satisfy problems \eqref{eq:state_ud} and \eqref{eq:state_un}, respectively.
\begin{remark}\label{eq:equivalence_of_classical_formulations}
	If $(\Omega, (\uu, p))$ is a solution of \eqref{eq:FSP}, then $\uud = \uun = \uu$.
	Hence, $J_{KV}(\Omega) = J_{D}(\Omega) = J_{N}(\Omega) = 0$.
	Conversely, if $J_{i}(\Omega) = 0$, $i\in \{KV,D,N\}$, are defined by \eqref{eq:KV_method}, \eqref{eq:tracking_Dirichlet_method}, and \eqref{eq:tracking_Neumann_method}, respectively, then for $i = KV$ ($J_{KV}(\Omega) = 0$), then we have that $\uud = \uun = \uu$ and $\uud = \uun = \uu$ is a solution of problem \eqref{eq:FSP}.
	For $i = D$ ($J_{D}(\Omega) = 0$), we have that $\uun \cdot \nn = 0$ almost everywhere on $\Sigma$ and $\uu = \uun$ is a solution of problem \eqref{eq:FSP}.
	Similarly, the same holds for $i = N$ ($J_{N}(\Omega) = 0$).
	Therefore, $(\Omega, (\uu, p))$ is a solution of \eqref{eq:FSP} if and only if $J_{i}(\Omega) = 0$, $i\in \{KV,D,N\}$.
\end{remark}
\begin{remark}
	The identity $J_{KV}(\Omega) = 0$ is equivalent to the existence of a constant $\lambda_{0} \in \mathbb{R}$ such that $(\uud,\pd) = (\uud,\pn+\lambda_{0})$ \cite{BouchonPeichlSayehTouzani2017}.
\end{remark}
\begin{remark}
In addition to the abovementioned classical approaches, there is another domain-integral-type penalization that may be considered.
More exactly, one may opt to penalize, instead of $J_{KV}$, by the cost functional 
\[
	J_{L^2}(\Omega) := \frac{1}{2} \intO{\abs{\un - \ud}^2}. 
\]
Compared with $J_{KV}$, however, the shape gradient of the above cost function is more complicated in structure and would be numerically expensive to evaluate.
In fact, it can be shown that in this case one would need to solve five real BVPs: two state equations, two adjoint equations (see \cite{LaurainPrivat2012} for the case of the Bernoulli problem), and one to approximate the mean curvature of the free boundary implicitly.
\end{remark}
In all of the minimization problems presented above, the infimum has always to be taken over all sufficiently smooth domains.
On the other hand, we emphasize that the cost function $J_{N}$ requires more regularity for $\pd$ and $\ud$ to be well defined.
As a consequence, it may be impractical to utilize this functional in numerical experiments where high regularity of the state variables is not guaranteed.
For the feasibility of these approaches, with $\vect{g} \equiv \vect{0}$, we refer the readers to \cite{Kasumba2014}.
In that paper, Kasumba computed explicitly the shape derivatives of the functionals $J_{KV}$, $J_{D}$, and $J_{N}$ via the chain rule approach.
Then, he utilizes the gradient information in a boundary variation method in a preconditioned steepest descent algorithm to solve the given shape optimization problems.

\textit{New strategy and novelty}. In this investigation, we want to offer yet another shape optimization technique to solve \eqref{eq:FSP}.
More exactly, we want to propose here a novel application of the so-called coupled complex boundary method or CCBM for solving stationary free surface problems.
The point of departure of the method is somewhat similar to \cite{Tiihonen1997}, but applies the concept of complex PDEs.
The idea is straightforward: we couple the Dirichlet and Neumann data in a Robin boundary condition in such a way that the Dirichlet data and the Neumann data are the respective real and imaginary parts of the Robin boundary condition.
In effect, the boundary conditions that have to be satisfied on the free boundary are transformed into one condition that needs to be satisfied on the domain.
Consequently, the formulation requires the introduction of a new cost function that has not been studied yet in the literature (see \eqref{eq:cost_function}).
Moreover, the new reformulation seems to be advantageous in the sense that it leads to a volume integral and produces more regular adjoint state than those of boundary-data tracking-type cost functionals.

The coupled complex boundary method was first introduced by Cheng et al. in \cite{Chengetal2014} as a method to solve an inverse source problem (see also \cite{Chengetal2016}) and was then used to deal with a Cauchy problem stably in \cite{Chengetal2016}.
It is then applied to an inverse conductivity problem with one measurement in \cite{Gongetal2017} and also to parameter identification in elliptic problems in \cite{Zhengetal2020}.
In a much recent paper, CCBM was also applied in solving inverse obstacle problems by Afraites in \cite{Afraites2022}.
Moreover, in the context of free boundary problems, the method was first applied by Rabago in \cite{Rabago2022} to solve the exterior case of the Bernoulli problem.
As far as we are concern, there is no work yet dealing with FSPs via CCBM, except in \cite{Ouiassaetal2022} where the authors applied CCBM as a numerical approach to solve an inverse Cauchy Stokes problem.
Therefore, the novelty of this work comes from the application of CCBM in treating free surface problems, especially as a numerical resolution to the free boundary problem \eqref{eq:FSP} to which we found some merit over the classical shape optimization methods of tracking the boundary data using a least-squares misfit cost function.

\textit{Contributions to the literature.} The main contributions and highlights of this study are listed as follows.
\begin{itemize}
	\item As opposed to classical shape optimization approaches for free surface problems \cite{Kasumba2014,BouchonPeichlSayehTouzani2017}, our state problem is a complex PDE system (see equation \eqref{eq:ccbm} in the next section), and to the best our knowledge, the Stokes equation with a Robin condition having complex Robin coefficient has not been considered yet in any of the previous investigations -- at least in the context of shape optimization.
	The problem therefore is new and hence warrants some initial examination.
	Although the approach to establish the well-posedness of the problem follows the same argumentations as in the real case, the discussion of the topic (which we provide here) is not yet available in the existing literature.
	See subsection \ref{subsec:well-posedness_of_state_problem} for the discussion.
	\item As a consequence of the CCBM formulation, a new cost functional for the free surface problem -- which apparently has not been examined yet in the literature -- is first introduced in this work; see equation \eqref{eq:cost_function}. 
	\item The study focuses on the rigorous computation of the first-order shape derivative of the cost functional associated with CCBM (refer to Proposition \ref{prop:the_shape_derivative_of_the_cost}) using only the H\"{o}lder continuity of the state variables (see Lemma \ref{lem:holder_continuity}) -- differing from the classical approach which uses either the material or shape derivatives of the states (see, e.g., \cite{Kasumba2014,BouchonPeichlSayehTouzani2017,Rabago2022}).
	In fact, our approach bypasses the computation of the said derivatives.
	\item The proof of the H\"{o}lder continuity of the state variables proceeds in a slightly different manner from the usual strategy found, for instance, in \cite{BacaniPeichl2013,IKP2006,IKP2008}; see the proof of Lemma \ref{lem:holder_continuity}.
		By applying the same technique used in the aforementioned papers, it seems difficult to eliminate the pressure variables to get a uniform estimate for the difference between the velocity solution of the transformed and the steady case of the Stokes problem.
		To get around this difficulty, we instead use the corresponding expansions of some specific expressions appearing in the computation.
		See the proof of Lemma \ref{lem:holder_continuity}.
	\item As far as we are aware of, previous numerical investigations on the free boundary problem with the Stokes equation (such as that of \cite{Kasumba2014}) only dealt with two dimensional cases.
		In this study, the proposed shape optimization method is also tested in solving a test problem in three dimensions.
		In fact, it will be shown in the numerical part of the paper that the proposed formulation has a smoothing effect -- specially in the case of three dimensions -- in approximating a solution to the overdetermined system of complex partial differential equations.
		See the discussion on subsection \ref{subsubsection:example_in_3d}.
\end{itemize}

The remainder of the paper is organized as follows.
In Section \ref{sec:shape_optimization_formulation}, we will demonstrate how the free surface problem \eqref{eq:FSP} is formulated into a shape optimization problem via the coupled complex boundary method.
The well-posedness of the CCBM formulation is also discussed in this section.  
Meanwhile, we devote Section \ref{sec:Computation_of_the_shape_derivative} in computing the boundary integral expression of the shape gradient of the cost associated with the proposed shape optimization problem.
The section of the paper starts with a brief discussion of some ideas from shape calculus needed in the study, followed by a short list of identities from tangential shape calculus.  
For the main result, which pertains to the shape gradient of the cost, the expression will be characterized rigorously via the rearrangement method in the spirit of \cite{IKP2006,IKP2008}.
In Section \ref{sec:Numerical_Approximation}, the continuous formulation is discretized, and a numerical algorithm based on Sobolev gradient method for solving the discrete shape optimization problem is developed and implemented.
The section therefore is divided into two parts: the discussion of the iterative scheme (subsection \ref{subsec:Numerical_Algorithm}) and the presentation of the numerical experiments carried out in two and three dimensions (subsection \ref{subsec:Numerical_Examples}).
The main content of the paper ends with Section \ref{sec:Conclusions_and_Future_Work} where we issue a short conclusion about the study and a statement of future work.

Additionally, the paper also contains three appendices wherein we provide the proofs of some lemmas (Appendix \ref{appxsec:proofs}), show the computation of some identities used in the investigation (Appendix \ref{appxsubsec:computations}), and demonstrate the derivation of the shape gradient expression via the chain rule approach (Appendix \ref{appxsubsec:shape_derivatives_of_the_cost_via_chain_rule}).
\section{CCBM in shape optimization settings}
\label{sec:shape_optimization_formulation}
We present here the proposed CCBM formulation of \eqref{eq:FSP} and discuss the well-posedness of the state problem.

\subsection{The coupled complex boundary method formulation} 
The coupled complex boundary method suggests to write the boundary conditions on the free boundary as one condition.
This means to consider the complex boundary value problem
\begin{equation}
\label{eq:ccbm}
\left\{\arraycolsep=1.4pt\def\arraystretch{1.1}
\begin{array}{rcll}
	- \alpha \Delta \uu + \nabla p	&=& \ff			&\quad\text{in $\Omega$},\\
	\nabla \cdot \uu				&=& \vect{0} 		&\quad\text{in $\Omega$},\\
	\uu	 					&=& \vect{g}			&\quad\text{on $\Gamma$},\\
	-p\nn + \alpha \dn{\uu}	 + i (\uu \cdot \nn)\nn			&=& \vect{0}		&\quad\text{on $\Sigma$},
\end{array}
\right.
\end{equation}
where $i = \sqrt{-1}$ stands for the unit imaginary number.
Letting $(\uu,p) := (\ur + i \ui, p_{r} + i \pim)$ denote the solution of \eqref{eq:ccbm}, then it can be verified that the real vector-valued functions $\ur$ and $\ui$, and real scalar valued functions $p_{r}$ and $\pim$, respectively satisfy the real PDE systems:
\begin{equation}
\label{eq:ccbm_real_part}
\left\{\arraycolsep=1.4pt\def\arraystretch{1.1}
\begin{array}{rcll}
	- \alpha \Delta \ur + \nabla \pr	&=& \ff			&\quad\text{in $\Omega$},\\
	\nabla \cdot \ur				&=& \vect{0} 		&\quad\text{in $\Omega$},\\
	\ur	 					&=& \vect{g}			&\quad\text{on $\Gamma$},\\
	-\pr\nn + \alpha \dn{\ur}				&=& (\ui \cdot \nn)\nn		&\quad\text{on $\Sigma$},
\end{array}
\right.
\end{equation}
\begin{equation}
\label{eq:ccbm_imaginary_part}
\left\{\arraycolsep=1.4pt\def\arraystretch{1.1}
\begin{array}{rcll}
	- \alpha \Delta \ui + \nabla \pim	&=& \vect{0}			&\quad\text{in $\Omega$},\\
	\nabla \cdot \ui				&=& \vect{0} 		&\quad\text{in $\Omega$},\\
	\ui	 					&=& \vect{0}		&\quad\text{on $\Gamma$},\\
	-\pim\nn + \alpha \dn{\ui}				&=& - (\ur \cdot \nn)\nn		&\quad\text{on $\Sigma$},
\end{array}
\right.
\end{equation}
%
%

From now on, if there is no confusion, we represent the real and imaginary parts of any function (vector-valued or scalar valued) by attaching the function with the subscript $\cdot\, _{r}$ and $\cdot\, _{i}$, respectively.
%
%
%
\begin{remark}
\label{rem:equivalence}
	Observe from the previous systems of real PDEs that if $\ui = \vect{0}$ and $\pim = 0$ in $\Omega$, then we have $\ui = \dn{\ui} = 0$ and $\pim = 0$ on $\Sigma$, and thus $\ur \cdot \nn = 0$ on $\Sigma$.
	From \eqref{eq:ccbm_real_part} and \eqref{eq:ccbm_imaginary_part}, we see that the pair $(\Omega,\ur)$ solves the original free boundary problem \eqref{eq:FSP}. 
	Conversely, if $(\Omega,\ur)$ is the solution to the free boundary problem \eqref{eq:FSP}, then clearly $\ur$ and $\ui$ satisfy \eqref{eq:ccbm_real_part} and \eqref{eq:ccbm_imaginary_part}.
\end{remark}
%
%
%
	We infer from the previous remark that the original free boundary problem \eqref{eq:FSP} can be recast into an equivalent shape optimization problem given as follows.
\begin{problem}
\label{eq:main_problem}
	Given a fixed interior boundary $\Gamma$ and a function $\ff$, find an annular domain $\Omega$, with the exterior boundary denoted by $\Sigma:=\partial \Omega\setminus \Gamma$, and a function $\uu := \uu(\Omega)$ such that $\ui = 0$ in $\Omega$ and $\uu = \ur + i\ui$ solves the PDE system \eqref{eq:ccbm}.
\end{problem}

Unless otherwise specified, we assume throughout the paper that $\Omega$ is a bounded, non-empty, connected subset of $\mathbb{R}^{d}$, and of class (at least) $\mathcal{C}^{1,1}$.
\subsection{Well-posedness of the state problem}\label{subsec:well-posedness_of_state_problem}
We discuss here the well-posedness of the PDE system \eqref{eq:ccbm}.
To this end, for simplicity, we carry out the analysis on the basis of homogenous Dirichlet boundary conditions on fixed boundary $\Gamma$.
Extension to non-homogeneous Dirichlet boundary conditions can be accomplished by standard techniques. 

To proceed, we first introduce some notations.

\textbf{Notations.} Let us start with the notations for some operators and operations.
We recall, for clarity, the following standard operators/operations
\[
	\nabla := \begin{pmatrix} \dfrac{\partial}{\partial x_{1}} \\ \vdots \\ \dfrac{\partial}{\partial x_{d}} \end{pmatrix}
	\quad\text{and}\quad
	\frac{\partial}{\partial \nn} 
	= \nn \cdot \nabla
	= \begin{pmatrix} n_{1} \\ \vdots \\ n_{d} \end{pmatrix} 
		\cdot
		\begin{pmatrix} \dfrac{\partial}{\partial x_{1}} \\ \vdots \\ \dfrac{\partial}{\partial x_{d}} \end{pmatrix}
	= \nn^{\top} \nabla	
	= \sum_{i=1}^{d} n_{i} \dfrac{\partial}{\partial x_{i}},
\]
where $\nn = (n_{1}, \ldots, n_{d})^{\top}$ is the outward unit normal vector.
Also, we denote the inner product\footnote{Another notation for the inner product with different usage will also be introduced shortly in the succeeding discussion.} of two (column) vectors $\vect{a}, \vect{b} \in \mathbb{R}^{d}$ in $\mathbb{R}^{d}$ by $\vect{a} \cdot \vect{b} := \vect{a}^{\top} \vect{b} \equiv \langle \vect{a}, \vect{b} \rangle_{\mathbb{R}^{d}} = \langle \vect{a}, \vect{b} \rangle$, respectively (where the latter notation is used when there is no confusion).

For a vector-valued function $\uu := (u_{1}, u_{2}, \ldots, u_{d})^{\top} : \Omega \to \mathbb{R}^{d}$, the gradient of $\uu$, denoted by $\nabla \uu$, 
a second-order tensor defined as $\nabla \uu = (\nabla \uu)_{ij} := (\partial u_{j}/\partial x_{i})_{i,j = 1, \ldots, d}$, where $(\nabla \uu)_{ij}$ is the entry at the $i$th row and $j$th column.
Meanwhile, the Jacobian of $\uu$, denoted by $D\uu$ (the total derivative of $\uu$), is the transpose of the gradient (i.e., $D\uu = (D\uu)_{ij} = (\partial u_{i}/\partial x_{j})_{i,j=1,\ldots,d} = \nabla^{\top} \uu$). 
Accordingly, we write the $\nn$-directional derivative of $\uu$ as $\partial_{\nn} \uu := (D\uu)\nn$.
Other notations are standard and will only be emphasized for clarity.

The notations for the function spaces used in the paper are as follows.
We denote by $W^{m,p}(\Omega)$ the standard real Sobolev space with the norm $\|\cdot\|_{W^{m,p}(\Omega)}$.
Let $W^{0,p}(\Omega)=L^p(\Omega)$, and particularly, we write $H^m(\Omega)$ to represent the space
\[
	W^{m,2}(\Omega) = \{ u \in L^{2}(\Omega) \mid D^{\beta} u \in L^{2}(\Omega) , \ \text{for $0 \leqslant |\beta| \leqslant m $}\},
\]
where $D^{\beta}$ is the weak (or distributional) partial derivative operator and $\beta$ is a multi-index.
Its corresponding inner product is given by $( \cdot , \cdot)_{m,\Omega}$ and norm 
\[
	\|\cdot\|_{H^m(\Omega)} = \sum_{|\beta| \leqslant m} \intO{ |D^{\beta} u|^{2} }.
\]

For vector-valued functions, we define the Sobolev space $H^m(\Omega)^{d}$ as follows:
\[
	H^m(\Omega)^{d} := \{\uu = (u_{1}, u_{2}, \ldots, u_{d})^{\top} \mid u_{i} \in H^{m}(\Omega), \ \text{for $i = 1, 2, \ldots, d$} \}.
\]
Its associated norm is given by
\[
	\|\uu\|_{H^{m}(\Omega)^{d}} = \left( \sum_{i=1}^{d} \|u_{i}\|^{2}_{H^{m}(\Omega)} \right)^{1/2}.
\]
Similar definition is given when $\Omega$ is replaced by the boundary $\partial\Omega$.

We let $\HH^m(\Omega)^{d}$ be the complex version of $H^m(\Omega)^{d}$ with the inner product $(\!(\cdot, \cdot)\!)_{m,\Omega,d}$ and norm $\vertiii{\, \cdot\, }_{m,\Omega,d}$ defined respectively as follows: 
\[
\text{for all $\uu, \vv \in \HH^m(\Omega)^{d}$}, \quad (\!( \uu, \vv )\!)_{m,\Omega,d} = \sum_{j=1}^{d}(\uu, \cv)_{m,\Omega}\quad\text{and}\quad
\vertiii{\vv}_{m,\Omega,d} = \sqrt{ (\!( \vv,\vv )\!)_{m,\Omega, d}}.
\]
%
%
%
Also, for ease of writing, we will use the following notations in the paper:
\begin{align*}
	X &:= \HH^{1}(\Omega)^{d}, \qquad\quad {\Vgamma} := \HH^{1}_{\Gamma,\vect{0}}(\Omega)^{d}, \qquad\quad  {\oV} := \HH^{1}_{\vect{0}}(\Omega)^{d},\\
	Q &:= \LL^2(\Omega), \qquad\quad \oQ = \left\{ q \in Q \mid \intO{q} = 0\right\},\\
	\oVzero &:= \left\{ \vv \in \oV \mid (\nabla \cdot \vv, q) = 0, \ \forall q \in \oQ \right\} \qquad  (\nabla \cdot \vv, q)  := (\nabla \cdot \vv, q)_{\Omega}  := \intO{\bar{q} \nabla \cdot \vv}\\
			&\ = \left\{ \vv \in \oV \mid (\nabla \cdot \vv, q) = 0, \ \forall q \in Q \right\} \\
			&\ = \left\{ \vv \in \oV \mid \nabla \cdot \vv = 0 \right\}, \\
	\oVperp & := \left\{ \vv \in \oV \mid ( \vv, \ww )_{X} = 0, \ \forall \ww \in \oV \right\}, 
			\qquad (\vv,\ww)_{X} := \intO{ ( \nabla \vv : \nabla \cw + \vv \cdot \cw) }.
\end{align*}
In above, the operation `$:$' stands for the tensor scalar product and is defined as
\[
	\nabla \vv : \nabla \cw = \left( \sum_{i,j=1}^{d} \frac{\partial {\vv}_{j}}{\partial {x}_{i}} \frac{\partial {\cw}_{j}}{\partial {x}_{i}} \right) \in \mathbb{R}.
\]
Because $\oVzero$ is a closed subspace of $\oV$, we have the decomposition $\oV = \oVzero \oplus \oVperp$.
Also, with the above definitions, we sometimes write
\[
	\vertiii{\vv}_{X} = \vertiii{\vv}_{1,\Omega,d} 
	\quad \text{and} \quad
	\vertiii{q}_{Q} = \vertiii{q}_{0,\Omega}, 
\]
and drop $d$ when there is no confusion.

Lastly, throughout the paper, $c$ will denote a generic positive constant which may have a different value at different places.
Also, we occasionaly use the symbol ``$\lesssim$'' which means that if $x \lesssim y$, then we can find some constant $c > 0$ such that $x \leqslant c y$.
Of course, $y \gtrsim x$ is defined as $x \lesssim y$.
Other notations are standard and will only be emphasized for clarity.

Next we quote the following lemma which is essential in our argumentation.
\begin{lemma}\label{lem:for_inf_sup_lemma}
	Let $\Omega$ be a connected set and consider the following maps 
	\begin{align*}
		\mathcal{T} &: \oV \longmapsto \oQ, \qquad \mathcal{T} \bphi := -\nabla \cdot \bphi, \quad (\bphi \in \oV)\\
		\mathcal{T}_{\perp} &: \oVperp \longmapsto \oQ, \qquad \mathcal{T}_{\perp} \bphi := -\nabla \cdot \bphi, \quad (\bphi \in \oVperp).
	\end{align*}
	Then, the following results holds: $\mathcal{T}$ is surjective, $\mathcal{T}_{\perp} \in \mathscr{L}(\oV,\oQ)$, $\mathcal{T}_{\perp}$ is bijective, and $\mathcal{T}_{\perp} \in \operatorname{Isom}(\oVperp,\oQ)$, i.e., there exists $\mathcal{T}_{\perp}^{-1} : \oQ \to \oVperp$, $\mathcal{T}_{\perp}^{-1} \in \mathscr{B}:=\mathscr{L}(\oQ,\oVperp)$. 
\end{lemma}
\begin{proof}
	The proof of the lemma follows similar arguments as in the real case (see, e.g.,\cite[Chap. 1.2]{GiraultRaviart1986}), so we omit it.
\end{proof}
Now we exhibit the well-posedness of the complex PDE system \eqref{eq:ccbm}.
On this purpose, we introduce the following forms: 
\begin{equation}\label{eq:forms_for_the_state_problem}
\left\{
	\begin{aligned}
		\aaa({\bphi},{\bpsi}) &= \intO{\alpha \nabla {\bphi} : \nabla {\cbpsi}} + i \intS{({\bphi} \cdot \nn)({\cbpsi} \cdot \nn)}, \quad \forall {\bphi}, {\bpsi} \in {\Vgamma},\\
		b(\bphi,\lambda) &= -\intO{\bar{\lambda} ( \nabla \cdot \bphi) }, \quad \forall {\bphi}\in {\Vgamma}, \ \forall \lambda \in {Q}, \\
		F(\bpsi) &= \intO{\ff \cdot \cbpsi},  \quad \forall {\bpsi}\in {\Vgamma}.
 	\end{aligned}
\right.
\end{equation}
Hence, we can state the weak formulation of \eqref{eq:ccbm} as follows:
find $({\uu},p) \in \Vgamma \times Q$ such that
    \begin{equation}\label{eq:ccbm_weak_form}
    \left\{\arraycolsep=1.4pt\def\arraystretch{1.1}
    \begin{array}{rcll}
    	\aaa({\uu},{\bphi}) + b(\bphi,p)	&=& F(\bphi),		&\quad\forall {\bphi}\in {\Vgamma},\\ 
    				b(\uu,\lambda)	&=& 0,			&\quad\forall \lambda \in {Q}.\\ 
    \end{array}
    \right.
    \end{equation}
Our first proposition is given next.
\begin{proposition}
	For a given $\ff \in L^{2}(\Omega)^{d}$, the weak formulation of \eqref{eq:ccbm} given above admits a unique solution $(\uu,p) \in {\Vgamma} \times Q$ which depends continuously on the data.
	Moreover, we have
	%
	%
	\[
		\vertiii{\uu}_{X} , \ \vertiii{p}_{Q} \lesssim \vertiii{\ff}_{0,\Omega}.
	\] 
\end{proposition}
%

%
The proof of the proposition is based on standard arguments as in the real case.
Indeed, the validity of the statement follows from the continuity of the sesquilinear form $a(\cdot,\cdot)$ on $X\times X$, its coercivity in $X$, i.e., to show that there exists a constant $c_{a}>0$ such that
\begin{equation}\label{eq:coercivity}
	\Re(a(\bphi,\bphi)) \geqslant c_{a} \vertiii{\bphi}_{X}, \quad \forall \bphi \in {\Vgamma},
\end{equation}
the continuity of $F$ in $X$, and on an inf-sup condition.
For the last condition, it needs to be shown that the bilinear form $b$ satisfies the condition that there is a constant $\beta_{0} > 0$ such that
\begin{equation}
\label{eq:inf_sup_condition}
	\inf_{\substack{{\lambda} \in {Q}\\ {\lambda}\neq0}} \sup_{\substack{\bphi \in {\Vgamma}\\ \bphi \neq \vect{0}}} \frac{b(\bphi,{\lambda})}{\vertiii{\bphi}_{X} \vertiii{{\lambda}}_{Q} } \geqslant \beta_{0}.
\end{equation}

Let us confirm the above results for completeness.
Let $\ff \in L^{2}(\Omega)^{d}$ and $\bphi, \bpsi \in \Vgamma$, ${\lambda} \in {Q} $.

\medskip 
\noindent \textbf{Continuity} We have the following computations
	\begin{align*}
		|a(\bphi,\bpsi)| &\leqslant c \left( \vertiii{\nabla \bphi}_{0,\Omega} \vertiii{\nabla \cbpsi}_{0,\Omega} + \vertiii{\bphi}_{0,\Sigma} \vertiii{\cbpsi}_{0,\Sigma}\right)
			%
			%
			\leqslant c \vertiii{\bphi}_{1,\Omega} \vertiii{\bpsi}_{1,\Omega},\\
		|F(\bphi)| &= \left| \intO{\ff \cdot \cvarphi} \right| 
			\leqslant \vertiii{\ff}_{0,\Omega} \vertiii{\cbphi}_{0,\Omega}
			\leqslant \vertiii{\ff}_{0,\Omega} \vertiii{\bphi}_{1,\Omega},\\
		|b(\bphi,{\lambda})| &= \left| \intO{\overline{{\lambda}} \cdot \nabla \cdot \bphi} \right| 
			\leqslant \vertiii{{\lambda}}_{Q} \vertiii{ \nabla \cdot \bphi }_{0,\Omega}
			\leqslant c \vertiii{{\lambda}}_{Q} \vertiii{ \nabla \bphi }_{0,\Omega}
			\leqslant c \vertiii{{\lambda}}_{Q} \vertiii{ \bphi }_{1,\Omega}.		
	\end{align*}
\textbf{Coercivity} We also have the inequality condition
	\[
		\Re(a(\bphi,\bphi)) = \alpha \vertiii{\nabla \bphi}_{0,\Omega}^{2} \geqslant c_{a} \vertiii{\bphi}_{1, \Omega}^{2},
	\]
	for some constant $c_{a} > 0$, which confirms \eqref{eq:coercivity}.

\medskip	
\noindent \textbf{Inf-Sup condition} Next, we want to show that we can find $\beta_{0} > 0$ such that
	\begin{equation}\label{eq:inf_sup_nts}
	\sup_{\substack{\bphi \in {\Vgamma}\\ \bphi \neq \vect{0}}} \frac{b(\bphi,{\lambda})}{\vertiii{\bphi}_{X} } \geqslant \beta_{0} \vertiii{{\lambda}}_{Q},\quad \forall {\lambda} \in Q.
	\end{equation}
	Let ${\lambda} \in Q$ be fixed arbitrarily.
	We define ${\lambda} = \olambda + {\lambda}_{\ast}$ with $\olambda \in \oQ$ and ${\lambda}_{\ast}:= \dfrac{1}{|\Omega|}\displaystyle\intO{{\lambda}} \in \mathbb{C}$.
	By Lemma \ref{lem:for_inf_sup_lemma}, there exists $\obphi \in \oVperp$ such that $- \nabla \cdot \obphi = \olambda$ ($= \mathcal{T}_{\perp} \obphi$).
	We let $\widetilde{\bphi}_{\ast}$ be a fixed function in $\Vgamma$ such that $\intS{\widetilde{\bphi}_{\ast} \cdot \nn} \neq 0$ and define
	\begin{equation}\label{eq:definition_of_w}
		{\bphi}_{\ast} := \widetilde{\bphi}_{\ast}\left( \intS{\widetilde{\bphi}_{\ast} \cdot \nn}\right)^{-1} \in \Vgamma
		\quad \text{satisfying} \quad \intS{{\bphi}_{\ast} \cdot \nn} = -1.
	\end{equation}
	With the above function, we consider $\bphi = \obphi + t_{0} {\lambda}_{\ast} \widetilde{\bphi}_{\ast} \in \Vgamma$ where $t_{0} > 0$.
	This leads to the following inequality
	\begin{equation}\label{eq:supremum_estimate}
		\sup_{\substack{\bphi \in {\Vgamma}\\ \bphi \neq \vect{0}}} \frac{b(\bphi,{\lambda})}{\vertiii{\bphi}_{X} } 
			\geqslant \frac{b(\obphi + t_{0} {\lambda}_{\ast} \widetilde{\bphi}_{\ast} , \olambda + {\lambda}_{\ast})}{\vertiii{\obphi + t_{0} {\lambda}_{\ast} \widetilde{\bphi}_{\ast} }_{X} } 
				=: \frac{N_{1}}{N_{2}}.
	\end{equation}
	For the numerator $N_{1}$, we have the following computations
	\begin{align*}
		N_{1} 
		&= -(\obphi + t_{0} {\lambda}_{\ast} \widetilde{\bphi}_{\ast} , \olambda + {\lambda}_{\ast})\\
		&= -(\nabla \cdot \obphi, \olambda) 
			- (\nabla \cdot \obphi, {\lambda}_{\ast}) 
			- t_{0} {\lambda}_{\ast} (\nabla \cdot \widetilde{\bphi}_{\ast}, \olambda) 
			- t_{0} {\lambda}_{\ast} (\nabla \cdot \widetilde{\bphi}_{\ast}, {\lambda}_{\ast})\\
		&\stackrel{\langle 1\rangle}{=} \vertiii{\olambda}_{Q}^{2}
			- t_{0} {\lambda}_{\ast} (\nabla \cdot \widetilde{\bphi}_{\ast}, \olambda) 
			- t_{0} {\lambda}_{\ast} (\nabla \cdot \widetilde{\bphi}_{\ast}, {\lambda}_{\ast})\\		
		&\stackrel{\langle 2\rangle}{\geqslant} \vertiii{\olambda}_{Q}^{2}
			- t_{0} |{\lambda}_{\ast}| \vertiii{\widetilde{\bphi}_{\ast}}_{1,\Omega} \vertiii{\olambda}_{Q}
			+ t_{0} |{\lambda}_{\ast}|^{2}\\	
		&\stackrel{\langle 3\rangle}{=} \vertiii{\olambda}_{Q}^{2}
			- c_{0} t_{0} |{\lambda}_{\ast}| \vertiii{\olambda}_{Q}
			+ t_{0} |{\lambda}_{\ast}|^{2}, \qquad (\mathbb{R}^{+} \ni c_{0}=\vertiii{\widetilde{\bphi}_{\ast}}_{1,\Omega})\\		
		&\stackrel{\langle 4\rangle}{\geqslant} \left( 1 - \frac{c_{0} t_{0}}{2\varepsilon_{0}}\right) \vertiii{\olambda}_{Q}^{2}
			+ t_{0} \left( 1 - \frac{\varepsilon_{0} c_{0}}{2}\right) |{\lambda}_{\ast}|^{2}, \qquad (\varepsilon_{0} > 0),\\	
		&\geqslant
			\min\left\{  \left( 1 - \frac{c_{0} t_{0}}{2\varepsilon_{0}}\right), t_{0} \left( 1 - \frac{\varepsilon_{0} c_{0}}{2}\right) \right\} \left( \vertiii{\olambda}_{Q}^{2} +  |{\lambda}_{\ast}|^{2} \right)\\
		&=: c_{1}(c_{0}, t_{0}, \varepsilon) h({\lambda}) =: c_{1} h({\lambda}).
	\end{align*}
	Here, $\langle 1\rangle$ is obtained from the fact that $\obphi \in \oVzero$ and the assumption that $-\nabla \cdot \obphi = \olambda$, $\langle 2\rangle$ follows from Green's theorem, together with \eqref{eq:definition_of_w}, $\langle 3\rangle$ is due to the assumption that $\widetilde{\bphi}_{\ast}$ is fixed, so $\vertiii{\widetilde{\bphi}_{\ast}}_{1,\Omega}$ equates to some constant $c_{0} > 0$, while inequality $\langle 4\rangle$ is obtained because of Peter-Paul inequality\footnote{Here, of course, $\varepsilon_{0} > 0$ is chosen such that $c_{1}>0$. For example, this condition holds if we choose $\varepsilon_{0} = c_{0}^{-1} > 0$ and $t_{0} = \varepsilon c_{0}^{-1} = c_{0}^{-2}>0$.} applied to the product $|{\lambda}_{\ast}| \vertiii{\olambda}_{Q}$.
	
	Let us further estimate below the sum $h({\lambda})$.
	Note that we have the following calculations
	\begin{align*}
		\vertiii{{\lambda}}_{Q}^{2}
			= \vertiii{\olambda + {\lambda}_{\ast}}_{Q}^{2}
			&= \intO{ (\olambda + {\lambda}_{\ast})^{2}}
			= \intO{ \left( |\olambda|^{2} + 2 \olambda {\lambda}_{\ast} + |{\lambda}_{\ast}|^{2} \right)},\\
			&= \vertiii{\olambda}_{Q}^{2} + |{\lambda}_{\ast}|^{2} |\Omega|, \qquad (\olambda \in \oQ, \ {\lambda}_{\ast} \in \mathbb{C})\\
			&\begin{cases} 
                             \ \leqslant & \max\{1,|\Omega|\} h({\lambda}) =: c_{2}^{-1} h({\lambda}) \\ 
                             \ \geqslant &\min\{1,|\Omega|\} h({\lambda}) =: c_{3}^{-1} h({\lambda}).
                           \end{cases}
	\end{align*}
	Thus, there actually exist constants $c_{2}, c_{3} > 0$ such that the following inequalities hold
	\[
		c_{2} \vertiii{{\lambda}}_{Q}^{2} \leqslant h({\lambda}) \leqslant c_{3} \vertiii{{\lambda}}_{Q}^{2}.
	\]
	By this estimate, we get
	\begin{equation}\label{eq:numerator_estimate}
		b(\obphi + t_{0} {\lambda}_{\ast} \widetilde{\bphi}_{\ast} , \olambda + {\lambda}_{\ast}) 
		\geqslant c_{1} c_{2} \vertiii{{\lambda}}_{Q}^{2},
	\end{equation}
	for the numerator in \eqref{eq:supremum_estimate}.
	
	For the denominator $N_{2}$ in \eqref{eq:supremum_estimate}, we have the following estimations
	\begin{align*}
	N_{2}
	&= \vertiii{\obphi}_{X} + t_{0} |{\lambda}_{\ast}| \vertiii{\widetilde{\bphi}_{\ast} }_{X}\\
	&= \vertiii{\mathcal{T}_{\perp}^{-1} \olambda}_{1,\Omega} + t_{0} c_{0} |{\lambda}_{\ast}| \\
	&= \left\|\mathcal{T}_{\perp}^{-1}\right\|_{\mathscr{B}} \vertiii{\olambda}_{Q} + t_{0} c_{0} |{\lambda}_{\ast}| \\
	&\leqslant \max\left\{ \left\|\mathcal{T}_{\perp}^{-1}\right\|_{\mathscr{B}} , t_{0} c_{0} \right\} \left( \vertiii{\olambda}_{Q} +  |{\lambda}_{\ast}| \right) \\
	&\leqslant \sqrt{2}\max\left\{ \left\|\mathcal{T}_{\perp}^{-1}\right\|_{\mathscr{B}} , t_{0} c_{0} \right\} \left( \vertiii{\olambda}_{Q}^{2} +  |{\lambda}_{\ast}|^{2} \right)^{1/2} \\
		&\leqslant \sqrt{2c_{3}}\max\left\{ \left\|\mathcal{T}_{\perp}^{-1}\right\|_{\mathscr{B}} , t_{0} c_{0} \right\} \vertiii{{\lambda}}_{Q}.
	\end{align*}
	In above, we have used the fact that $\mathcal{T}_{\perp}^{-1}$ is a bounded linear operator.
	Meanwhile, $\left\| \cdot \right\|_{\mathscr{B}}$ denotes the operator norm for linear operators in $\mathscr{B} =\mathscr{L}(\oQ,\oVperp)$.
	Combining the above estimate with \eqref{eq:numerator_estimate}, we finally get
	\[
		\sup_{\substack{\bphi \in {\Vgamma}\\ \bphi \neq \vect{0}}} \frac{b(\bphi,{\lambda})}{\vertiii{\bphi}_{X} } 
		\geqslant \frac{N_{1}}{N_{2}} \geqslant \frac{c_{1} c_{2} \vertiii{{\lambda}}_{Q}^{2}}{\sqrt{2c_{3}}\max\left\{ \left\|\mathcal{T}_{\perp}^{-1}\right\|_{\mathscr{B}} , t_{0} c_{0} \right\}\vertiii{{\lambda}}_{Q}} =: \beta_{0} \vertiii{{\lambda}}_{Q}.
	\]
	This proves \eqref{eq:inf_sup_nts}, and thus \eqref{eq:inf_sup_condition}.
\subsection{The proposed shape optimization formulation} 
To solve Problem \ref{eq:main_problem}, we introduce the cost functional
\begin{equation}\label{eq:cost_function}
	J(\Omega) = \frac12\|\ui\|^2_{L^{2}(\Omega)^{d}} + \frac12\|\pim\|^2_{L^2(\Omega)^{d}} = \frac12 \intO{\left( |\ui|^2 + |\pim|^2 \right)},
\end{equation}
where $(\ui,\pim)$ is subject to the state problem \eqref{eq:ccbm}.
Notice that compared to $J_{KV}$, the cost function $J$ only requires the solution of a single complex PDE problem to be solved.
The optimization problem we consider here is the problem of minimizing $J(\Omega)$ over a set of admissible domains $\mathcal{O}_\text{ad}$, where $\mathcal{O}_\text{ad}$ is essentially the set of $\mathcal{C}^{k,1}$, $k\geqslant1$, $k\in \mathbb{N}$, (non-empty) doubly connected domains with (fixed) interior boundary $\Gamma$ and (free) exterior boundary $\Sigma$. 
In other words, we consider the shape optimization problem that reads as follows: find $\Omega$ such that
\begin{equation}\label{eq:shape_optimization_problem}
	J(\Omega) = \min_{\tilde{\Omega} \in \mathcal{O}_\text{ad} } J(\tilde{\Omega}).
\end{equation}
We note that it is actually enough to consider $\Gamma$ to be only Lipschitz regular to derive the shape derivative of the cost function, but for simplicity we also assume it to be $\mathcal{C}^{k,1}$ regular. 
We emphasize that, in this paper, we will not tackle the question of existence of optimal shape solution to \eqref{eq:shape_optimization_problem}.
Instead, we will tacitly assume the existence of solution to the free surface problem \eqref{eq:FSP} and adopt the shape optimization formulation \eqref{eq:shape_optimization_problem} to resolve \eqref{eq:FSP} numerically.

To numerically solve the optimization problem $J(\Omega) \to \inf$, we will apply a shape-gradient-based descent method based on finite element method (FEM).
We will not, however, employ any kind of adaptive mesh refinement in our numerical scheme as opposed to \cite{RabagoAzegami2019a,RabagoAzegami2019b,RabagoAzegami2020}.
In this way, we can further assess the stability of the new method in comparison with the classical least-squares method of tracking the Dirichlet data.
The expression for the shape derivative of the cost will be exhibited in the next section using \textit{shape calculus} \cite{DelfourZolesio2011,HenrotPierre2018,MuratSimon1976,Simon1980,SokolowskiZolesio1992}.
However, as opposed to \cite{Kasumba2014,BouchonPeichlSayehTouzani2017}, our strategy to obtain the shape gradient expression does not require the knowledge of the shape derivative of the state, nor its material derivative.

Throughout the rest of the paper, we shall refer to our proposed shape optimization method simply by CCBM.
\section{Computation of the shape derivative}\label{sec:Computation_of_the_shape_derivative}
The main purpose of this section is to characterize the first-order shape derivative of the cost function $J$ given in Proposition \ref{prop:the_shape_derivative_of_the_cost}. 

\subsection{Some concepts from shape calculus}
Before we derive the shape gradient of the cost function $J$, we briefly review in this section some important concepts and results from shape calculus.

Consider a (convex) bounded hold-all domain $U$ of class $\mathcal{C}^{k,1}$, $k \geqslant 1$, strictly containing $\overline{\Omega}$ (the closure of $\Omega$) and define $T_{t}$ as the \textit{perturbation of the identity} $id$\footnote{Here, $id$ is also used to denote the identity matrix in $d$-dimension. If there is no confusion, this abuse of notation is used throughout the paper.} given by the map
\begin{equation}\label{eq:poi}
	T_{t} = T_{t}({\VV}) = id + t \VV, \qquad (T_{t} : \overline{U} \longmapsto \mathbb{R}^{d}),
\end{equation}
where $\VV := (\theta_{1}, \theta_{2}, \ldots, \theta_{d})^{\top} \in \mathbb{R}^{d}$ is a $t$-independent deformation field belonging to the admissible space
\begin{equation}
\label{eq:space_for_V}
	\vect{\Theta}^{k}:=\{\VV\in \mathcal{C}^{k,1}(\overline{U})^{d} \mid \VV = \vect{0} \ \text{on} \ \Gamma \cup \partial U\},\footnote{Here, and throughout the paper, $\mathcal{C}^{k,1}(\,\cdot\,)^{d} := \mathcal{C}^{k,1}(\,\cdot\,; \mathbb{R}^{d})$. Similarly, $\mathcal{C}^{k,1}(\,\cdot\,)^{d\times d} := \mathcal{C}^{k,1}(\,\cdot\,; \mathbb{R}^{d \times d})$}
\end{equation}
and $k \in \mathbb{N}$, later on specified as what is needed.
Hereinafter, $t$ is assumed sufficiently small such that $T_{t}$ is a diffeomorphism from $\Omega \in \mathcal{C}^{k,1}$ onto its image.\footnote{In fact, $T_{t} : U \longmapsto U$ is a $\mathcal{C}^{k,1}$ diffeomorphism for sufficiently small $t$.}
In other words, we suppose that the reference domain $\Omega$ and its perturbation $\Omega_{t}$ have the same topological structure and regularity under the transformation $T_{t}$.
On this purpose, particularly for $k=1$ (thus, $\VV \in \vect{\Theta}^{1}$), we let $\varepsilon \in \mathbb{R}^{+}$ be sufficiently small such that $[t \mapsto T_{t}] \in \mathcal{C}^{1}(\mathcal{I},\mathcal{C}^{1,1}(\overline{U})^{d})$, where $\mathcal{I} := [0,\varepsilon]$.
Throughout the paper, when not specified, it is always assumed that (at least) $k=1$ so that $T_{t}$ becomes a $\mathcal{C}^{1,1}$ diffeomorphism.
 
By definition, we also have the perturbation $\Sigma_{t}:=T_{t}(\Sigma)$ and $\Gamma_{t}:=T_{t}(\Gamma) \equiv \Gamma$ where the latter identity is due to the fact that $\VV = \vect{0}$ on $\Gamma$.
In addition, of course, $\Omega_{0}=\Omega$ and $\Sigma_{0}=\Sigma$.
Accordingly, the set of all admissible domains $\mathcal{O}_{ad}$ is given as follows
\begin{equation}\label{eq:admissible_domains}
	\mathcal{O}_{ad} = \left\{T_{t}({\VV})(\overline{\Omega}) \subset U \mid \Omega \in \mathcal{C}^{k,1}, k \in \mathbb{N}, t \in \mathcal{I}, \VV \in \vect{\Theta}^{k} \right\}. 
\end{equation}

The functional $J : \mathcal{O}_{ad} \to \mathbb{R}$ has a directional \textit{first-order} \textit{Eulerian derivative} at $\Omega$ in the direction of the field $\VV$ if the limit
\begin{equation}
\label{eq:limit_J}
\lim_{t \searrow0} \frac{J(\Omega_{t}) - J(\Omega)}{t} =: {d}J(\Omega)[\VV]
\end{equation}
exists (see, e.g., \cite[Sec. 4.3.2, Eq. (3.6), p. 172]{DelfourZolesio2011}). 
The functional $J$ is said to be \textit{shape differentiable} at $\Omega$ in the direction of $\VV$ if the map $\VV \mapsto {d}J(\Omega)[\VV]$ is linear and continuous.
In this case, we refer to the map as the \textit{shape gradient} of $J$.

We introduce a few more notations for ease of writing.
We denote by $DT_{t}$ the Jacobian matrix of $T_{t}$ (i.e., $(DT_{t})_{ij} = \partial (T_{t})_{i}/\partial x_{j}$) and write the inverse and inverse transpose of this matrix by $(DT_{t})^{-1}$ and $(DT_{t})^{-\top} := ((DT_{t})^{\top})^{-1}$, respectively.
For convenience, we define
\[
	I_{t} := \det \, DT_{t},\quad
	A_{t} := I_{t}(DT_{t}^{-1})(DT_{t})^{-\top},\quad
	\text{and} \quad B_{t} := I_{t} \abs{\Mt \nn}, \quad \Mt:=(DT_{t})^{-\top}.
\]
For $t \in \mathcal{I}$, $I_{t}$ is positive.
At $t=0$, it is evident that $I_{0} = 1$, $A_{0} = id$ (because $DT_{0}^{-1} = id$ and $(DT_{0})^{-\top} = id$), and $B_{0} = 1$. 
Moreover, on $\mathcal{I}$, the maps $t \mapsto I_t$, $t \mapsto A_t$, and $t \mapsto B_t$ are continuously differentiable.
That is, for $t \in \mathcal{I}$, we have (see, e.g., \cite{IKP2006,IKP2008})
\begin{equation}\label{eq:regular_maps}
\left\{
\begin{aligned}
	[t \mapsto I_t] &\in \mathcal{C}^{1}(\mathcal{I},\mathcal{C}(\overline{\Omega})), \\
	[t \mapsto A_t] &\in \mathcal{C}^{1}(\mathcal{I},\mathcal{C}(\overline{\Omega})^{d \times d}),\\
	[t \mapsto B_t] &\in \mathcal{C}^{1}(\mathcal{I},\mathcal{C}(\Sigma)).
\end{aligned}
\right.
\end{equation}
Furthermore, $[t \mapsto I_t] \in \mathcal{C}^{1}(\mathcal{I},\mathcal{C}^{0,1}(\overline{U}))$ and $[t \mapsto A_t] \in \mathcal{C}(\mathcal{I},\mathcal{C}(\overline{U})^{d \times d})$.

The derivatives of the maps in \eqref{eq:regular_maps} are respectively given as follows:
\begin{equation}\label{eq:limits_of_maps}
\begin{aligned}
	\frac{d}{dt}I_{t} \big\rvert_{t=0}
		&= \lim_{t\to 0} \frac{I_{t} - 1}{t} = \dive \VV,\\
	\quad \frac{d}{dt}A_{t} \big\rvert_{t=0}
		& = \lim_{t\to 0} \frac{A_{t} - {id}}{t}
		= (\dive \VV){id} -  {D} \VV - ({D} \VV)^\top =: A,
	\\\quad\frac{d}{dt}B_{t} \big\rvert_{t=0}
		& =  \lim_{t\to 0} \frac{B_{t} - 1}{t}
		= \dive_{\Sigma} \VV
		= \dive \VV \big\rvert_{\Sigma} - ({D} \VV\nn)\cdot\nn,
\end{aligned}
\end{equation}
where ${\operatorname{div}}_{\Sigma} \VV$ denotes the tangential divergence of the vector $\VV$ on $\Sigma$.
Meanwhile, for simplicity we write $\Vn$ instead of $\VV \cdot \nn$. 

Additionally to the above results, we specifically point out that $(d/dt) DT_{t} \big\rvert_{t=0} = D\VV \in \mathcal{C}^{0,1}(\overline{U})^{d\times d}$ which implies that $[t \mapsto DT_{t}] \in \mathcal{C}^{1}(\mathcal{I},\mathcal{C}^{0,1}(\overline{U})^{d\times d})$.
On the other hand, we have $[t \mapsto T_{t}^{-1}] \in \mathcal{C}(\mathcal{I},\mathcal{C}^{1}(\overline{U})^{d})$ from which we see that $[t \mapsto (DT_{t})^{-\top}] \in \mathcal{C}^{1}(\mathcal{I},\mathcal{C}(\overline{U})^{d\times d})$.

We also assume that for $t \in \mathcal{I}$, we have
\begin{equation}\label{eq:bounds_At_and_Bt}
	0 < \Lambda_{1} \leqslant I_{t} \leqslant \Lambda_{2}
	\qquad \text{and} \qquad 0 < \Lambda_{3}\abs{\xi}^{2} \leqslant A_{t}\xi \cdot \xi \leqslant \Lambda_{4}\abs{\xi}^{2},
\end{equation}
for all $\xi \in \mathbb{R}^{d}$, for some constants $\Lambda_{1}$, $\Lambda_{2}$, $\Lambda_{3}$, and $\Lambda_{4}$ ($\Lambda_{1} < \Lambda_{2}$, $\Lambda_{3} < \Lambda_{4}$).

Furthermore, as we intend to refer any function $\bphi_{t} : \Omega_{t} \to \mathbb{R}^{d}$ to the reference domain using the transformation $T_{t}$, we shall make use of the notation $\bphi^{t}$ which is essentially the composite function $\bphi_{t} \circ T_{t}:\Omega \to \mathbb{R}^{d}$.

Lastly, to close the section we list some properties of the Jacobian of the map $T_{t}$ and state an auxiliary result that will be helpful with our investigation.
Their proofs are provided in Appendix \ref{appxsec:proofs}.

\begin{lemma}\label{lem:properties_of_the_Jacobian}
	Let $\Omega$ and $U$ be nonempty bounded open connected subsets of $\mathbb{R}^{d}$ (defined as before).
	Also, let $T_{t}$ be given by \eqref{eq:poi} and $\VV \in \vect{\Theta}^{k}$, $k \geqslant 1$.
	Then, we have the following results:
	\begin{itemize}
		\item $DT_{t}$ and $DT_{t}^{-1}$ are Lipschitz continuous in $U$.
		\item there exists some constant $C>0$ such that $\abs{DT_{t}^{-1}(x)}_{\infty} < C < \infty$, for all $x \in \overline{U}$.
		\item $\lim_{t \to 0} \dfrac{1}{t}\left( DT_{t} - I\right) = D\VV$ and $\lim_{t \to 0} \dfrac{1}{t}\left[ (DT_{t})^{-1} - I\right] = -D\VV$.
	\end{itemize}
\end{lemma}
\begin{lemma}\label{lem:convergence_of_functions}
		Let $\VV \in \vect{\Theta}^{k}$, $k \geqslant 1$. Then, we have the following
		\begin{enumerate}
			\item[(i)] for any $f \in L^{p}(U)$, $p \geqslant 2$, we have
			\[
				\lim_{t \to 0} \|f \circ T_{t} - f\|_{L^{p}(U)} = 0;
			\]
			\item[(ii)] for any $f \in W^{1,p}(U)$, $p \geqslant 2$, we have
			\[
				\lim_{t \to 0} \|f \circ T_{t} - f\|_{W^{1,p}(U)} = 0.
			\]
			\item[(iii)] for $f \in W^{2,p}(U)$, $p \geqslant 1$, the mapping $t \mapsto f \circ T_{t}$ from $\mathcal{I} \to W^{1,p}(\Omega)$ is differentiable at $t=0$ and the derivative is given by
			\[
				\lim_{t \to 0} \frac{1}{t} \left( f\circ T_{t} - f\right) = Df\VV.
			\]
			\item[(iv)] for $f \in H^{1}(U)$. Then, the map $t \to I_{t} f \circ T_{t}$ from $\mathcal{I}$ to $L^{2}(\Omega)$ is differentiable at $t=0$ and the derivative is given by
			\[
				\lim_{t \to 0} \frac{1}{t} \left( I_{t}f\circ T_{t} - f\right) = \operatorname{div}(f\VV).
			\]
		\end{enumerate}
\end{lemma}
For $f \in L^{p}(U)$, $p \geqslant 1$, the limit $\lim_{t\to0} f \circ T_{t}  = f$ in $L^{p}(\Omega)$,
where $\overline{\Omega} \subset U$, also holds; see \cite[Lem. 3.4]{IKP2006}.

Here we note that similar results to the previous lemma for the case of vector-valued functions also hold.
For instance, for $\bphi = (\varphi_{1},\varphi_{2},\ldots,\varphi_{d})^{\top}\in H^{2}(U)^{d}$, we have
\begin{equation}\label{eq:convergence_of_vector_valued_functions_1}
\begin{aligned}
	\lim_{t \to 0} \|\bphi \circ T_{t} - \bphi\|_{H^{1}(U)^{d}} 
	&= \lim_{t \to 0} \left( \sum_{i=1}^{d} \|\varphi_{i} \circ T_{t} - \varphi_{i} \|_{H^{1}(U)^{d}} \right)^{1/2}\\
	&= \left( \sum_{i=1}^{d} \lim_{t \to 0} \|\varphi_{i} \circ T_{t} - \varphi_{i} \|_{H^{1}(U)^{d}} \right)^{1/2}
	= 0.
\end{aligned}
\end{equation}
Moreover, the map $t \mapsto \bphi \circ T_{t}$ from $\mathcal{I} \to H^{1}(\Omega)^{d}$ is differentiable at $t=0$ and we have
\begin{equation}\label{eq:convergence_of_vector_valued_functions_2}
				\lim_{t \to 0} \frac{1}{t} \left( \bphi\circ T_{t} - \bphi\right) = D\bphi\VV.
\end{equation}
Furthermore, the mapping $t \to I_{t} \bphi \circ T_{t}$ from $\mathcal{I}$ to $L^{2}(\Omega)^{d}$ is differentiable at $t=0$ and the following limit holds
\begin{equation}\label{eq:convergence_of_vector_valued_functions_3}
				\lim_{t \to 0} \frac{1}{t} \left( I_{t}\bphi\circ T_{t} - \bphi\right) = \vect{\nabla} \cdot(\bphi \otimes \VV),
\end{equation}
where
\begin{equation}\label{eq:vector_type_divergence}
	\vect{\nabla} \cdot(\bphi \otimes \VV) = ( \operatorname{div}(\varphi_{1}\VV), \operatorname{div}(\varphi_{2}\VV), \ldots, \operatorname{div}(\varphi_{d}\VV) )^{\top}.
\end{equation}
and $\otimes$ denotes the outer product (i.e., $\bphi \otimes \VV$ gives us the matrix $\bphi \VV^{\top} = (\varphi_{j} \theta_{k})_{jk}$ where $j,k = 1, 2, \ldots, d$).

The proof of the above results follows the same argumentation in validating Lemma \ref{lem:convergence_of_functions} which is provided in Appendix \ref{appxsubsec:proof_of_convergence_of_functions}.
\subsection{Some identities from trangential shape calculus}
\label{subsec:some_identities_from_tangential_shape_calculus}

As we will see in the next section, the expression for the shape gradient of $J$ contains expressions from \textit{tangential shape calculus} (see, e.g., \cite[Chap. 9, Sec. 5]{DelfourZolesio2011} and see also \cite[Sec. 5.4.3, pp. 216--221]{HenrotPierre2018} for a detailed discussion).
So, to facilitate our investigation, we summarize below some notations that will appear in our subsequent discussions.
\begin{definition}[{\cite[Chap. 9, Sec. 5.2, eq. (5.17) -- (5.19), p. 497]{DelfourZolesio2011}}]\label{def:tangential_operators}
	For any domain $\Omega \subset \mathbb{R}^{d}$ with $\mathcal{C}^{1,1}$ smooth boundary $\Gamma:=\partial \Omega$, the following expressions are well-defined:
	\begin{itemize}
	\item the \textit{tangential gradient} of $\psi \in \mathcal{C}^{1}(\Gamma)$ is given by 
	\[
		\nabla_{\Gamma} \psi := \nabla \tilde{\psi} \Big|_{\Gamma} - \frac{\partial \tilde{\psi}}{\partial \nn}\nn \in \mathcal{C}^{0}(\Gamma; \mathbb{R}^{d});
	\]
	\item the \textit{tangential Jacobian matrix} of a vector function $\bphi \in \mathcal{C}^{1}(\Gamma; \mathbb{R}^{d})$ is given by
	\[
		D_{\Gamma} \bphi = D \tilde{\bphi} \Big|_{\Gamma} - D\tilde{\bphi}\nn\otimes\nn \in \mathcal{C}^{0}(\Gamma; \mathbb{R}^{d \times d});
	\]
	\item the \textit{tangential divergence} of a vector function $\bphi \in \mathcal{C}^{1}(\Gamma; \mathbb{R}^{d})$ is given by
	\[
		\nabla_{\Gamma} \cdot \bphi := \dive_{\Gamma} \bphi = \dive \tilde{\bphi} \Big|_{\Gamma} - D\tilde{\bphi}\nn\cdot\nn \in \mathcal{C}^{0}(\Gamma).
	\]	
	\end{itemize}
	In above definitions, $\tilde{\psi}$ and $\tilde{\bphi}$ are any corresponding $\mathcal{C}^{1}$ extensions of $\psi$ and $\bphi$ into a neighborhood of $\Gamma$.
\end{definition}
With the above definitions, the formulas given in the next lemma can easily be shown.
\begin{lemma}[{\cite[Chap. 9, Sec. 5.5, eq. (5.26) -- (5.27), p. 498]{DelfourZolesio2011}}]\label{lem:tangential_formulas}
	Consider a $\mathcal{C}^{1,1}$ domain $\Omega$ with boundary $\Gamma := \partial \Omega$.
	Then, for $\psi \in H^{1}(\Gamma)$ and $\bphi \in \mathcal{C}^{1}(\Gamma; \mathbb{R}^{d})$ the following identities hold:
	\begin{itemize}
		\item \textit{tangential divergence formula} (see also \cite[Lem. 2.63, eq. (2.140), p. 91]{SokolowskiZolesio1992}): \[\dive_{\Gamma} (\psi \bphi) = \nabla_{\Gamma} \psi \cdot \bphi + \psi \dive_{\Gamma} \bphi;\]
		\item \textit{tangential Stokes' formula}: \[\displaystyle \intG{\dive_{\Gamma} \bphi} = \intG{\kappa \bphi \cdot \nn};\]
		\item \textit{tangential Green's formula}: \[\displaystyle \intG{\left( \nabla_{\Gamma} \psi \cdot \bphi + \psi \dive_{\Gamma} \bphi \right)} = \intG{\kappa \psi \bphi \cdot \nn}.\]
	\end{itemize}
	As a consequence of the previous identity, we have
	\[
	\intG{\nabla_{\Gamma} \psi \cdot \bphi} = - \intG{\psi \dive_{\Gamma} \bphi}, \qquad \text{whenever $\bphi \cdot \nn = 0$}. 
	\]
\end{lemma}
Another version of the tangential Green's formula (whose proof can be found, for instance, in \cite{MuratSimon1976}) is given in the next lemma (see, e.g., \cite[Chap. 2, eq. (2.144), p. 92]{SokolowskiZolesio1992}).
\begin{lemma}\label{lem:tangential_Greens_formula}
	Let $U$ be a bounded domain of class $\mathcal{C}^{1,1}$ and $\Omega \subset U$ with boundary $\Gamma:=\partial \Omega$.
	Consider the functions $\bphi \in \mathcal{C}^{1,1}(\overline{U};\mathbb{R}^{d})$ and $\psi \in W^{2,1}(U)$.
	Then, the following equation holds
	\[
	\intS{\left( \nabla \psi \cdot \bphi + \psi \, \dive_{\Sigma} \bphi \right) }
		=  \intS{ \left( \dn{\psi}  +\psi\, \dive_{\Sigma} \nn \right) {\bphi \cdot \nn}}.
	\]
\end{lemma}

For a smooth enough function $\bphi$, we denote its second normal derivative as $\dnn{\bphi}$.
Let us denote by $\mathcal{M}_{d\times d}$ the space of the matrix of size $d \times d$.
The tangential differential operators will be represented with the subscript $\cdot\, _{\Sigma}$.
Particularly, for $\bphi \in \vect{W}^{1,1}(\partial \Omega)$ and $\vect{M} \in W^{1,1}(\partial \Omega; \mathcal{M}_{d\times d})$, the following operators are defined on $\Gamma:=\partial \Omega$:
\begin{itemize}
	\item \textit{tangential vectorial gradient operator}: \[\nabla_{\Gamma} \bphi := \nabla \bphi - (\nabla \bphi\nn) \otimes \nn;\footnote{Given a vector $\bphi :=(\varphi_{1}, \ldots, \varphi_{d})^{\top}$, we note the relation $(D_{\Gamma}\bphi)^{\top} = (\nabla_{\Gamma} \varphi_{1}, \ldots, \nabla_{\Gamma} \varphi_{d}) = \nabla_{\Gamma} \bphi$, where $\nabla_{\Gamma} \varphi_{i}$, $i=1,\ldots, d$, is a column vector.}\]
	\item \textit{tangential vectorial divergence operator} \cite[Chap. 9, Sec. 5.2, eq. (5.11), p. 496]{DelfourZolesio2011}: \[\dive_{\Gamma} \vect{M} := \dive \vect{M} - (\nabla \vect{M}\nn)\nn;\]
	\item \textit{Laplace-Beltrami operator} \cite[Chap. 9, Sec. 5.3, eq. (5.12), p. 496]{DelfourZolesio2011}: \[\Delta_{\Gamma} \bphi := \dive_{\Gamma}(\nabla_{\Gamma}\bphi).\]
\end{itemize}
In relation to the last definition, the Laplace-Beltrami operator can be decomposed in the following way (see, e.g., \cite[p. 28]{DelfourZolesio2011} or \cite[Prop. 5.4.12, eq. (5.59), p. 220]{HenrotPierre2018})
\begin{equation}
\label{eq:Laplace_Beltrami_operator_definition}
	\Delta \bphi = \Delta_{\Sigma} \bphi + \kappa D\bphi\nn + D^{2}\bphi\nn \cdot \nn.
\end{equation}

In the next section, we will exhibit the computation of the shape gradient of $J$.
\subsection{Computation of the shape gradient}
\label{subsec:shape_derivative_of_the_cost_using_the_Eulerian_derivatives}
We compute the shape derivative of $J$ rigorously via rearrangement method in the spirit of \cite{IKP2006} given that $\ff \in H^{1}(U)^{d}$, the $\Omega$ of class $\mathcal{C}^{1,1}$, and $\VV \in \sfTheta^{1}$.
This approach has the advantage of bypassing the need to calculate the material derivative of the state.
Moreover, instead of the $\mathcal{C}^{1}$ continuity of the map $t \mapsto u^{t}$ around $0$, we only need the H\"{o}lder continuity of the state $u^{t}$ -- the composition $u_{t} \circ T_{t}$ which is defined on $\Omega$. 
Our main result is given in the following proposition.
%
\begin{proposition}
	\label{prop:the_shape_derivative_of_the_cost}
	Let $\ff \in H^{1}(U)^{d}$, $\Omega$ be of class $\mathcal{C}^{1,1}$, and $\VV \in \sfTheta^{1}$.
	Then, $J$ is shape differentiable and its shape derivative is given by ${d}J(\Omega)[\VV] = \intS{{\ggb} \nn \cdot \VV}$ where\footnote{Here $\nabla_{\Gamma}$ is the tangential gradient operator on $\Gamma$. The intrinsic definition of the operator is given, for instance, in \cite[Chap. 5., Sec. 5.1, p. 492]{DelfourZolesio2011}.}
	\begin{equation}\label{eq:shape_gradient}
	\begin{aligned}
		{\ggb} &= \Im\left\{  \vect{B}[\Vn] \cdot \overline{\vv} \right\} + \frac12 \left( |\ui|^2 + |\pim|^2 \right) \Vn.
	\end{aligned}
	\end{equation}
	In above expression of the shape gradient of $J$, $\vect{B}[\Vn]$ is given by
	\[
	\begin{aligned}
	\vect{B}[\Vn] 
		&= \ff \Vn - \nabla_{\Sigma} (p\Vn) + \dive_{\Sigma}{[\alpha (\nabla_{\Sigma} \uu) \Vn]} + i \dive_{\Sigma} (\Vn \uu)\nn
			+ i ({\uu} \cdot \nn) \nabla_{\Sigma}\Vn \\
		&\qquad - \kappa \left[ - p\nn + i ( {\uu} \cdot \nn ) \nn \right] \Vn,
	\end{aligned}
	\]
	and $\kappa$ denotes the mean curvature of the free boundary $\Sigma$, $\uu = \ur + i \ui$ and $p = p_{r} + i \pim$ is the unique pair of solution to \eqref{eq:ccbm}, and the pair of adjoints $(\vv,q)$, where $\vv = \vr + i \vi$ and $q = q_{r} + i q_{i}$, uniquely solves the adjoint system
    \begin{equation}
    \label{eq:adjoint_system}
    \left\{\arraycolsep=1.4pt\def\arraystretch{1.1}
    \begin{array}{rcll}
    	- \alpha \Delta \vv + \nabla q	&=& \ui			&\quad\text{in $\Omega$},\\
    	-\nabla \cdot \vv			&=& \pim 		&\quad\text{in $\Omega$},\\
    	\vv	 					&=& \vect{0}			&\quad\text{on $\Gamma$},\\
    	-q\nn + \alpha \dn{\vv}		 - i (\vv \cdot \nn)\nn	&=& \vect{0}		&\quad\text{on $\Sigma$}.
    \end{array}
    \right.
    \end{equation}	
\end{proposition}
The proof of the above result relies on several lemma that we first prove below.
To start, we present the weak formulation of \eqref{eq:adjoint_system}.
On this purpose, we introduce the following forms: 
\begin{equation}\label{eq:forms_for_the_adjoint_problems}
\left\{
	\begin{aligned}
		\tilde{\aaa}({\bphi},{\bpsi}) &= \intO{\alpha \nabla {\bphi} : \nabla {\cbpsi}} - i \intS{({\bphi} \cdot \nn)({\cbpsi} \cdot \nn)}, \quad \forall {\bphi}, {\bpsi} \in {\Vgamma},\\
		\tilde{F}(\bpsi) &= \intO{\ui\cdot \cbpsi},  \quad \forall {\bpsi}\in {\Vgamma}.
 	\end{aligned}
\right.
\end{equation}
Hence, we may write the weak formulation of \eqref{eq:adjoint_system} as follows:	
        find $({\vv},q) \in \Vgamma \times Q$ such that
        \begin{equation}\label{eq:adjoint_system_weak_form}
        \left\{\arraycolsep=1.4pt\def\arraystretch{1.1}
        \begin{array}{rcll}
        	\tilde{\aaa}({\vv},{\bpsi}) + b(\bpsi,q)	&=& \tilde{F}(\bpsi),	&\quad\forall {\bpsi}\in {\Vgamma},\\ 
        				b(\vv,\mu)	&=& (\mu, \pim),			&\quad\forall \mu \in {Q}.
        \end{array}
        \right.
        \end{equation}
        \begin{remark}\label{rem:well_posedness_adjoint_problem}
	The well-posedness of the above problem can be verified using similar arguments issued in subsection \ref{subsec:well-posedness_of_state_problem} for the state problem \eqref{eq:ccbm} (provided, of course, \eqref{eq:ccbm_weak_form} admits a unique solution $({\uu},p) \in \Vgamma \times Q$).
	Particularly, one can show that an inf-sup condition (cf. \eqref{eq:inf_sup_condition}) holds for the above variational problem.
	\end{remark}
%
%
%
%
\begin{remark}\label{rem:higher_regularity_of_{t}he_adjoint}
	For $\Omega$ of class $\mathcal{C}^{k+1,1}$, $k$ a non-negative integer, it can be shown that the weak solution $(\uu, p) \in X \times Q$ to the variational problem \eqref{eq:ccbm_weak_form} is also $\HH^{k+2}(\Omega)^{d} \times H^{k+1}(\Omega)$.
	In particular, $\ui \in H^{k+2}(\Omega)^{d}$ and $\pim \in H^{k+1}(\Omega)$.
	Consequently, we find that the weak solution $(\vv,q)$ of problem \eqref{eq:adjoint_system_weak_form} is not only in $\Vgamma \times Q$, but is also an element of $\HH^{k+4}(\Omega)^{d} \times H^{k+3}(\Omega)$. 
\end{remark}
At this juncture, we introduce the following sesquilinear forms $\aat, \taat \in \Vgamma \times \Vgamma \to \mathbb{R}$ and linear forms, $b^{t} : \Vgamma \times Q \to \mathbb{R}$ and $F^{t}, \tilde{F}^{t} : \Vgamma \to \mathbb{R}$ (which are essentially the transformed versions of the forms listed in \eqref{eq:forms_for_the_state_problem} and \eqref{eq:forms_for_the_adjoint_problems}) defined as follows:
\begin{equation}\label{eq:transformed_forms}
\left\{
\begin{aligned}
		\aat({\bphi},{\bpsi}) &= \intO{\alpha A_{t} \nabla {\bphi} : \nabla {\cbpsi}} + i \intS{\frac{B_{t}}{|\Mt\nn|^{2}} ( \Mt^{\top} {\bphi} \cdot \nn)( \Mt^{\top} {\cbpsi} \cdot \nn)},\\
		b^{t}(\bphi,\lambda) &= -\intO{ I_{t} \bar{\lambda} ( \Mt^{\top}: \nabla \bphi) }, \\
		F^{t}(\bphi) &= \intO{I_{t} \ff^{t} \cdot \cbphi}, \qquad \text{where} \ \ff^{t} = \ff_{t} \circ T_{t} \in H^{1}(U)^{d},\\
		\taat({\bphi},{\bpsi}) &= \intO{\alpha A_{t} \nabla {\bphi} : \nabla {\cbpsi}} - i \intS{\frac{B_{t}}{|\Mt\nn|^{2}} ( \Mt^{\top} {\bphi} \cdot \nn)( \Mt^{\top} {\cbpsi} \cdot \nn)},\\
		\tilde{F}^{t}(\bphi) &= \intO{I_{t} \vect{h} \cdot \cbphi}, \qquad \text{where $\vect{h} \in X$ is a given function}.
\end{aligned}
\right.
\end{equation}
The first of the several lemmas that we need is given as follows.
\begin{lemma}\label{lem:boundedness_of_sesquiliner_and_linear_forms}
	For $t \in \mathcal{I}$, the sesquilinear forms $\aat$ and $\taat$ defined on $X \times X$ are bounded and coercive on $X \times X$.
	In addition, the linear forms $F^{t}(\bphi)$ and $\tilde{F}^{t}(\bphi)$ are also bounded.
	Moreover, the bilinear form $b$ satisfies the condition that there is a constant $\beta_{1} > 0$ such that
	\begin{equation}
	\label{eq:inf_sup_condition_for_transformed_problem}
		\inf_{\substack{{\lambda} \in {Q}\\ {\lambda}\neq0}} \sup_{\substack{\bphi \in {\Vgamma}\\ \bphi \neq \vect{0}}} \frac{b^{t}(\bphi,{\lambda})}	{\vertiii{\bphi}_{X} \vertiii{{\lambda}}_{Q} } \geqslant \beta_{1}.
	\end{equation}

\end{lemma}
\begin{proof}
	The proof is similar to the arguments used in the discussion issued in subsection \ref{subsec:well-posedness_of_state_problem}.
	This time, however, one has to take into account the properties of $A_{t}$ and $B_{t}$ given in \eqref{eq:regular_maps} and the bounds given in \eqref{eq:bounds_At_and_Bt} to prove the given results.
\end{proof}
The next lemma is concerned about the well-posedness of the transported perturbed version of \eqref{eq:ccbm_weak_form}.
\begin{lemma}
	\label{lem:transported_problem}
	The pair $(\uu^{t},p^{t}) = (\ur^{t} + i \ui^{t},\pr^{t} + i \pim^{t})$ uniquely solves in $\Vgamma \times Q$ the variational problem
        \begin{equation}\label{eq:transformed_ccbm_weak_form}
        \left\{\arraycolsep=1.4pt\def\arraystretch{1.1}
        \begin{array}{rcll}
        	\aat({\uu^{t}},{\bphi}) + b^{t}(\bphi,p^{t})	&=& F^{t}(\bphi),		&\quad\forall {\bphi}\in {\Vgamma},\\ 
        					b^{t}(\uu^{t},\lambda)	&=& 0,			&\quad\forall \lambda \in {Q}.
        \end{array}
        \right.
        \end{equation}
\end{lemma}	
\begin{proof}
	The functions ${\uu_{t}} \in \HH^{1}(\Omega_{t})^{d}$ and ${p_{t}} \in \LL^{2}(\Omega_{t})$ solve the variational equation
	\[
        \left\{
        	\begin{aligned}
        		\intOt{\alpha \nabla {\uu_{t}} : \nabla {\cbphi}} + i \intSt{({\uu_{t}} \cdot \nn_{t}) ({\cbphi} \cdot \nn_{t})} &- \intOt{p_{t} ( \nabla \cdot \cbphi) }\\ 
			&= \intOt{\ff_{t} \cdot \cbphi}, \quad \forall {\bphi} \in {\HH_{\Gamma, \vect{0}}^{1}(\Omega_{t})}^{d},\\
		- \intOt{\bar{\lambda} ( \nabla \cdot \uu_{t}) } &= 0, \quad \forall \lambda \in {\LL^{2}(\Omega_{t})}.
         	\end{aligned}
        \right.
	\]
	By change of variables (cf. \cite[pp. 482--484]{DelfourZolesio2011}), together with the identities
	\begin{align*}
		\uu^{t} &= {\uu_{t}} \circ T_{t},\qquad \text{where}\ \uu_{t} = {\uu_{rt}} + i {\uu_{it}},\\
		(\nabla \bphi_{t}) \circ T_{t} & 
			=  DT_{t}^{-\top} \nabla \bphi^{t},\qquad \text{where}\ \bphi_{t} \in \HH^{1}(\Omega_{t}), \ \bphi^{t} \in \HH^{1}(\Omega),\\
		(\nabla \cdot \bphi_{t}) \circ T_{t} & 
			= (DT_{t})^{-1} : \nabla \bphi^{t},
	\end{align*}
	and the transformation \cite[Thm. 4.4, p. 488]{DelfourZolesio2011}
	\[
	(\uu \cdot \nn)_{t} = (\uu_{t} \cdot \nn_{t}) \circ T_{t} = \uu^{t} \cdot \nn^{t} = \uu^{t} \cdot \frac{\Mt\nn}{\abs{\Mt\nn}},
	\]
	we get
	\[
        \left\{
        	\begin{aligned}
	  \intO{\alpha A_{t} \nabla {\uu^{t}} : \nabla {\cbphi}} + \intS{\frac{B_{t}}{|\Mt\nn|^{2}} ({\uu^{t}} \cdot \Mt \nn)({\cbphi} \cdot \Mt \nn)} &- \intO{ I_{t} p^{t} ( \Mt^{\top} : \nabla \cbphi) }	\\		&= \intO{I_{t} \ff^{t} \cdot \cbphi}, \quad \forall {\bphi}\in {\Vgamma},\\
	-\intO{ I_{t} \bar{\lambda} ( \Mt^{\top} : \nabla {\uu^{t}}) } &= 0, \quad \forall \lambda \in {Q}.
         	\end{aligned}
        \right.
	\]
        In light of the notations listed in \eqref{eq:transformed_forms}, we recover \eqref{eq:transformed_ccbm_weak_form}.
	
	The rest of the proof is similar to the arguments used in subsection \ref{subsec:well-posedness_of_state_problem} combined with the properties of $A_{t}$ and $B_{t}$ issued in \eqref{eq:regular_maps} and \eqref{eq:bounds_At_and_Bt}, and together with Lemma \ref{lem:boundedness_of_sesquiliner_and_linear_forms}.
	Concerning uniqueness of solution, the proof is also standard, so we omit it.
	This proves the lemma.
\end{proof}
\begin{lemma}\label{lem:boundedness_of_the_transformed_state}
	For $t \in \mathcal{I}$, the solutions $(\uu^{t},p^{t})$ of \eqref{eq:transformed_ccbm_weak_form} are uniformly bounded in $X \times Q$.
	More precisely, for all $t \in \mathcal{I}$, there exists a constant $c>0$ independent of $t$ such that
	\[
		\vertiii{\uu^{t}}_{X}, \
		\vertiii{p^{t}}_{Q} \lesssim \vertiii{\ff}_{1,U}.
	\] 
\end{lemma}
\begin{proof}
	For the (uniform) boundedness of $\uu^{t}$ in $X$ for $t \in \mathcal{I}$, the result is established by taking $\bphi = \uu^{t} \in \Vgamma$ in \eqref{eq:transformed_ccbm_weak_form},
	and then applying the properties of $A_{t}$ and $B_{t}$ given in \eqref{eq:regular_maps} and \eqref{eq:bounds_At_and_Bt}, 
	and the coercivity of $\aat$ on $X \times X$ (Lemma \ref{lem:boundedness_of_sesquiliner_and_linear_forms}).
	Indeed, from the first equation in \eqref{eq:transformed_ccbm_weak_form} together with the second equation $b^{t}(\uu^{t},\lambda)	 = 0$ with $\lambda = p^{t} \in Q$, we have
	\[
        	\aat({\uu^{t}},{\uu^{t}}) + b^{t}({\uu^{t}},p^{t}) = F^{t}({\uu^{t}}) 
	\qquad \Longleftrightarrow \qquad
	 	\aat({\uu^{t}},{\uu^{t}}) = F^{t}({\uu^{t}}).
	\]
	This yields the following inequality
	\[
	c \vertiii{\uu^{t}}_{X}^{2}
		\leqslant \Re\{\aat(\uu^{t},\uu^{t})\}
		\leqslant \abs{\aat(\uu^{t},\uu^{t})}
		\leqslant 
		\left| \intO{I_{t} \ff^{t} \cdot \overline{\uu}^{t}} \right|
		\leqslant \left\| I_{t} \ff^{t} \right\|_{0,\Omega} \vertiii{\overline{\uu}^{t}}_{0,\Omega},
	\]
	where the constant $c>0$ appearing on the leftmost side of the inequality is the coercivity constant of $\aat$ whose existence is ensured by Lemma \ref{lem:boundedness_of_sesquiliner_and_linear_forms}.
	To further get an estimate for the rightmost side of the above inequality, we note of the following calculation
	\begin{equation}\label{eq:bound_for_ft}
		\left\| I_{t} \ff^{t} \right\|_{0,{\Omega}}^{2}
		= \intO{(I_{t} \circ T_{t}^{-1}) \circ T_{t}(x) \ff^{2} \circ T_{t}(x) I_{t}}
		= \intOt{ I_{t} \circ T_{t}^{-1} \ff^{2} }
		\leqslant \Lambda_{2} \|\ff\|^{2}_{0,U},
	\end{equation}
	where in the last inequality we used the bound for $I_{t}$ given in \eqref{eq:bounds_At_and_Bt}.
	Hence, we have
	\[
		\vertiii{\uu^{t}}_{X}^{2}
		\leqslant c^{-1}\sqrt{\Lambda_{2}} \|\ff\|_{0,U} \vertiii{\overline{\uu}^{t}}_{0,\Omega} 
		\leqslant c^{-1}\sqrt{\Lambda_{2}} \|\ff\|_{0,U} \vertiii{\uu^{t}}_{X}. 
	\]
	This gives us a priori estimate
	\begin{equation}\label{eq:bound_for_ut}
		\vertiii{\uu^{t}}_{X} \leqslant c \|\ff\|_{0,U},
	\end{equation}
	for some constant $c>0$, which shows that $\uu^{t}$ is bounded in $X$ for $t \in \mathcal{I}$.
	
	For the boundedness of $p^{t}$ in $Q$ for $t \in \mathcal{I}$, we make use of \eqref{eq:inf_sup_condition_for_transformed_problem}, which is equivalent to the following
	\[
		\sup_{\substack{\bphi \in {\Vgamma}\\ \bphi \neq \vect{0}}} \frac{b^{t}(\bphi,{\lambda})}{\vertiii{\bphi}_{X} } \geqslant \beta_{1} \vertiii{{\lambda}}_{Q},\quad \forall {\lambda} \in Q.
	\]
	We let $\lambda = p^{t} \in Q$.
	Then, from \eqref{lem:transported_problem}, the bounds in \eqref{eq:bounds_At_and_Bt}, and equation \eqref{eq:bound_for_ft}, we have the following calculation
	\begin{align*}
		\beta_{1} \vertiii{{p^{t}}}_{Q}
			&\leqslant \sup_{\substack{\bphi \in {\Vgamma}\\ \bphi \neq \vect{0}}} \frac{b^{t}(\bphi,{p^{t}})}{\vertiii{\bphi}_{X} }
			=\sup_{\substack{\bphi \in {\Vgamma}\\ \bphi \neq \vect{0}}} \vertiii{\bphi}_{X}^{-1}
				\Bigg\{ F^{t}(\bphi) - \aat({\uu^{t}},{\bphi}) \Bigg\}\\
			&\leqslant 
			\sup_{\substack{\bphi \in {\Vgamma}\\ \bphi \neq \vect{0}}} \vertiii{\bphi}_{X}^{-1}
				\Bigg\{ 
				 	\intO{I_{t} \ff^{t} \cdot \cbphi} 
					-  \intO{\alpha A_{t} \nabla {\uu^{t}} : \nabla {\cbphi}} \\
			&\qquad\qquad\qquad\quad		
					- \intS{\frac{I_{t}}{|\Mt\nn|} ({\uu^{t}} \cdot \Mt \nn)({\cbphi} \cdot \Mt \nn)} 
			 	\Bigg\}\\
			&\leqslant 
			c \sup_{\substack{\bphi \in {\Vgamma}\\ \bphi \neq \vect{0}}} \vertiii{\bphi}_{X}^{-1}
				\Bigg\{ 
				 	 \| \ff \|_{0,U} \vertiii{\cbphi}_{0,\Omega} 
					+  \vertiii{\nabla {\uu^{t}}}_{0,\Omega} \vertiii{\nabla {\cbphi}}_{0,\Omega}
					+ \vertiii{{\uu^{t}}}_{0,\Sigma} \vertiii{{\cbphi}}_{0,\Sigma} 
			 	\Bigg\}\\
			&\leqslant 
			c \left(  \| \ff \|_{0,U} + \vertiii{\uu^{t}}_{X}  \right).								
	\end{align*}
	Therefore, using the estimate for $\vertiii{\uu^{t}}_{X}$ given in \eqref{eq:bound_for_ut}, we arrive at
	\[
		\vertiii{{p^{t}}}_{Q} \leqslant c \| \ff \|_{0,U},
	\]
	for some constant $c>0$.
	The desired inequalities then follow from the fact that $\ff \in H^{1}(U)^{d}$.
	This proves the lemma.
\end{proof}
To complete our preparations, we next prove the following lemma concerning the H\"{o}lder continuity of the state variables $\uu^{t}$ and $p^{t}$ with respect to $t$.
\begin{lemma}\label{lem:holder_continuity}
	Let $(\uu,p) \in \Vgamma \times Q$ be the solution of \eqref{eq:ccbm_weak_form}.
	Then, the following limit holds
	\[
		\lim_{t\to0} \frac{1}{\sqrt{t}} \left( \vertiii{\uu^{t} - \uu}_{X} + \vertiii{p^{t} - p}_{Q} \right) = 0,
	\] 
	where $({\uu^{t}}, p^{t}) \in \Vgamma \times Q$ solves \eqref{eq:transformed_ccbm_weak_form}, for $t \in \mathcal{I}$.
\end{lemma}
\begin{proof}
	Let us first note that the assumption that $\VV \in \vect{\Theta}^{1}$ implies that
	\[
		\|\VV\|_{\mathcal{C}^{1,1}(\overline{U})^{d}} = |\VV|_{\infty} + |D\VV|_{\infty} + \sup_{x \neq y} \frac{|D\VV(x) - D\VV(y)|_{\infty}}{|x-y|} < \infty\footnote{Here and throughout the paper, we use $|\cdot|_{\infty}$ to denote the $L^{\infty}(\overline{U})$ norm for simplicity.}
	\]
	which implies further that $ |\VV|_{\infty}$ and $|D\VV|_{\infty}$ are both finite.
	
	Hereafter, we proceed in two steps:
	\begin{description}
		\item[\underline{\textnormal{Step 1}.}] We first show that $\lim_{t \to 0} \uu^{t} = \uu$ in $X$ and $\lim_{t \to 0} p^{t} = p$ in $Q$. 
		\item[\underline{\textnormal{Step 2}.}] Then, we validate our claim that $\lim_{t\to0} \dfrac{1}{\sqrt{t}} \left( \vertiii{\uu^{t} - \uu}_{X} + \vertiii{p^{t} - p}_{Q} \right) = 0$.
	\end{description}

	\underline{\textnormal{Step 1}.} We consider the difference between the variational equations \eqref{eq:transformed_ccbm_weak_form} and \eqref{eq:ccbm_weak_form} and define $\yt := \uu^{t} - \uu$ and $r^{t} := p^{t} - p$.
	By making $\varepsilon > 0$ smaller if necessary, it can be shown that the following expansions hold for sufficiently small $t \in \mathcal{I}$,
	\begin{equation*}\label{eq:approximations}
	\left\{
	\begin{aligned}
		I_{t} &= 1 + t \operatorname{div} \VV + t^{2} \tilde{\rho}(t,\VV), \quad \text{where} \ \tilde{\rho} \in \mathcal{C}(\mathbb{R},\mathcal{C}^{0,1}(U)),\\
		\Mt^{\top} &= (DT_{t})^{-1} = (id + t D\VV)^{-1} = \sum_{j=0}^{\infty} (-1)^{j} t^{j} (D\VV)^{j}, \quad \text{where}\ 0 \leqslant |t| \leqslant \varepsilon < |\lambda_{max}|^{-1},
	\end{aligned}
	\right.
	\end{equation*}
	where $\lambda_{max}$ is the dominant eigenvalue (i.e., the eigenvalue with the largest modulus) of $D\VV$.
	Below we use the above expansions and specifically write $\Mt^{\top}$ as follows
	\begin{equation}\label{eq:expansion_of_Mt}
		\Mt^{\top} = (DT_{t})^{-1} = (id + t D\VV)^{-1} = id - t D\VV + O{(t^{2})} id.
	\end{equation}
	Henceforth, we denote
	\[
		\rho(t):=t^{2} \tilde{\rho}(t,\VV),
		\quad R(t) := O{(t^{2})} id,
		\quad \rho_{1}(t):=t \tilde{\rho}(t,\VV),
		\quad\text{and}\quad R_{1}(t) := O{(t)} id.\footnote{In some occasions, the remainder $\rho(t)$ and $R(t)$ may have a different structure for its exact expression. Nevertheless, these expressions are always of order $O(t^{2})$. The same is true for $\rho_{1}(t)$ and $R_{1}(t)$. We abuse the use of these notations since the exact expressions are not actually of interest in our argumentations.}
	\]
	
	Now, let us consider the variational equation
	\[
		b^{t}(\uu^{t},\lambda) - b(\uu,\lambda) 
		= -\intO{ I_{t} \bar{\lambda} ( \Mt^{\top} : \nabla {\uu^{t}}) } - \left( - \intO{\overline{\lambda} \nabla \cdot \uu } \right)
		= 0, \quad \forall \lambda \in Q.
	\]
	Applying \eqref{eq:expansion_of_Mt} to the above equation, we get
	\begin{equation}\label{eq:divergence_equal_zero_computation}
	\begin{aligned}
	&b^{t}(\uu^{t},\lambda) - b(\uu,\lambda) \\
		&\quad = -\intO{ (1 + t \nabla \cdot \VV + \rho(t)) \bar{\lambda} [ (id - t D\VV + R(t)) : \nabla {\uu^{t}}] } - \left( - \intO{\overline{\lambda} \nabla \cdot \uu } \right)\\
		&\quad = - \intO{\overline{\lambda} \nabla \cdot (\uu^{t} - \uu) } 
			- \intO{ t (\nabla \cdot \VV + \rho(t)) \overline{\lambda} ( \Mt^{\top} : \nabla {\uu^{t}})}
			\\ &\quad \qquad  - \intO{ \overline{\lambda} [ (-t D\VV + R(t) ): \nabla \uu^{t} ] }\\
		&\quad = 0, \quad \forall \lambda \in Q.
	\end{aligned}
	\end{equation}
	Taking $\lambda = r^{t} = p^{t} - p \in Q$, we get
	\begin{equation}\label{eq:identity_for_rtyt} 
	\begin{aligned}
	 \intO{r^{t} \nabla \cdot \overline{\vect{y}}^{t} }  
	 	&= - \intO{ t (\nabla \cdot \VV + \rho(t)) r^{t} ( \Mt^{\top} : \nabla \overline{\uu}^{t})}\\
		&\qquad	- \intO{ r^{t} [ (-t D\VV + R(t) ): \nabla \overline{\uu}^{t} ] }.	
	\end{aligned}
	\end{equation}
	%
	
	%
	%
	On the other hand, we also have the equation
	\[
		\aat(\uu^{t},\bphi) + b^{t}(\bphi,p^{t}) - \aaa(\uu,\bphi) - b(\bphi,p) = F^{t}(\bphi) - F(\bphi),  \qquad \forall {\bphi} \in \Vgamma,
	\]
	which is equivalent to
	\begin{equation}\label{eq:difference_equation}
		\aaa(\uu^{t} - \uu, \bphi) + b(\bphi,p^{t} - p) = \Phi^{t}(\bphi),\qquad \forall {\bphi} \in \Vgamma,
	\end{equation}
	where
	\begin{equation}\label{eq:big_Phi_sup_t}
	\begin{aligned} 
		\Phi^{t}(\bphi) &= - \intO{ \alpha  (A_{t} - {id}) \nabla {\uu^{t}} : \nabla {\cbphi}}\\
		&\qquad - i \intS{ \left( \frac{B_{t}}{|\Mt\nn|^{2}}  - 1 \right) ( \Mt^{\top} {\uu^{t}} \cdot \nn ) ( \Mt^{\top} {\cbphi} \cdot \nn) } \\
		&\qquad - i \intS{ [ (\Mt^{\top} - id) {\uu^{t}} \cdot \nn ] (\Mt^{\top} {\cbphi} \cdot \nn) } 
					- i \intS{ ( \Mt^{\top} {\uu^{t}} \cdot \nn ) [ ( \Mt^{\top} - id) {\cbphi} \cdot \nn] } \\
		&\qquad + \intO{ (I_{t} -1) p^{t} ( \Mt^{\top} : \nabla \cbphi) } 
					+ \intO{ p^{t} [ (\Mt^{\top} - id) : \nabla \cbphi ] } \\
		&\qquad + \intO{ ( I_{t} \ff^{t} - \ff) \cdot \cbphi}, \qquad (\bphi \in \Vgamma),
	\end{aligned}
	\end{equation}
	and $a: \Vgamma \times \Vgamma \to \mathbb{R}$ and $b: \Vgamma \times Q \to \mathbb{R}$ are respectively the sesquilinear and linear forms given in \eqref{eq:forms_for_the_state_problem}.

	By choosing $\bphi = \yt \in \Vgamma$ in \eqref{eq:difference_equation} and utilizing identity \eqref{eq:identity_for_rtyt}, it follows that
	\begin{align*}
	c_{a}\vertiii{\yt}_{X}^{2}
		& \ \leqslant \Re\{\aaa(\yt,\yt)\}\\
		& \ \leqslant \abs{\aaa(\yt,\yt)} \\
		&\ = \left|{\intO{\alpha \nabla {\yt} : \nabla {\cyt}} + i \intS{({\yt} \cdot \nn)({\cyt} \cdot \nn)}}\right| \\
		&\ = \Bigg| \Phi^{t}(\yt) + \intO{r^{t} \nabla \cdot \cyt} \Bigg| \\
		%
		%
		&\ = \Bigg| - \intO{ \alpha (A_{t} - {id}) \nabla {\uu^{t}} : \nabla {\cyt}}\\
		&\qquad \ - i \intS{ \left( \frac{B_{t}}{|\Mt\nn|^{2}}  - 1 \right) ( \Mt^{\top} {\uu^{t}} \cdot \nn ) ( \Mt^{\top} {\cyt} \cdot \nn) } \\
		&\qquad \ - i \intS{ [ (\Mt^{\top} - id) {\uu^{t}} \cdot \nn ] ( \Mt^{\top} {\cyt} \cdot \nn) } \\
		&\qquad \	- i \intS{ ( \Mt^{\top} {\uu^{t}} \cdot \nn ) [ ( \Mt^{\top} - id) {\cyt} \cdot \nn] } \\
		&\qquad \ + \intO{ (I_{t} -1) p^{t} ( \Mt^{\top} : \nabla \cyt) } 
					+ \intO{ p^{t} [ (\Mt^{\top} - id) : \nabla \cyt ] } \\		
		&\qquad \ + \intO{ (I_{t} \ff^{t} - \ff) \cdot \cyt}\\
		&\qquad \ - t \intO{ (\nabla \cdot \VV + \rho_{1}(t)) r^{t} ( \Mt^{\top} : \nabla \overline{\uu}^{t})}
			- t \intO{ r^{t} [ (- D\VV + R_{1}(t) ): \nabla \overline{\uu}^{t} ] } \Bigg| \\		
		&\ \leqslant |\alpha| \left| A_{t} - {id} \right|_{\infty} \vertiii{\nabla {\uu^{t}}}_{0,\Omega} \vertiii{\nabla {\cyt}}_{0,\Omega}\\
		&\quad \ + \left| \frac{B_{t}}{|\Mt\nn|^{2}}  - 1 \right|_{\infty} \abs{ \Mt^{\top} }_{\infty}^{2}  \vertiii {\uu^{t}}_{0,\Sigma} \vertiii{ {\cyt} }_{0,\Sigma} \\
		&\quad \ + 2 \left| \Mt^{\top} - id \right|_{\infty} \abs{ \Mt^{\top} }_{\infty} \vertiii {\uu^{t}}_{0,\Sigma} \vertiii{ {\cyt} }_{0,\Sigma} \\ 
		&\quad \ + \left( \left| I_{t} - 1 \right|_{\infty} \abs{ \Mt^{\top} }_{\infty} + \left| \Mt^{\top} - id \right|_{\infty} \right) \vertiii{p^{t}}_{0,\Omega} \vertiii{ \nabla {\cyt} }_{0,\Omega} \\		
		&\quad \ + \|I_{t} \ff^{t} - \ff\|_{0,\Omega} \vertiii{ \cyt }_{0,\Omega}\\
		%
		%
		%
		&\ \quad + |t| \left( \abs{\nabla \cdot \VV}_{\infty} + \abs{\rho_{1}(t)}_{\infty} \right) \abs{ \Mt^{\top} }_{\infty}
			\vertiii{r^{t}}_{0,\Omega} \vertiii{ \nabla \overline{\uu}^{t} }_{0,\Omega}\\
		&\ \quad + |t| \left( \abs{D\VV}_{\infty} + \abs{R_{1}{(t)}}_{\infty} \right) 
			\vertiii{r^{t}}_{0,\Omega} \vertiii{ \nabla \overline{\uu}^{t} }_{0,\Omega}.			
	\end{align*}
	Therefore, we have the estimate
	\begin{equation}\label{eq:first_estimate}
		c_{a}\vertiii{\yt}_{X}^{2}
		\leqslant {m}_{t} \vertiii{ \yt }_{X} + t {\Xi}^{t} \vertiii{r^{t}}_{Q},	
	\end{equation}
	where
	\begin{equation}\label{eq:coefficients}
	\left\{
	\begin{aligned} 
		{m}_{t} &:= \left( |\alpha| \left| A_{t} - {id} \right|_{\infty} + \left| \frac{B_{t}}{|\Mt\nn|^{2}}  - 1 \right|_{\infty} \abs{ \Mt^{\top} }_{\infty}^{2} + 2 \left| \Mt^{\top} - id \right|_{\infty} \abs{ \Mt^{\top} }_{\infty} \right) \vertiii {\uu^{t}}_{X}\\
		&\qquad + \left( \left| I_{t} - 1 \right|_{\infty} \abs{ \Mt^{\top} }_{\infty} + \left| \Mt^{\top} - id \right|_{\infty} \right) \vertiii{p^{t}}_{Q}
			+ \|I_{t} \ff^{t} - \ff\|_{0,\Omega},\\
		{\Xi}^{t} &:=  \left[ \left( \abs{\nabla \cdot \VV}_{\infty} + \abs{\rho_{1}(t)}_{\infty} \right) \abs{ \Mt^{\top} }_{\infty} + \abs{D\VV}_{\infty} + \abs{R_{1}{(t)}}_{\infty} \right] \vertiii{ {\uu}^{t} }_{X}.	
	\end{aligned}
	\right.
	\end{equation}

	Now, let us take $\lambda = r^{t} = p^{t}-p \in Q$ in the inf-sup condition \eqref{eq:inf_sup_nts} and consider equation \eqref{eq:difference_equation}.
	We have
	\begin{equation}\label{eq:inf_sup_estimate_inequality}
	\begin{aligned}
		\beta_{0} \vertiii{{r^{t}}}_{Q} & \leqslant
		\sup_{\substack{\bphi \in {\Vgamma}\\ \bphi \neq \vect{0}}} \frac{b(\bphi,{r^{t}})}{\vertiii{\bphi}_{X} } \\
		&= \sup_{\substack{\bphi \in {\Vgamma}\\ \bphi \neq \vect{0}}} \vertiii{\bphi}_{X}^{-1} 
			\Big\{ \Phi^{t}(\bphi) - \aaa(\yt, \bphi) \Big\}\\
		&= \sup_{\substack{\bphi \in {\Vgamma}\\ \bphi \neq \vect{0}}} \vertiii{\bphi}_{X}^{-1} 
			\Bigg\{ |\alpha| \left| A_{t} - {id} \right|_{\infty} \vertiii{\nabla {\uu^{t}}}_{0,\Omega} \vertiii{\nabla {\cbphi}}_{0,\Omega}\\
		&\qquad\qquad\qquad\qquad + \left| \frac{B_{t}}{|\Mt\nn|^{2}}  - 1 \right|_{\infty} \abs{ \Mt^{\top} }_{\infty}^{2}  \vertiii {\uu^{t}}_{0,\Sigma} \vertiii{ {\cbphi} }_{0,\Sigma} \\
		&\qquad\qquad\qquad\qquad + 2 \left| \Mt^{\top} - id \right|_{\infty} \abs{ \Mt^{\top} }_{\infty} \vertiii {\uu^{t}}_{0,\Sigma} \vertiii{ {\cbphi} }_{0,\Sigma} \\ 
		&\qquad\qquad\qquad \qquad + \left| I_{t} -1 \right|_{\infty} \abs{ \Mt^{\top} }_{\infty} \vertiii{p^{t}}_{0,\Omega} \vertiii{ \nabla {\cbphi} }_{0,\Omega} \\
		&\qquad\qquad\qquad \qquad + \left| \Mt^{\top} - id \right|_{\infty}  \vertiii{p^{t}}_{0,\Omega} \vertiii{ \nabla {\cbphi} }_{0,\Omega} \\
		&\qquad\qquad\qquad \qquad + \left\| I_{t} \ff^{t} -\ff \right\|_{0,\Omega} \vertiii{ {\cbphi} }_{0,\Omega}\\
		&\qquad\qquad\qquad \qquad + |\alpha| \vertiii{\nabla \yt }_{0,\Omega} \vertiii{\nabla {\cbphi}}_{0,\Omega} + \vertiii{ \yt }_{0,\Sigma} \vertiii{ {\cbphi}}_{0,\Sigma} \Bigg\}\\
		&\leqslant m_{t} + \max\{ |\alpha|, 1\} \vertiii{ \yt }_{X}. 
	\end{aligned}
	\end{equation}
        Therefore, we have the inequality
        \begin{equation}\label{eq:estimate_for_pt_final}
       	 	\vertiii{{r^{t}}}_{Q} \leqslant \beta_{0}^{-1} \left( {m}_{t} + \max\{ |\alpha|, 1\} \vertiii{ \yt }_{X} \right).
        \end{equation}
	Going back \eqref{eq:first_estimate} and utilizing the above estimate, we get 
	\begin{equation}\label{eq:second_estimate_eliminating_pt}
	\begin{aligned}
		c_{a}\vertiii{\yt}_{X}^{2}
		&\leqslant {m}_{t} \vertiii{ \yt }_{X} + t {\Xi}^{t} \left[ \beta_{0}^{-1} \left( {m}_{t} + \max\{ |\alpha|, 1\} \vertiii{ \yt }_{X} \right) \right]\\
		&= t {\Xi}^{t} \beta_{0}^{-1} {m}_{t}  + \left(  {m}_{t}  + t {\Xi}^{t} \beta_{0}^{-1} \max\{ |\alpha|, 1\} \right) \vertiii{ \yt }_{X}.
	\end{aligned}
	\end{equation}	
	We apply Peter-Paul inequality to $\left(  {m}_{t}  + t {\Xi}^{t} \beta_{0}^{-1} \max\{ |\alpha|, 1\} \right) \vertiii{ \yt }_{X}$ to obtain the following estimate
	\begin{align*}
		\left(  {m}_{t}  + t {\Xi}^{t} \beta_{0}^{-1} \max\{ |\alpha|, 1\} \right) \vertiii{ \yt }_{X}
			& \leqslant \frac{\left(  {m}_{t}  + t {\Xi}^{t} \beta_{0}^{-1} \max\{ |\alpha|, 1\} \right)^{2}}{2\varepsilon_{1}} + \frac{\varepsilon_{1}}{2} \vertiii{ \yt }_{X}^{2}.
	\end{align*}
	for some constant $\varepsilon_{1} > 0$.
	We choose (and fixed) $\varepsilon_{1}$ such that $\bar{c}:=c(c_{a},\varepsilon_{1}) := 2c_{a} - \varepsilon_{1} > 0$ so that, from our first estimate \eqref{eq:second_estimate_eliminating_pt}, we have
	\begin{equation}\label{eq:final_estimate_for_yt_bounded_above}
	\vertiii{\yt}_{X}
		\leqslant \bar{c}^{-\frac{1}{2}} \left(  t {\Xi}^{t} \beta_{0}^{-1} {m}_{t} + \frac{\left(  {m}_{t}  + t {\Xi}^{t} \beta_{0}^{-1} \max\{ |\alpha|, 1\} \right)^{2}}{2\varepsilon_{1}} \right)^{1/2}.	
	\end{equation}
	We note that, at $t=0$, $I_{t} |\Mt\nn|^{-1} = (1) (|id \nn|^{-1}) = 1$.
	Thus, in view of \eqref{eq:coefficients} together with Lemma \ref{lem:boundedness_of_the_transformed_state}, \eqref{eq:convergence_of_vector_valued_functions_1} and the vector-valued version of Lemma \ref{lem:convergence_of_functions} (refer to \eqref{eq:convergence_of_vector_valued_functions_3}), we see that ${m}_{t} \to 0$ as $t \to 0$.
	Moreover, it is not hard to see that ${\Xi}^{t}$ is (uniformly) bounded for all $t \in \mathcal{I}$ because of Lemma \ref{lem:boundedness_of_the_transformed_state}.
	Therefore, we can conclude -- by Lebesgue's dominated convergence theorem -- that
	\begin{equation}\label{eq:limit_of_ut}
	\lim_{t \to 0} \vertiii{\yt}_{X} = 0
		\qquad \Longleftrightarrow \qquad
		\lim_{t \to 0} \uu^{t} = \uu \quad \text{in $X$}.
	\end{equation}

	Similarly, based from the above discussion and from \eqref{eq:inf_sup_estimate_inequality}, we know that the terms on the right-hand side of the said inequality vanish as $t \to 0$.
	Thus, we also have the limit
	\begin{equation}\label{eq:limit_of_pt}
	\lim_{t \to 0} \vertiii{r^{t}}_{Q} = 0
		\qquad \Longleftrightarrow \qquad	
		\lim_{t \to 0} p^{t} = p \quad \text{in $Q$}.
	\end{equation}

	\underline{\textnormal{Step 2}.}	 Before we proceed to the last part of the proof, we note that, for sufficiently small $t > 0$, $\frac{1}{t}\yt \in \Vgamma$ and $\frac{1}{t}r^{t} \in Q$.
	In addition, we recall that the derivatives of $I_{t}$ and $\Mt$ with respect to $t$ exists in $L^{\infty}(\Omega)$ and $L^{\infty}(\Omega)^{d\times d}$, respectively.

	Now, to finish the proof, we go back to the computations in the previous step (referring particularly to \eqref{eq:final_estimate_for_yt_bounded_above}), to obtain, after dividing by $t > 0$, 
	\begin{align*}
		\frac{1}{t}\vertiii{\yt}_{X}^{2}
		\leqslant \bar{c}^{-1} \left( {\Xi}^{t} \beta_{0}^{-1} {m}_{t} + \frac{1}{2\varepsilon_{1}} \left(  \frac{1}{\sqrt{t}}{m}_{t}  + \sqrt{t} {\Xi}^{t} \beta_{0}^{-1} \max\{ |\alpha|, 1\} \right)^{2} \right).	
	\end{align*}
	Observe that, we have
	\begin{align*}
	\frac{1}{\sqrt{t}}{m}_{t} 
	&= \sqrt{t}\left( |\alpha| \left| \frac{A_{t} - {id}}{t} \right|_{\infty} + \left| \frac{1}{t}  \left( \frac{B_{t}}{|\Mt\nn|^{2}}  - 1 \right) \right|_{\infty} \abs{ \Mt^{\top} }_{\infty}^{2} \right) \vertiii {\uu^{t}}_{X}\\
	&\qquad + \sqrt{t}\left( 2 \left| \frac{\Mt^{\top} - id}{t} \right|_{\infty} \abs{ \Mt^{\top} }_{\infty} \right) \vertiii {\uu^{t}}_{X}\\
		&\qquad + \sqrt{t}\left( \left| \frac{I_{t} - 1}{t} \right|_{\infty} \abs{ \Mt^{\top} }_{\infty} + \left| \frac{\Mt^{\top} - id}{t} \right|_{\infty} \right) \vertiii{p^{t}}_{Q} + \sqrt{t} \left\|\frac{I_{t} \ff^{t} - \ff}{t} \right\|_{0,\Omega}.
	\end{align*}
	Thus, using \eqref{eq:regular_maps}, \eqref{eq:limits_of_maps}, \eqref{eq:convergence_of_vector_valued_functions_3}, Lemma \ref{lem:boundedness_of_the_transformed_state}, \eqref{eq:limit_of_ut}, and \eqref{eq:limit_of_pt}, we deduce that the following limit holds
	\begin{equation}\label{eq:limit_ratio_yt_over_t}
		\lim_{t \to 0 }\frac{1}{t}\vertiii{\uu^{t} - \uu}_{X}^{2} = 0
		\qquad \Longleftrightarrow \qquad
		\lim_{t \to 0 }\frac{1}{\sqrt{t}}\vertiii{\uu^{t} - \uu}_{X} = 0.
	\end{equation}

	Similarly, from estimate \eqref{eq:estimate_for_pt_final}, we know that (after dividing by $\sqrt{t} > 0$) the following inequality holds
	\begin{align*}
		\frac{1}{\sqrt{t}}\vertiii{{r^{t}}}_{Q} \leqslant \beta_{0}^{-1} \left[ \frac{1}{\sqrt{t}} {m}_{t} + \max\{ |\alpha|, 1\} \left( \frac{1}{\sqrt{t}} \vertiii{ \yt }_{X} \right) \right].
	\end{align*}
	Again, applying \eqref{eq:regular_maps}, \eqref{eq:limits_of_maps}, \eqref{eq:convergence_of_vector_valued_functions_3}, and Lemma \ref{lem:boundedness_of_the_transformed_state}, but now combined with \eqref{eq:limit_ratio_yt_over_t}, we infer that
	\begin{equation}\label{eq:limit_ratio_rt_over_t}
		\lim_{t \to 0 }\frac{1}{t}\vertiii{p^{t} - p}_{Q}^{2} = 0
		\qquad \Longleftrightarrow \qquad
		\lim_{t \to 0 }\frac{1}{\sqrt{t}}\vertiii{p^{t} - p}_{Q} = 0.
	\end{equation}	

		In above argumentations of taking the limit (equations \eqref{eq:limit_ratio_yt_over_t} and \eqref{eq:limit_ratio_rt_over_t}), it is understood that we are taking a sequence $\{t_{n}\}$ such that $\lim_{n \to \infty} t_{n} = 0$ and we have exploited the fact that the following sequences
	\[
		\left\{ \frac{I_{t_{n}}  -1}{t_{n}} \right\},
		\quad \left\{ \frac{M_{t_{n}}^{\top}  -id}{t_{n}} \right\},
		\quad \left\{ \frac{A_{t_{n}}  -id}{t_{n}} \right\},
		\quad \text{and}
		\ \ \left\{ \frac{1}{t_{n}}\left( \frac{B_{t_{n}}}{|M_{t_{n}}^{\top} \nn|}  - 1 \right)\right\},
	\]
	are (uniformly) bounded in $L^{\infty}$ (because of \eqref{eq:regular_maps}, \eqref{eq:limits_of_maps}, and \eqref{eq:bounds_At_and_Bt}).
	In addition, Lemma \ref{lem:boundedness_of_the_transformed_state} was employed above in the manner that, for the sequence $\{t_{n}\}$ (passing to a subsequence if necessary), the solutions $(\uu^{t_{n}},p^{t_{n}})$ of \eqref{eq:transformed_ccbm_weak_form} are uniformly bounded in $X \times Q$.
	Note that these results allow us to interchange the limit and the supremum in obtaining, particularly, equation \eqref{eq:limit_ratio_rt_over_t}.
	Finally, combining \eqref{eq:limit_ratio_yt_over_t} and \eqref{eq:limit_ratio_rt_over_t} concludes the lemma.
\end{proof} 
We are now in the position to prove Proposition \ref{prop:the_shape_derivative_of_the_cost} without using the shape derivative of $\uu$ and $p$.
Before we dealt with the proof, we make a few preparations by proving the following four lemmas which primarily contain several identities.
These results will serve useful in deriving the boundary expression for the shape gradient in accordance with Hadamard-Zol\'{e}sio structure theorem \cite[Chap. 9, Sec. 3.4, Thm. 3.6, p. 479]{DelfourZolesio2011}.
\begin{lemma}\label{lem:divergence_identity_limit}
	The following equation is satisfied by the state solution $\uu \in X$ of \eqref{eq:ccbm_weak_form}:
	\[
		\lim_{t \to 0} b(\ut - \uu, \lambda) 
		= - \lim_{t \to 0} \intO{\overline{\lambda} \nabla \cdot \left(\frac{\uu^{t} - \uu}{t}\right)}
		= - \intO{\overline{\lambda} (D\VV : \nabla \uu)}, \quad \forall \lambda \in {Q}.
	\]
\end{lemma}
\begin{proof}
	Let $(\uu^{t},p^{t}) \in \Vgamma \times Q$ be as in Lemma \ref{lem:transported_problem}.
	We recall equation \eqref{eq:divergence_equal_zero_computation} and divide both sides by $t > 0$, to obtain the following equation
	\begin{equation}\label{eq:estimate_for_divergence_equation}
	\begin{aligned}
		- \intO{\overline{\lambda} \nabla \cdot \left(\frac{\uu^{t} - \uu}{t}\right)} 
		&= \intO{ \left(\nabla \cdot \VV + \frac{\rho(t)}{t} \right) \overline{\lambda} ( \Mt^{\top} : \nabla {\uu^{t}})}\\
		&\qquad + \intO{ \overline{\lambda} \left[ \left(-D\VV + \frac{R(t)}{t} \right): \nabla \uu^{t} \right]}, \quad \forall \lambda \in {Q}.
	\end{aligned}
	\end{equation} 
	Using the properties of $\rho(t)$ and $R(t)$, we get -- by letting $t$ goes to zero in \eqref{eq:estimate_for_divergence_equation} and in view of the last paragraph in the proof of the previous lemma -- the following limit
	\[
			- \lim_{t \to 0} \intO{\overline{\lambda} \nabla \cdot \left(\frac{\uu^{t} - \uu}{t}\right)}
			= \intO{(\nabla \cdot \VV) \overline{\lambda} (\nabla \cdot \uu)} - \intO{\overline{\lambda} (D\VV : \nabla \uu)}, \quad \forall \lambda \in {Q}.
	\]
	Because $\dive \uu = 0$ in $\Omega$, we arrive at the desired equation, ending the proof of the lemma.
\end{proof}
We next prove three identities which prove to be useful in simplifying the expression for the shape gradient.
\begin{lemma}\label{lem:useful_identity}
	Let $\VV \in \sfTheta^{1}$ and $\bphi$ be sufficiently smooth vector such that $\dive \bphi = 0$ in $\Omega \in \mathcal{C}^{1,1}$.
	Then, we have the following identity
	\begin{equation}\label{eq:useful_identity}
		\langle D\VV \bphi - D\bphi \VV, \nn \rangle
			= \langle \bphi , \nabla_{\Sigma} \Vn \rangle  + \langle (D\VV \nn \otimes \nn) \bphi, \nn \rangle - \langle \VV , \nabla_{\Sigma}(\bphi \cdot \nn) \rangle + \Vn \dive_{\Sigma}\bphi.
	\end{equation}
\end{lemma}
\begin{proof}
	To obtain the identity, we recall the definition of some of the tangential operators in Definition \ref{def:tangential_operators} and consider some identities (on $\Sigma$) given as follows:
	\begin{align*}
		D\bphi &= D_{\Sigma} \bphi + D\bphi \nn \otimes \nn;\\
		\dive_{\Sigma} \bphi &= \dive \bphi - D\bphi \nn \cdot \nn = \dive \bphi - \dn{\bphi} \cdot \nn;\\
		\nabla_{\Sigma}(\bphi \cdot \nn) &= D_{\Sigma}\bphi^{\top} \nn + D_{\Sigma}\nn^{\top} \bphi
		= D_{\Sigma}\bphi^{\top} \nn + D_{\Sigma}\nn \bphi.
	\end{align*}
	We have the following computations:
	\begin{align*}
		&\langle D\VV \bphi - D\bphi \VV, \nn \rangle\\
		&\qquad= \langle D_{\Sigma} \VV\bphi, \nn \rangle + \langle (D\VV \nn \otimes \nn) \bphi, \nn \rangle - \langle D\bphi \VV, \nn \rangle\\
		&\qquad= \langle D_{\Sigma} \VV\bphi, \nn \rangle + \langle (D\VV \nn \otimes \nn) \bphi, \nn \rangle - \langle \VV, D\bphi^{\top}\nn \rangle\\
		&\qquad= \langle D_{\Sigma} \VV\bphi, \nn \rangle + \langle (D\VV \nn \otimes \nn) \bphi, \nn \rangle - \langle \VV, D_{\Sigma}\bphi^{\top}\nn \rangle
			- \langle \VV, [\nn (D\bphi \nn)^{\top}]\nn \rangle \\
		&\qquad= \langle D_{\Sigma} \VV\bphi, \nn \rangle + \langle (D\VV \nn \otimes \nn) \bphi, \nn \rangle 
			- \langle \VV, \nabla_{\Sigma}(\bphi \cdot \nn) \rangle
			- \langle \VV, D_{\Sigma}\nn\bphi \rangle
			- \langle D\bphi \nn, \nn \rangle \Vn \\
		&\qquad= \langle \bphi, D_{\Sigma} \VV^{\top}\nn +  D_{\Sigma}\nn\VV \rangle + \langle (D\VV \nn \otimes \nn) \bphi, \nn \rangle 
			- \langle \VV, \nabla_{\Sigma}(\bphi \cdot \nn) \rangle
			- \langle D\bphi \nn, \nn \rangle \Vn \\					
		&\qquad= \langle \bphi, \nabla_{\Sigma} \Vn \rangle + \langle (D\VV \nn \otimes \nn) \bphi, \nn \rangle 
			- \langle \VV, \nabla_{\Sigma}(\bphi \cdot \nn) \rangle
			+ \Vn \dive_{\Sigma} \bphi.							
	\end{align*}
	This proves the lemma.
\end{proof}
Note that, by using tangential divergence formula, we can further write \eqref{eq:useful_identity} into
	\begin{equation}\label{eq:useful_identity_second_formula}
		\langle D\VV \bphi - D\bphi \VV, \nn \rangle
			= \dive_{\Sigma}(\Vn \bphi)  + \langle (D\VV \nn \otimes \nn) \bphi, \nn \rangle - \langle \VV , \nabla_{\Sigma}(\bphi \cdot \nn) \rangle.
	\end{equation}
\begin{lemma}\label{lem:divergence_expansion}
	Let $\VV = (\theta_{1}, \ldots, \theta_{d})^{\top} \in \sfTheta^{1}$ and $\bphi = (\varphi_{1}, \ldots, \varphi_{d})^{\top}$ be a sufficiently smooth vector in $\Omega$.
	Then, the following identity holds
	\begin{equation}\label{eq:divergence_expansion}
		\dive (\nabla\bphi^{\top} \VV) = (\VV \cdot \nabla) (\nabla \cdot \bphi) + D\VV : \nabla \bphi
			= (\VV \cdot \nabla) (\nabla \cdot \bphi) + \nabla\bphi^{\top} : D\VV^{\top}.
	\end{equation}
\end{lemma}	
\begin{proof}
	Let us first note that
	\[
		\nabla\bphi^{\top} \VV = \sum_{j=1}^{d} \sum_{i=1}^{d} \frac{\partial \varphi_{j}}{\partial x_{i}} \theta_{i} \vect{e}_{j}.
	\]
	Therefore, direct computation of the divergence of $\nabla\bphi^{\top} \VV$ gives us
	\begin{align*}
		\dive(\nabla\bphi^{\top} \VV)
			&= \sum_{j=1}^{d} \frac{\partial}{\partial x_{j}} \sum_{i=1}^{d} \frac{\partial \varphi_{j}}{\partial x_{i}} \theta_{i}
			= \sum_{j=1}^{d} \sum_{i=1}^{d} \frac{\partial^{2} \varphi_{j}}{\partial x_{i}\partial x_{j}} \theta_{i}
				+ \sum_{j=1}^{d} \sum_{i=1}^{d} \frac{\partial \varphi_{j}}{\partial x_{i}} \frac{\partial \theta_{i}}{\partial x_{j}}\\
			%
			&= 	\left( \sum_{i=1}^{d} \theta_{i} \frac{\partial}{\partial x_{i}} \right)
				\cdot \left(\sum_{j=1}^{d} \frac{\partial \varphi_{j}}{\partial x_{j}} \right) 	
				+ D\VV : \nabla \bphi\\
			&= 	(\VV \cdot \nabla) (\nabla \cdot \bphi)+ D\VV : \nabla \bphi,
	\end{align*}
	proving the identity.
\end{proof}
\begin{lemma}\label{lem:useful_equivalence}
	Let $\VV = (\theta_{1}, \ldots, \theta_{d})^{\top} \in \sfTheta^{1}$ and $\bphi = (\varphi_{1}, \ldots, \varphi_{d})^{\top}$ be a given vector.
	Then, the following equation holds
	\begin{equation}\label{eq:useful_equivalence}
		\nn \cdot (D\VV \nn \otimes \nn)\bphi - (D\VV\nn \cdot \nn)(\bphi \cdot \nn) = 0. 
	\end{equation}
\end{lemma}
\begin{proof}
	We first note of the following identities:
	\[
		D\VV\nn \otimes \nn = \sum_{j,k =1}^{d} n_{j} \sum_{i=1}^{d} \theta_{ki} n_{i} \vect{e}_{k} \vect{e}_{j}^{\top}
		\quad \text{and}\quad
		(D\VV\nn \otimes \nn)\bphi = \sum_{j,k =1}^{d} \varphi_{j} n_{j} \sum_{i=1}^{d} \theta_{ki} n_{i} \vect{e}_{k} \vect{e}_{j}^{\top}.
	\]
	Therefore, we have the following computations
	\begin{align*}
		\nn \cdot (D\VV\nn \otimes \nn)\bphi
			= \sum_{k,j =1}^{d} n_{k} (\varphi_{j} n_{j}) \sum_{i=1}^{d} \theta_{ki} n_{i}
			&= \left( \sum_{k =1}^{d} n_{k} \sum_{i=1}^{d} \theta_{ki} n_{i} \right) \left( \sum_{k =1}^{d} \varphi_{j} n_{j} \right)\\
			&= (D\VV \nn \cdot \nn )(\bphi \cdot \nn),
	\end{align*}
	which proves the equation.
\end{proof}
We now provide the proof of Proposition \ref{prop:the_shape_derivative_of_the_cost}.
\begin{proof}[Proof of Proposition \ref{prop:the_shape_derivative_of_the_cost}]
	The proof essentially proceeds in two parts.
	First, we evaluate the limit $\lim_{t\to0} \frac{1}{t}\left( J(\Omega_{t}) - J(\Omega) \right)$.
	Then, using the regularity of the domain as well as the state and adjoint variables (to be introduced below), we characterized the boundary integral expression for the computed limit.
	
	We begin by applying the following domain transformation formula (see, e.g., \cite[p. 77]{SokolowskiZolesio1992} or \cite[Chap. 9, Sec. 4.1, eq. (4.2), p. 482]{DelfourZolesio2011}):
	\[
		\intOt{\varphi_{t}} = \intO{\varphi_{t} \circ T_{t} I_{t}} = \intO{\varphi^{t} I_{t}},
	\]
	which holds for $\varphi_{t} \in L^{1}(\Omega_{t})$
combined with the identity $\eta^{2} - \zeta^{2} = (\eta - \zeta)^{2} + \zeta (\eta - \zeta)$ to obtain the following sequence of calculations:
\begin{align*}
	J(\Omega_{t}) - J(\Omega) 
	&= \frac12 \intOt{\left( |\uu_{it}|^2 + |p_{it}|^2 \right)} - \frac12 \intO{\left( |\ui|^2 + |\pim|^2 \right)}\\
	&= \frac12 \intO{\left( I_{t} |\uu_{i}^{t}|^2 - |\ui|^2 \right)} + \frac12 \intO{\left( |p_{i}^{t}|^2 + |\pim|^2 \right)}\\
	&= \frac{1}{2} \intO{ (I_{t} - 1) (\abs{\uu_{i}^{t}}^{2} - \abs{\uu_{i}}^{2}) }	
		+ \frac{1}{2} \intO{ (I_{t} - 1) \abs{\uu_{i}}^{2} }	\\
	&\qquad	+ \frac{1}{2} \intO{ (\abs{\uu_{i}^{t} - \uu_{i}}^{2}) }	
		+ \intO{ (\uu_{i}^{t} - \uu_{i}) \cdot \uu_{i} }	\\
	&\qquad + \frac{1}{2} \intO{ (I_{t} - 1) (\abs{\pim^{t}}^{2} - \abs{\pim}^{2}) }	
		+ \frac{1}{2} \intO{ (I_{t} - 1) \abs{\pim}^{2} }	\\
	&\qquad	+ \frac{1}{2} \intO{ (\abs{\pim^{t} - \pim}^{2}) }	
		+ \intO{ (\pim^{t} - \pim) \pim }\\
	%
	%
	&=: \sum_{i=1}^{8}{J}_{i}(t).		
\end{align*}
	In view of \eqref{eq:limits_of_maps}$_{1}$ and using Lemma \ref{lem:holder_continuity}, we infer that
	\[
		\dot{J}_{1}(0) = \dot{J}_{3}(0) = \dot{J}_{5}(0) = \dot{J}_{7}(0).
	\]
	Moreover, \eqref{eq:limits_of_maps}$_{1}$ also reveals that
	\[
		\dot{J}_{2}(0) + \dot{J}_{6}(0) 
			= \frac{1}{2}  \intO{\operatorname{div}\VV ( \abs{\uu_{i}}^{2} + \abs{\pim}^{2}) }.
	\]
	At this point, we put into use the expansion of the divergence of the product of a scalar function $\psi$ and a vector field $\bphi$ given by 
	\begin{equation}\label{eq:identity_expansion}
		\dive(\psi \bphi) = \psi \dive \bphi + \bphi \cdot \nabla \psi,
	\end{equation}
	and then apply Green's theorem to obtain the following equivalent form of the previous computed expression:
	\begin{equation}\label{eq:distributed_gradient_first_part}
	\begin{aligned}
		\dot{J}_{2}(0) + \dot{J}_{6}(0) 
			&= - \frac{1}{2}  \intO{ \VV \cdot \nabla ( \abs{\uu_{i}}^{2} + \abs{\pim}^{2}) }
				+ \frac{1}{2}  \intO{\operatorname{div}[\VV ( \abs{\uu_{i}}^{2} + \abs{\pim}^{2})] }\\
			&= - \intO{ \VV \cdot ( \nabla \uu_{i} \uu_{i} + \pim \nabla \pim) }
				+ \frac{1}{2}  \intS{ ( \abs{\uu_{i}}^{2} + \abs{\pim}^{2}) \Vn }\\
			&= - \intO{ (\uu_{i} \cdot \nabla \uu_{i}^{\top} \VV + \pim \VV \cdot \nabla \pim) }
				+ \frac{1}{2}  \intS{ ( \abs{\uu_{i}}^{2} + \abs{\pim}^{2}) \Vn }.								
	\end{aligned}
	\end{equation}
	The computations of the remaining two expressions $\dot{J}_{4}(0)$ and $\dot{J}_{8}(0)$ need much more work.
	This require the utilization of the adjoint system \eqref{eq:adjoint_system_weak_form}.
	Indeed, since $\yt = \uu^{t} - \uu \in \Vgamma$ and $r^{t} = p^{t} - p \in Q$, we can write 
	\begin{align*}
		&{J}_{4}(t) + {J}_{8}(t)\\
		&\qquad = \Im\left\{ \intO{\alpha \nabla \overline{\vv} \cdot \nabla \yt } + i \intS{(\overline{\vv} \cdot \nn) (\yt \cdot \nn) }
			- \intO{ \overline{q} \operatorname{div} \yt}
			- \intO{ r^{t} \operatorname{div} \vv} \right\}\\
		&\qquad \equiv \Im\left\{ a(\uu^{t} - \uu, \vv) + b(\vv, p_{t} - p) + b(\uu^{t} - \uu, q) \right\}.
	\end{align*}
	Therefore, in light of \eqref{eq:difference_equation} with $\bphi = \vv \in \Vgamma$, we get
	\[
		{J}_{4}(t) + {J}_{8}(t) = \Im\{ \Phi^{t}(\vv) + b(\uu^{t} - \uu, q) \},
	\]
	where $\Phi^{t}(\vv)$ is given by \eqref{eq:big_Phi_sup_t} with $\bphi = \vv$.
	Using \eqref{eq:limits_of_maps}, \eqref{eq:convergence_of_vector_valued_functions_3}, 
	and Lemma \ref{lem:divergence_identity_limit}, we obtain the following limit
\begin{equation}\label{eq:distributed_gradient}
\begin{aligned}
		\dot{J}_{4}(0) + \dot{J}_{8}(0)
		&= \lim_{t \to 0} \frac{1}{t}\left[ \Im\{ \Phi^{t}(\vv) + b(\uu^{t} - \uu, q) \} \right]\\
		&= \Im\Big\{ - \intO{ \alpha A \nabla {\uu} : \nabla {\overline{\vv}}}\\
		&\quad\qquad - i \intS{ (\dive \VV + D\VV \nn \cdot \nn) ( {\uu} \cdot \nn ) ( {\overline{\vv}} \cdot \nn) } \\
		&\quad\qquad - i \intS{ [ (-D\VV) {\uu} \cdot \nn ] ( {\overline{\vv}} \cdot \nn) } 
					- i \intS{ ( {\uu} \cdot \nn ) [ ( -D\VV ) {\overline{\vv}} \cdot \nn] } \\
		%
		%
		&\quad\qquad + \intO{ (\nabla \cdot \VV) p ( \nabla \cdot \overline{\vv}) } 
					+ \intO{ p [ (-D\VV) : \nabla \overline{\vv} ] } \\ 
		&\quad\qquad + \intO{ [ \vect{\nabla} \cdot (\ff \otimes \VV) ] \cdot \overline{\vv} } 
					- \intO{\overline{q} (D\VV : \nabla \uu)} \Big\}\\
		&=: \Im\left\{ \sum_{i=1}^{8} K_{i} \right\},
	\end{aligned}
\end{equation}
	where $\vect{\nabla} \cdot (\ff \otimes \VV)$ is evaluated as in \eqref{eq:vector_type_divergence}.
	%
	%
	%
Before we go further with our computation, we gather in the next few lines some identities that will be useful in our calculations.
We put into use the weak formulation of the state and the adjoint state problem given in \eqref{eq:ccbm_weak_form} and \eqref{eq:adjoint_system_weak_form}, respectively.
In these variational equations -- since $\VV \in \sfTheta$ and $\uu, \vv \in \HH^{2}(\Omega)^{d}$ -- we can respectively take $\bphi = \nabla \vv^{\top} \VV \in \Vgamma$ and $\bpsi = \nabla \uu^{\top} \VV \in \Vgamma$, apply integration by parts (twice), and then use the boundary conditions on $\Sigma$ to obtain the following equations:
\begin{equation} \label{eq:state_tested}
\begin{aligned} 
	&-\intO{\alpha \Delta \uu \cdot \nabla \overline{\vv}^{\top} \VV } + \intS{\alpha \dn{\uu} \cdot \nabla \overline{\vv}^{\top} \VV}
		+ \intO{\nabla p \cdot \nabla \overline{\vv}^{\top} \VV }- \intS{p \nn \cdot \nabla \overline{\vv}^{\top} \VV}\\	
	&\quad  = -\intO{\alpha \Delta \uu \cdot \nabla \overline{\vv}^{\top} \VV } + \intS{\alpha \dn{\uu} \cdot \nabla \overline{\vv}^{\top} \VV}
		- \intO{ p \dive ( \nabla \overline{\vv}^{\top} \VV ) } \\
	&\quad = \intO{\ff \cdot \nabla \overline{\vv}^{\top} \VV }
			- i \intS{ (\uu \cdot \nn)\nn \cdot \nabla \overline{\vv}^{\top} \VV }
\end{aligned}
\end{equation}
and
\begin{equation}
\begin{aligned} 	\label{eq:adjoint_tested}
	&-\intO{\alpha \Delta \overline{\vv} \cdot \nabla \uu^{\top} \VV} + \intS{\alpha \dn{\overline{\vv}} \cdot \nabla \uu^{\top} \VV}
		+ \intO{\nabla \overline{q} \cdot \nabla \uu^{\top} \VV }- \intS{\overline{q} \nn \cdot \nabla \uu^{\top} \VV}\\
	&\quad  = -\intO{\alpha \Delta \overline{\vv} \cdot \nabla \uu^{\top} \VV} + \intS{\alpha \dn{\overline{\vv}} \cdot \nabla \uu^{\top} \VV}
		- \intO{ \overline{q} \dive ( \nabla \uu^{\top} \VV ) } \\	
			& \quad = \intO{\uu_{i} \cdot \nabla \uu^{\top} \VV} - i \intS{ (\overline{\vv} \cdot \nn)\nn \cdot \nabla \uu^{\top} \VV}.
\end{aligned}
\end{equation}
Now, let us simplify or expand some of the integrals $K_{i}$ above.
\begin{itemize}	
	\item Rewriting $K_{7}$. For the last integral, we note of the following expansion of the integrand:
	\begin{align*}
		[ \vect{\nabla} \cdot (\ff \otimes \VV) ] \cdot \overline{\vv}
			& = [ \vect{\nabla} \cdot (\ff \VV^{\top}) ] \cdot \overline{\vv}\\
			&= \sum_{j=1}^{d} \left( \VV \cdot \nabla f_{j} + f_{j} \dive \VV \right) \bar{v}_{j}\\
			&= \sum_{j=1}^{d} \VV \cdot \nabla f_{j} \bar{v}_{j} + \left( \sum_{j=1}^{d} f_{j} \bar{v}_{j} \right) \dive \VV\\ 
			&= \sum_{j=1}^{d} \sum_{k=1}^{d} \frac{\partial f_{j}}{\partial x_{k}} \theta_{k} \bar{v}_{j} + \left( \ff \cdot \overline{\vv} \right) \dive \VV\\ 
			&= \sum_{i=1}^{d} \sum_{k=1}^{d} \frac{\partial f_{j}}{\partial x_{k}} \theta_{k} \vect{e}_{i} \cdot \sum_{l=1}^{d} \bar{v}_{l} \vect{e}_{l} + \left( \ff \cdot \overline{\vv} \right) \dive \VV\\ 
			&= \nabla \ff^{\top} \VV \cdot \overline{\vv} + \left( \ff \cdot \overline{\vv} \right) \dive \VV.
	\end{align*}
	At this juncture, we utilize the identity \eqref{eq:identity_expansion} to get the following equation
	\begin{align*}
		\intO{ \left[ (\ff \cdot \overline{\vv}) \dive \VV + \ff \cdot \nabla \overline{\vv}^{\top} \VV + \nabla \ff^{\top} \VV \cdot \overline{\vv} \right] } 
			&= \intO{ \left( \ff \cdot \overline{\vv} \dive \VV + \VV \cdot \nabla (\ff \cdot \overline{\vv})\right) }\\
			&= \intO{ \dive \left[ (\ff \cdot \overline{\vv}) \VV \right] }.
	\end{align*}
	Thus, by the previous computation together with the divergence theorem, we may write $K_{7}$ as follows
	\begin{equation}\label{eq:K7}
		K_{7} = \intO{ [ \vect{\nabla} \cdot (\ff \otimes \VV) ] \cdot \overline{\vv} } 
			= \intS{ (\ff \cdot \overline{\vv}) \Vn } - \intO{ \ff \cdot \nabla \overline{\vv}^{\top} \VV }
			=: K_{71} + K_{72}.
	\end{equation}
	\item Rewriting $K_{5}$. Using the definition of the tangential divergence of a vector function (see Definition \ref{def:tangential_operators}) and based from the discussion given in \textit{Step 1} of the proof of Proposition \ref{prop:shape_derivative_with_sufficient_regularity_assumptions} issued in Appendix \ref{appxsubsec:shape_derivatives_of_the_cost_via_chain_rule} (refer to equation \eqref{eq:p_divergence}), we have
	\begin{align*}
		&\intO{ (\nabla \cdot \VV) p (\nabla \cdot \overline{\vv})} \\
			&\qquad = \intO{ \{ \dive[p(\nabla \cdot \overline{\vv})\VV] - (\VV \cdot \nabla)[p(\nabla \cdot \overline{\vv})] \}}\\
			&\qquad = \intS{ p(\nabla \cdot \overline{\vv})\Vn} - \intO{ (\VV \cdot \nabla)[p(\nabla \cdot \overline{\vv})]}\\
			&\qquad = \intS{ [p \dive_{\Sigma} \overline{\vv} + \dn{\overline{\vv}} \cdot (p\nn)] \Vn} - \intO{(\VV \cdot \nabla)[p(\nabla \cdot \overline{\vv})]}\\
			&\qquad= \intS{ [\kappa p\nn \Vn - \nabla_{\Sigma} (p\Vn)] \cdot \overline{\vv}} 
				+ \intS{ \dn{\overline{\vv} \cdot p\nn \Vn }}
				- \intO{(\VV \cdot \nabla)[p(\nabla \cdot \overline{\vv})]}\\
			&\qquad= \intS{ [\kappa p\nn \Vn - \nabla_{\Sigma} (p\Vn)] \cdot \overline{\vv}} 
				+  \intS{ \dn{\overline{\vv} \cdot p\nn \Vn }}\\
			&\qquad\qquad	\qquad
				- \intO{(\nabla \cdot \overline{\vv})\VV \cdot \nabla p}
				- \intO{p (\VV \cdot \nabla)(\nabla \cdot \overline{\vv})}\\
			&\qquad= \intS{ [\kappa p\nn \Vn - \nabla_{\Sigma} (p\Vn)] \cdot \overline{\vv}} 
				+ \intS{ \dn{\overline{\vv} \cdot p\nn \Vn }}\\
			&\qquad\qquad	\qquad
				+ \intO{(\VV \cdot \nabla p ) \pim}
				- \intO{p (\VV \cdot \nabla)(\nabla \cdot \overline{\vv})}\\				
			&\qquad= K_{51} + K_{52} + K_{53} + K_{54}.	
	\end{align*}
	In the last equality above (i.e., for the expression $K_{53}$), we applied the identity
	\[
		-\intO{(\VV \cdot \nabla p) (\nabla \cdot \overline{\vv})} = \intO{(\VV \cdot \nabla p)\pim},
	\] 
	which was obtained by taking $\mu = \VV \cdot \nabla p \in Q$ in the second equation of \eqref{eq:adjoint_system_weak_form}.
	\item Rewriting $K_{2}$. Using the identity $\dive_{\Sigma} \VV = \dive \VV - D\VV \nn \cdot \nn$ on $\Sigma$ and the tangential formula, we can expand $K_{2}$ as follows:
	\begin{align*}
		K_{2} &= - i \intS{ (\dive \VV + D\VV \nn \cdot \nn) ( {\uu} \cdot \nn ) ( {\overline{\vv}} \cdot \nn) }\\
			&= - i \intS{ \dive_{\Sigma} \VV ( {\uu} \cdot \nn ) ( {\overline{\vv}} \cdot \nn) }
				- 2 i \intS{ D\VV \nn \cdot \nn ( {\uu} \cdot \nn ) ( {\overline{\vv}} \cdot \nn) } \\ 
			&= - i \intS{ \Vn \kappa ( {\uu} \cdot \nn ) \nn \cdot {\overline{\vv}} } + \intS{\nabla_{\Sigma} [( {\uu} \cdot \nn ) ( {\overline{\vv}} \cdot \nn)] \cdot \VV}\\
			&\qquad	- 2 i \intS{ D\VV \nn \cdot \nn ( {\uu} \cdot \nn ) ( {\overline{\vv}} \cdot \nn) } \\ 	
			&=: K_{21} + K_{22} + K_{23}.			
	\end{align*}
	\item Rewriting $K_{1}$. For the first integral, we make use of the following formula:
\begin{equation}\label{eq:expansion}
\begin{aligned}
	&- \intO{A\nabla \varphi \cdot \nabla \overline{\psi} }\\
	&\qquad = - \intO{ (\Delta \varphi) \VV \cdot \nabla \overline{\psi} }
		- \intO{ (\Delta \overline{\psi}) \VV \cdot \nabla \varphi }
		+ \intS{\dn{\varphi}(\VV \cdot \nabla \overline{\psi})} \\
	&\qquad  \quad \qquad
		 + \intS{\dn{\overline{\psi}}(\VV \cdot \nabla \varphi)}
		- \intS{(\nabla \overline{\psi} \cdot \nabla \varphi)\Vn},
\end{aligned}
\end{equation}
which holds for all functions $\varphi$, $\psi \in \HHg(\Omega) \cap \HH^2(\Omega)$ and $\VV \in \sfTheta^{1}$ (see, e.g., \cite{DelfourZolesio2011} or the proof of Lemma 32 in \cite{BacaniPeichl2013}).
	Hence, with reference to equation \eqref{eq:state_tested} and \eqref{eq:adjoint_tested}, we get the following computations
	\begin{align*}
		&- \intO{ \alpha A \nabla {\uu} : \nabla {\overline{\vv}}} = \sum_{j=1}^{d} \intO{ \alpha A \nabla u_{j} \cdot \nabla \overline{v}_{j} }\\
	&\qquad = \sum_{j=1}^{d} \Bigg\{ \intO{ (- \alpha \Delta {u}_{j}) \VV \cdot \nabla \overline{{v}}_{j} }
		- \intO{ (\alpha \Delta \overline{{v}}_{j}) \VV \cdot \nabla {u}_{j} }
		+ \intS{\alpha\dn{{u}_{j}}(\VV \cdot \nabla \overline{{v}}_{j})} \\
	&\qquad  \quad \qquad
		 + \intS{\alpha\dn{\overline{{v}}_{j}}(\VV \cdot \nabla {u}_{j})}
		- \intS{\alpha(\nabla \overline{{v}}_{j} \cdot \nabla {u}_{j})\Vn} \Bigg\}\\
	%
	%
	%
	%
	%
	%
	%
	&\qquad =  \intO{\ff \cdot \nabla \overline{\vv}^{\top} \VV} - i \intS{(\uu \cdot \nn)\nn \cdot \nabla \overline{\vv}^{\top} \VV} \\
	&\qquad  \quad \qquad 
		+ \intO{\uu_{i} \cdot \nabla \uu^{\top} \VV} - i \intS{(\overline{\vv} \cdot \nn)\nn \cdot \nabla \uu^{\top} \VV}\\
	&\qquad  \quad \qquad
		- \intS{\alpha(\nabla_{\Sigma} \uu : \nabla_{\Sigma} \overline{\vv} )\Vn}
		- \intS{\alpha \dn{\uu} \cdot \dn{\overline{\vv}} \Vn} \\
	&\qquad  \quad \qquad		
		+ \intO{ p \dive ( \nabla \overline{\vv}^{\top} \VV ) } 
		+ \intO{ \overline{q} \dive ( \nabla \uu^{\top} \VV ) } \\						
	& \qquad =: K_{11} + K_{12} + K_{13} + K_{14} + K_{15} + K_{16} + K_{17} + K_{18}.	
	\end{align*}
	Notice here that $K_{11}$ will cancel out with $K_{72}$.
	Moreover, in light of Lemma \ref{lem:divergence_expansion}, the sum $K_{6} + K_{54} + K_{17}$ actually equates to zero.
	Similarly, using equation \eqref{eq:divergence_expansion} from the said lemma, we know that 
	\[
		K_{18} = \intO{ \overline{q} \dive ( \nabla \uu^{\top} \VV ) } = \intO{\overline{q} \left[ (\VV \cdot \nabla) (\nabla \cdot \uu) + D\VV : \nabla \uu \right]}.
	\]
	However, since $\dive \uu = 0$ in $\Omega$, then we actually have $K_{18} = \intO{ \overline{q} D\VV : \nabla \uu }$.
	This integral, in addition, also vanishes when combined with $K_{8}$. 
\end{itemize}
In above, we have obtained some expression where in we can utilize Lemma \ref{lem:useful_identity}.
Looking at the sum $K_{3} + K_{14}$, we can apply the identity \eqref{eq:useful_identity_second_formula} to get the following equation
\[
\begin{aligned}
	K_{3} + K_{14} 
	&= i \intS{(\overline{\vv} \cdot \nn) \nn \cdot D\VV \uu} - i\intS{(\overline{\vv} \cdot \nn)\nn \cdot \nabla \uu^{\top}\VV}\\
	&= i \intS{ (\overline{\vv} \cdot \nn) \dive_{\Sigma}(\Vn \uu) } 
		+ i \intS{ (\overline{\vv} \cdot \nn)\nn \cdot (D\VV \nn \otimes \nn)\uu }\\
	&\qquad	- i \intS{(\overline{\vv} \cdot \nn) \VV \cdot \nabla_{\Sigma}(\uu \cdot \nn)}\\
	&=: J_{1} + J_{2} + J_{3}.
\end{aligned}
\]
Meanwhile, for the sum $K_{4} + K_{12}$, we have
\[
\begin{aligned}
	K_{4} + K_{12} 
	&= i \intS{(\uu \cdot \nn) \nn \cdot D\VV \overline{\vv}} - i\intS{(\uu \cdot \nn)\nn \cdot \nabla \overline{\vv}^{\top}\VV}\\
	&= i \intS{ (\uu \cdot \nn) \nabla_{\Sigma}\Vn \cdot \overline{\vv} } 
		+ i \intS{ (\uu \cdot \nn)\nn \cdot (D\VV \nn \otimes \nn)\overline{\vv} }\\
	&\qquad \qquad \qquad 
		- i \intS{(\uu \cdot \nn) \VV \cdot \nabla_{\Sigma}(\overline{\vv} \cdot \nn)}
		- i \intS{ (\uu \cdot \nn)\nn \cdot \dn{\overline{\vv}} \Vn }\\
	&=: H_{1} + H_{2} + H_{3} + H_{4}.	
\end{aligned}
\]
Combining this with the sum $K_{16} + K_{52}$ and noting that $- p \nn + \alpha \dn{\uu} + i (\uu \cdot \nn)\nn = 0$ on $\Sigma$, we further get
\[
\begin{aligned}
	K_{4} + K_{12} + K_{16} + K_{52}
	& = H_{1} + H_{2} + H_{3} - \intS{ [ - p \nn + \alpha \dn{\uu} + i (\uu \cdot \nn)\nn ] \cdot \dn{\overline{\vv}} \Vn }\\
	& = H_{1} + H_{2} + H_{3}.	
\end{aligned}
\]
At this point, we underline the observation that the sum $K_{22} + J_{3} + H_{3}$ vanishes.
Moreover, employing Lemma \ref{lem:useful_equivalence}, we also observe that
\begin{align*}
	&i \intS{ (\overline{\vv} \cdot \nn)\nn \cdot (D\VV \nn \otimes \nn)\uu }
	+ i \intS{ (\uu \cdot \nn)\nn \cdot (D\VV \nn \otimes \nn)\overline{\vv} } \\
	&\qquad\qquad - 2 i \intS{ D\VV \nn \cdot \nn ( {\uu} \cdot \nn ) ( {\overline{\vv}} \cdot \nn) } = 0,
\end{align*}
i.e., the sum $J_{2} + H_{2} + K_{23}$ also disappears.

Finally, summarizing our computations -- and after applying the tangential Green's formula (see Lemma \ref{lem:tangential_formulas}) and using the fact that $\nabla_{\Sigma} \uu \nn = \vect{0}$ (see Lemma \ref{lem:tangential_times_normal_vector}) -- we get the following equivalent expression for the sum $\dot{J}_{2}(0) + \dot{J}_{6}(0) + \dot{J}_{4}(0) + \dot{J}_{8}(0)$ (recall equations \eqref{eq:distributed_gradient} and \eqref{eq:distributed_gradient}):
\[
\begin{aligned}
		&\dot{J}_{2}(0) + \dot{J}_{6}(0) + \dot{J}_{4}(0) + \dot{J}_{8}(0)\\
		&\quad = - \intO{ [\uu_{i} \cdot \nabla \uu_{i}^{\top} \VV + \pim (\VV \cdot \nabla \pim) ] }
				+ \frac{1}{2}  \intS{ ( \abs{\uu_{i}}^{2} + \abs{\pim}^{2}) \Vn } \\
		&\qquad	+  \Im\Bigg\{ \intO{\uu_{i} \cdot \nabla \uu^{\top} \VV} + \intO{ \pim (\VV \cdot \nabla p )}\\
		&\qquad\quad	+ \intO{ \Big[ \ff \Vn - \nabla_{\Sigma} (p\Vn) + \dive_{\Sigma}{[\alpha (\nabla_{\Sigma} \uu) \Vn]} + i \dive_{\Sigma} (\Vn \uu)\nn
		+ i ({\uu} \cdot \nn) \nabla_{\Sigma}\Vn \Big] \cdot \overline{\vv} }
		\Bigg\}\\
		&\quad = \intS{ \left[ \Im\left\{  \vect{B}[\Vn] \cdot \overline{\vv} \right\} + \frac12 \left( |\ui|^2 + |\pim|^2 \right) \Vn \right] },
	\end{aligned}
\]
as desired.
This ends the proof of Proposition \ref{prop:the_shape_derivative_of_the_cost}.
\end{proof}
\begin{remark}
	The computed expression \eqref{eq:shape_gradient} for the shape gradient of $J$ obtained via rearrangement method clearly agrees with the structure of the same derivative derived through the classical chain rule approach issued in Appendix \ref{appxsubsec:shape_derivatives_of_the_cost_via_chain_rule}.
	In fact, after some manipulations, the expression can equivalently be expressed as
	\[
	\vect{B}[\Vn] = \left[ \alpha\nabla \uu + ({\uu} \cdot \nn - p) id \right] \nabla_{\Sigma}\Vn
			- \left[ - \dn{p} \nn + \dnn{\uu} + i {(\dn{\uu} \cdot \nn) \nn} \right] \Vn -i \left({\uu} \cdot \nabla_{\Sigma}\Vn \right) \nn;
	\]
	see \eqref{eq:shape_derivative_boundary_condition} in Appendix \ref{appxsubsec:shape_derivatives_of_the_cost_via_chain_rule}.
\end{remark}
The following results can be drawn easily from \eqref{eq:shape_gradient}, \eqref{eq:adjoint_system}, and Remark \ref{rem:equivalence}.
\begin{corollary}[Necessary optimality condition]\label{cor:necessary_condition}
	Let the domain $\Omega^\ast$ be such that the state $(\uu, p)=(\uu(\Omega^\ast),p(\Omega^\ast))$, where $\uu = \ur + i \ui$ and $p = p_{r} + i \pim$, satisfies the free surface problem \eqref{eq:FSP}, i.e., we have
	\[
		-p\nn + \alpha \dn{\uu}= \vect{0}	
		\quad\text{and} \quad
		\uu \cdot \nn = 0 \quad \text{on $\Sigma^\ast$},
	\]
	or equivalently,
	\begin{equation}
	\label{eq:imaginary_is_zero}
		\ui = \vect{0} \quad \text{and} \quad \pim = 0 \quad \text{on $\Omega^\ast$},
	\end{equation}
	with $(\uu, p)$ satisfying \eqref{eq:ccbm}.
	Then, the domain $\Omega^\ast$ is stationary for the shape problem \eqref{eq:shape_optimization_problem} (the minimization problem $J(\Omega) = \frac12 \intO{\left( |\ui|^2 + |\pim|^2 \right)} \to \inf$, where $(\ui,\pim)$ is subject to the state problem \eqref{eq:ccbm}).
	That is, it fulfills the necessary optimality condition
	\begin{equation}
	\label{eq:optimality_condition}
		{d}J(\Omega^\ast)[\VV] = 0, \quad \text{for all $\VV \in \sfTheta^{1}$}.
	\end{equation}
\end{corollary}
\begin{proof}
	By the assumption that $\ui = \vect{0}$ and $\pim = 0$ on $\Omega^\ast$, one obtains -- in view of \eqref{eq:adjoint_system} -- that $\vect{v}_{i} = \vect{0}$ and $q_{i}=0$ on $\Omega^\ast$.
	Thus, it follows that $g_{\Sigma} = 0$ on $\Sigma^\ast$ which therefore gives us the conclusion ${d}J(\Omega^\ast)[\VV] = 0$, for any $\VV \in \sfTheta^{1}$.
\end{proof}
In connection with the previous conclusion, we point out that the solutions of the necessary condition \eqref{eq:optimality_condition} might exist such that the states does not satisfy equation \eqref{eq:imaginary_is_zero} (refer, for example to Figure \ref{fig:figure2} and Figure \ref{fig:figure3} in terms of numerics).
Nonetheless, only in the case of exact matching of boundary data a stationary domain $\Omega^\ast$ is a global minimum because $J(\Omega^\ast)=0$.
%
\section{Numerical approximation}
\label{sec:Numerical_Approximation}
The numerical resolution to our proposed shape optimization method to \eqref{eq:FSP} is realized through a Sobolev gradient-based method.
We implement the idea of the said approach via finite element method, along the lines of the author's previous work, see \cite{RabagoAzegami2020}.
Accordingly, this section is divided into two main subsections.
In the first part, we provide some details of the algorithm and then in the second part we give a concrete test example of the problem.

\subsection{Numerical algorithm}
\label{subsec:Numerical_Algorithm}
For completeness, we give below the important details of our algorithm.

\textit{Choice of descent direction.} 
As alluded above, the problem will be solved via finite element method. 
To this end, we make use of the Riesz representation of the shape gradient \eqref{eq:shape_gradient}.
The reason for this is that, the gradient given by $\ggb$ is only supported on the free boundary $\Sigma$, and if we use directly $\ggb\nn$ as the descent vector $\VV$, the perturbation may produce some unwanted oscillations on the free boundary which, of course, we do not want to happen.
So, we apply an extension-regularization technique by taking the descent direction $\VV$ as the solution in $\vect{H}^1_{\Gamma, \vect{0}}(\Omega)$ to the variational problem
\begin{equation}\label{eq:H1_gradient_method}
	a(\VV,\vect{\varphi})= - \intS{\ggb\nn \cdot \vect{\varphi}}, \quad  \text{for all $\vect{\varphi} \in \vect{H}^1_{\Gamma, \vect{0}}(\Omega)$},
\end{equation}
where $a$ is the ${\vect{H}^1(\Omega)}(:=H^1(\Omega;\mathbb{R}^d)$)-inner product in $d$-dimension, $d \in \{2, 3\}$.
In this manner, the \textit{Sobolev gradient} $\VV$ \cite{Neuberger1997} becomes a smoothed preconditioned extension of $-\ggb\nn$ over the entire domain $\Omega$.
Further discussion on discrete gradient flows for various shape optimization problems can be found in \cite{Doganetal2007}.

\textit{The main algorithm.}
To compute the $k$th domain $\Omega^{k}$, a typical procedure is given as follows:
\begin{description}
\setlength{\itemsep}{0.5pt}
	\item[1. \it{Initilization}] Choose an initial guess for $\Sigma^{0}$ (this gives us $\Omega^{0}$).  
	\item[2. \it{Iteration}] For $k = 0, 1, 2, \ldots$, do the following
		\begin{enumerate}
			\item[2.1] solve the state and adjoint state systems on the current domain $\Omega^{k}$;
			\item[2.2] choose $t^{k}>0$, and compute the Sobolev gradient $\VV^{k}$ on $\Omega^{k}$;
			\item[2.3] update the current domain by setting $\Omega_{k+1} = (\operatorname{id}+t^{k}\VV^{k})\Omega^{k}$. 
		\end{enumerate}
	\item[3. \it{Stop Test}] Repeat the \textit{Iteration} part until convergence.
\end{description}

\begin{remark}[Step-size computation]\label{rem:step_size}
In Step 2.2, the step size $t^{k}$ is computed via a backtracking line search procedure using the formula $t^{k} = \mu J(\Omega^{k})/|\VV^{k}|^2_{\mathbf{H}^1(\Omega^{k})}$ at each iteration, where $\mu > 0$ is a given real number.
This procedure is based on an Armijo-Goldstein-like condition for shape optimization method proposed in \cite[p. 281]{RabagoAzegami2020}.
Here, however, the step size parameter $\mu$ is not limited to the interval $(0,1)$ in contrast to \cite{RabagoAzegami2020}.
We do this to maximize the initial choice of the length of the descent vector and simply rely on the backtracking procedure to get a suitable step size.
\end{remark}

\begin{remark}[Evaluating the mean curvature] In Step 2.2, the mean curvature $\kappa$ of $\Sigma$ is evaluated as the divergence of some vector $\NN$, where $\NN \in {\vect{H}^1(\Omega)}$ is an extension of $\nn$ satisfying the equation $c_\text{N}  \intO{\nabla \vect{N}:\nabla \vect{\varphi}} + \intS{ \vect{N}\cdot \vect{\varphi} } = \intS{ \nn \cdot \vect{\varphi}}$, for all $\vect{\varphi} \in {\vect{H}^1(\Omega)}$, for some small real number $c_\text{N} >0$\footnote{This follows the same idea in computing the Riesz representation of the shape gradient.}.
We point that it is a requirement that $\NN$ is a \textit{unitary} extension of $\nn$ for the formula $\kappa = \dive_{\Sigma} \nn = \dive \NN - (\nabla \NN \nn)\nn = \dive \NN$ (on $\Sigma$) to hold; see, e.g., \cite[Lem. 2.14, p. 92]{SokolowskiZolesio1992}.
So, the said extension of the (outward) unit normal vector to $\Sigma$ may not be accurate.
However, the approximation $\nabla \cdot \NN$ of $\kappa$ is enough to get a good descent direction for the optimization procedure.
Moreover, such strategy easily applies to three dimensional cases.
A more accurate numerical computation of the mean curvature is of course possible, but here we are satisfied with our results using this approximation technique for $\kappa$.
In our numerical experiments, we take $c_\text{N} = 10^{-8}$.
\end{remark}

\begin{remark}[Stopping condition] \label{rem:stopping_condition}
We terminate the algorithm after it reached a maximum number of iterations.
\end{remark}

\begin{remark}[Incorporating the shape Hessian information in the numerical procedure]\label{rem:Newton_method}
The convergence behavior of our scheme can of course be improved by incorporating the shape Hessian information in the numerical procedure.
The drawback, however, of a second-order method is that, typically, it demands additional computational burden and time to carry out the calculation, especially when the Hessian is complicated \cite{NovruziRoche2000,Simon1989}. 
Here, we will not employ a second-order method to numerically solve the optimization problem as the first-order method already provides an excellent approximation of the optimal solution.
\end{remark}

\subsection{Numerical examples}\label{subsec:Numerical_Examples}
Let us now illustrate the feasibility and applicability of the coupled complex boundary method in the context of shape optimization for solving concrete examples of the free surface problem \eqref{eq:FSP}.
We first test our method in two dimensions, and carry out a comparison with the least-squares method of tracking the Dirichlet data (see Problem  \eqref{eq:tracking_Dirichlet_method}).
Also, since most of the previous studies in free surface problems (see \cite{Kasumba2014}) only dealt with problems in two dimensions, we also put our attention on testing the proposed method to three dimensional cases. 
\begin{remark}[Details of the computational setup and environment]\label{rem:details_or_algorithm}
The numerical simulations conducted here are all implemented in the programming software \textsc{FreeFem++}, see \cite{Hecht2012}.
Every variational problems involved in the procedure (except those that correspond to \eqref{eq:ccbm} and \eqref{eq:adjoint_system}) are solved using $\mathbb{P}_{1}$ finite element discretization. 
Moreover, all mesh deformations are carried out without any kind of adaptive mesh refinement as opposed to what has been usually done in earlier works, see \cite{RabagoAzegami2020}.
In this way, we can assess the stability of CCBM in comparison with the classical approach of tracking the Dirichlet data in least-squares sense.
The computations are all performed on a MacBook Pro with Apple M1 chip computer having 16GB RAM processors.
\end{remark}
\begin{remark}
We note that the advantage of employing a Sobolev-gradient method is that we can avoid the generation of a new triangulation of the domain at every iterative step. 
This is achieve by moving not only the boundary nodes, but also the internal nodes of the mesh triangulation at every iteration.
Doing so only requires one to generate an initial profile for the computational mesh. 
To carry out such approach, one simply need to solve the discretized version of \eqref{eq:H1_gradient_method} and then move the domain in the direction of the resulting vector field scaled with the pseudo-time step size $t^{k}$.
That is, we find ${{\VV_{h}^{k}}} \in \mathbb{P}_{1}(\Omega_{h}^{k})^{d}$ which satisfies the following system of equations
\[
	- \Delta {\VV_{h}^{k}} + {\VV_{h}^{k}}  		=  \vect{0} \ \ \text{in $\Omega_{h}^{k} $},\qquad
		{\VV_{h}^{k}} 				=  \vect{0} \ \ \text{on $\Gamma^h$},\qquad
 		{\nabla{\VV_{h}^{k}}} \cdot {\nn_{h}^{k}}	= -\ggb^{k} {\nn_{h}^{k}} \ \ \text{on $\Sigma_{h}^{k}$},
\]
where we suppose a polygonal domain $\overline{\Omega_{h}^{k}}$ and its triangulation $\mathcal{T}_h(\overline{\Omega_{h}^{k}}) = \{ K^{k}_{l} \} ^{N_e}_{l=1}$ ($K^{k}_{l}$ is a closed triangle for $d=2$, or a closed tetrahedron for $d=3$) are given, and $\mathbb{P}_1(\Omega_{h}^{k})^{d}$ denotes the $\mathbb{R}^d$-valued piecewise linear function space on $\mathcal{T}_h(\overline{\Omega_{h}^{k}})$.
Accordingly, we update the domain or, equivalently, move the nodes of the mesh by defining $\Omega_{k+1}^h$ and $\mathcal{T}_h(\overline{\Omega_{k+1}^h}) = \{ K^{k+1}_{l} \} ^{N_e}_{l=1}$ respectively as $\overline{\Omega^{k+1}_h}:=\left \{ x + t^k {\VV_{h}^{k}}(x) \ \middle\vert \ x \in \overline{\Omega_{h}^{k}} \right\}$ and $K^{k+1}_{l} := \left\{ x + t^k {\VV_{h}^{k}}(x) \ \middle\vert \ x \in K^{k}_l \right\}$, for all $k = 0,1,\ldots$. 
\end{remark}
We are now ready to give our first numerical example.
\subsubsection{Example in two dimensions}
In this example, we consider a gravity-like force $\ff = (-10x,-10y)$ that keeps the fluid on top of the circular domain of radius $0.4$.
The fluid flowing in the domain is triggered by an initial velocity.
As the fluid comes to rest, a free surface position that is concentric with the circular domain, is attained.
On the fixed boundary $\Gamma$, a homogenous Dirichlet boundary conditions is imposed (i.e., $\vect{g} = \vect{0}$), and we set the value of $\alpha$ to $0.01$. 
Meanwhile, we set the initial location (or initial guess) of the fluid free-boundary $\Sigma^{0}$ as follows
\[
	\Sigma^{0} := \left\{ (x,y) \in \mathbb{R}^{2} \mid \frac{x^2}{1} + \frac{y^{2}}{1.1^{2}} = 1\right\}.
\]

We discretize the computational domain with triangular elements generated by the (built-in) bi-dimensional anisotropic mesh generator of \textsc{FreeFEM++} \cite{Hecht2012}.
Here, the exterior (free) boundary is discretized with $70$ nodal points while the interior (fixed) boundary is discretized with $30$ discretization points.
The Stokes equations \eqref{eq:ccbm} (as well as the adjoint system \eqref{eq:adjoint_system}) are then discretized using the Galerkin finite-element method.
We use Taylor-Hood elements ($\mathbb{P}_{2}$-$\mathbb{P}_{1}$ finite elements) for the approximation of the velocity and pressure.
As a result, we obtain a set of linear algebraic equations that may be written in matrix form as $\vect{K} \bar{\vect{u}} = \vect{F}$
where $\vect{K}$ is the global system matrix, $\bar{\vect{u}}$ is the global vector of unknowns (velocities and pressures),
and $\vect{F}$ is a vector that includes the effects of body forces and boundary conditions.
This linear system is solved using the default \textsc{LU}-solver of \textsc{FreeFEM++}.

The results of the experiment are shown in Figure \ref{fig:figure1} which depict a cross comparison of the computed shapes (top figure), histories of cost values (left upper plot), and the histories of $H^{1}$ gradient norms (left lower plot).
Observe that the solution attained by the proposed approach CCBM coincides with the solution of the classical least-squares approach TD (tracking-Dirichlet).
Moreover, the convergence behavior of the two methods are actually comparable as evident in the plots of the histories of cost values and $H^{1}$ gradient norms.
At this point, we do not see much advantage of using the proposed method.
Nonetheless, the advantage we claim will be more apparent when it is applied in the case of three dimension.

Meanwhile, we also plotted the initial and final imaginary parts of the Stokes velocity and pressure profiles in Figure \ref{fig:figure2} as well as the corresponding values for the adjoint solutions in Figure \ref{fig:figure3}. 
Clearly, the imaginary parts reduced to very small magnitudes as it obtains an approximation shape solution.
For the computed Stokes and adjoint pressure at the approximate optimal solution, the (maximum) magnitudes are or order $10^{-5}$ and $10^{-7}$, respectively.
This numerically confirms Remark \ref{rem:equivalence} and the optimality condition issued in Corollary \ref{cor:necessary_condition}.
%
%
\begin{figure}[htp!]
\centering
\begin{multicols}{2}
\resizebox{1.0\linewidth}{!}{\includegraphics{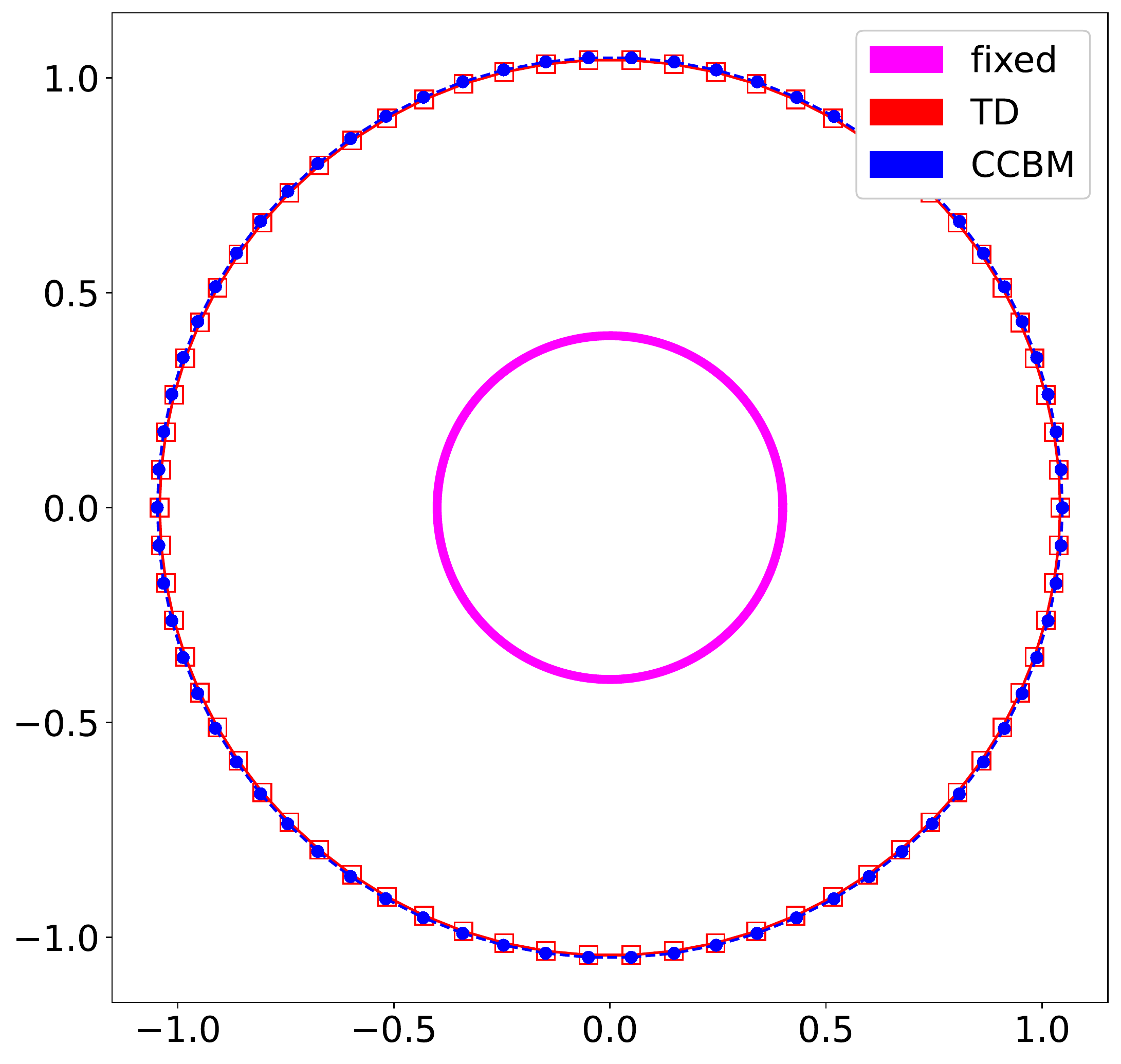}} \\
\resizebox{0.9\linewidth}{!}{\includegraphics{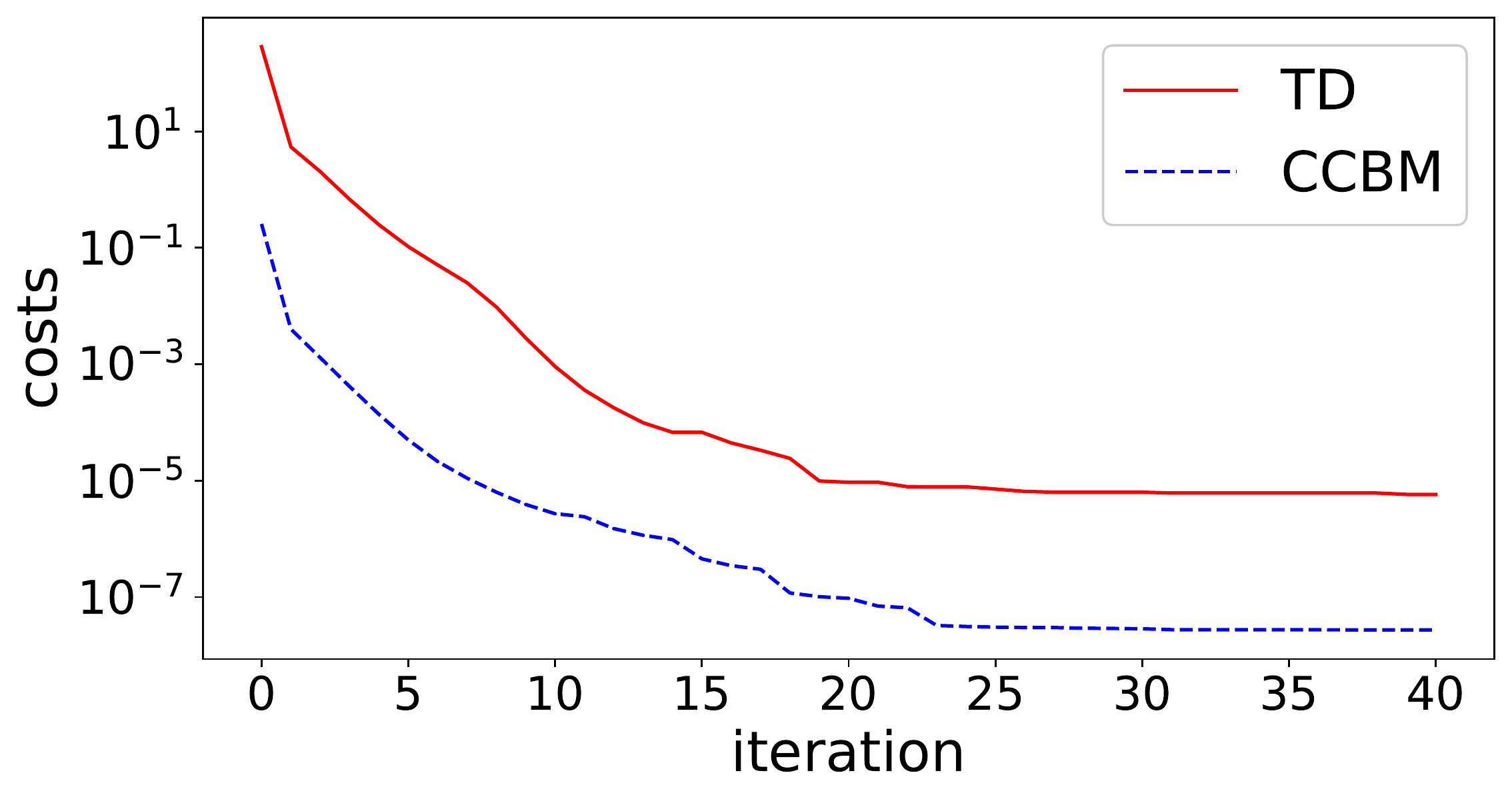}} \\[0.2em]
\resizebox{0.9\linewidth}{!}{\includegraphics{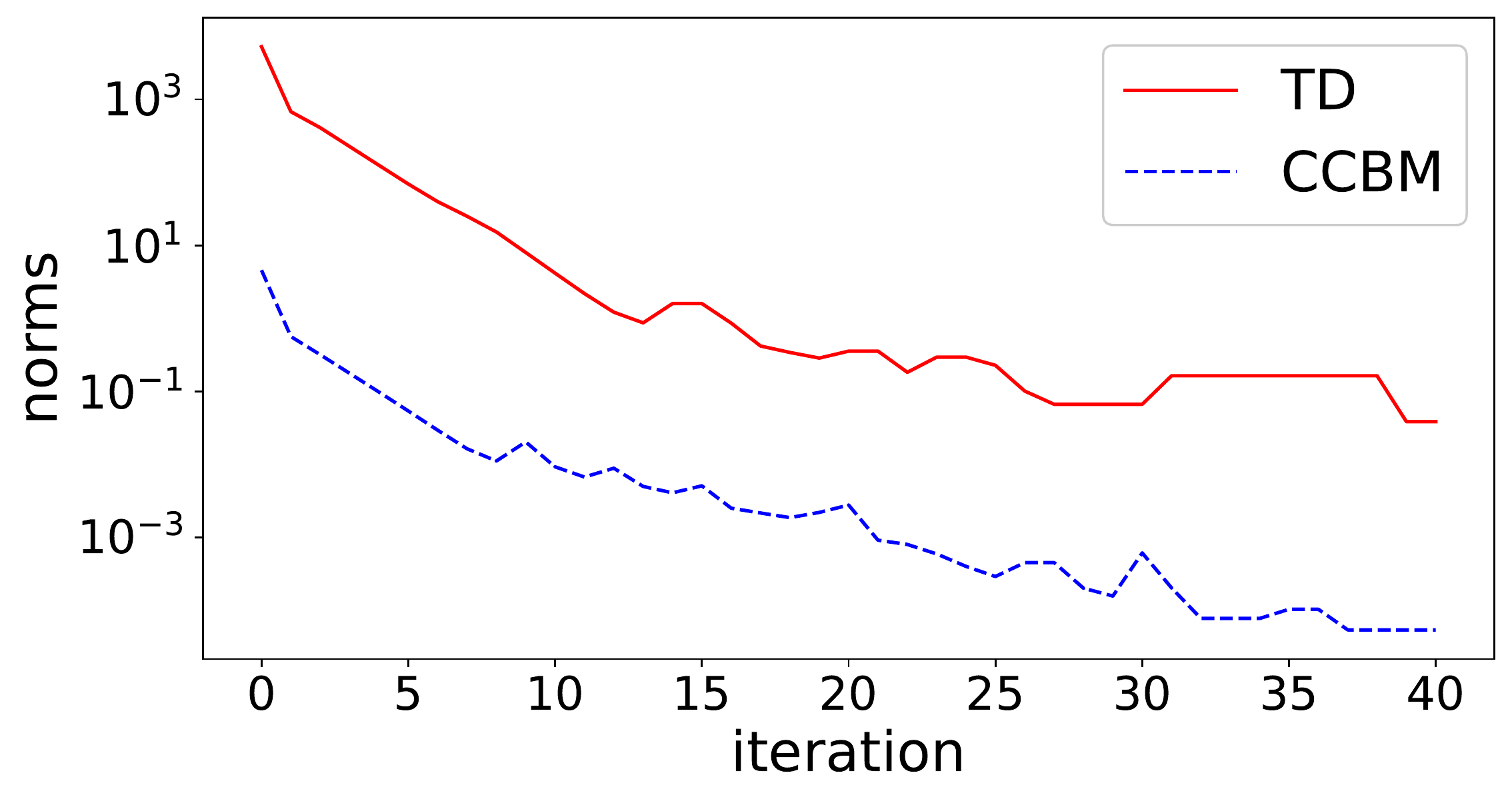}} 
\end{multicols}
\caption{Cross comparison of computed shapes (left figure), histories of cost values (left upper plot), and $H^{1}$ gradient norms (left lower plot)}
\label{fig:figure1}
\end{figure}
%
%
\begin{figure}[htp!]
\centering
\resizebox{0.24\linewidth}{!}{\includegraphics{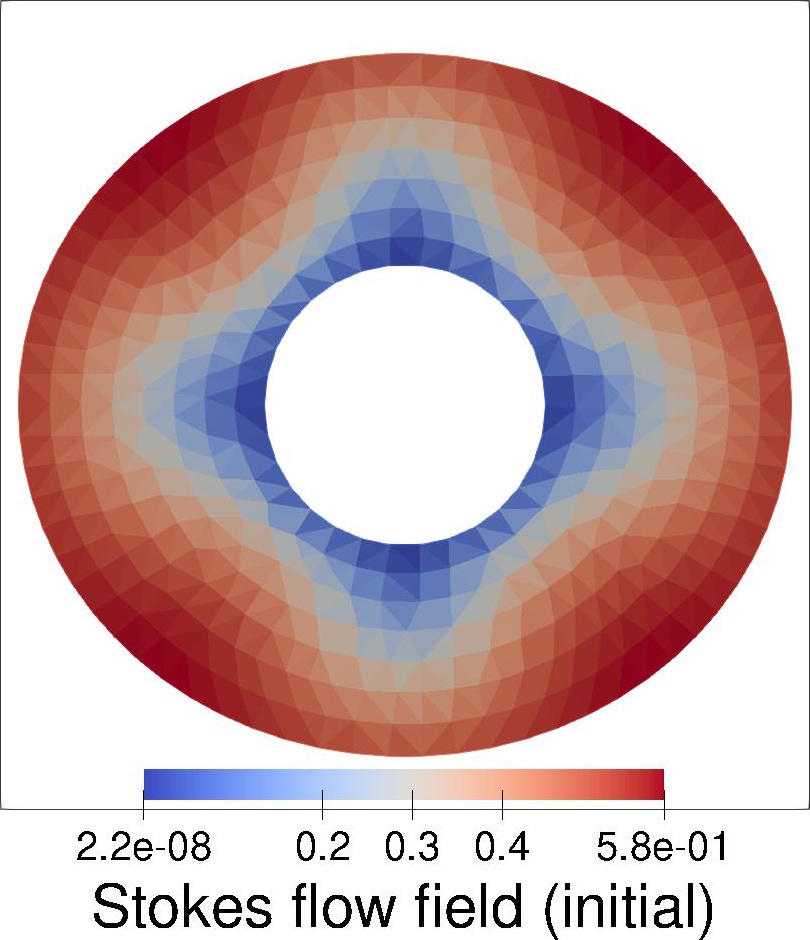}} \hfill
\resizebox{0.24\linewidth}{!}{\includegraphics{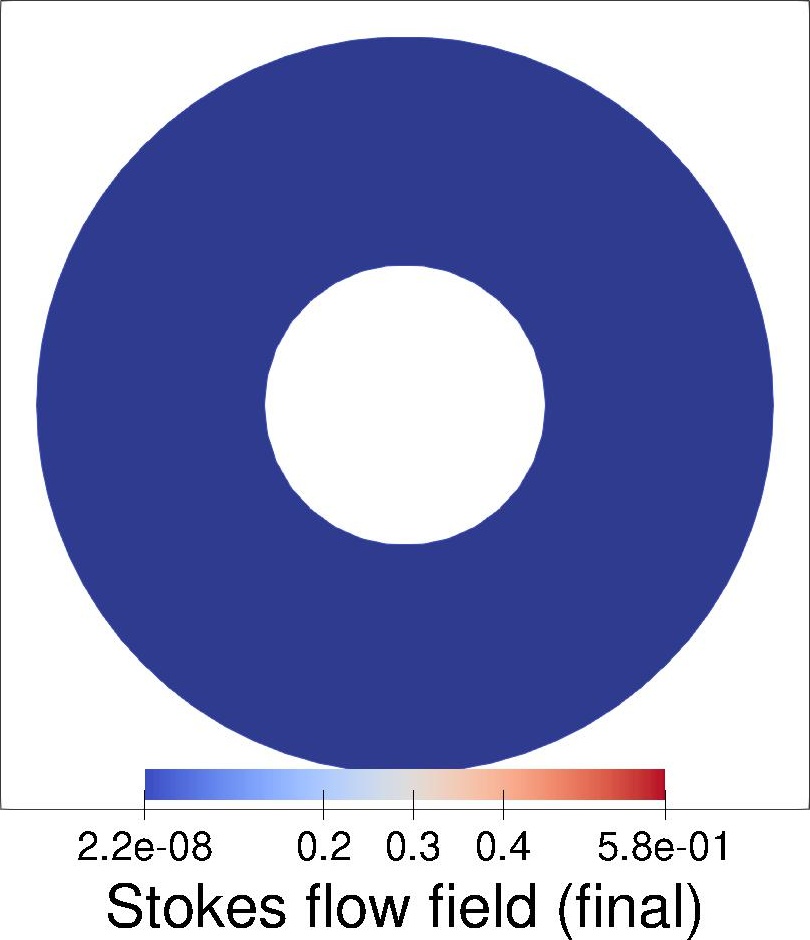}} \hfill
\resizebox{0.24\linewidth}{!}{\includegraphics{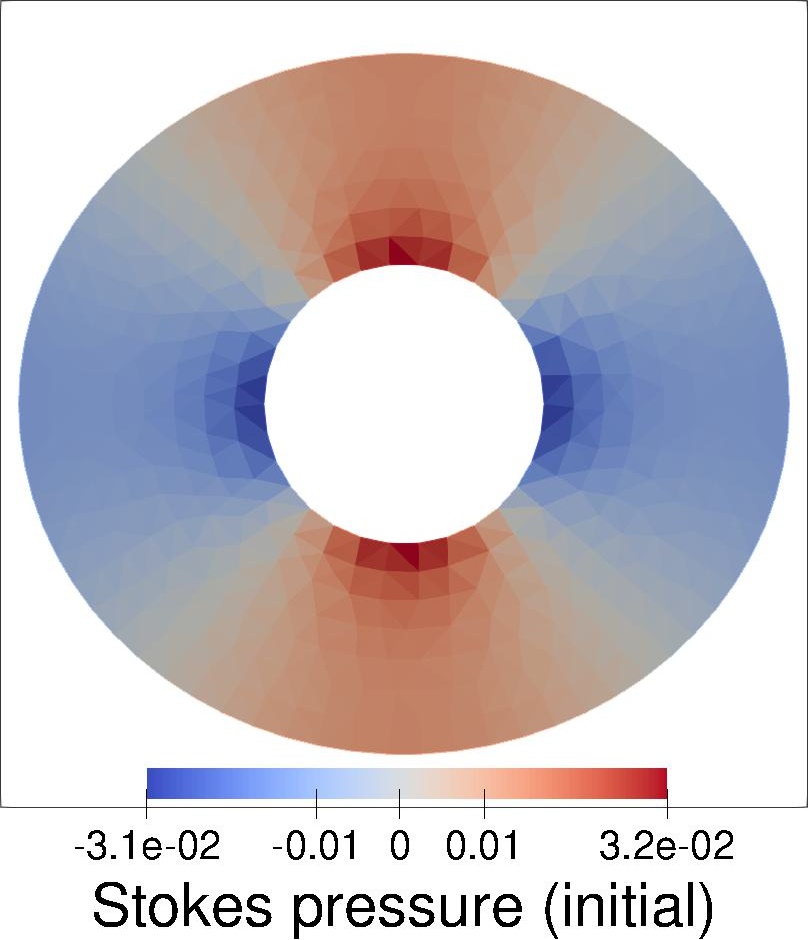}} \hfill
\resizebox{0.24\linewidth}{!}{\includegraphics{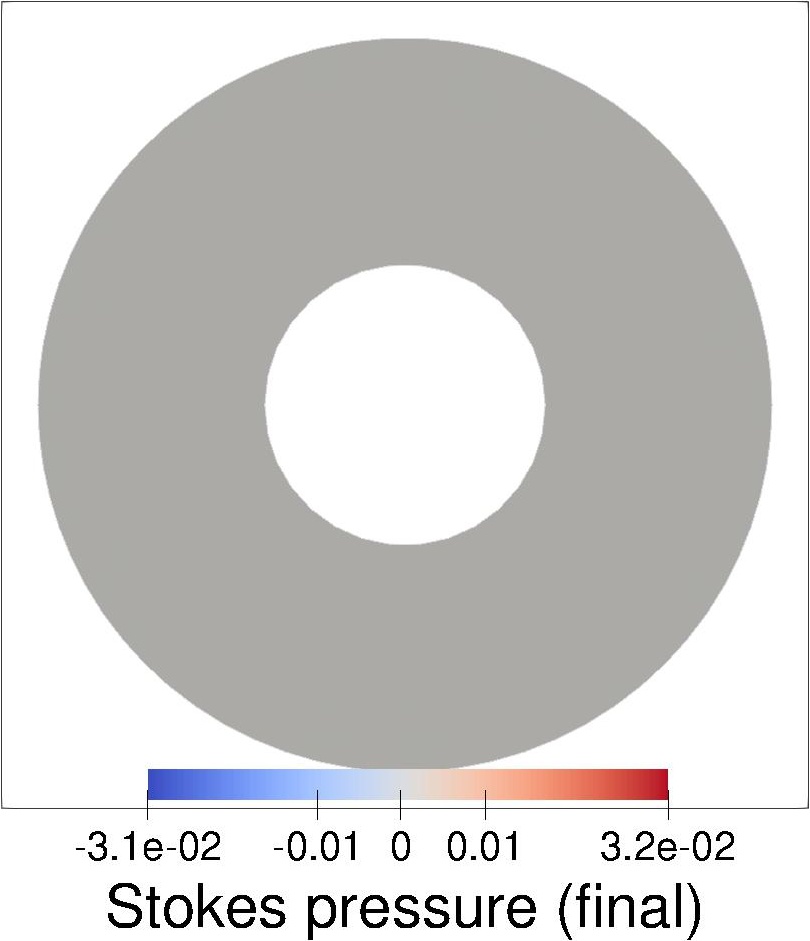}}
\caption{Imaginary parts of the initial and final flow field and pressure of the Stokes system}
\label{fig:figure2}
\end{figure}
%
%
\begin{figure}[htp!]
\centering
\resizebox{0.24\linewidth}{!}{\includegraphics{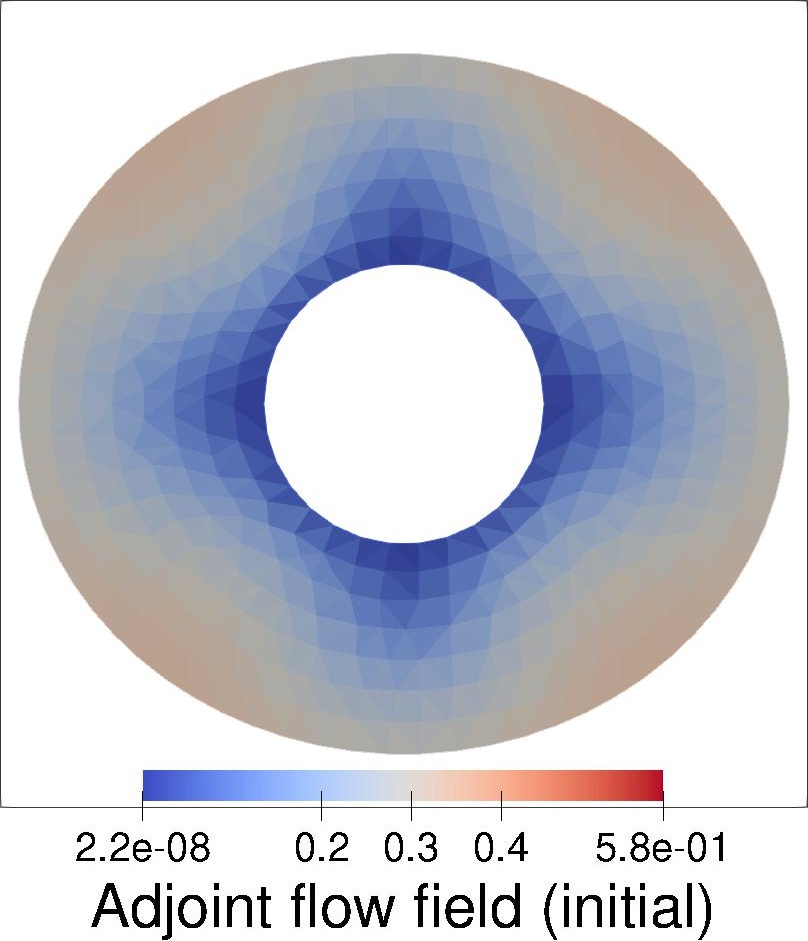}} \hfill
\resizebox{0.24\linewidth}{!}{\includegraphics{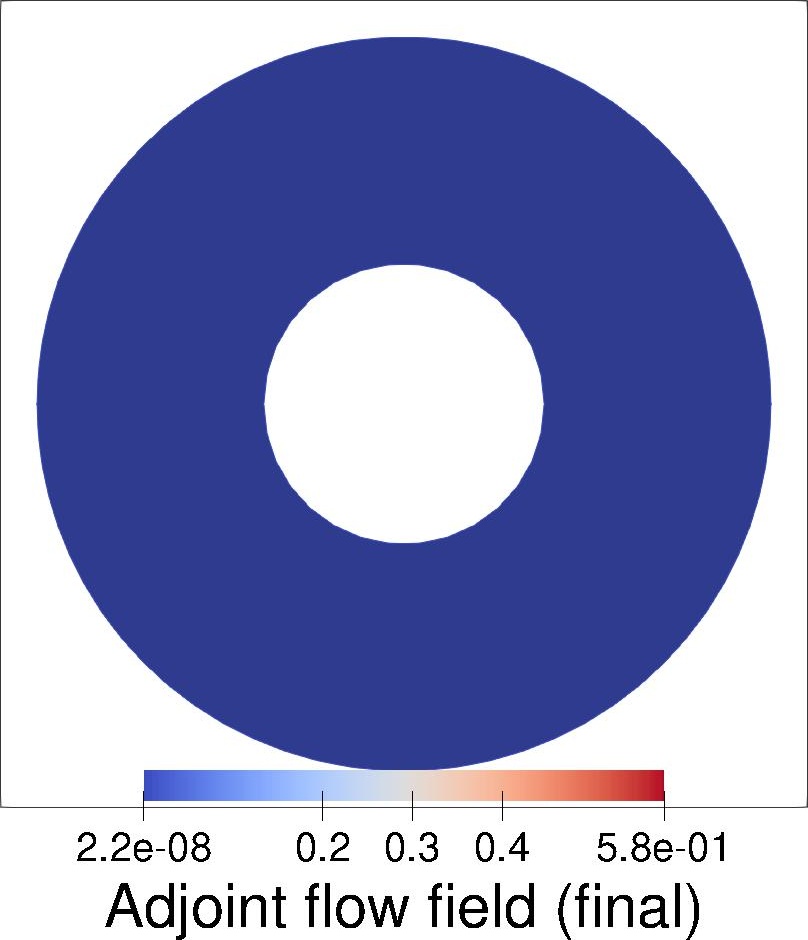}} \hfill
\resizebox{0.24\linewidth}{!}{\includegraphics{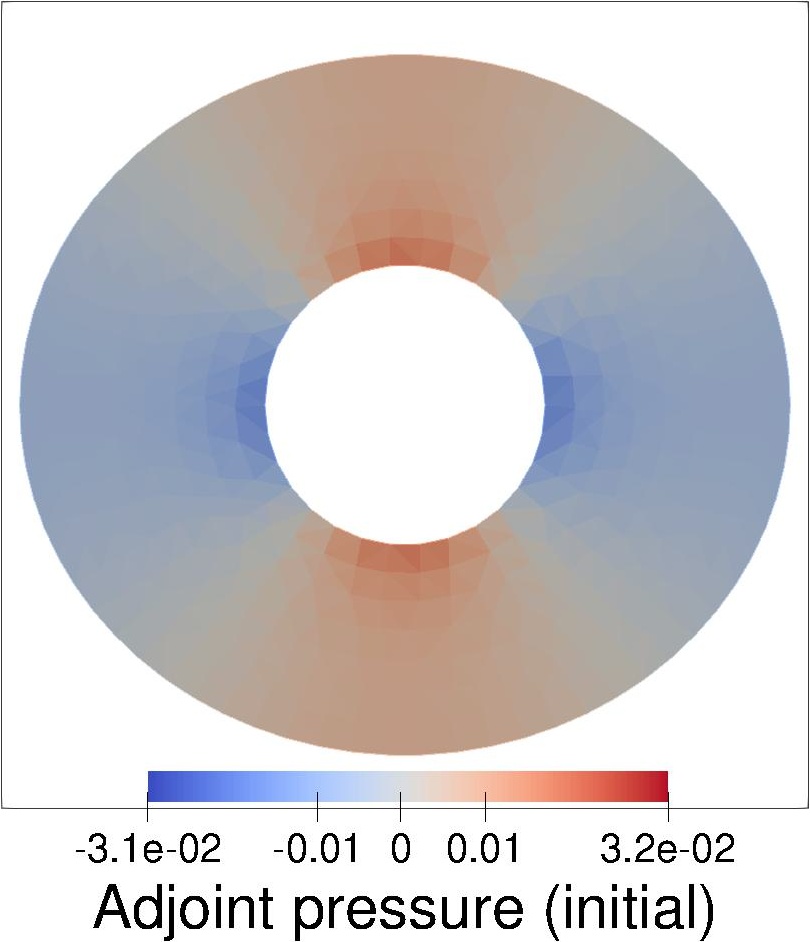}} \hfill
\resizebox{0.24\linewidth}{!}{\includegraphics{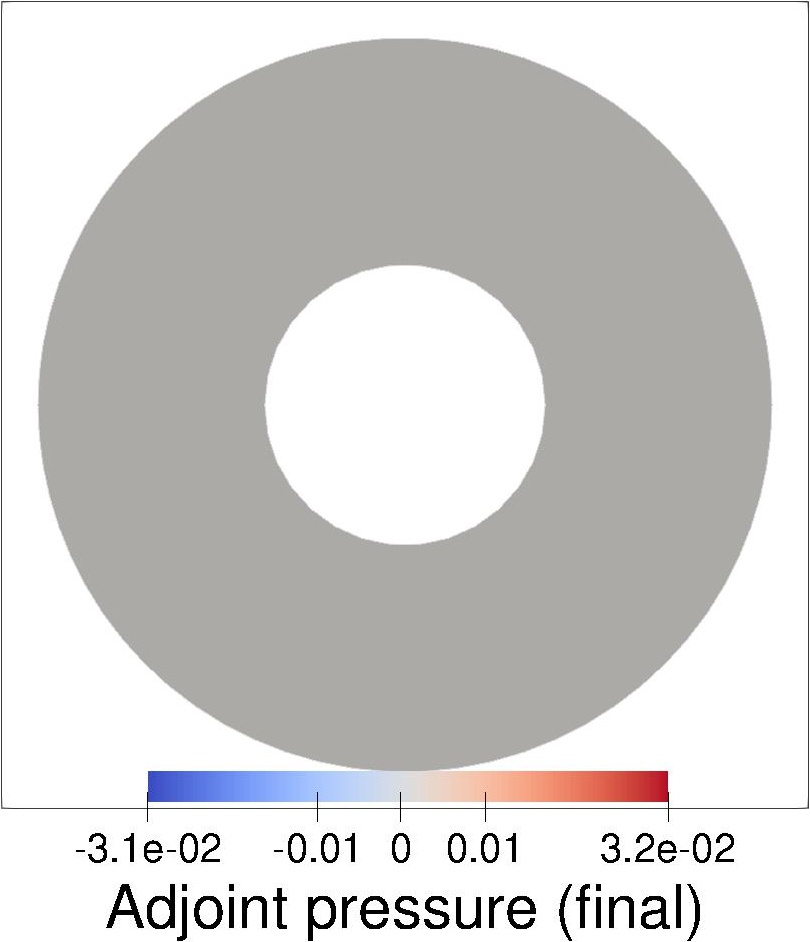}}
\caption{Imaginary parts of the initial and final flow field and pressure of the adjoint system}
\label{fig:figure3}
\end{figure}
\subsubsection{Example in three dimensions}\label{subsubsection:example_in_3d}
Now we consider a test case in three spatial dimensions.
The assumptions are similar to the previous example.
That is, we consider a gravity-like force $\ff = (-10x,-10y,-10z) \in \mathbb{R}^{3}$.
This force is assumed to keep the fluid to surround an object that is spherical in shape having radius equal to $0.4$.
Again, the fluid flowing in the domain is triggered by an initial velocity.
The value of $\alpha$ is again set to $0.01$ and this time we take the initial guess to be the boundary of a three-petal flower like shape that is given exactly in the figure shown in Figure \ref{fig:figure4a}.

We again look at the situations where we have coarse and fine mesh for the computational domain.
For the latter (initial) setup, the tetrahedrons have maximum and minimum mesh width of $h_{\max} \approx 0.665$ and $h_{\min} \approx 0.047$, respectively.
Meanwhile, for the test experiment with fine mesh we take $h_{\max} \approx 0.487$ and $h_{\min} \approx 0.047$.
The rest of the computational setup is similar to the case of two dimensions.

For coarse mesh, the computational results are shown in Figure \ref{fig:figure4} -- Figure \ref{fig:figure6d}.
It is evident from Figure \ref{fig:figure4} that the sequence of shape approximations due to CCBM is different from the classical approach TD.
Nevertheless, as we observed in Figure \ref{fig:figure5}, the computed optimal shapes obtained from the two methods almost coincide.
Meanwhile, Figure \ref{fig:figure6} illustrates the mesh profile of the computed optimal shape for each method which also show a close similarity between the two results.
The next two figures, Figure \ref{fig:figure6a} -- Figure \ref{fig:figure6b},\footnote{For the final pressure profile of the Stokes and the adjoint solutions, the maximum magnitude is found to be of order $10^{-3}$ and $10^{-4}$, respectively.} depict the initial and final imaginary parts of the Stokes' and adjoint's flow field and pressure profiles (magnitude) for the coarse-mesh experiment.
These numerical results -- similar to the observation made for two dimensional case -- corroborate the statement issued in Remark \ref{rem:equivalence} and the conclusion drawn in Corollary \ref{cor:necessary_condition}.

On the other hand, the results for a finer computational mesh are shown in Figure \ref{fig:figure7} -- Figure \ref{fig:figure9}.
In Figure \ref{fig:figure7} we notice a well-behaved (in the sense that the shape approximations are smooth) evolution of the initial domain to the optimal domain due to CCBM as opposed to the case of TD.
Observe from the latter (left plots in Figure \ref{fig:figure7}), that there are obvious dents in the computed optimal shape which we do not expect to appear.
Moreover, it becomes more evident as we made a cross comparison between the computed optimal shapes obtained from the two methods in Figure \ref{fig:figure8}.
For the (initial and) final imaginary parts of the Stokes' and adjoint flow field and pressure profiles (magnitude) for the case of finer mesh, the maximum magnitude is found to be of smaller values than that of the coarse mesh,  as expected. 

Additionally, it is obvious from Figure \ref{fig:figure9} that the optimal shape obtained via TD is less rounded that the one achieve using CCBM.
These results clearly show the merit of employing CCBM which we actually expect since it utilizes (naturally as a consequence of the formulation) a volume-integral.

Finally, in Figure \ref{fig:figure10}, we plot the graph of the histories of costs and gradient norms for the two methods.
It seems that, for coarse mesh (see left plot in Figure \ref{fig:figure10}), CCBM converges faster to a stationary point compared to TD.
Meanwhile, for finer mesh, it appears that graph corresponding to the histories of cost and gradient norm values for TD stops abruptly (as seen in the right plot of Figure \ref{fig:figure10}).
This is primarily due to the instability occurring in the computational mesh which was observed in previous figures.
Based from these results, it seems that CCBM is more robust compared to TD (as expected).
%
%
\begin{figure}[htp!]
\centering
\resizebox{0.32\linewidth}{!}{\includegraphics{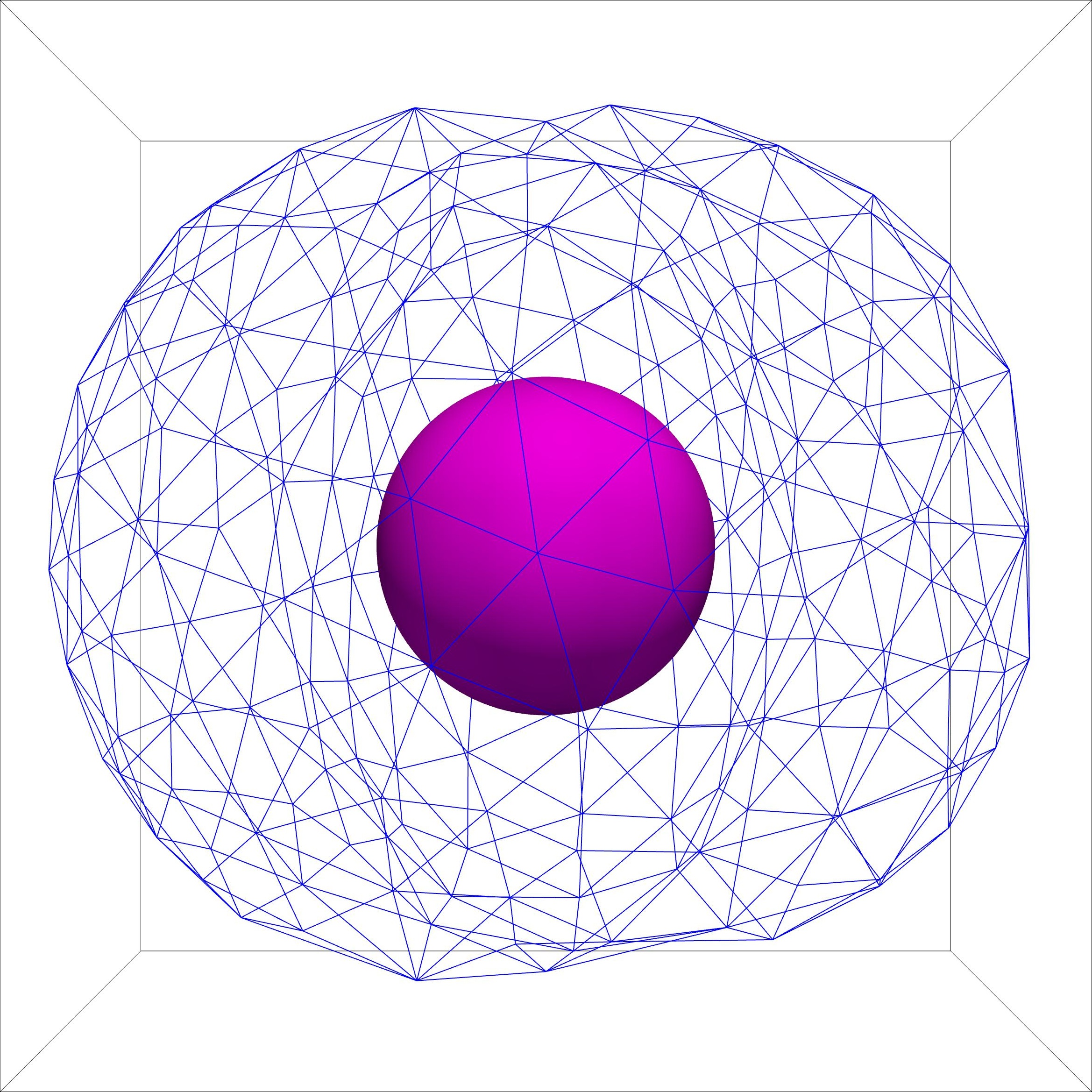}} \hfill
\resizebox{0.32\linewidth}{!}{\includegraphics{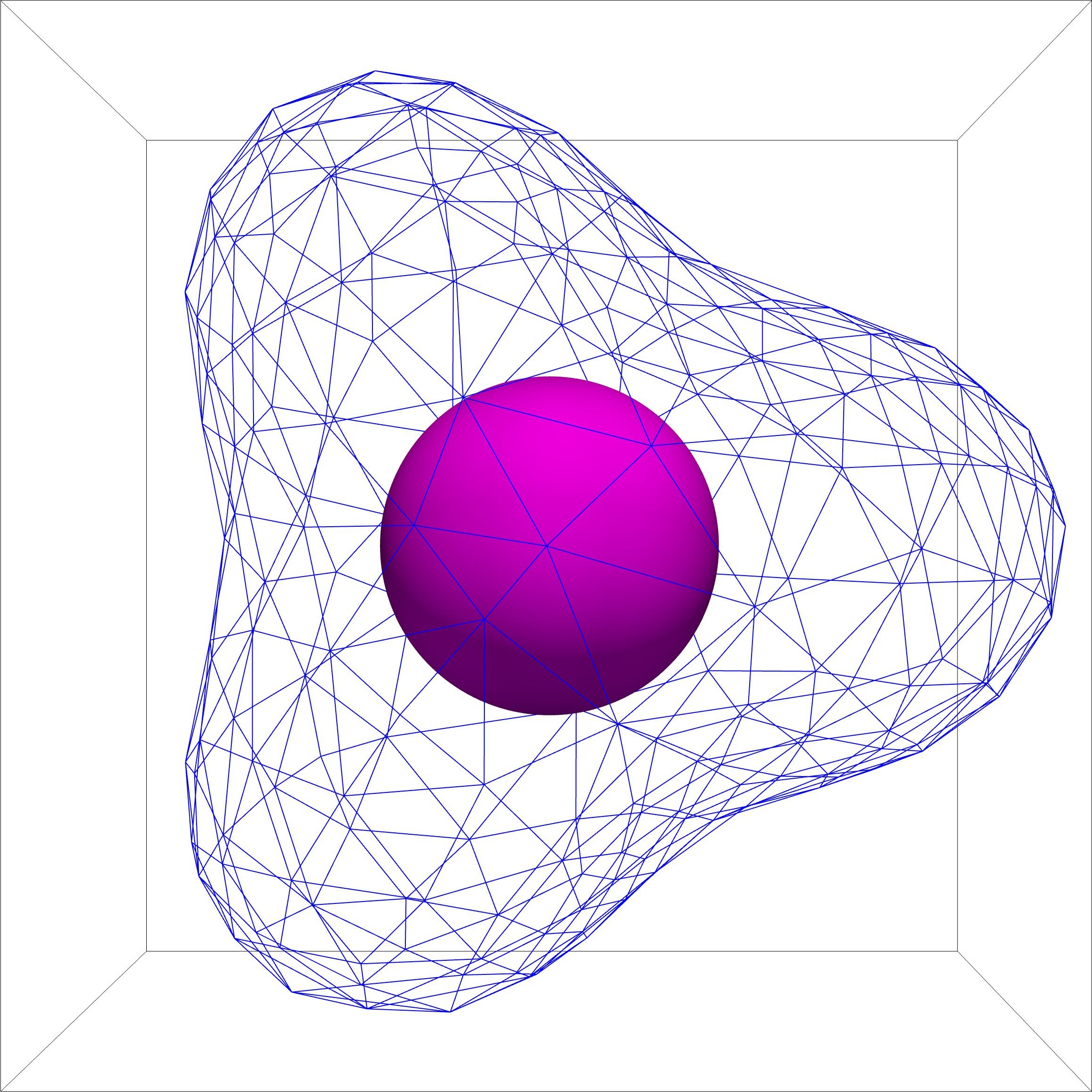}} \hfill
\resizebox{0.32\linewidth}{!}{\includegraphics{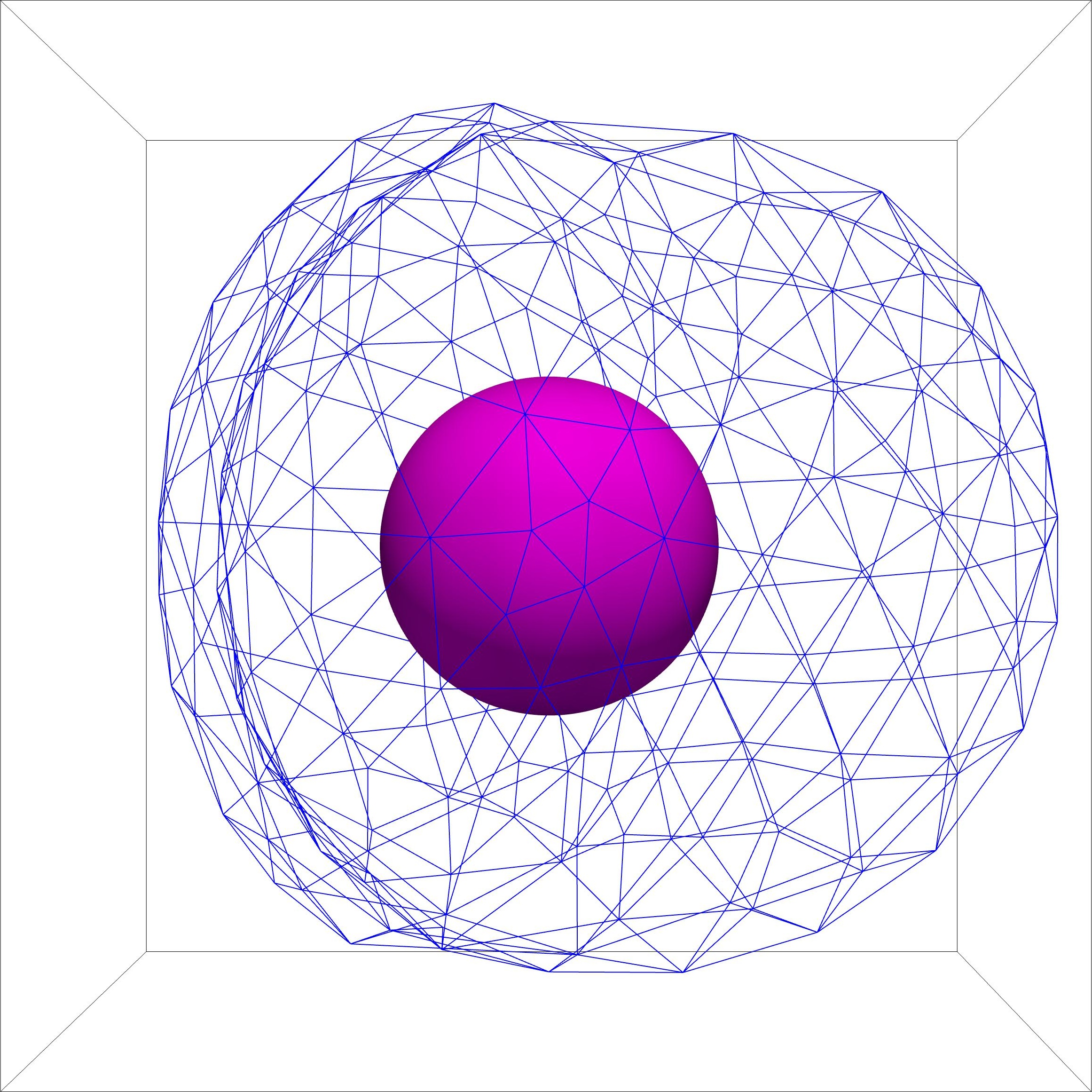}}
\caption{Mesh profile (coarse mesh) of the initial guess viewed on different positions}
\label{fig:figure4a}
\end{figure}
\begin{figure}[htp!]
\centering
\resizebox{0.32\linewidth}{!}{\includegraphics{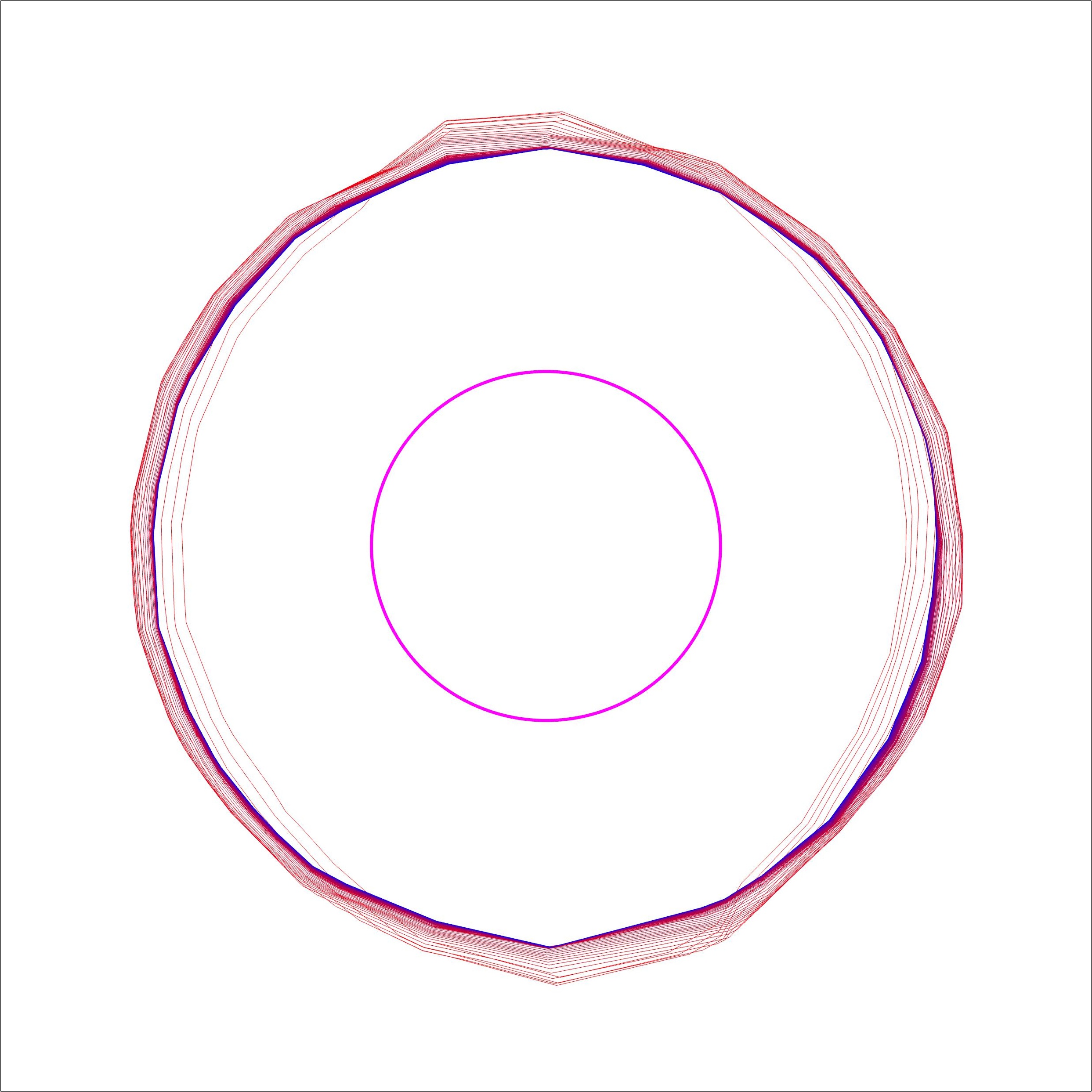}} \hfill
\resizebox{0.32\linewidth}{!}{\includegraphics{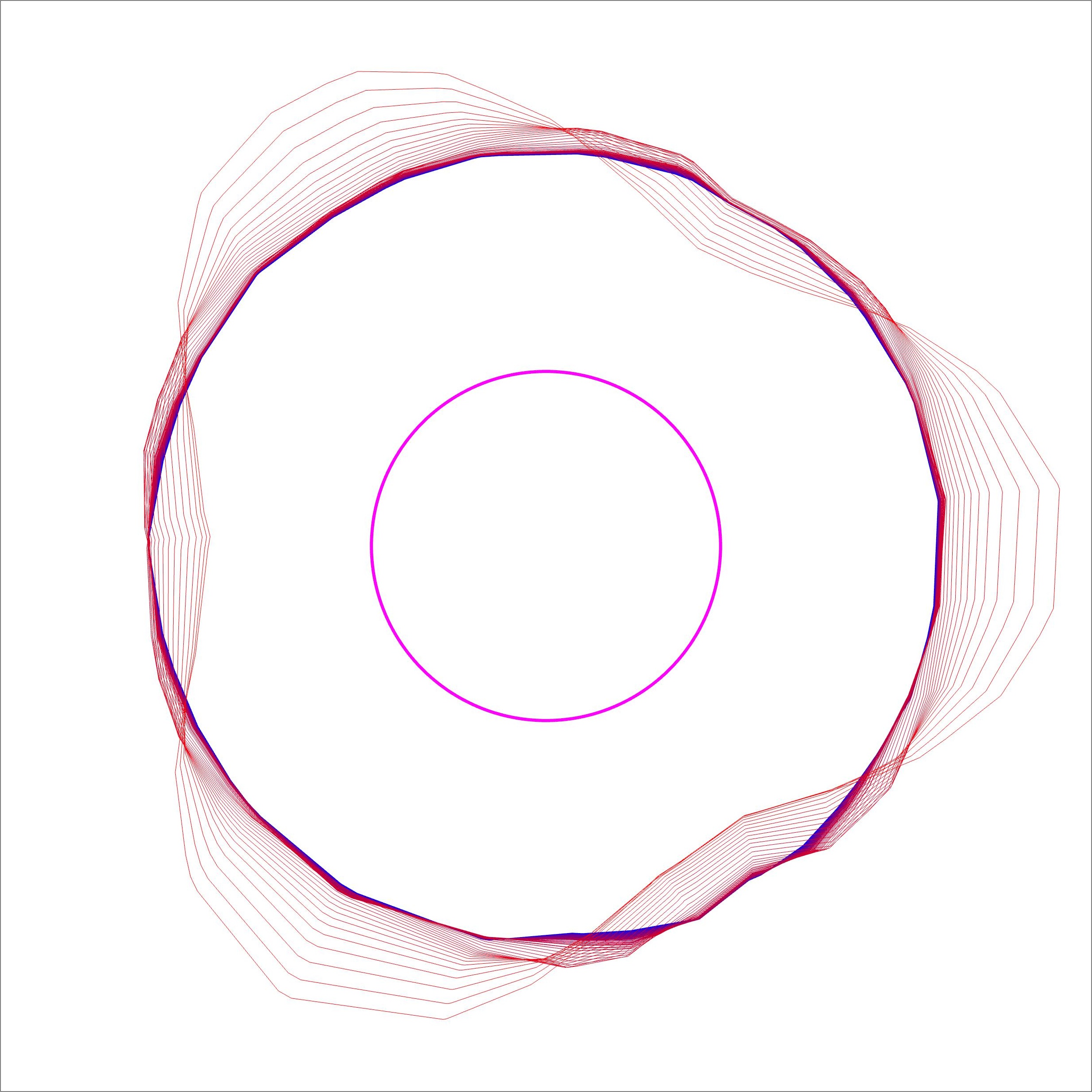}} \hfill
\resizebox{0.32\linewidth}{!}{\includegraphics{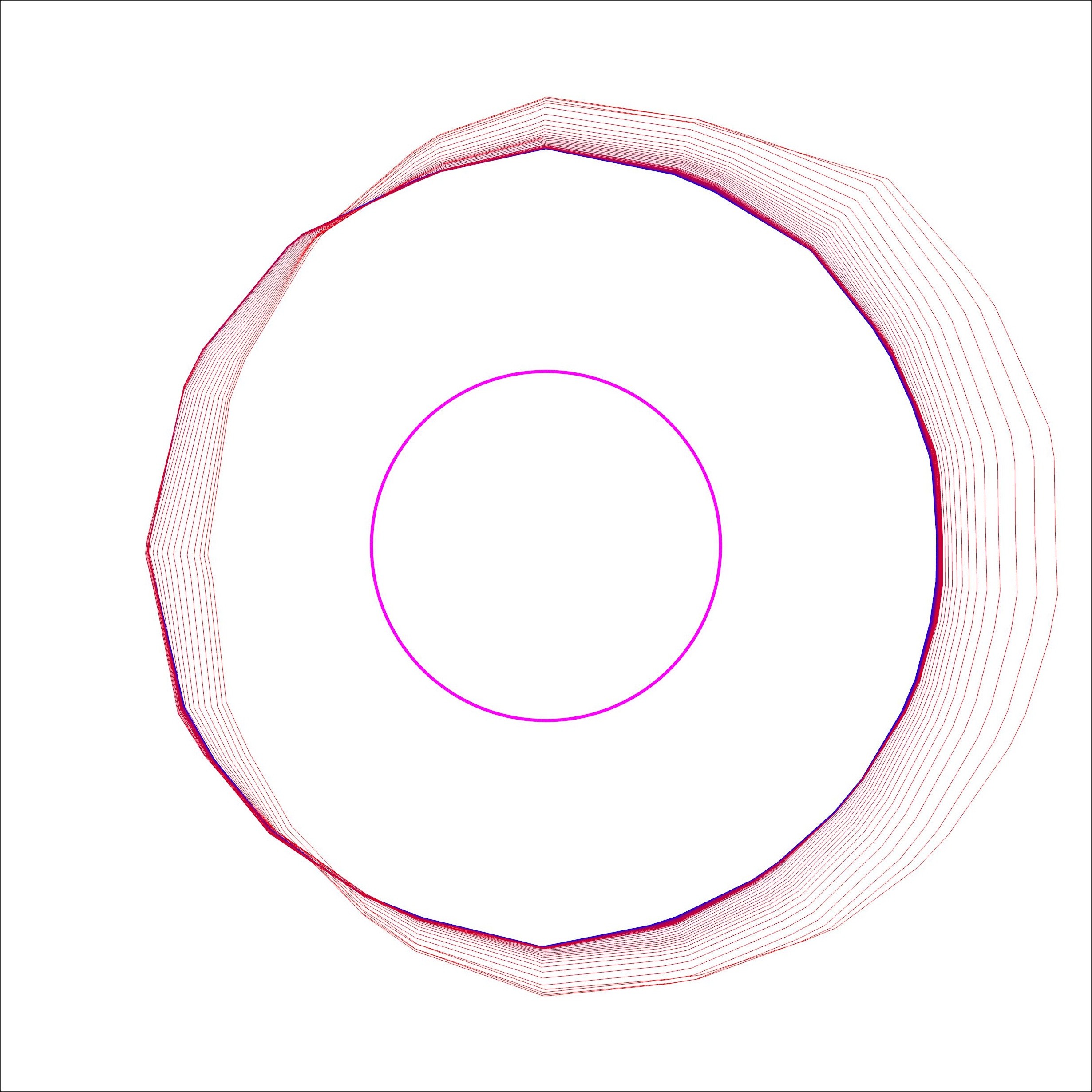}} \\[0.5em] 
\resizebox{0.32\linewidth}{!}{\includegraphics{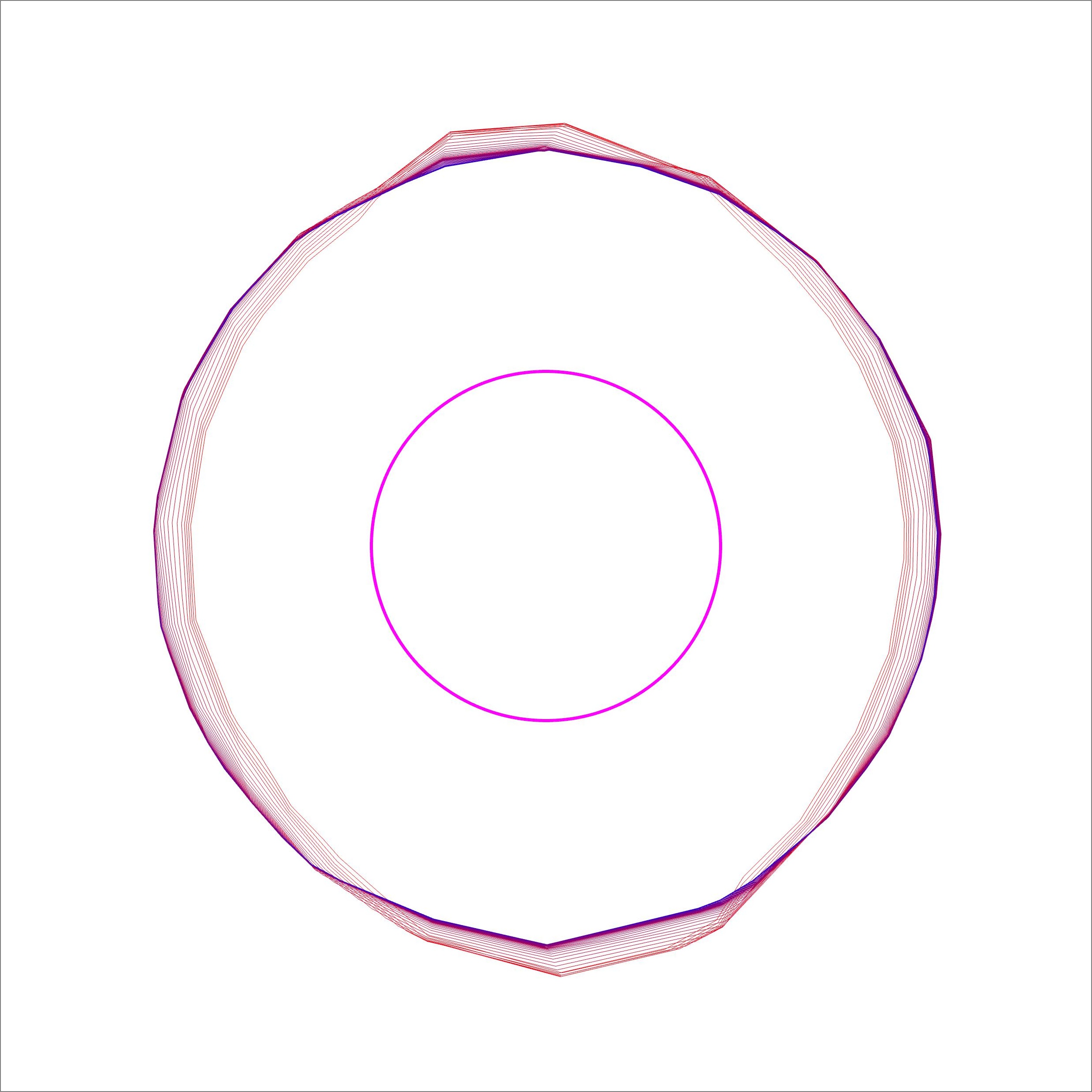}} \hfill
\resizebox{0.32\linewidth}{!}{\includegraphics{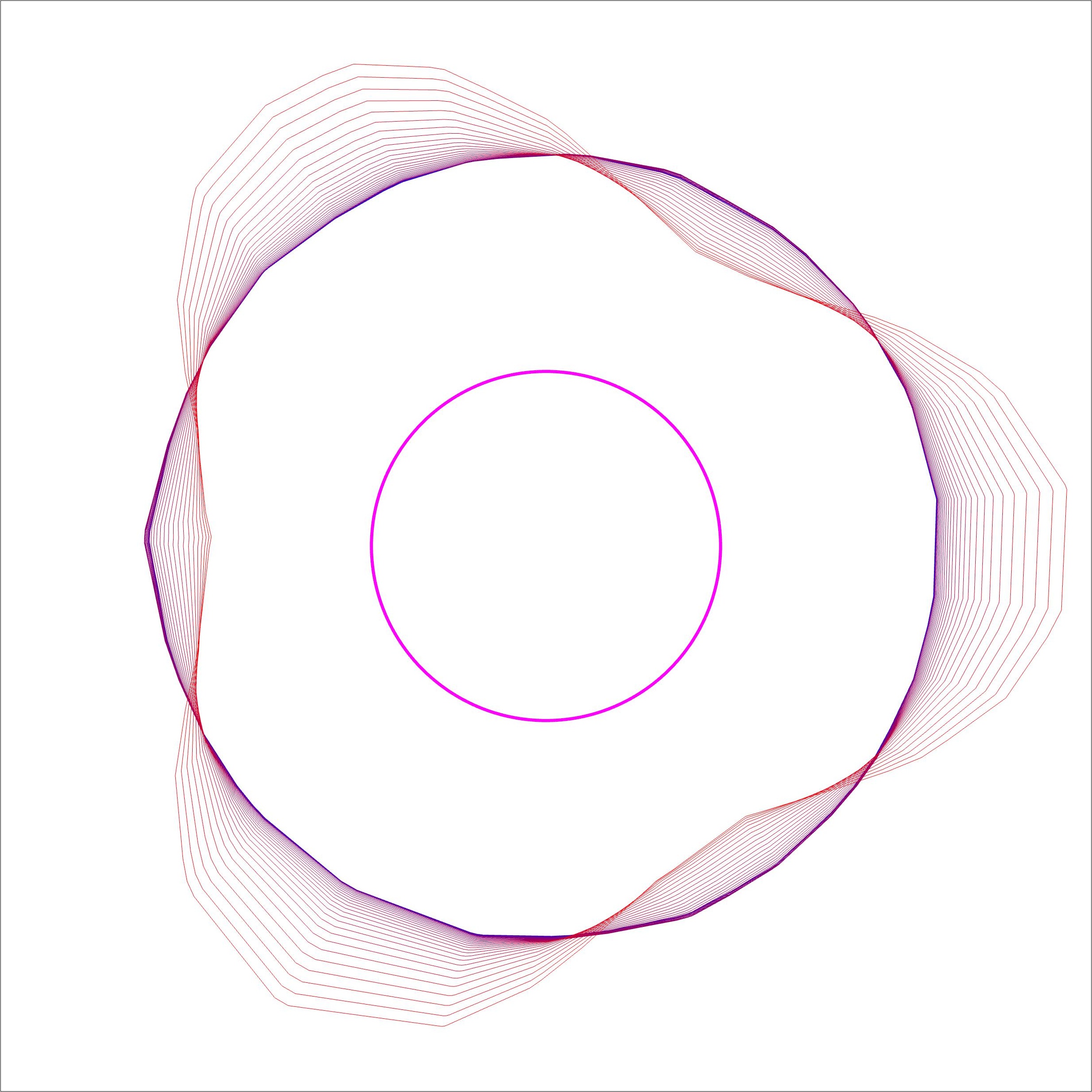}} \hfill
\resizebox{0.32\linewidth}{!}{\includegraphics{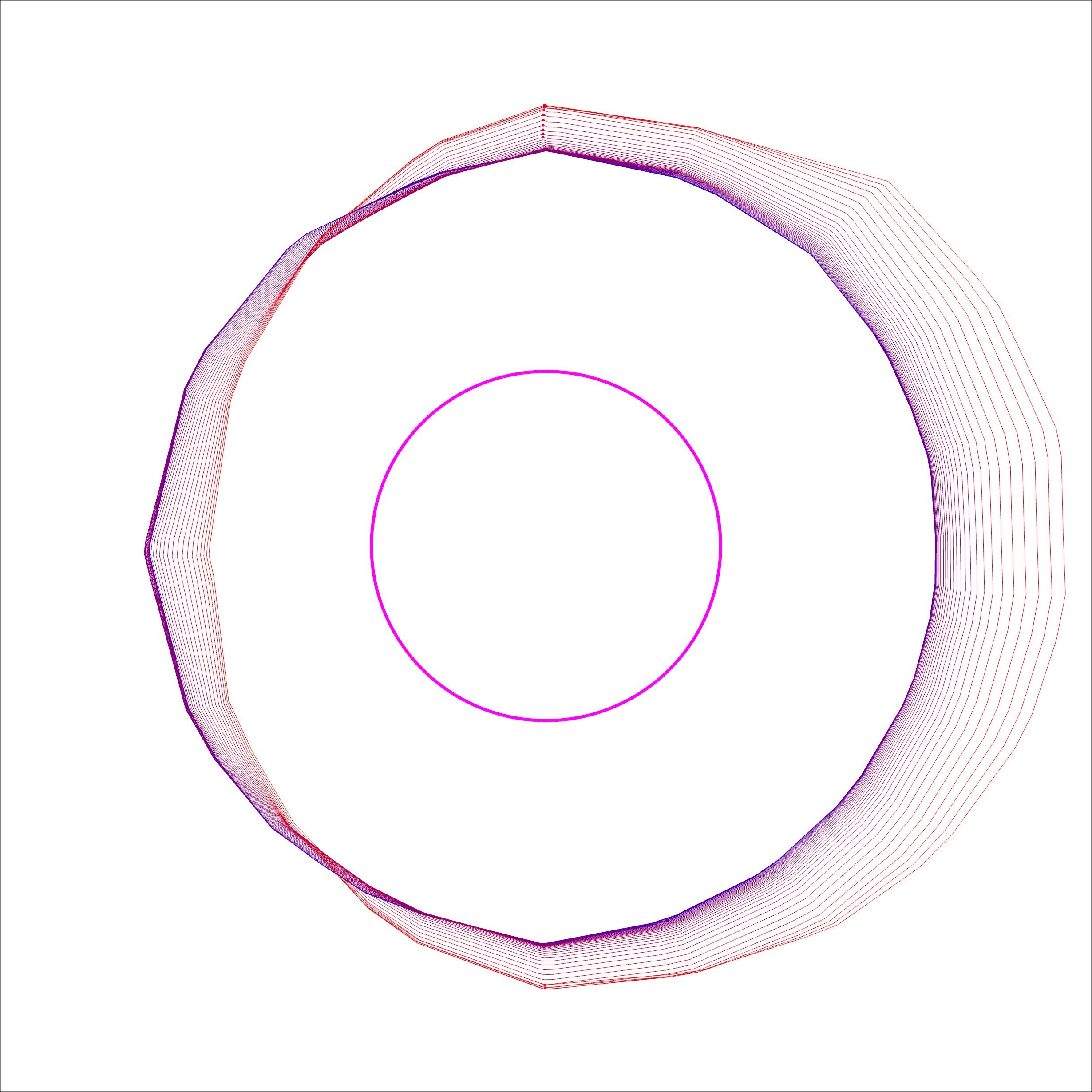}}
\caption{Shape histories of the free boundary computed using TD (upper plots) and CCBM (lower plots) with finer mesh viewed on the plane $xz$ (leftmost column), $yx$ (middle column), and $yz$ (rightmost column)}
\label{fig:figure4}
\end{figure}
%
%
%
%
\begin{figure}[htp!]
\centering
\resizebox{0.32\linewidth}{!}{\includegraphics{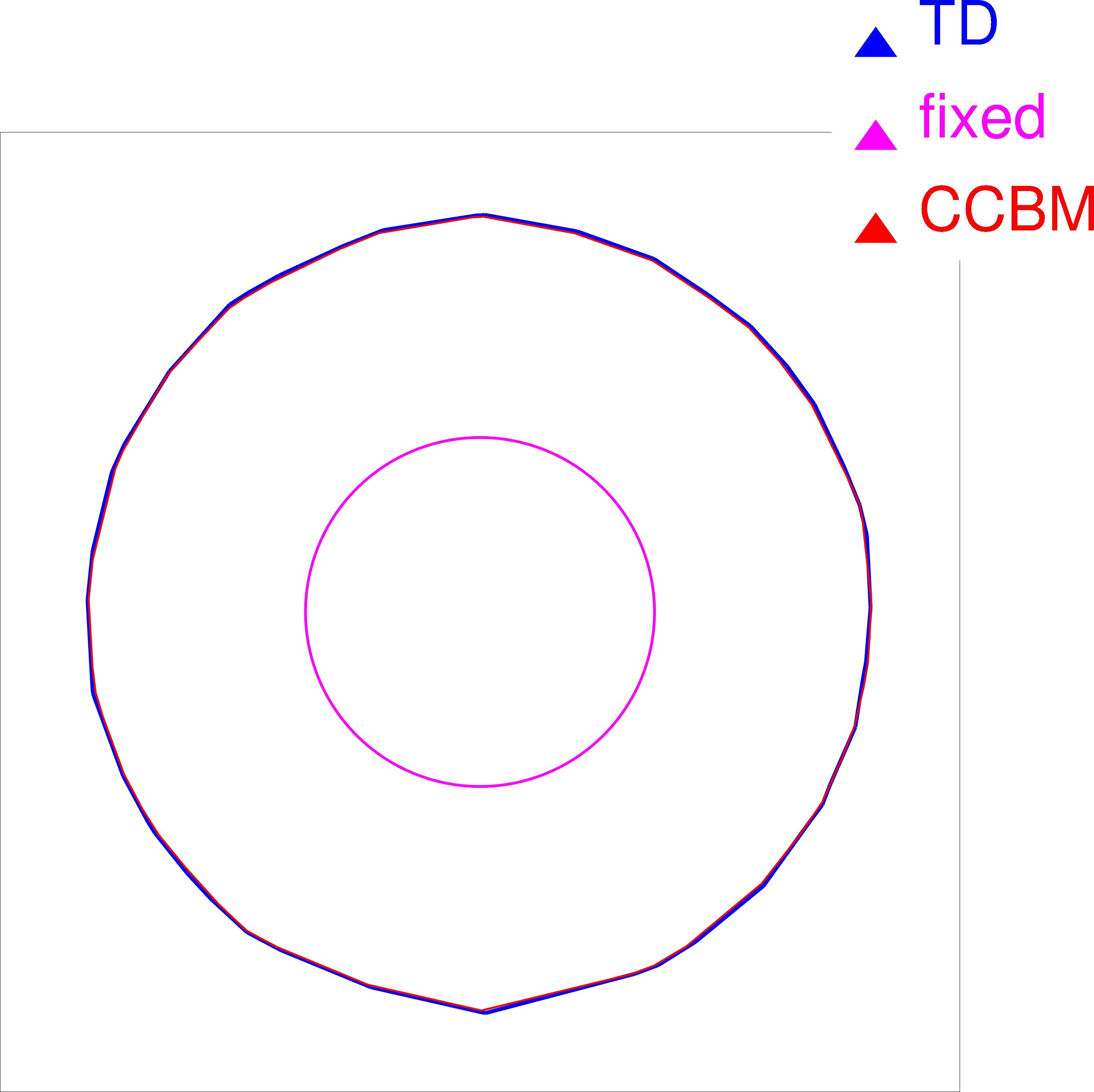}} \hfill
\resizebox{0.32\linewidth}{!}{\includegraphics{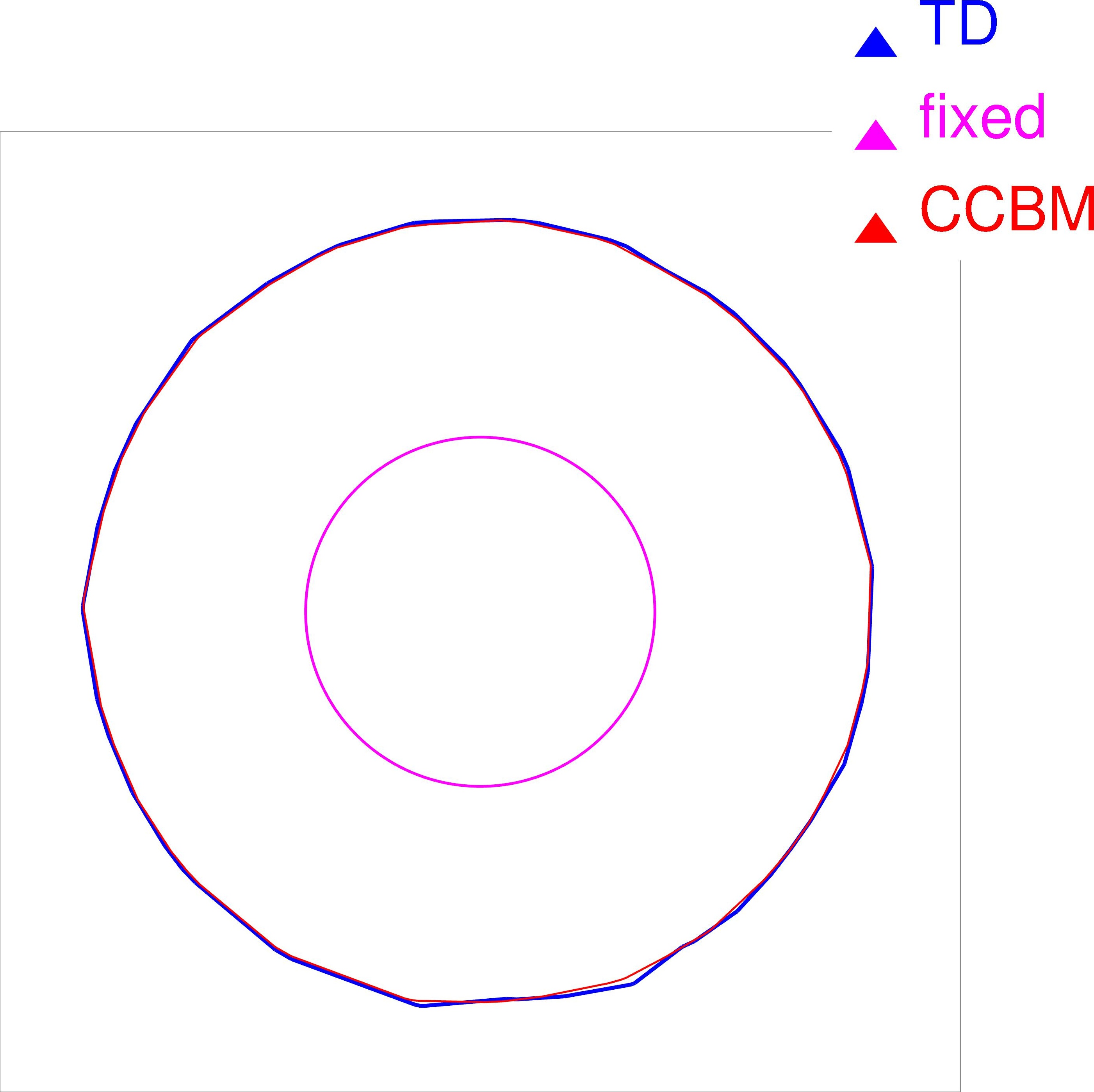}} \hfill
\resizebox{0.32\linewidth}{!}{\includegraphics{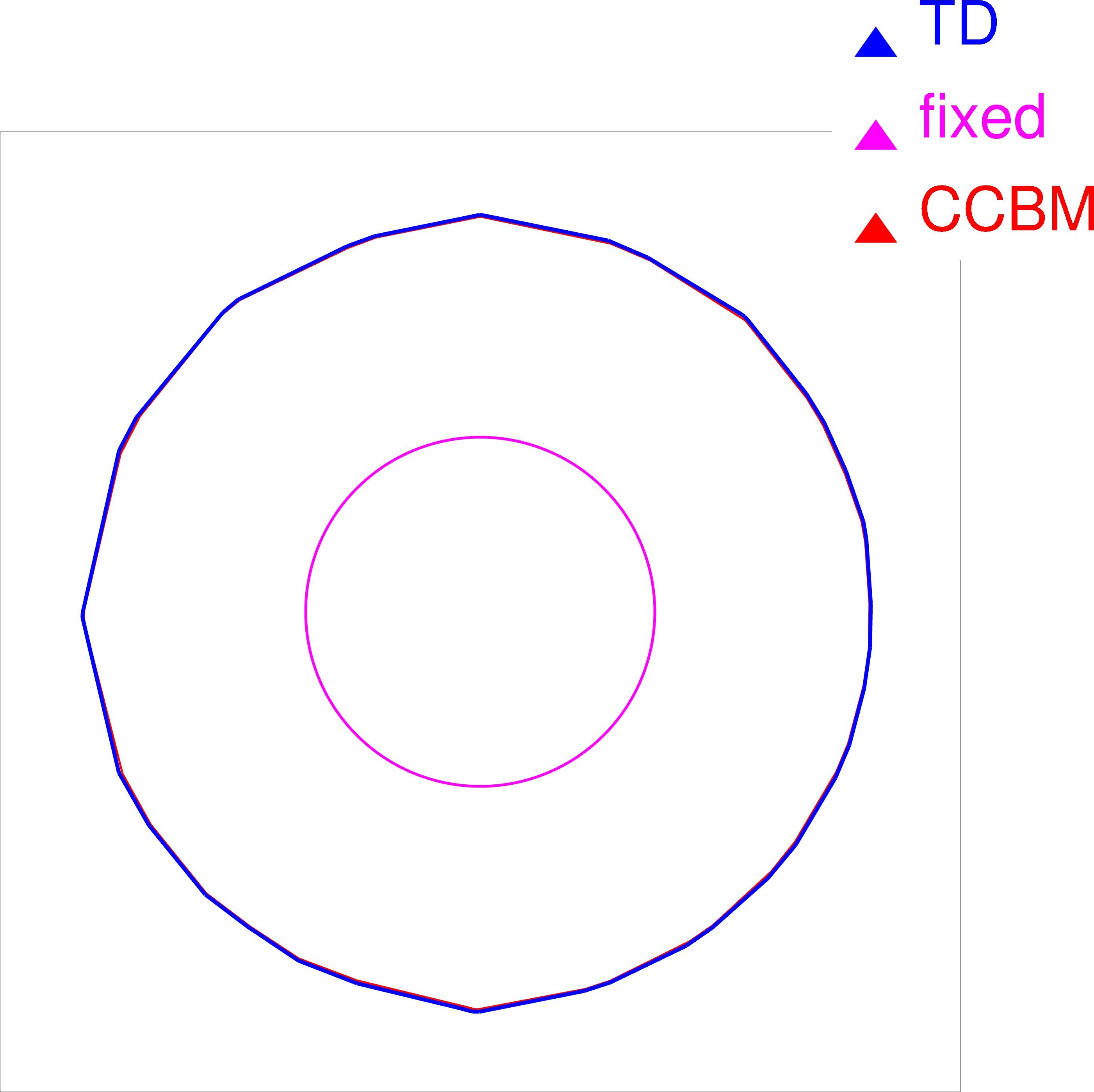}} \hfill
\caption{Cross comparison of computed shapes (case of coarser meshes) viewed on the plane $xz$ (leftmost plot), $yx$ (middle plot), and $yz$ (rightmost plot)}
\label{fig:figure5}
\end{figure}
%
%
%
%
\begin{figure}[htp!]
\centering
\resizebox{0.45\linewidth}{!}{\includegraphics{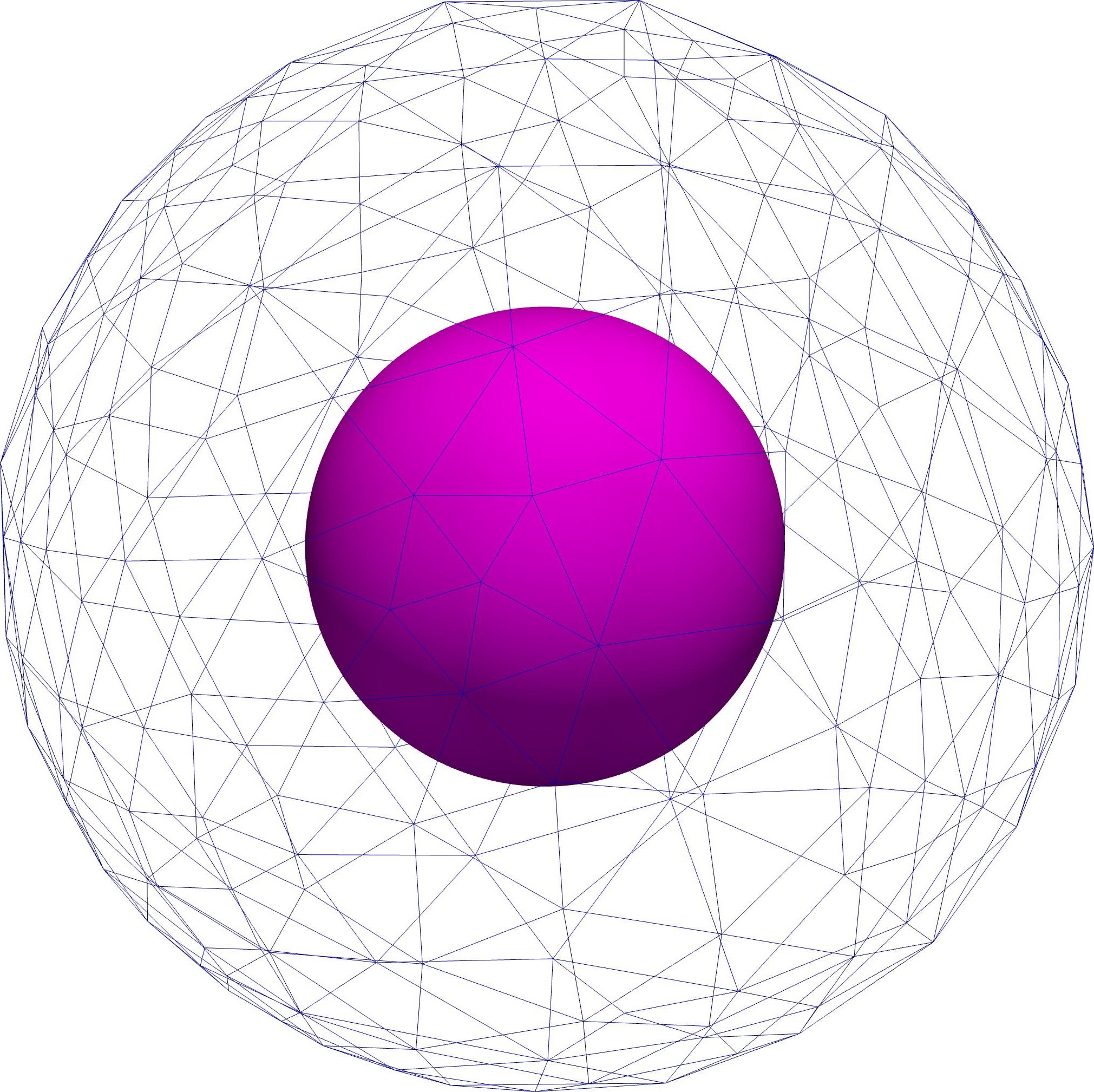}} \hfill
\resizebox{0.45\linewidth}{!}{\includegraphics{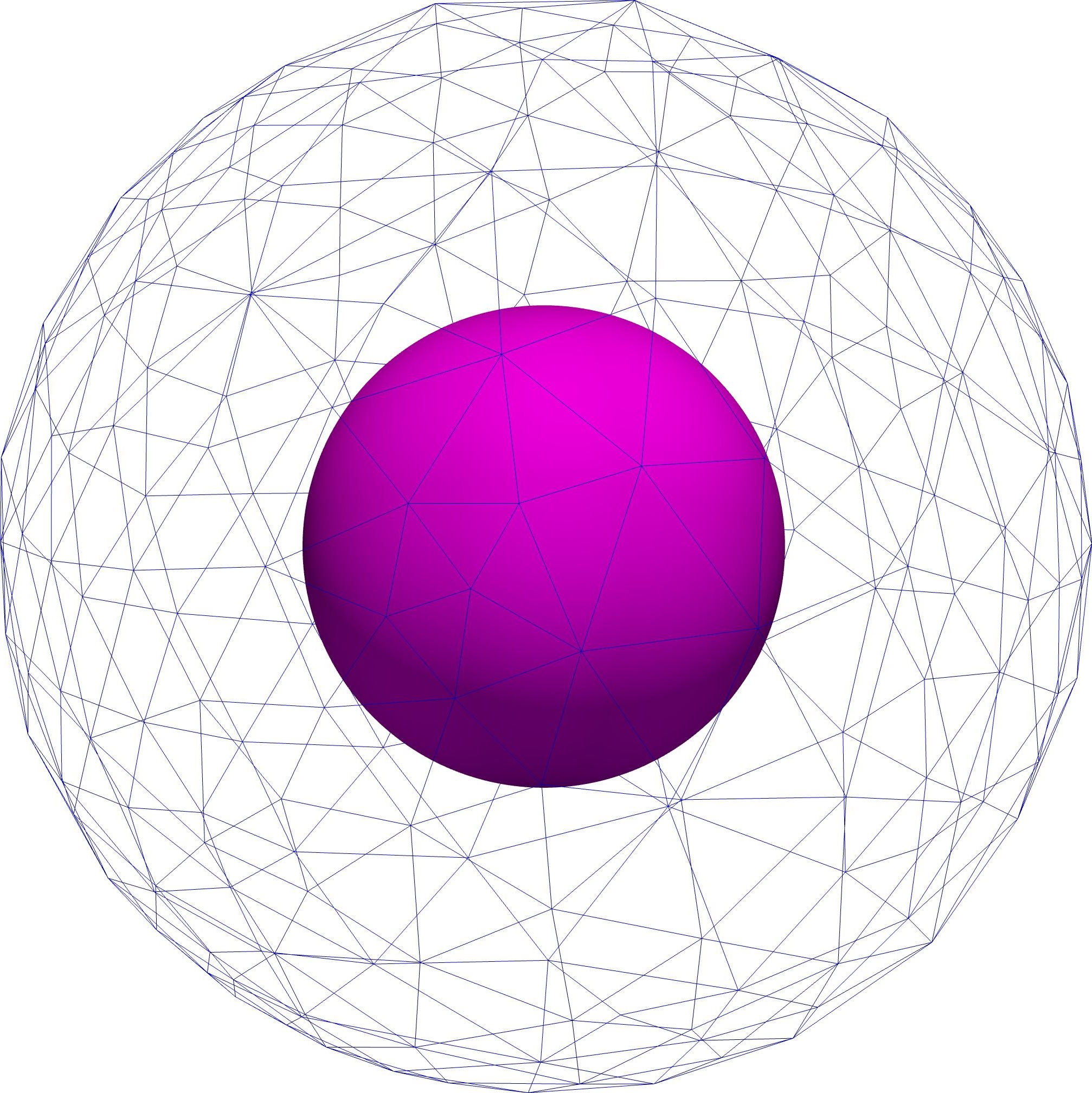}}
\caption{Mesh profile of computed shapes (TD: left plot, CCBM: right plot) with coarse mesh}
\label{fig:figure6}
\end{figure}
%
%
\begin{figure}[htp!]
\centering
\resizebox{0.32\linewidth}{!}{\includegraphics{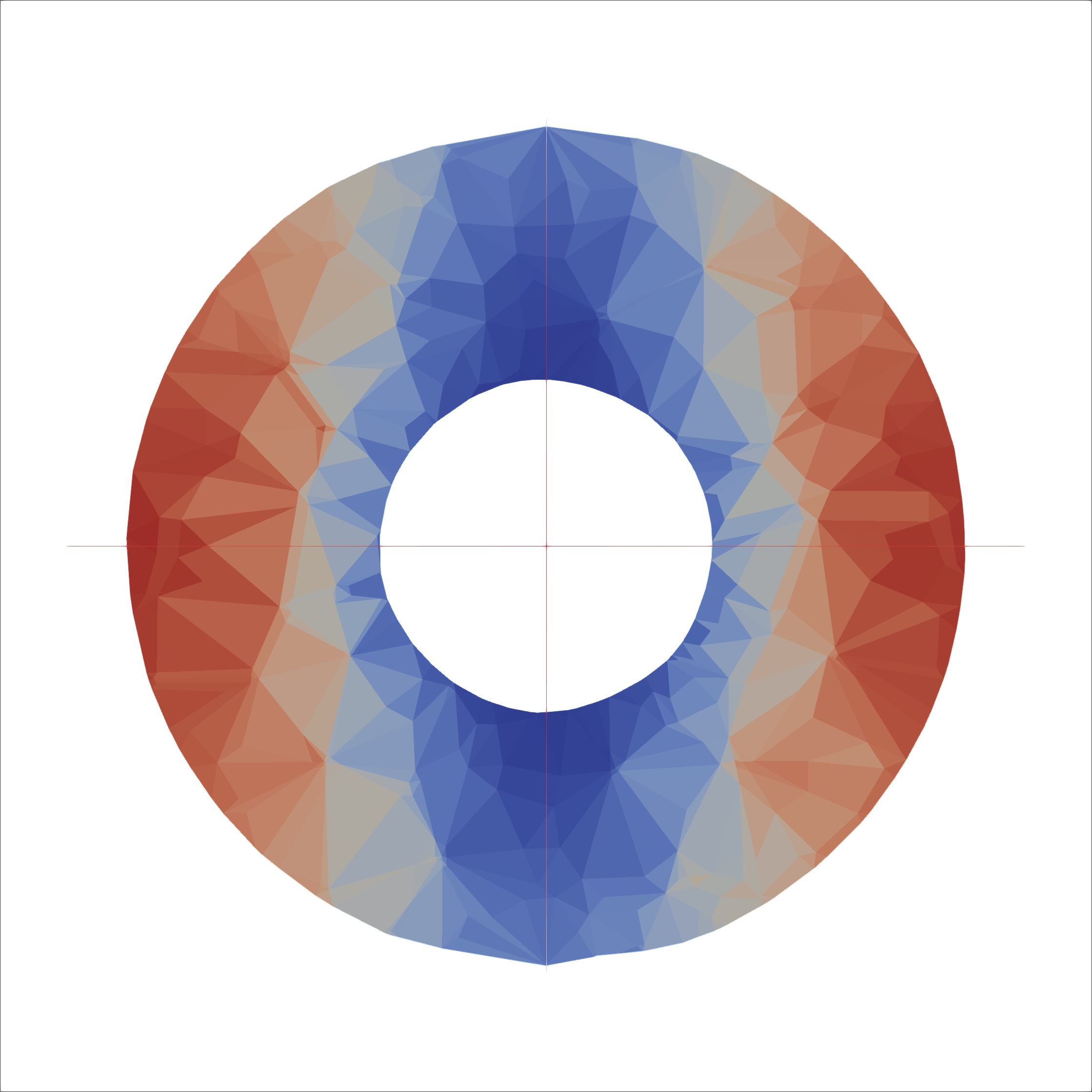}}
\resizebox{0.32\linewidth}{!}{\includegraphics{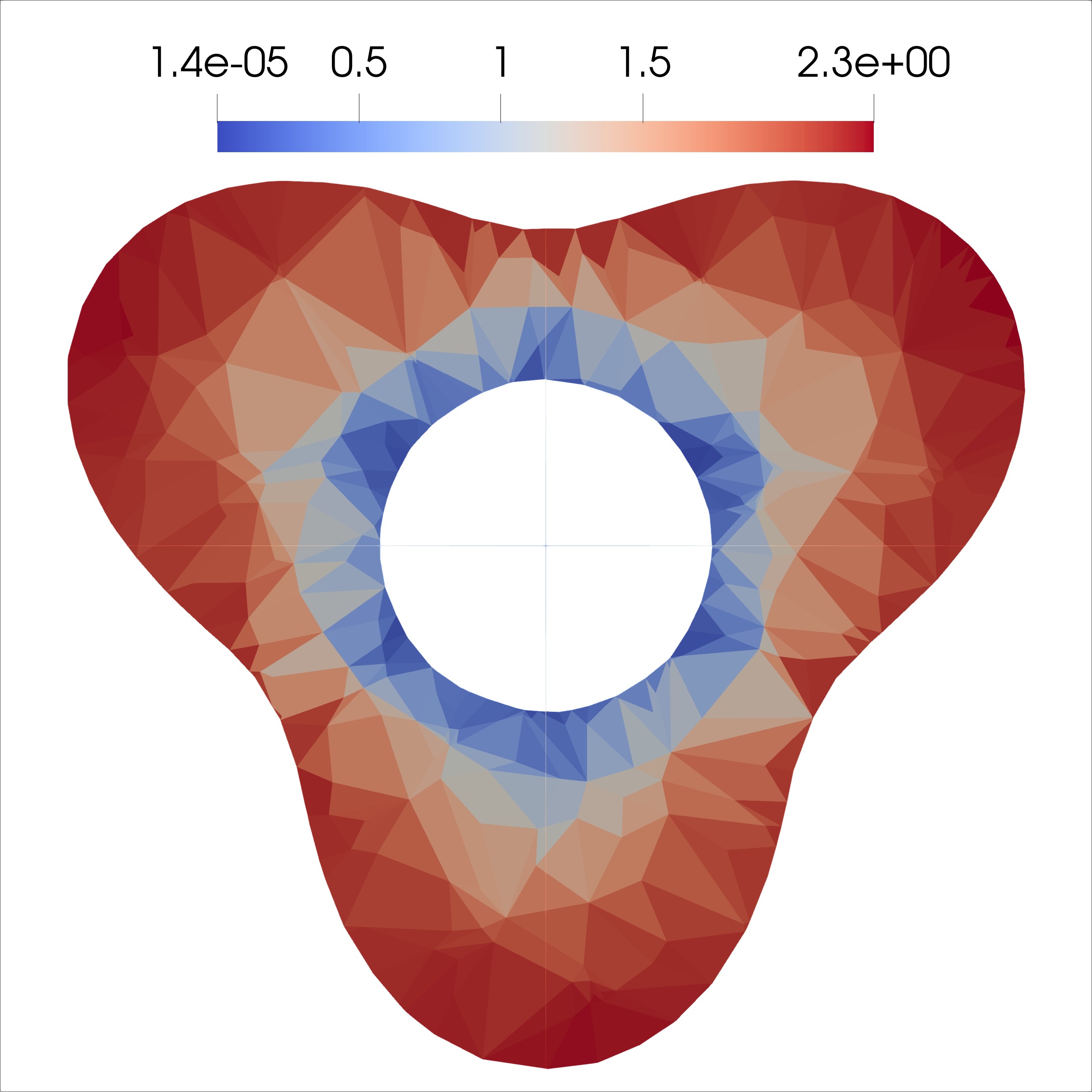}}
\resizebox{0.32\linewidth}{!}{\includegraphics{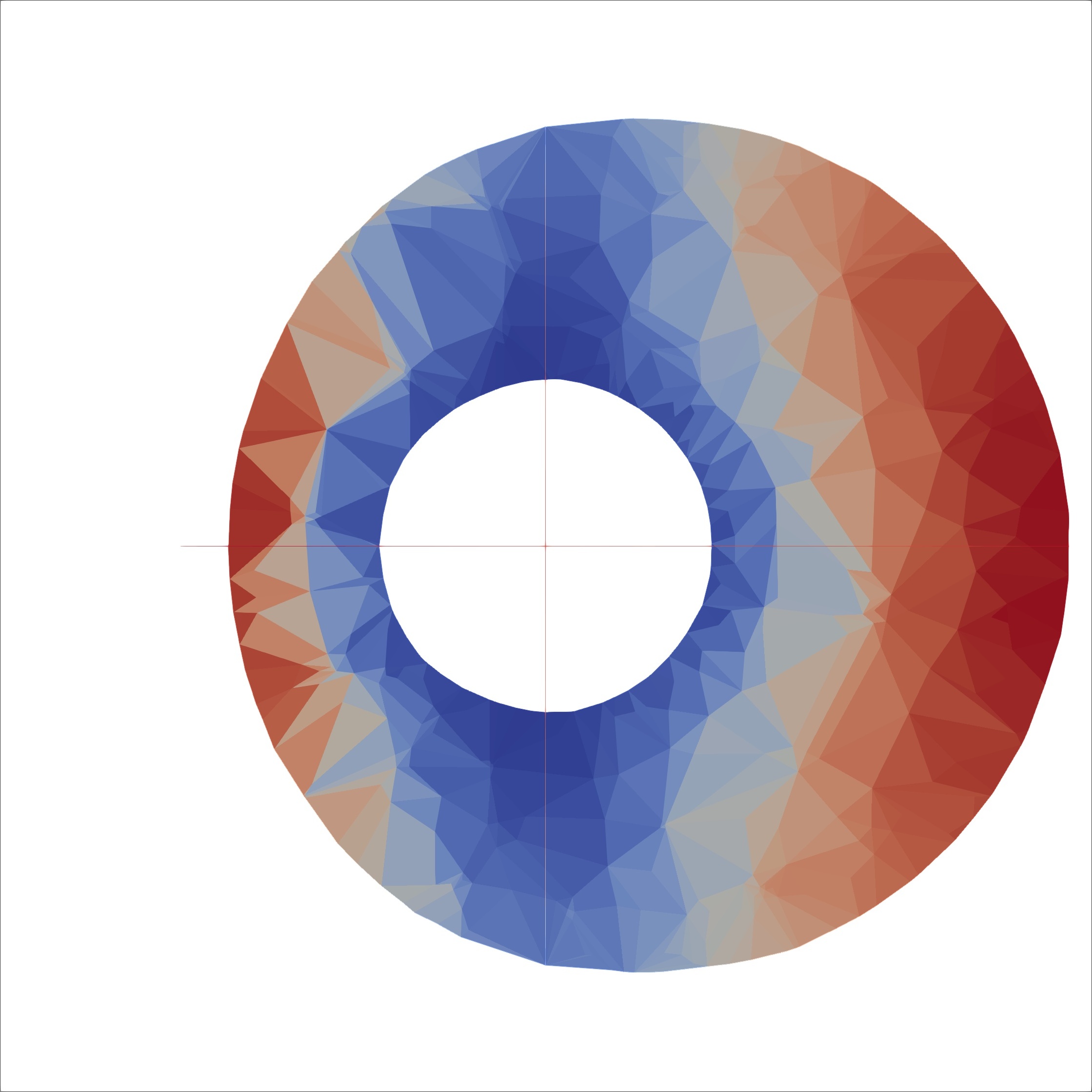}}\\[0.5em]
\resizebox{0.32\linewidth}{!}{\includegraphics{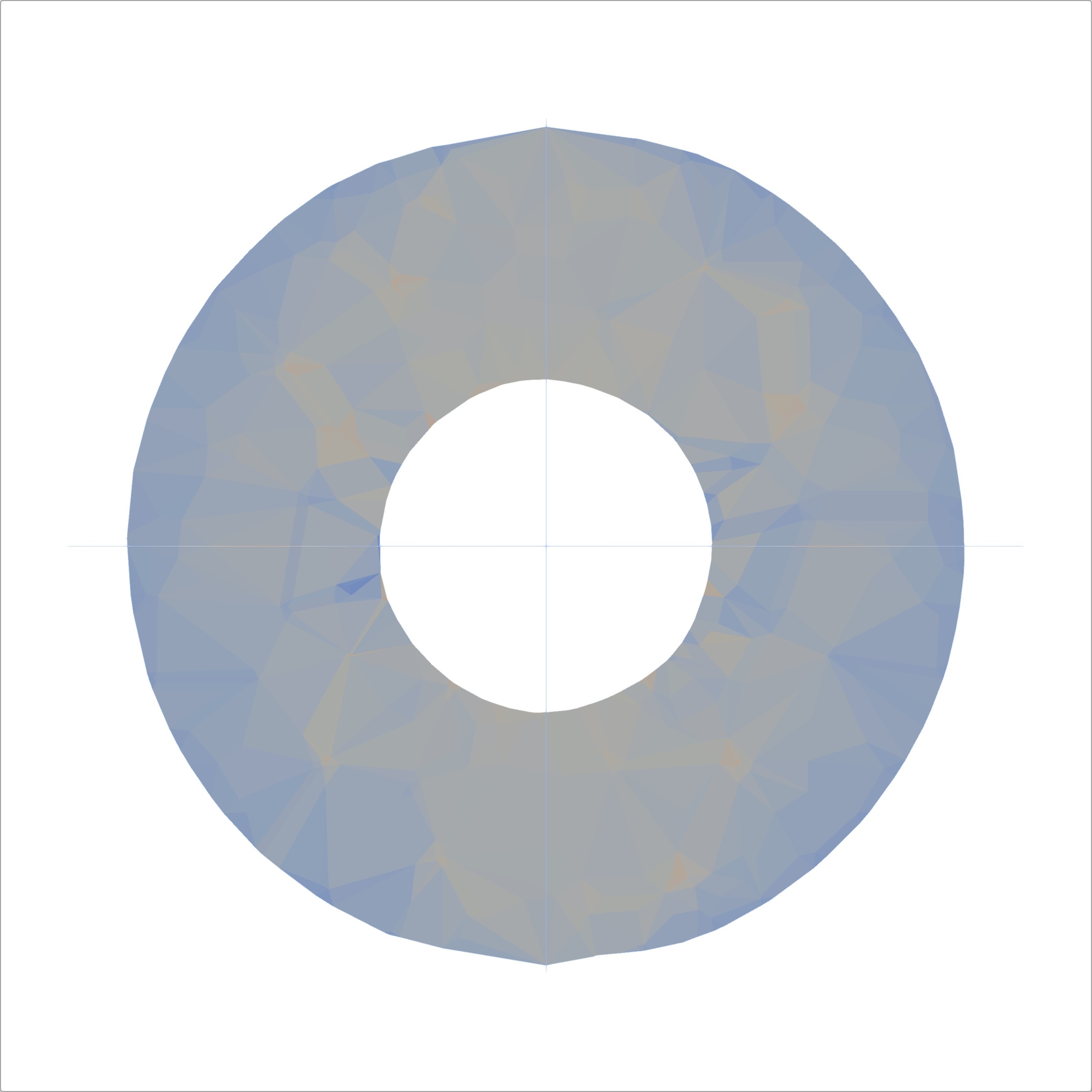}}
\resizebox{0.32\linewidth}{!}{\includegraphics{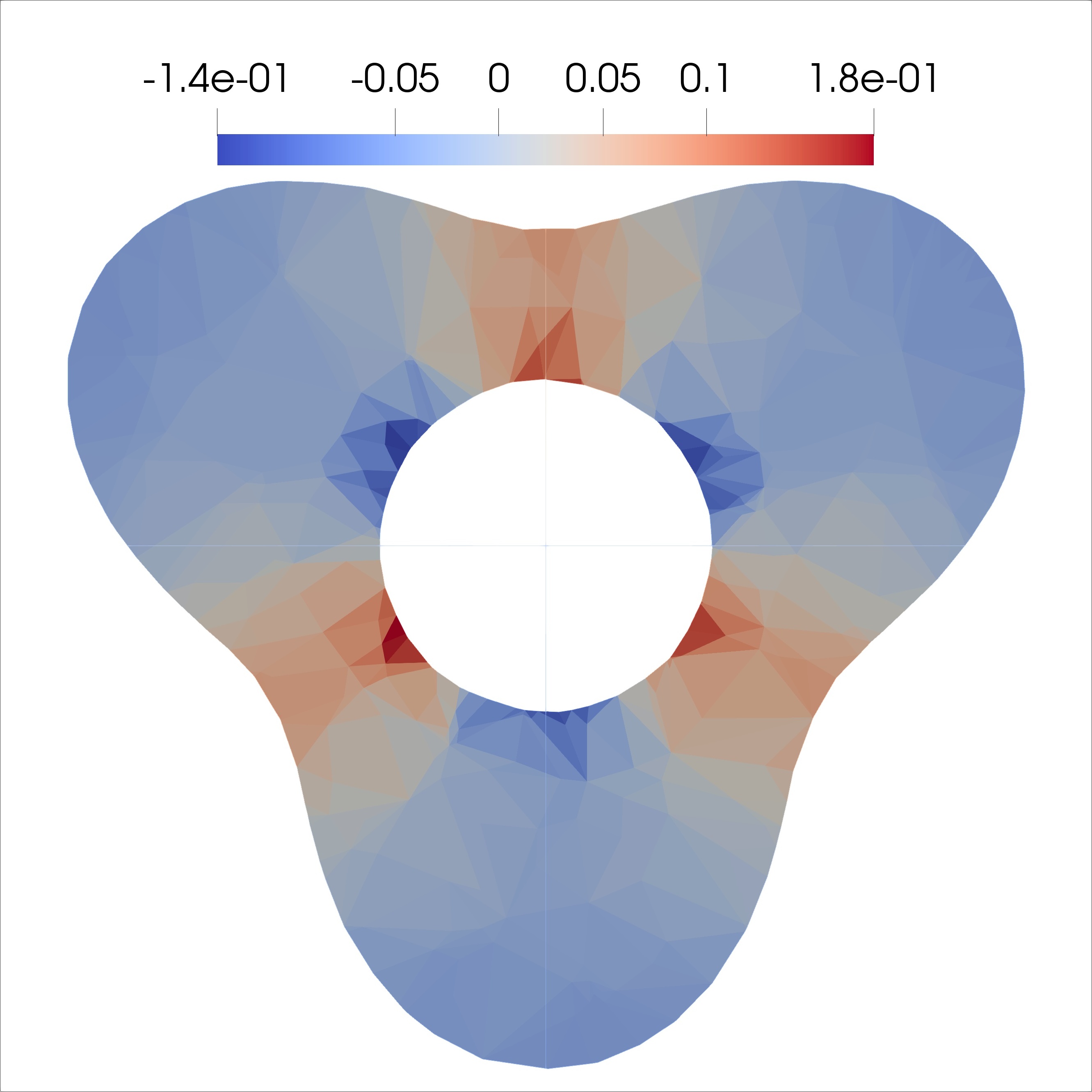}}
\resizebox{0.32\linewidth}{!}{\includegraphics{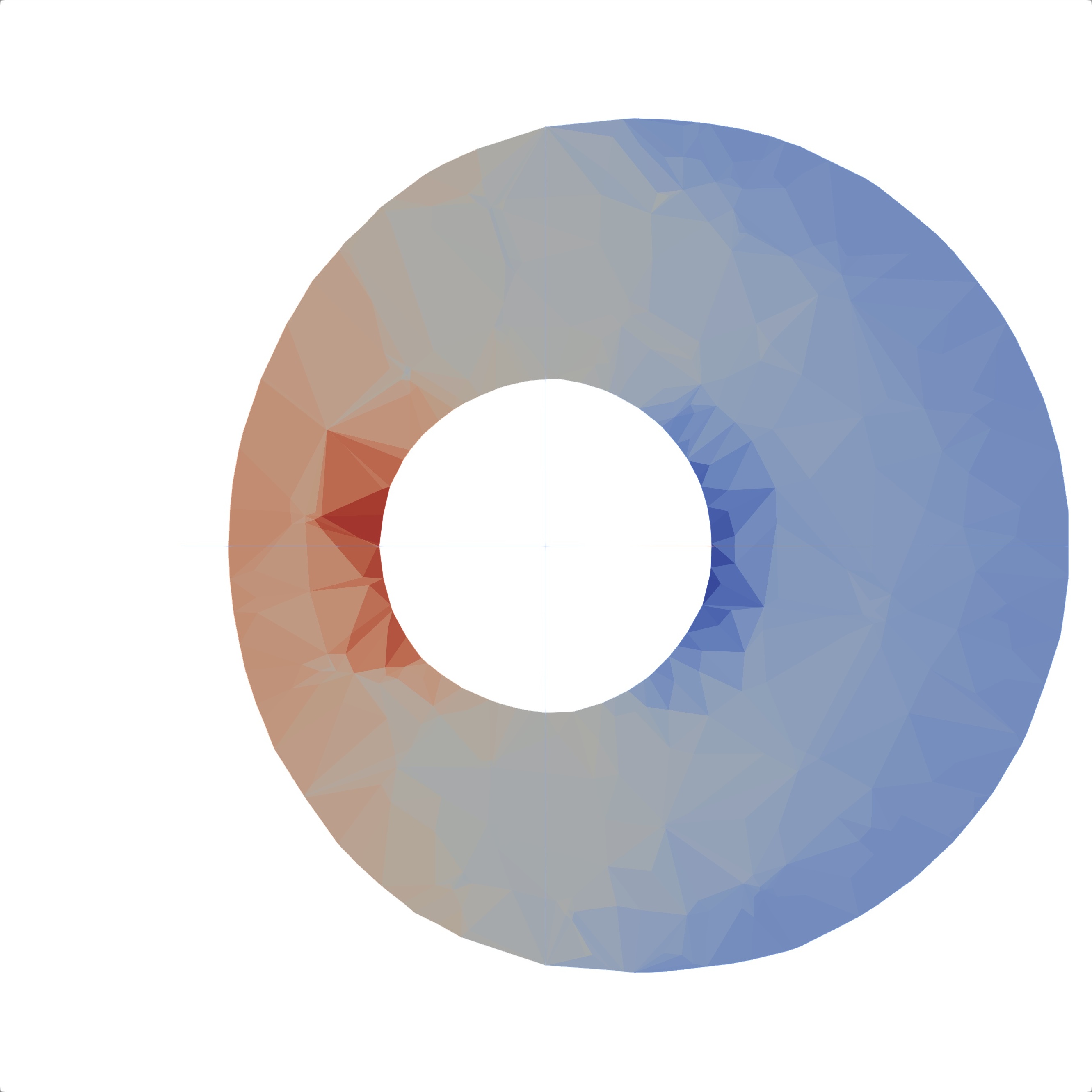}}\\[0.5em]
\resizebox{0.32\linewidth}{!}{\includegraphics{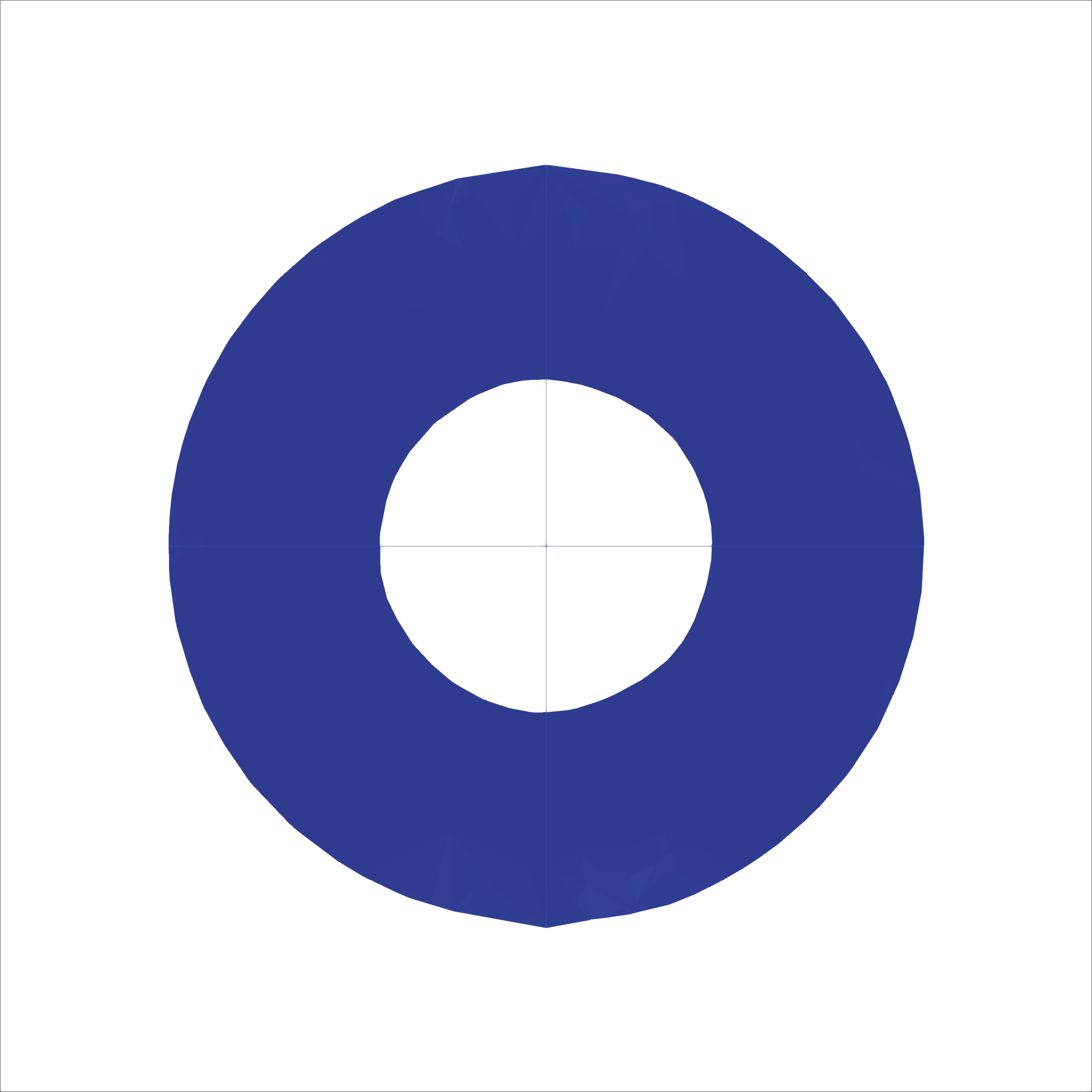}}
\resizebox{0.32\linewidth}{!}{\includegraphics{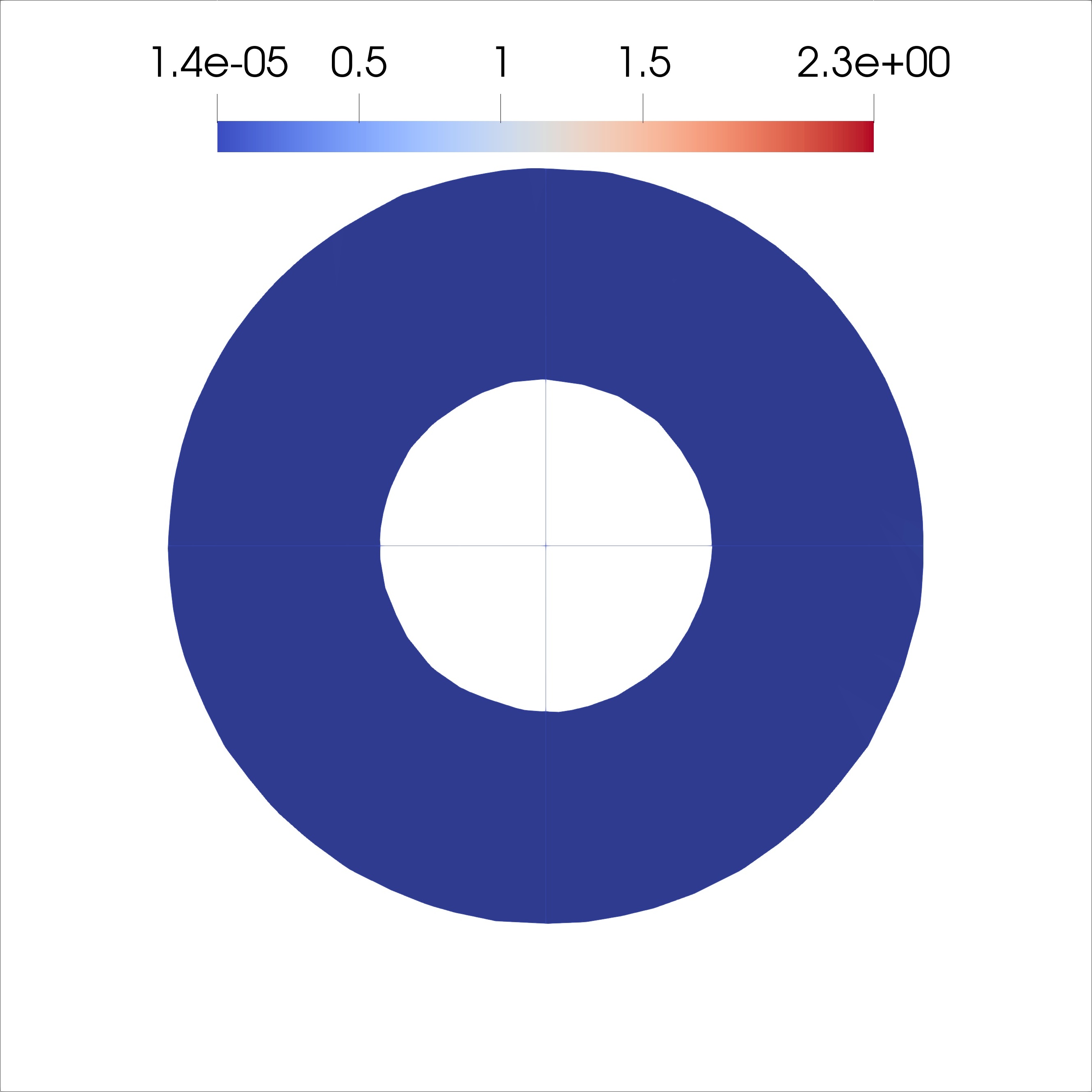}}
\resizebox{0.32\linewidth}{!}{\includegraphics{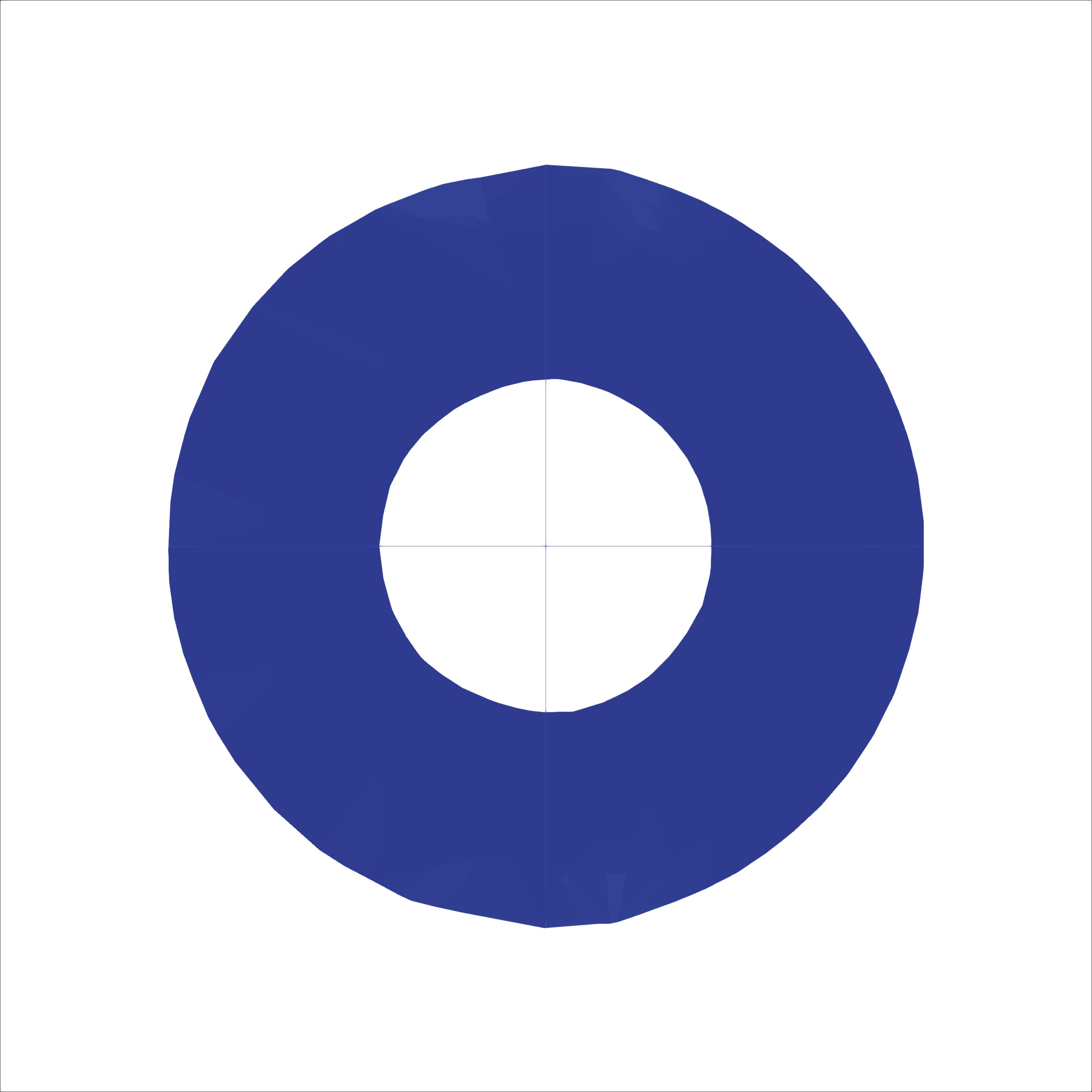}}\\[0.5em]
\resizebox{0.32\linewidth}{!}{\includegraphics{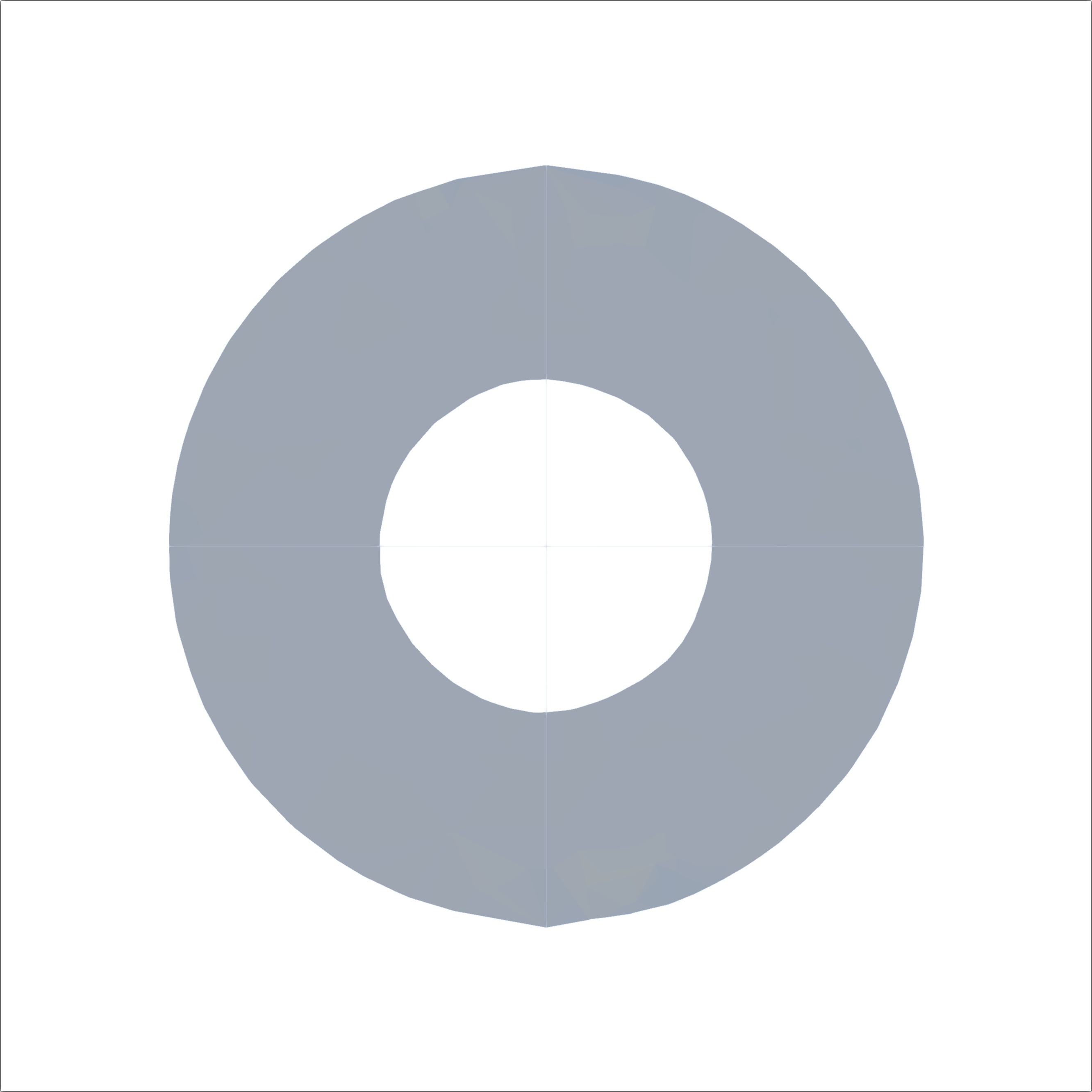}}
\resizebox{0.32\linewidth}{!}{\includegraphics{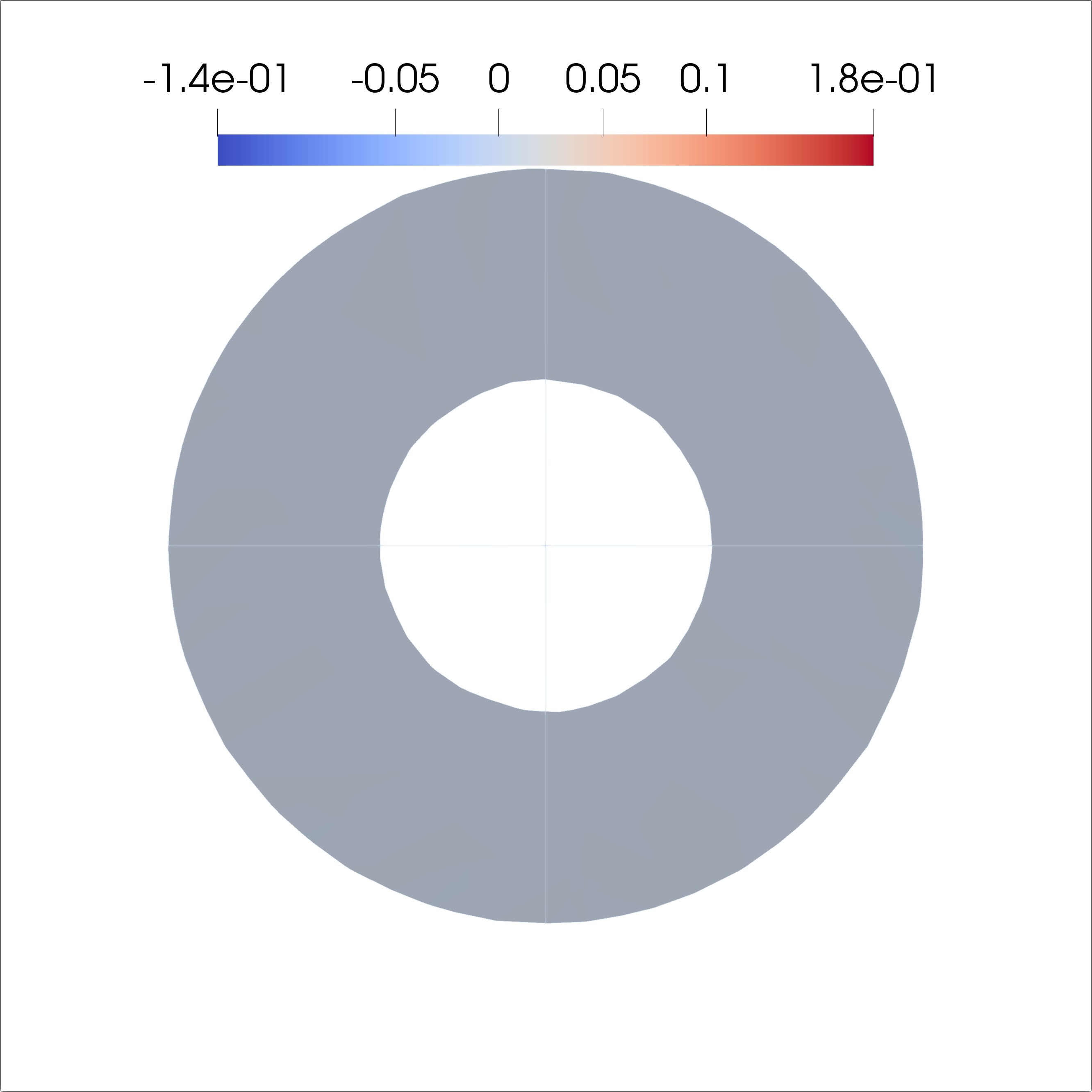}}
\resizebox{0.32\linewidth}{!}{\includegraphics{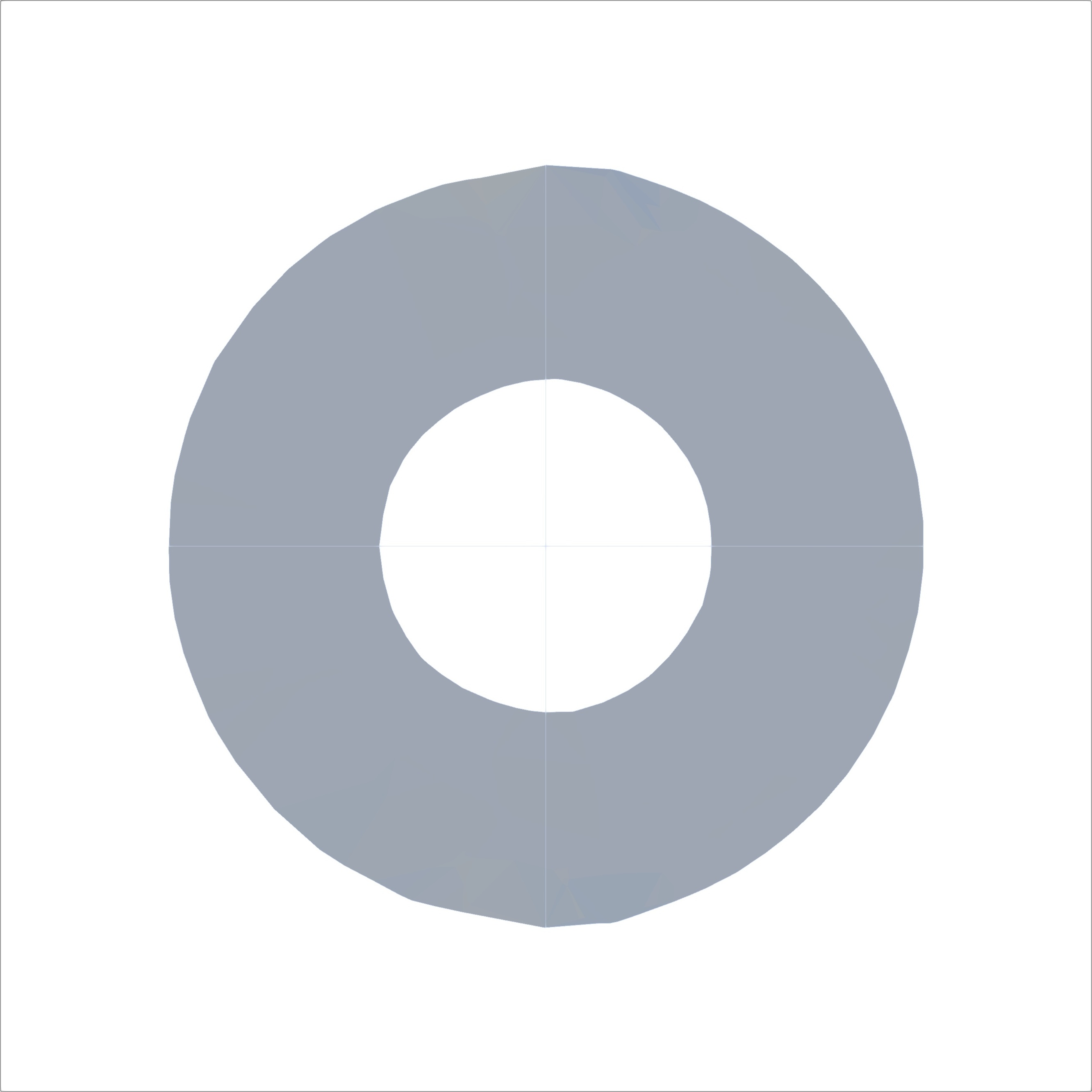}}
\caption{Stokes' flow field and pressure profiles (magnitude) at initial configuration (upper two top rows) and at the computed shape (lower two bottom rows) under coarse mesh viewed on $xz$ (leftmost column), $xy$ (middle column), and $yz$ (rightmost column) plane.}
\label{fig:figure6a}
\end{figure}
%
%
%
%
\begin{figure}[htp!]
\centering
\resizebox{0.32\linewidth}{!}{\includegraphics{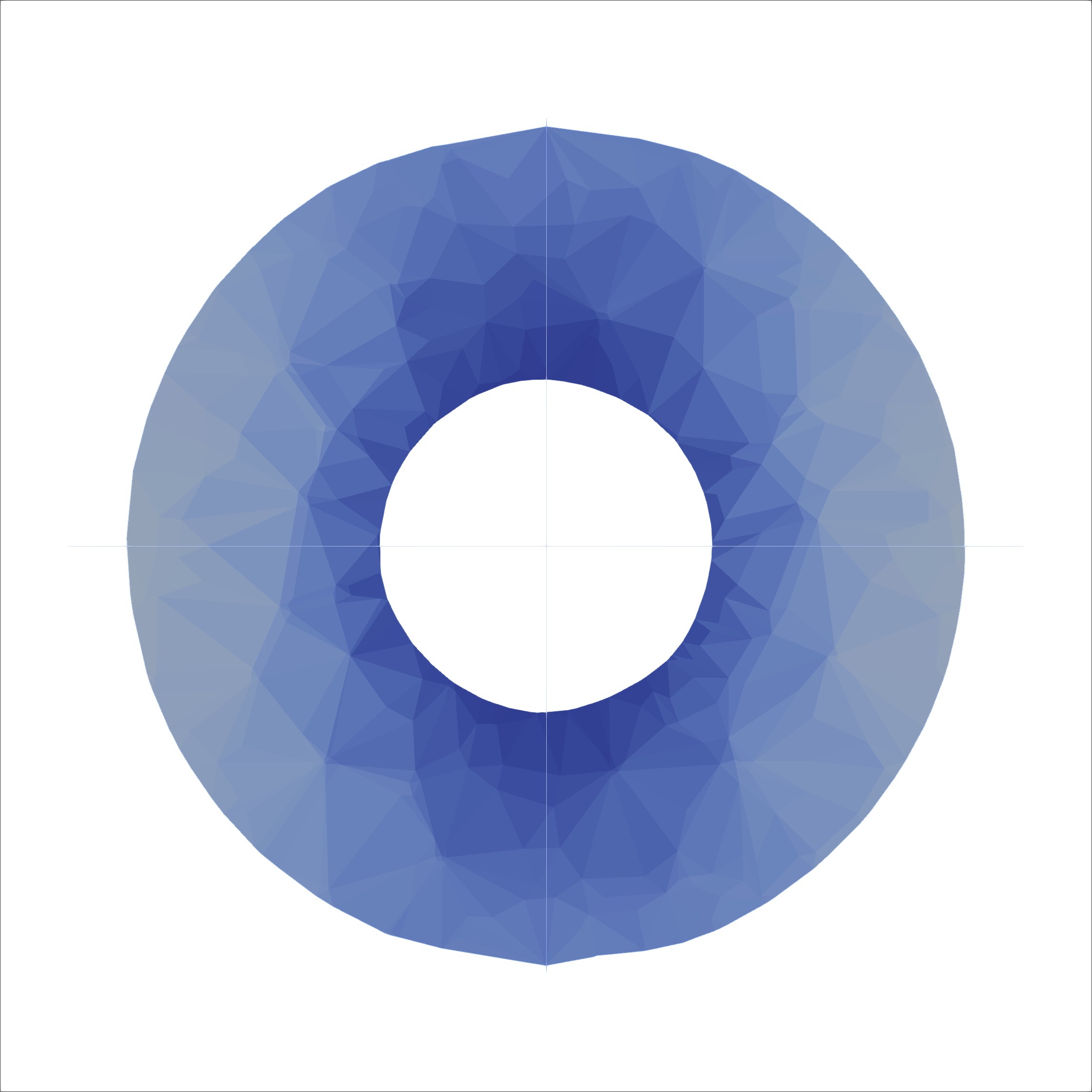}}
\resizebox{0.32\linewidth}{!}{\includegraphics{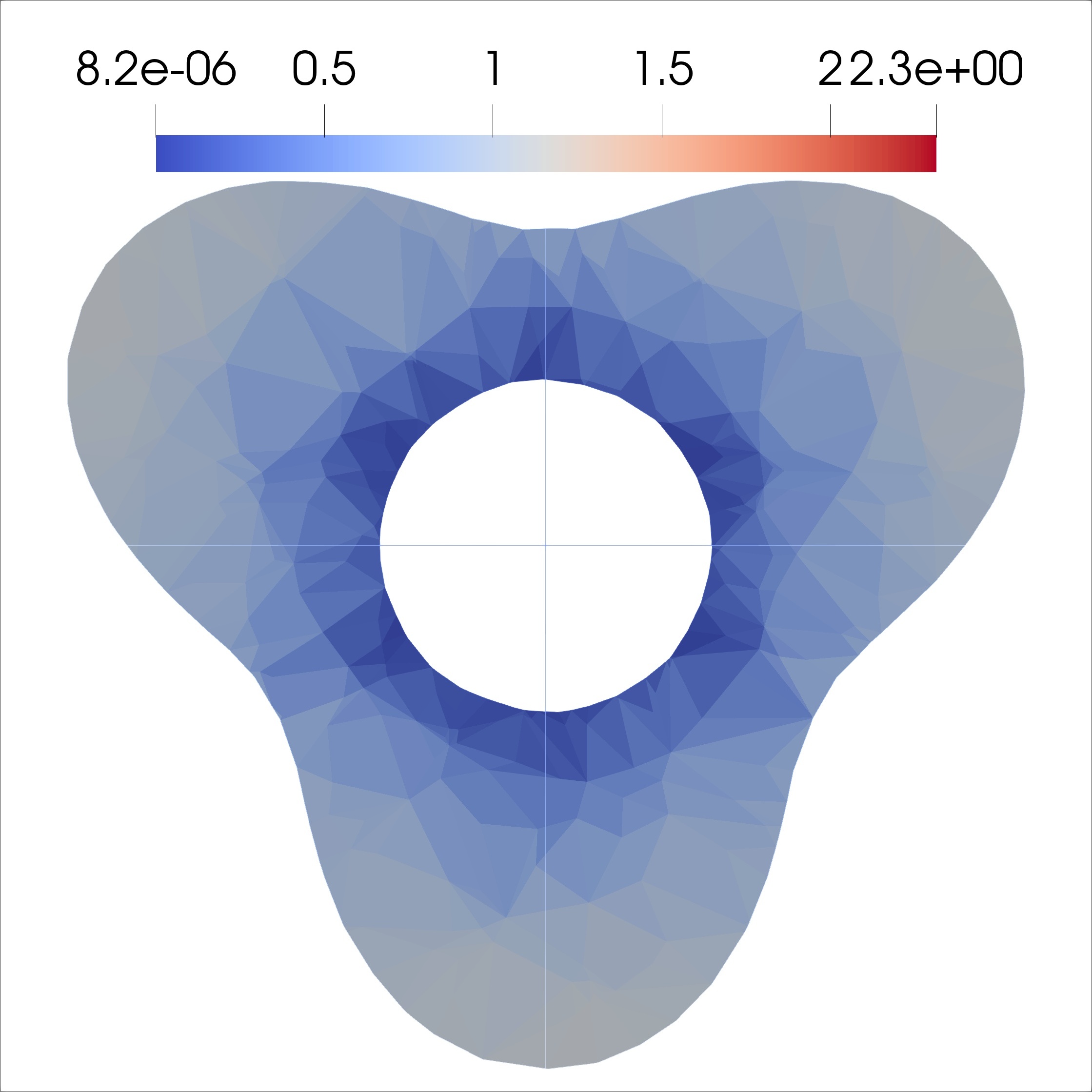}}
\resizebox{0.32\linewidth}{!}{\includegraphics{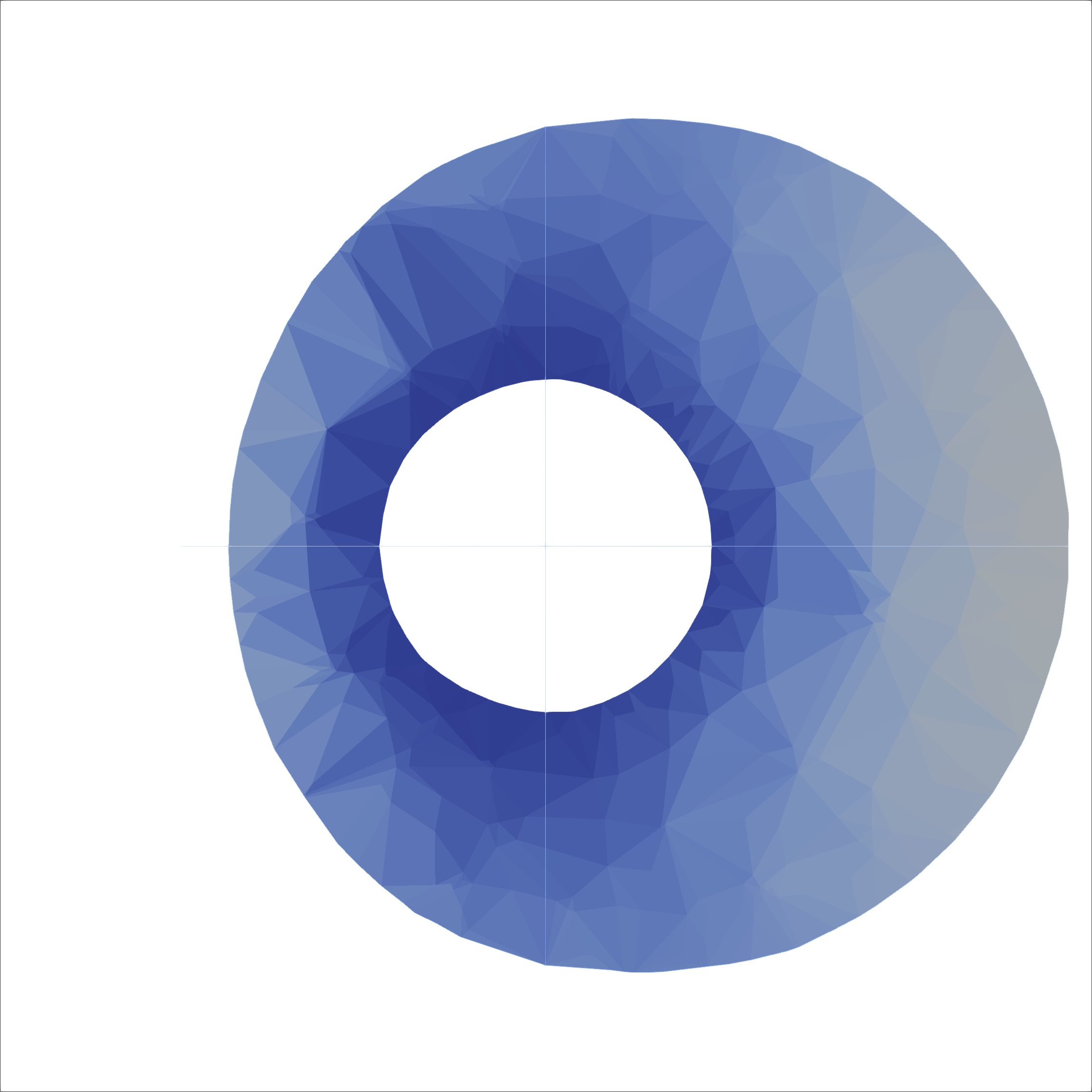}}\\[0.5em]
\resizebox{0.32\linewidth}{!}{\includegraphics{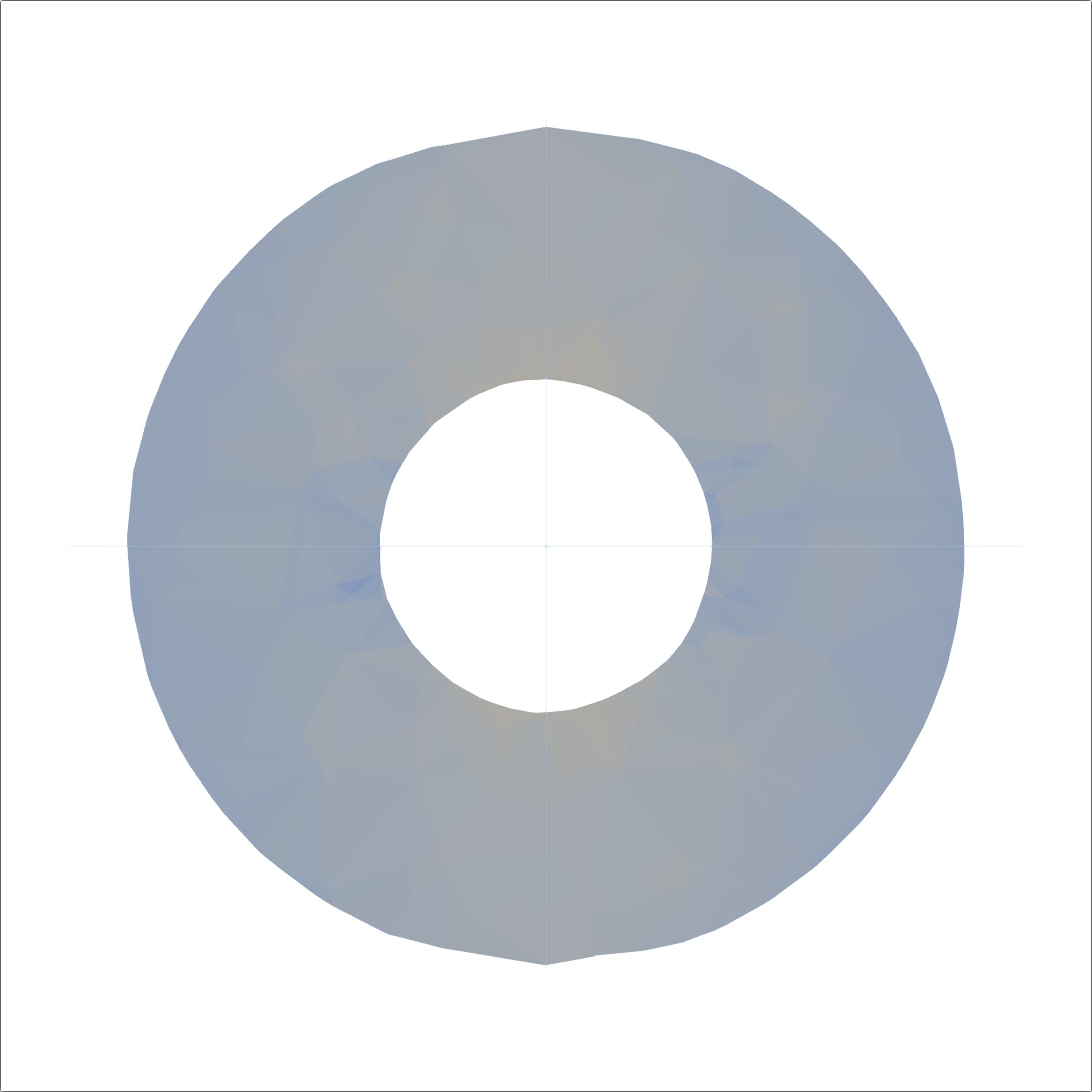}}
\resizebox{0.32\linewidth}{!}{\includegraphics{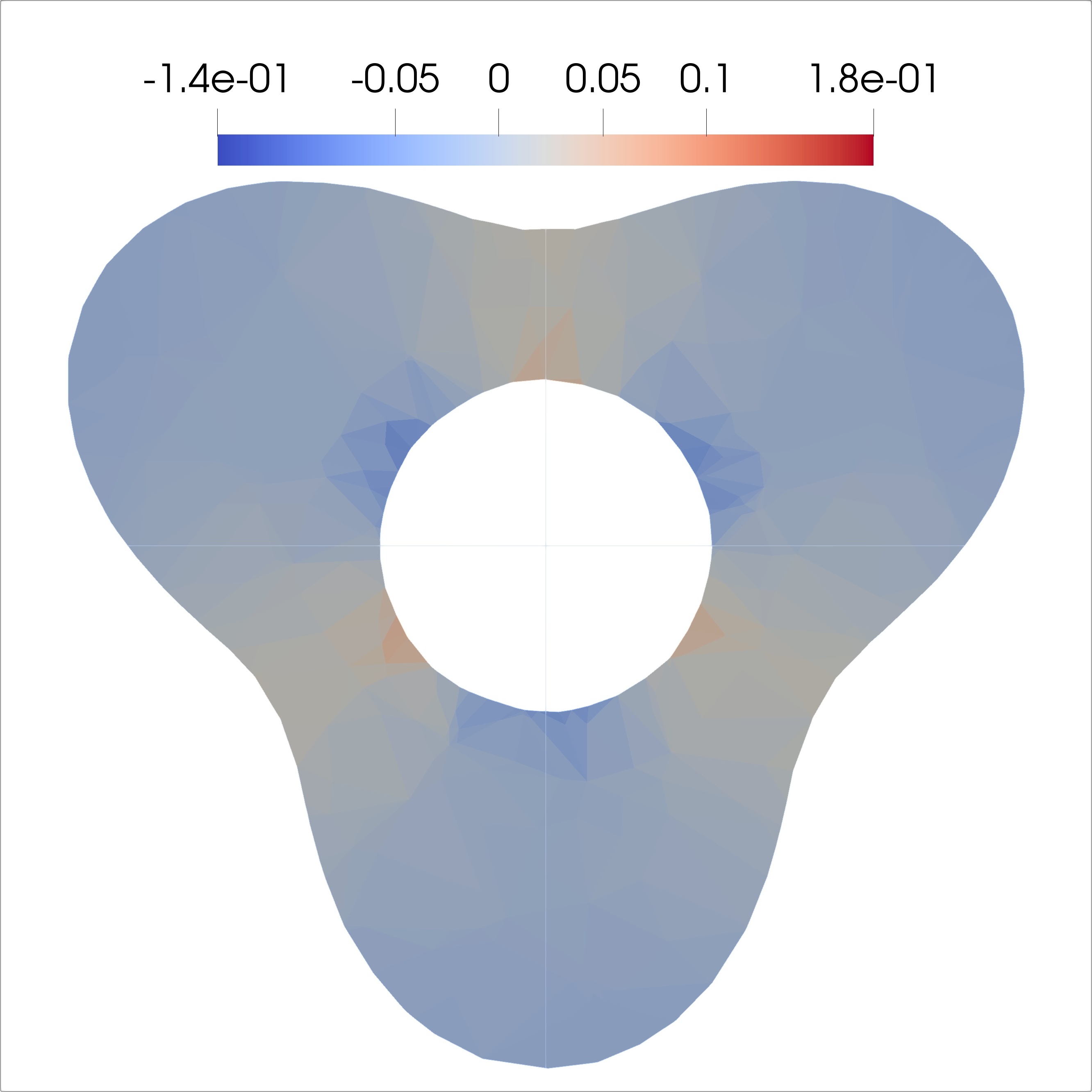}}
\resizebox{0.32\linewidth}{!}{\includegraphics{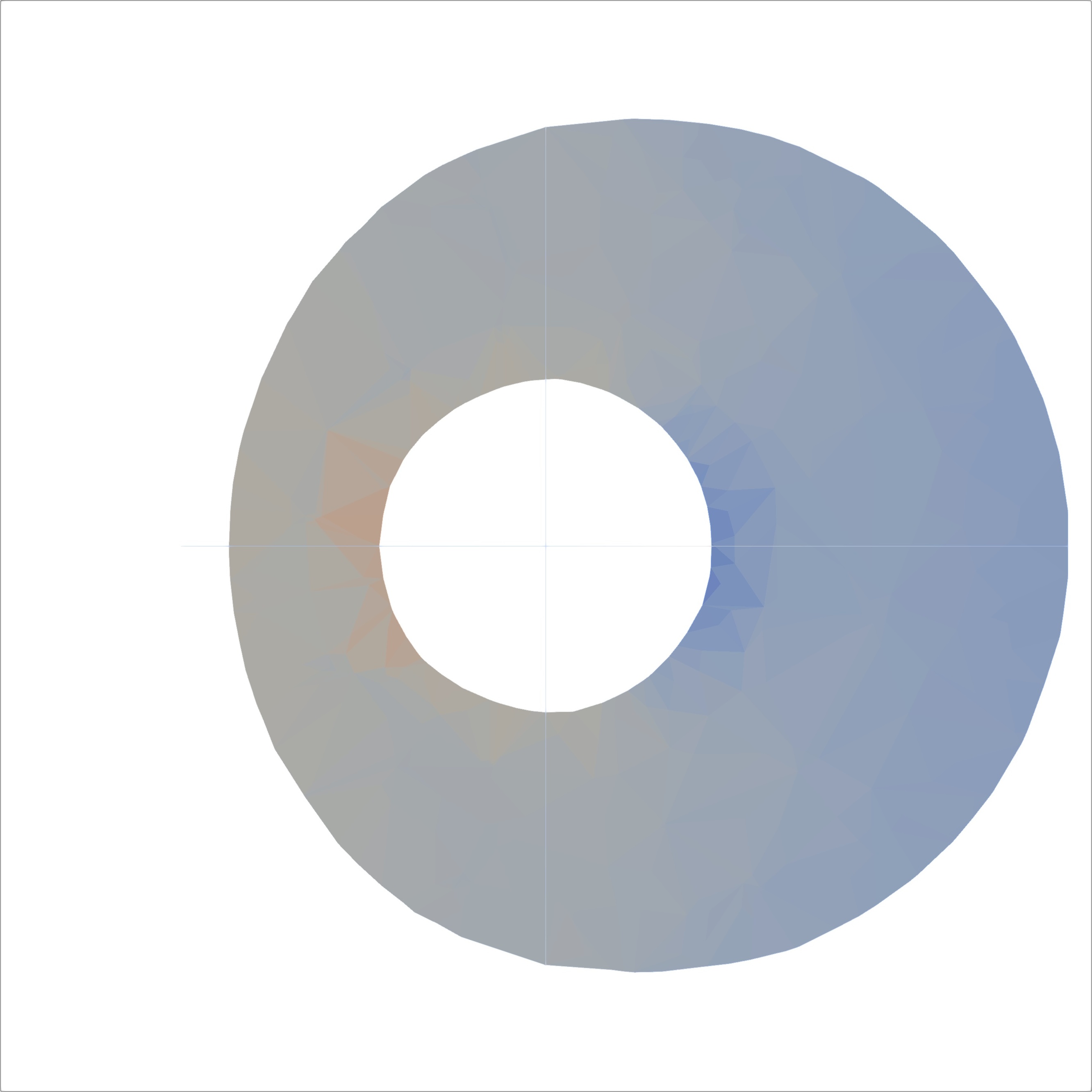}}\\[0.5em]
\resizebox{0.32\linewidth}{!}{\includegraphics{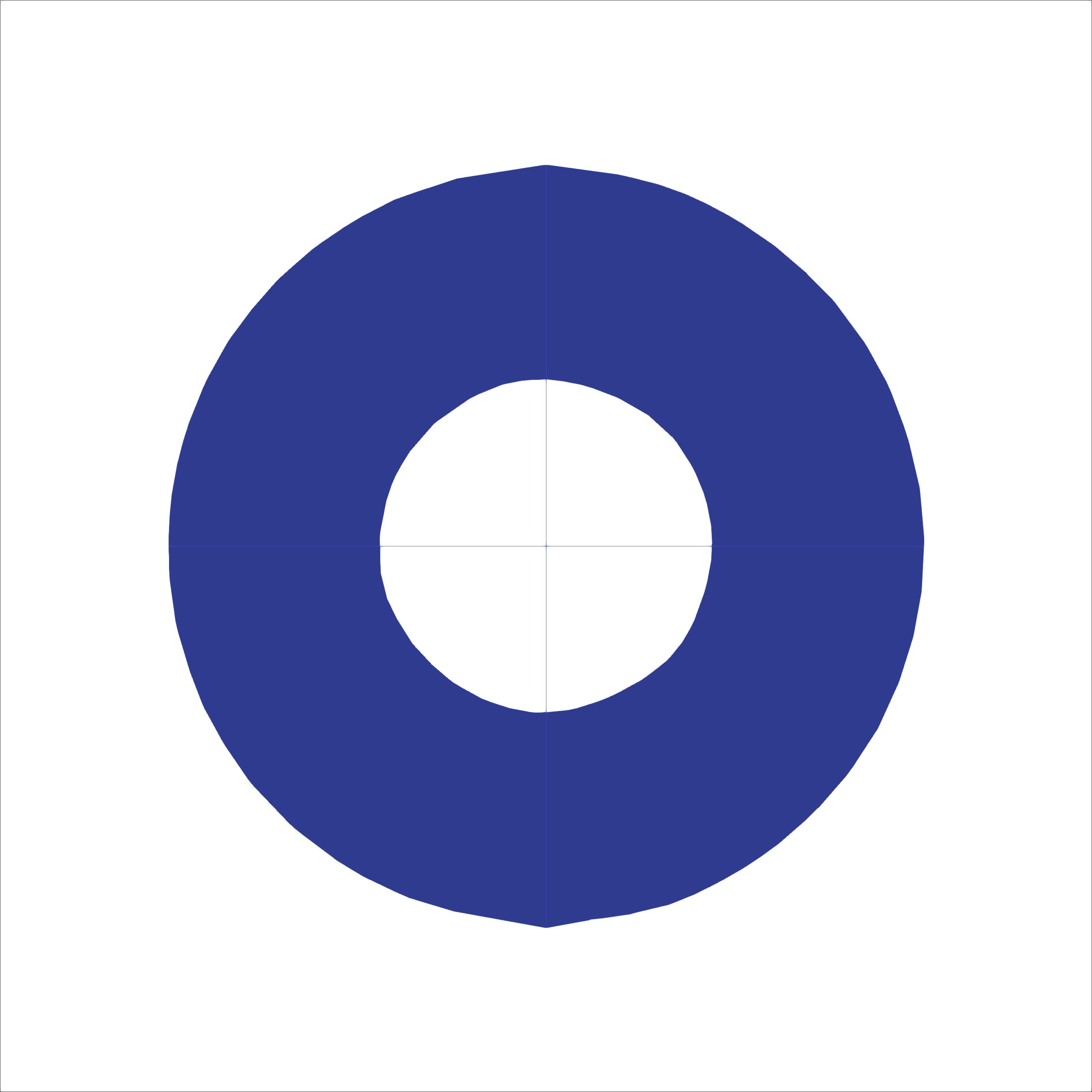}}
\resizebox{0.32\linewidth}{!}{\includegraphics{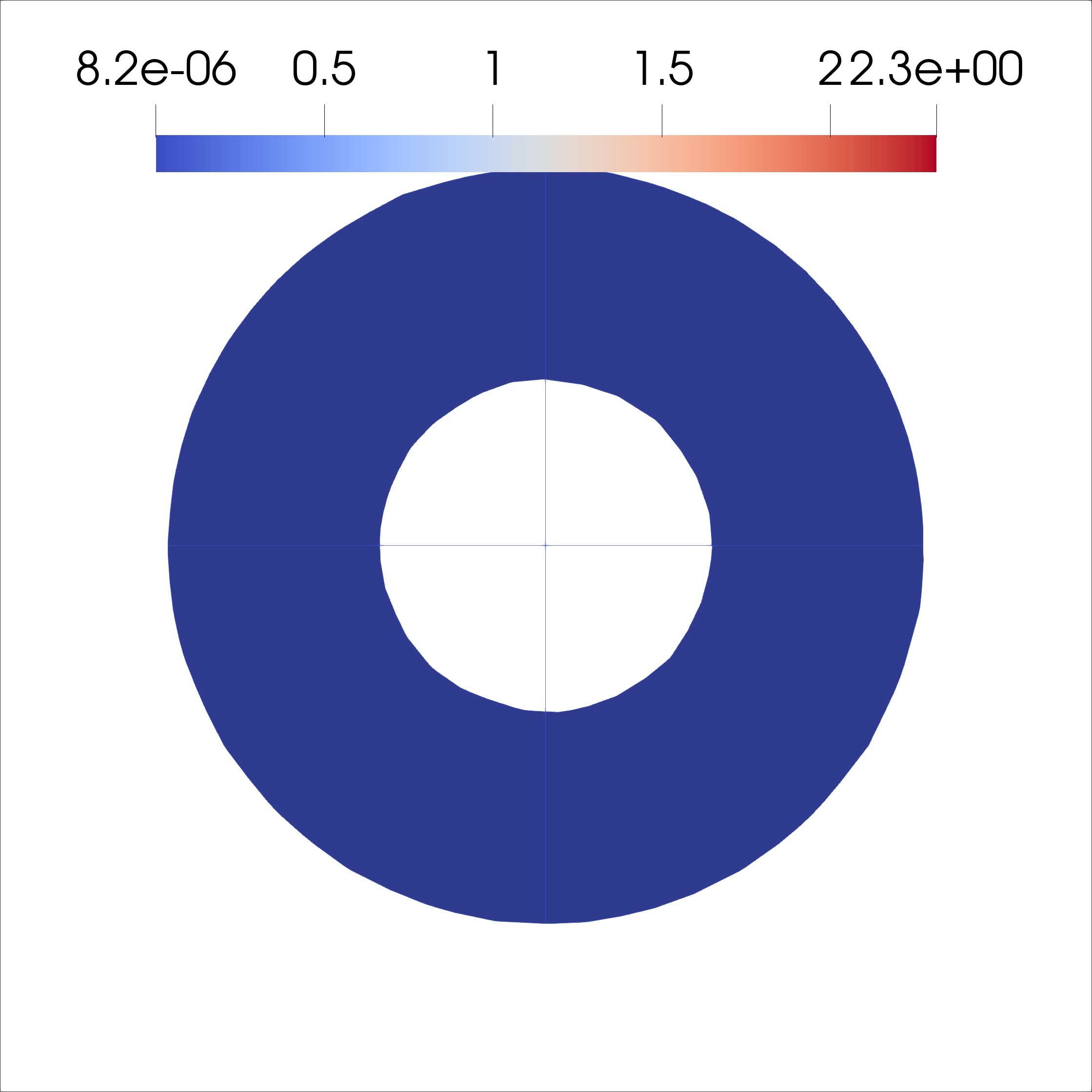}}
\resizebox{0.32\linewidth}{!}{\includegraphics{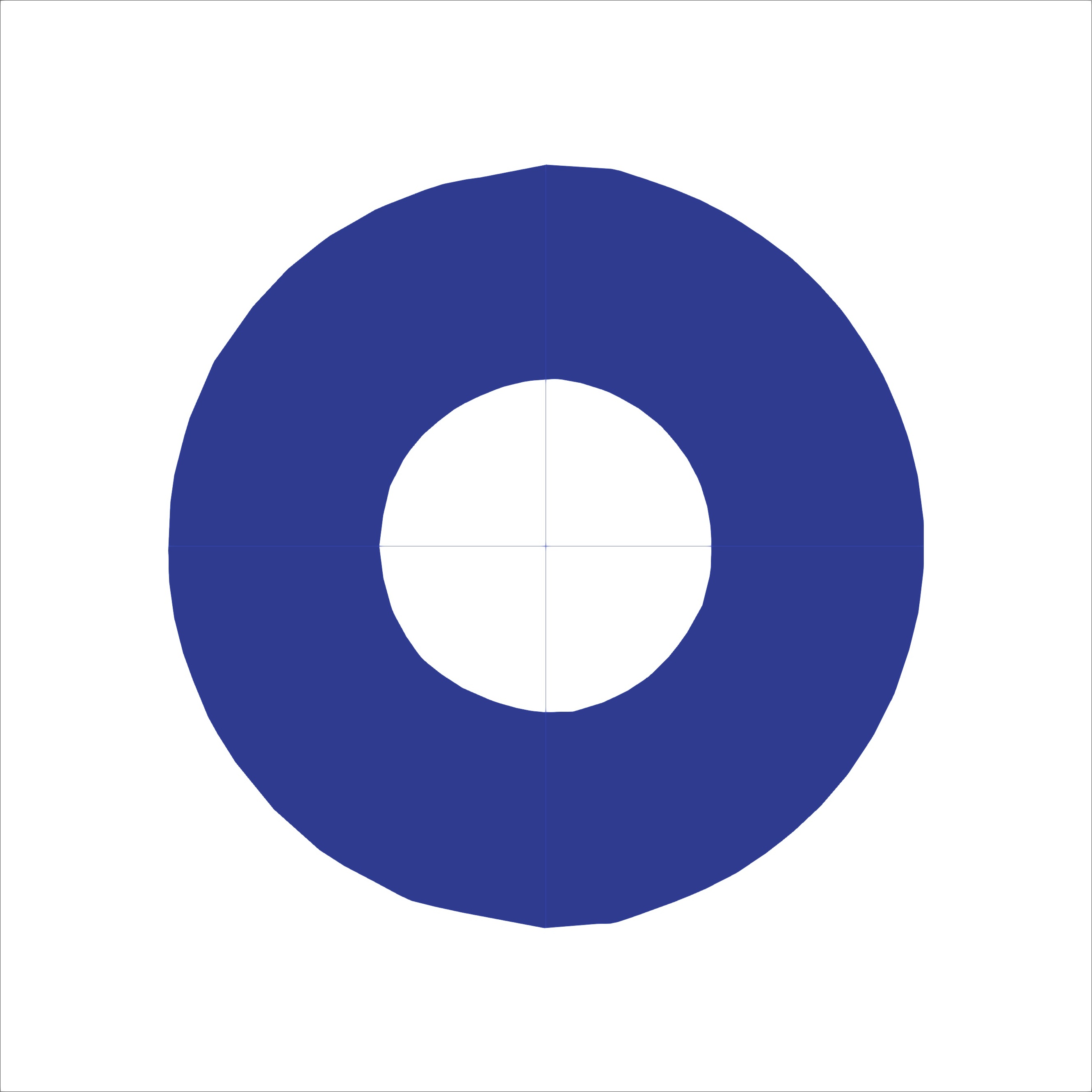}}\\[0.5em]
\resizebox{0.32\linewidth}{!}{\includegraphics{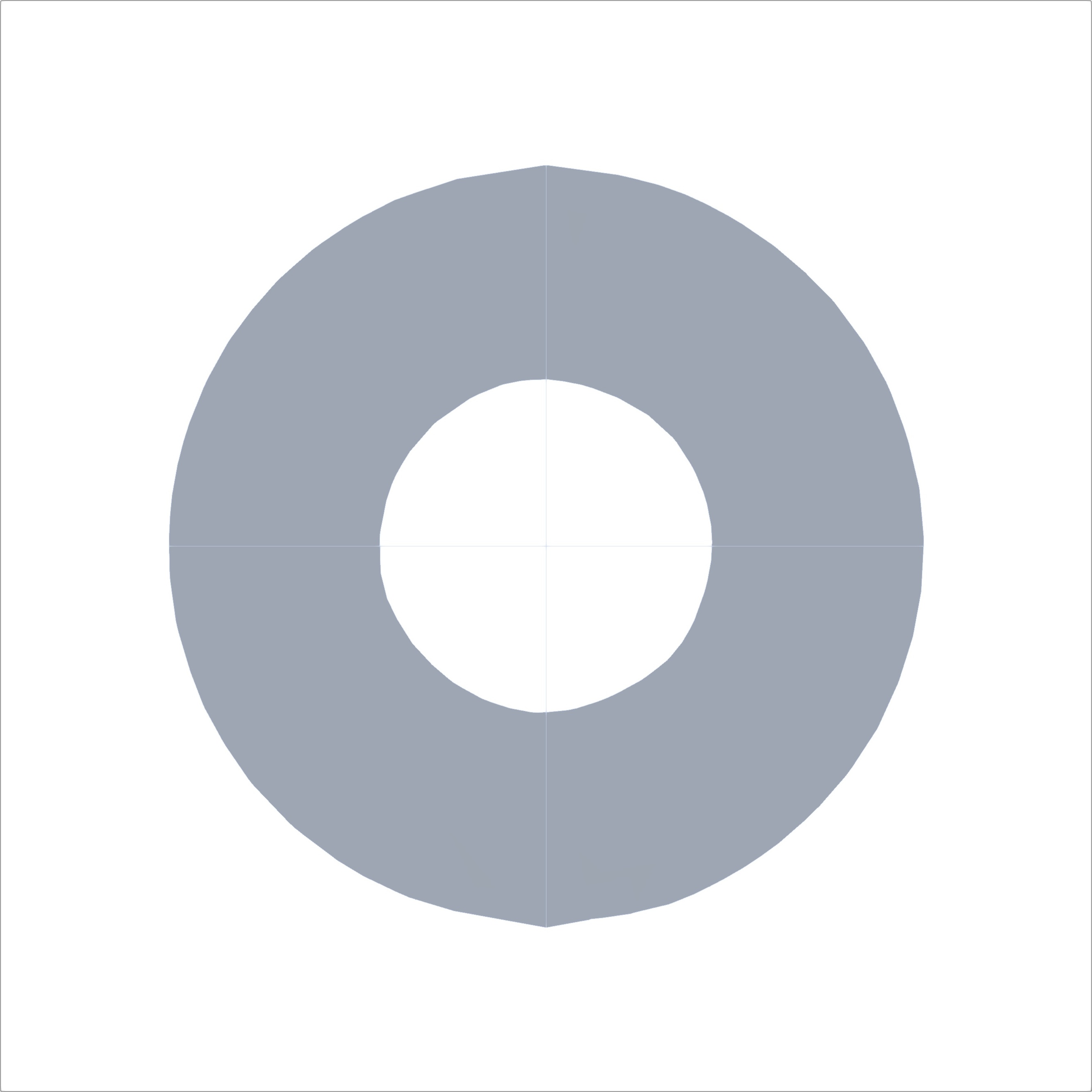}}
\resizebox{0.32\linewidth}{!}{\includegraphics{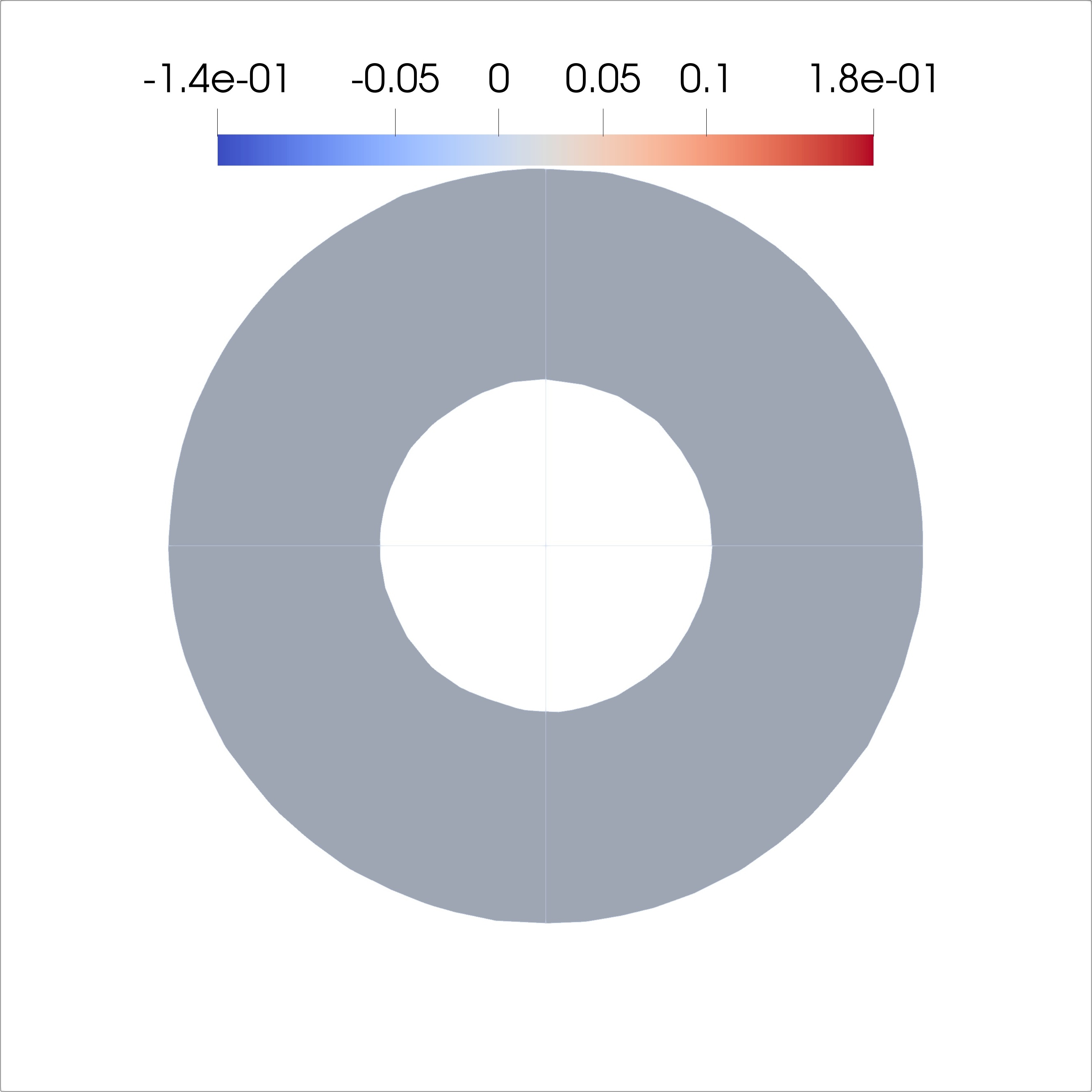}}
\resizebox{0.32\linewidth}{!}{\includegraphics{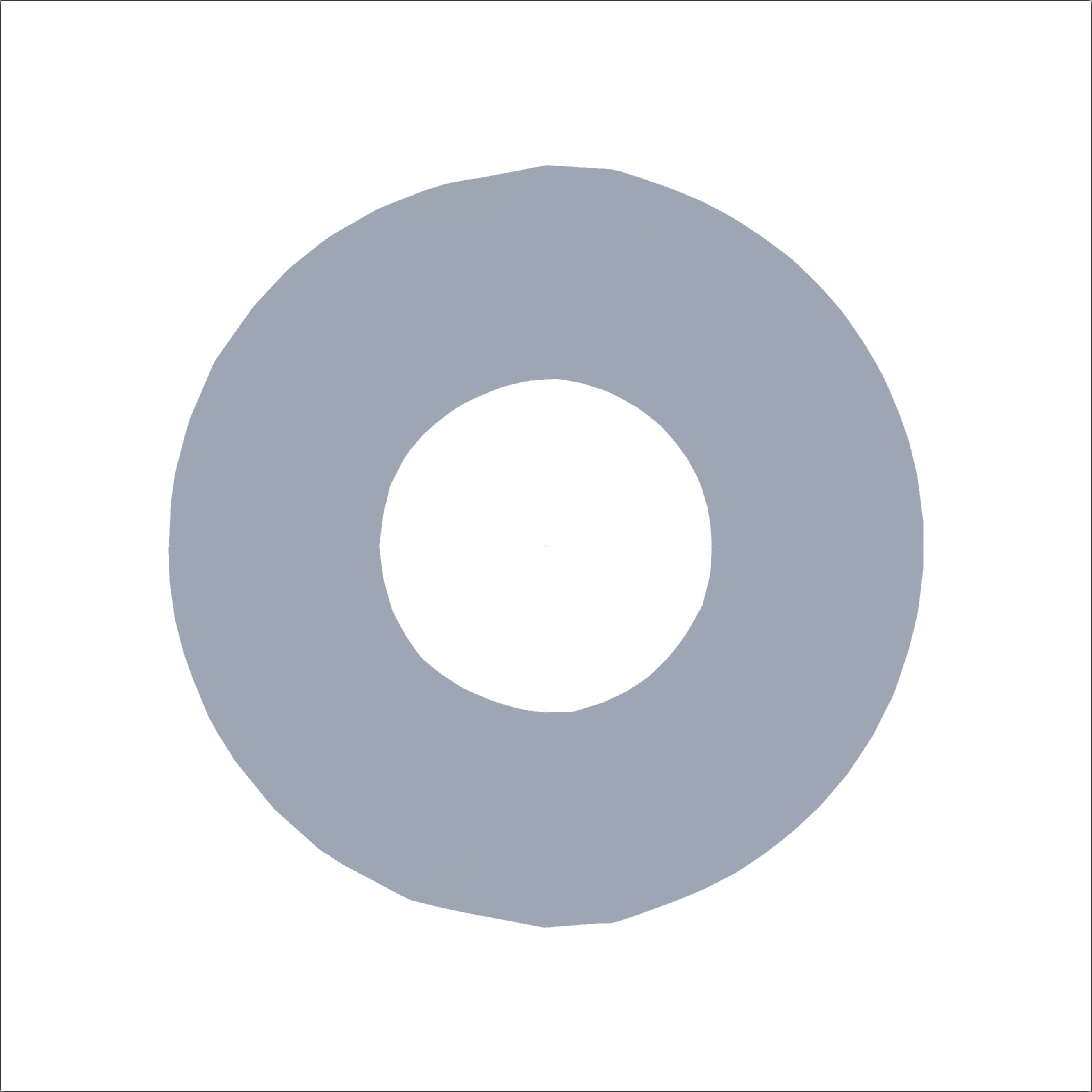}}
\caption{Adjoint's flow field and pressure profiles (magnitude) at initial configuration (upper two top rows) and at the computed shape (lower two bottom rows) under coarse mesh viewed on $xz$ (leftmost column), $xy$ (middle column), and $yz$ (rightmost column) plane.}
\label{fig:figure6b}
\end{figure}
\begin{figure}[htp!]
\centering
\resizebox{0.32\linewidth}{!}{\includegraphics{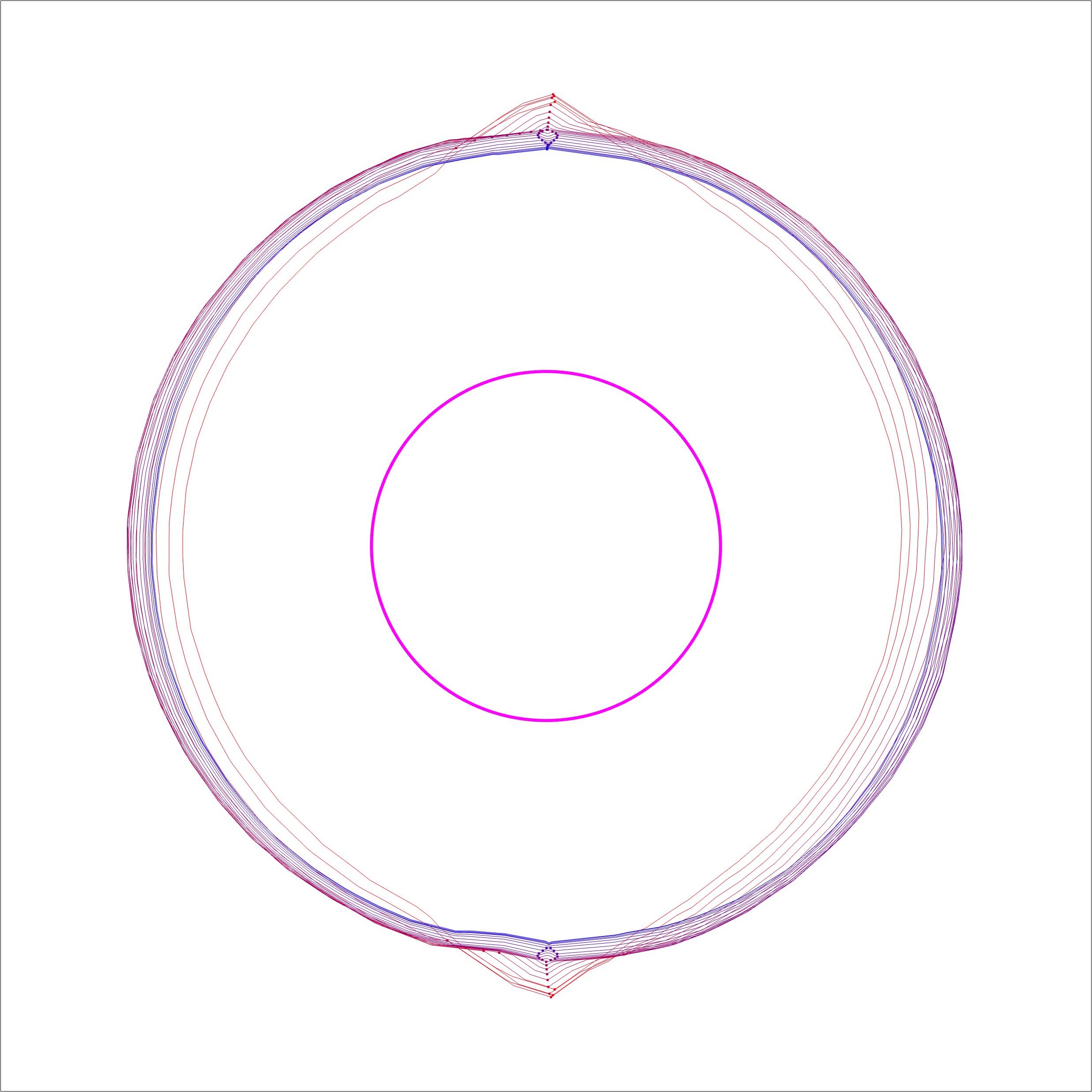}} \hfill
\resizebox{0.32\linewidth}{!}{\includegraphics{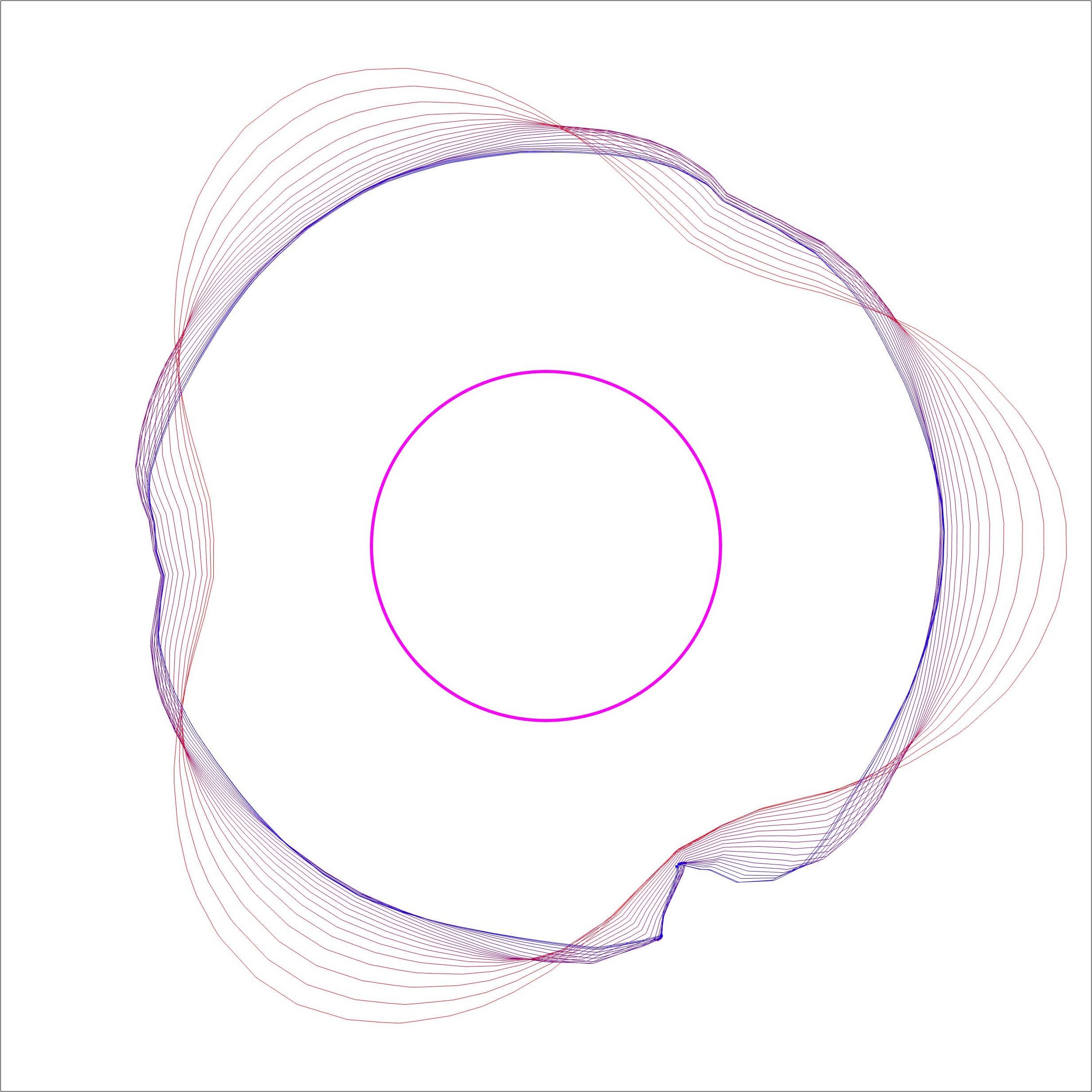}} \hfill
\resizebox{0.32\linewidth}{!}{\includegraphics{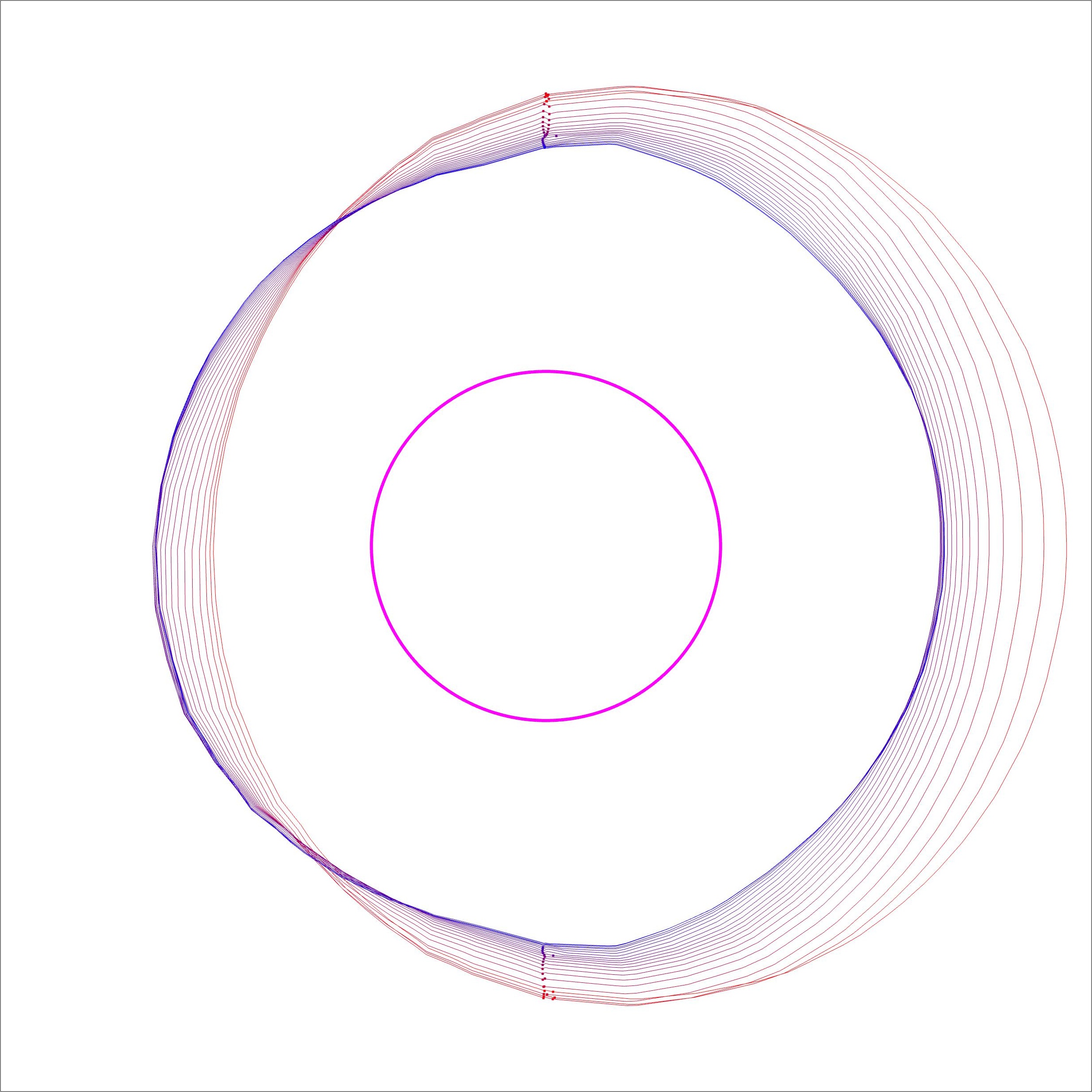}}  \\[0.5em]
\resizebox{0.32\linewidth}{!}{\includegraphics{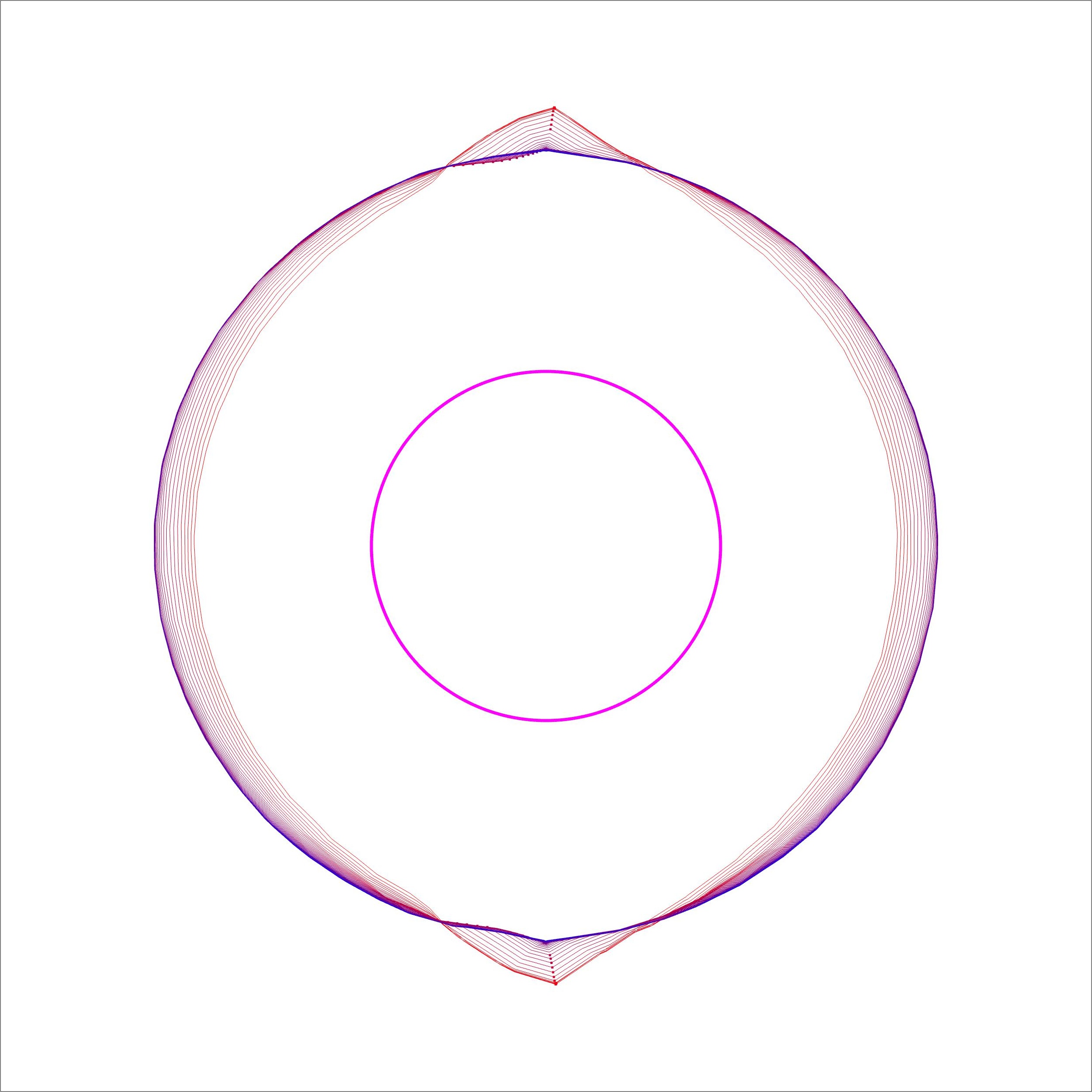}} \hfill
\resizebox{0.32\linewidth}{!}{\includegraphics{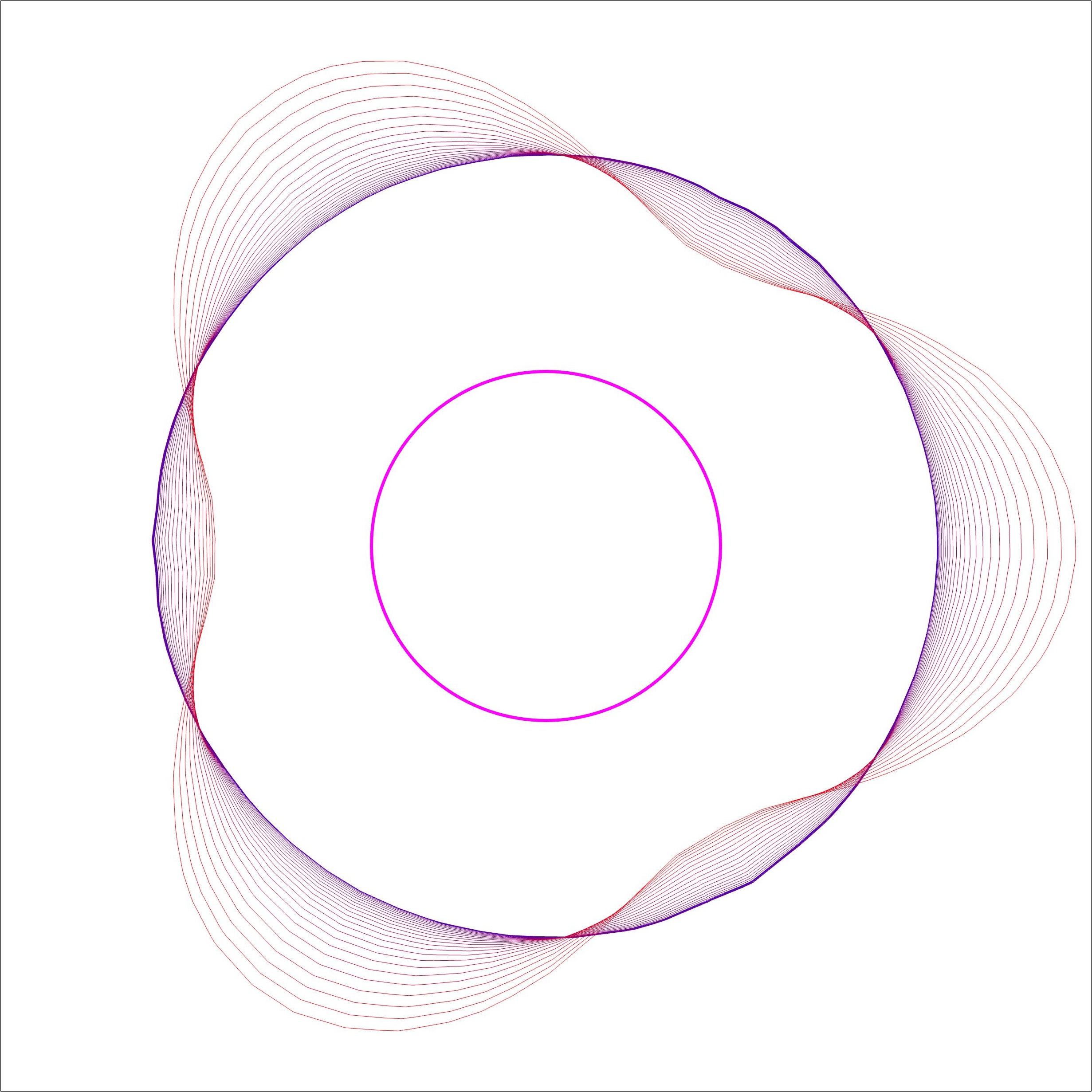}} \hfill
\resizebox{0.32\linewidth}{!}{\includegraphics{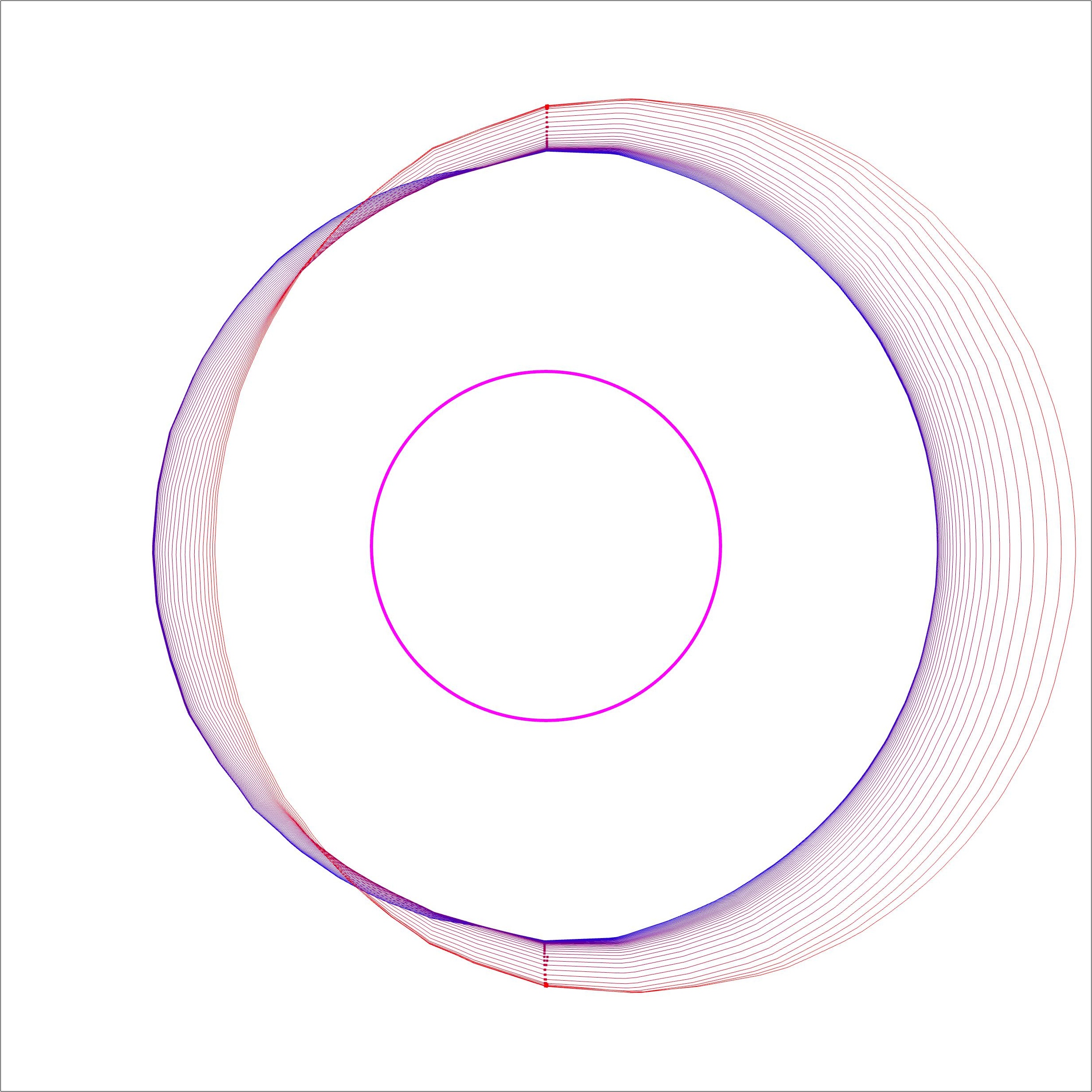}}
\caption{Shape histories of the free boundary computed using TD (upper plots) and CCBM (lower plots) with finer mesh viewed on the plane $xz$ (leftmost column), $yx$ (middle column), and $yz$ (rightmost column)}
\label{fig:figure7}
\end{figure}
%
%
%
%
\begin{figure}[htp!]
\centering
\resizebox{0.32\linewidth}{!}{\includegraphics{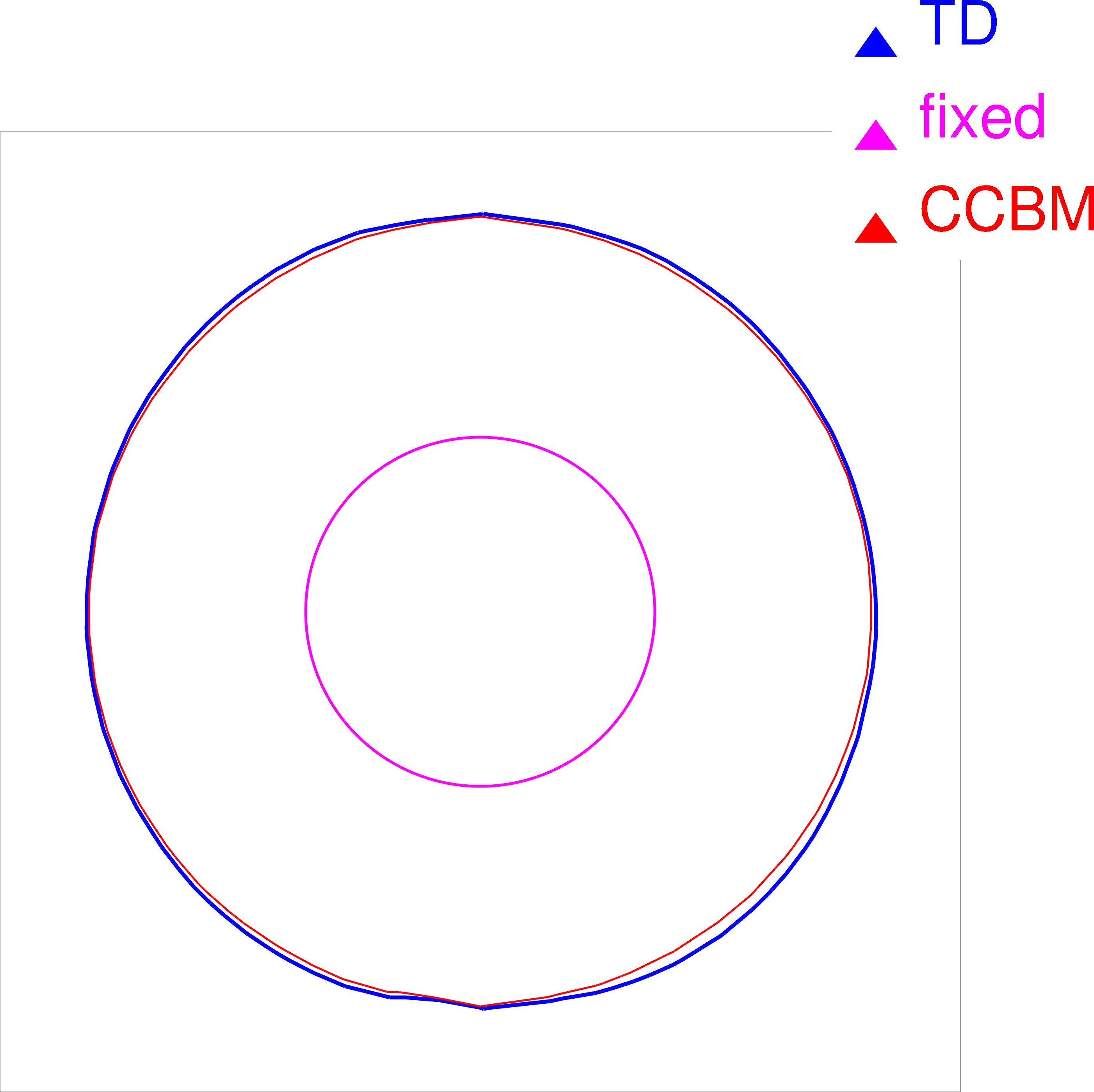}} \hfill
\resizebox{0.32\linewidth}{!}{\includegraphics{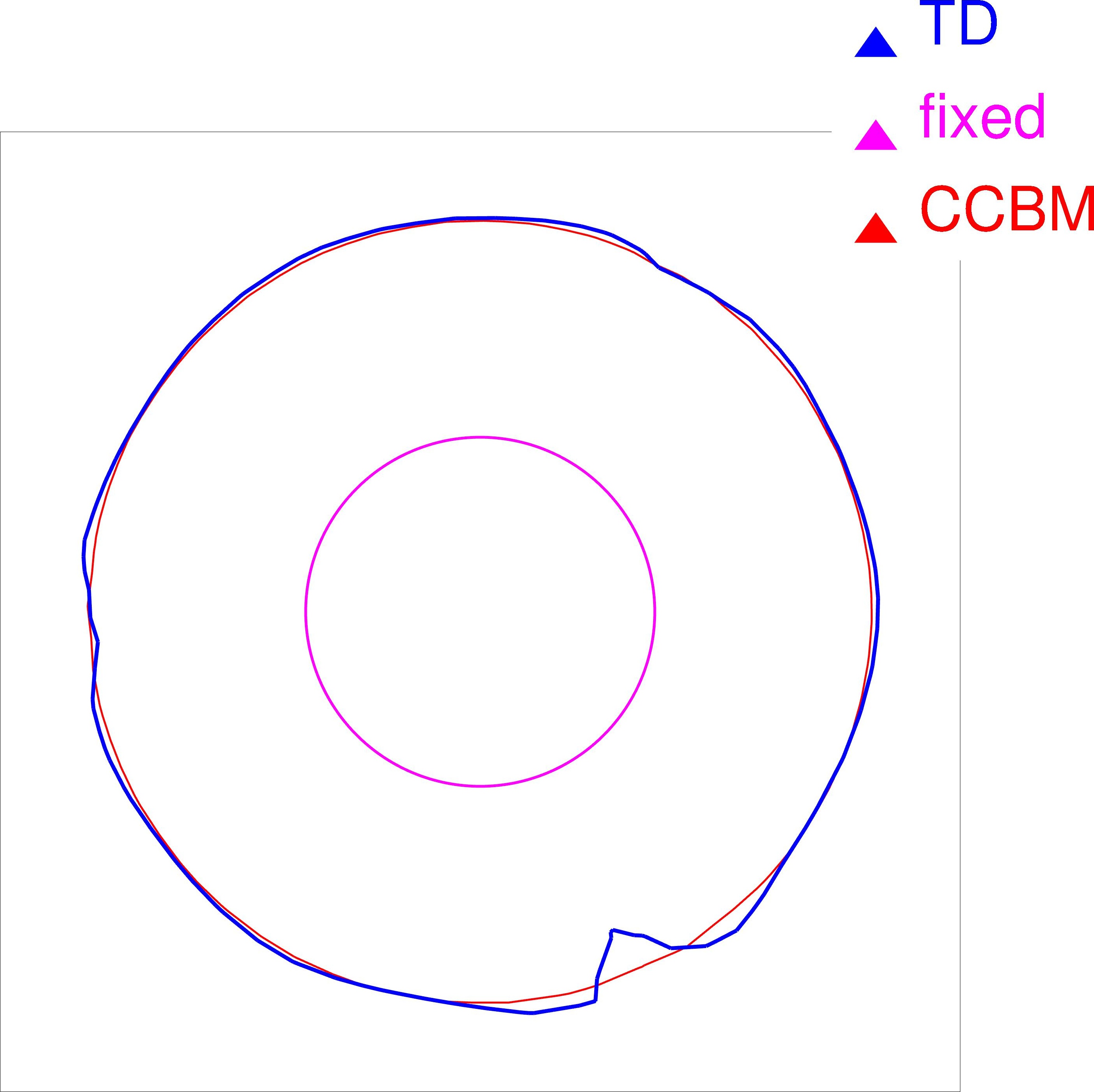}} \hfill
\resizebox{0.32\linewidth}{!}{\includegraphics{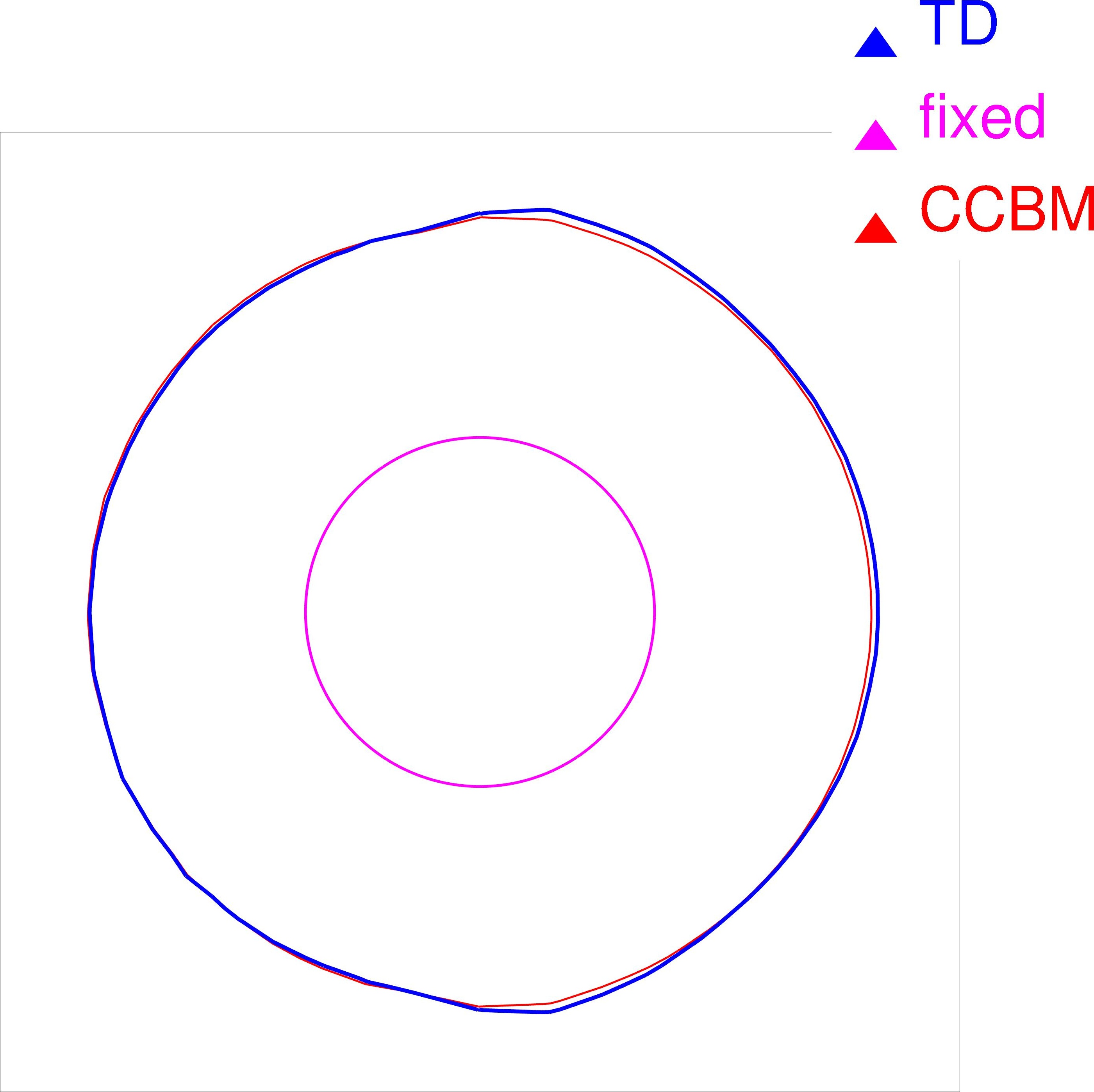}} \hfill
\caption{Cross comparison of computed shapes (case of finer meshes) viewed on the plane $xz$ (leftmost plot), $yx$ (middle plot), and $yz$ (rightmost plot)}
\label{fig:figure8}
\end{figure}
%
%
%
%
\begin{figure}[htp!]
\centering
\resizebox{0.45\linewidth}{!}{\includegraphics{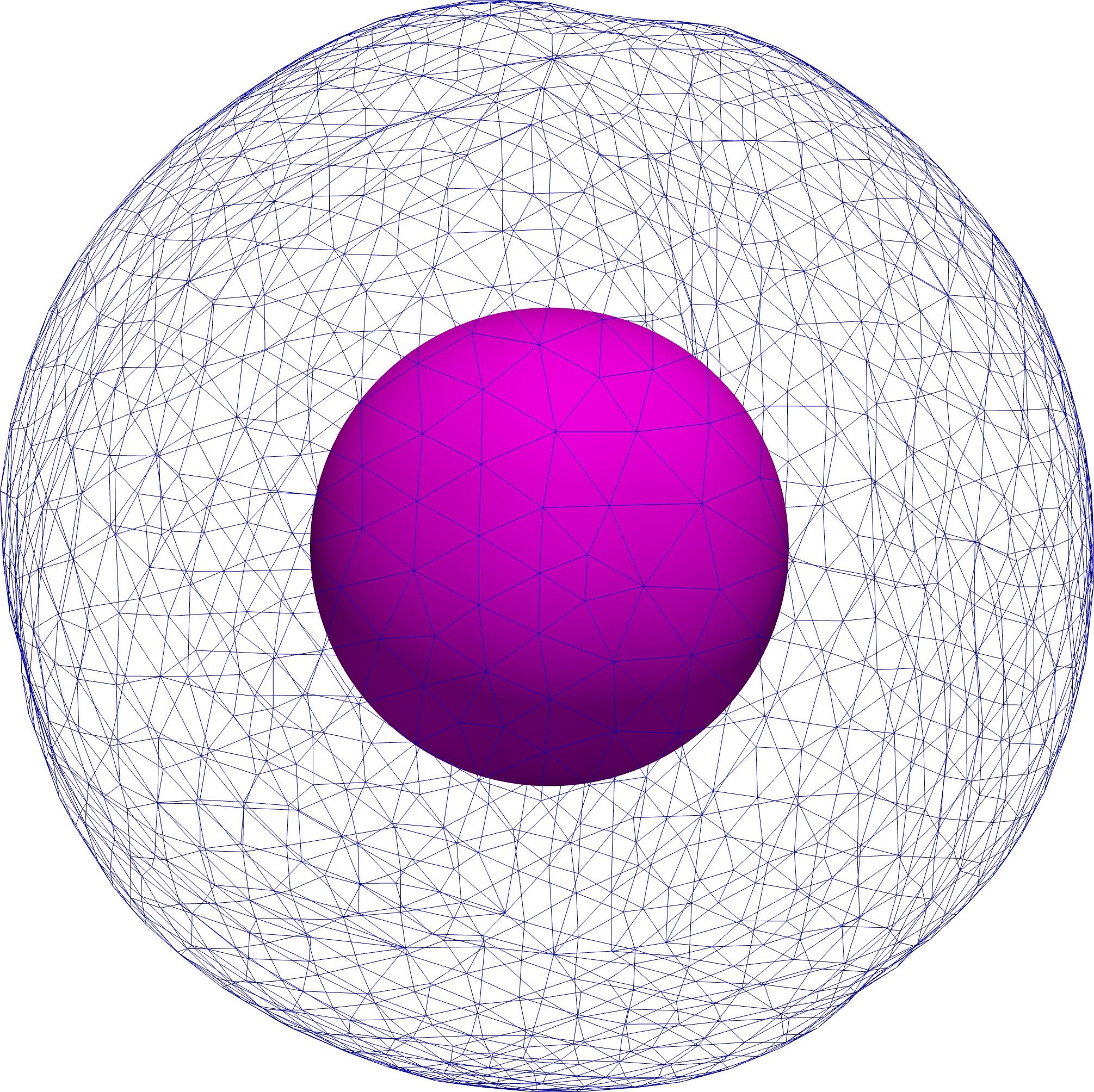}} \hfill
\resizebox{0.45\linewidth}{!}{\includegraphics{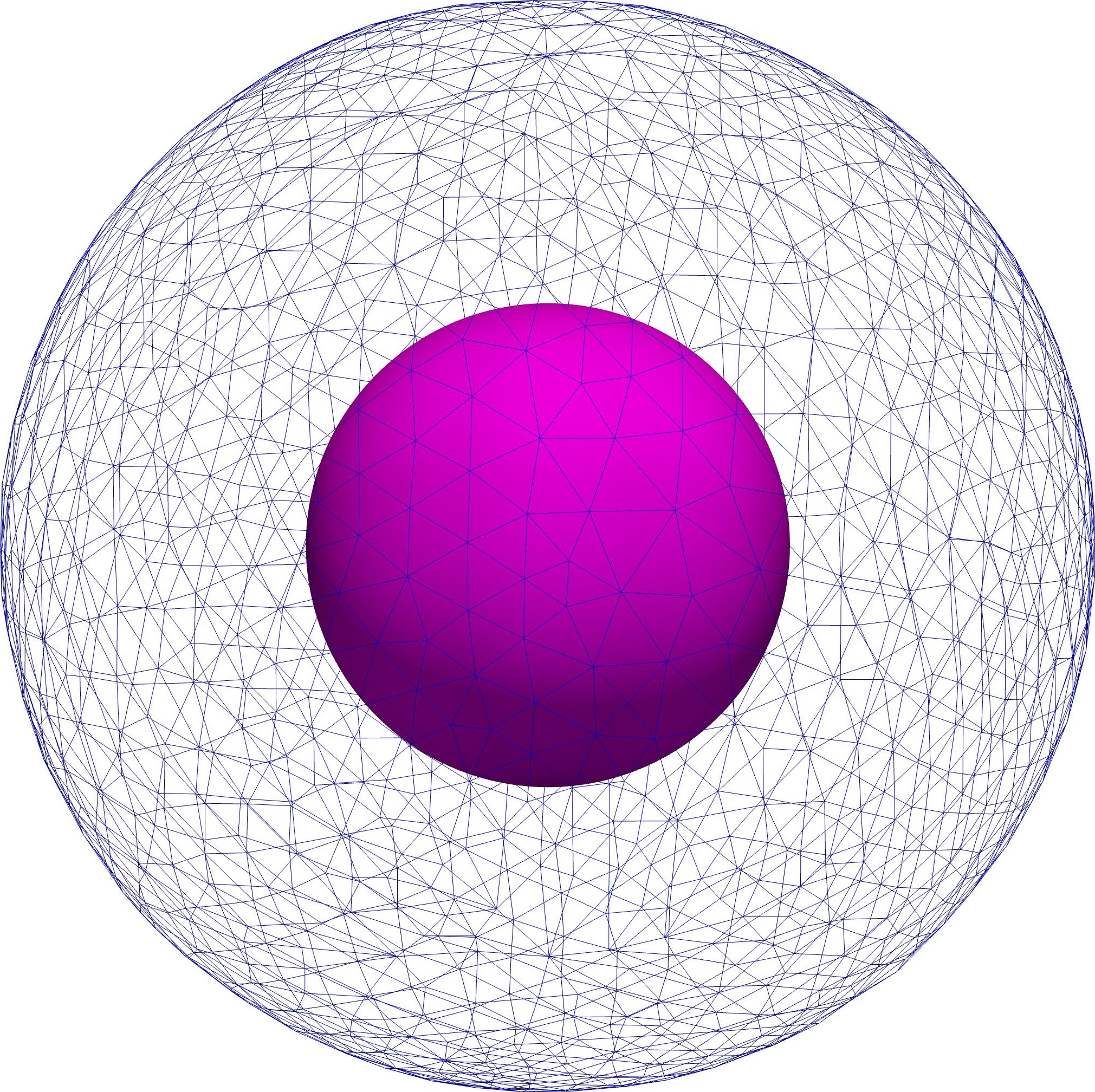}} 
\caption{Mesh profile of computed shapes (TD: left plot, CCBM: right plot) with finer mesh viewed on different planes}
\label{fig:figure9}
\end{figure}
%
%
%
%
\begin{figure}[htp!]
\centering
\resizebox{0.48\linewidth}{!}{\includegraphics{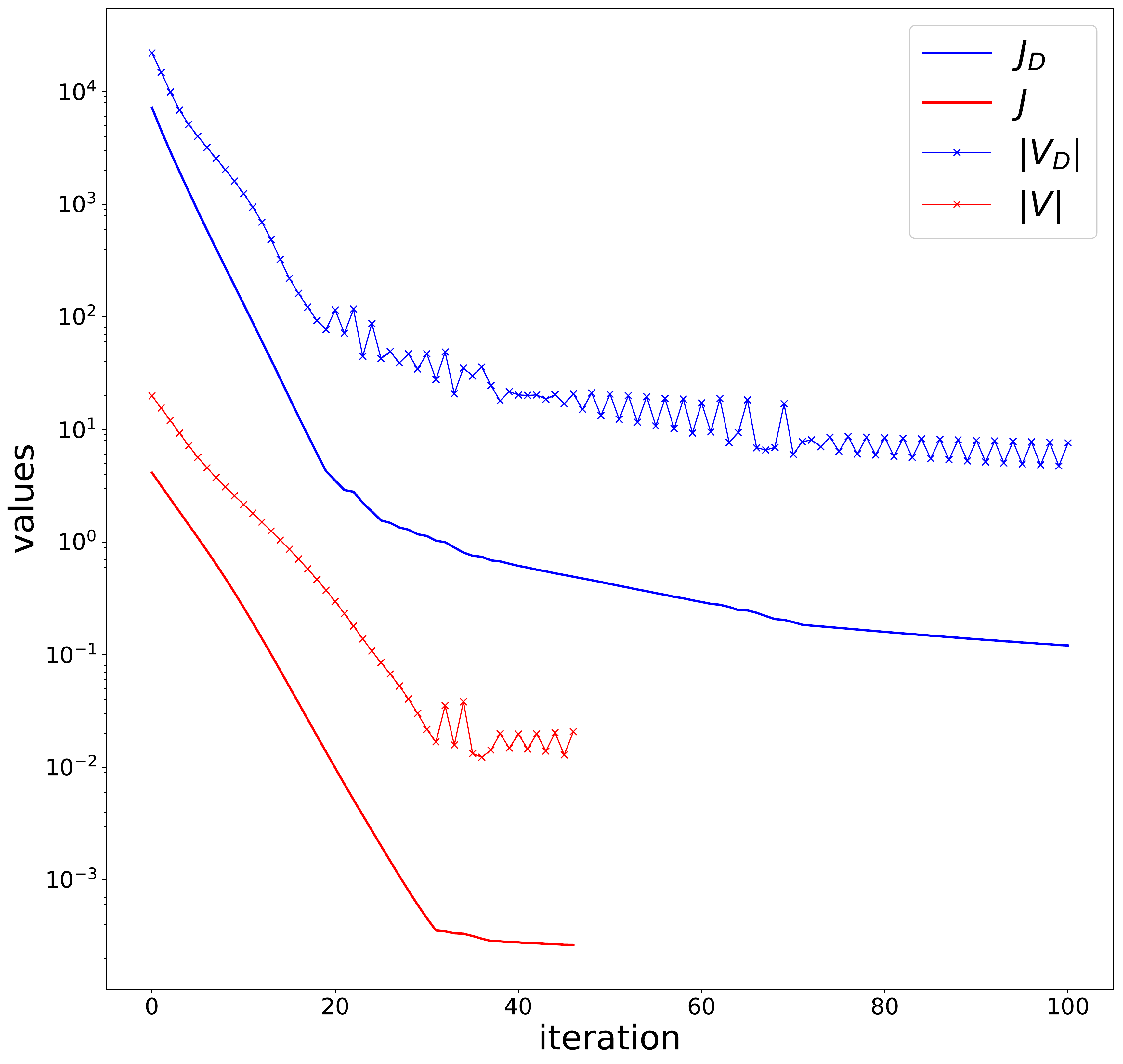}} \hfill
\resizebox{0.48\linewidth}{!}{\includegraphics{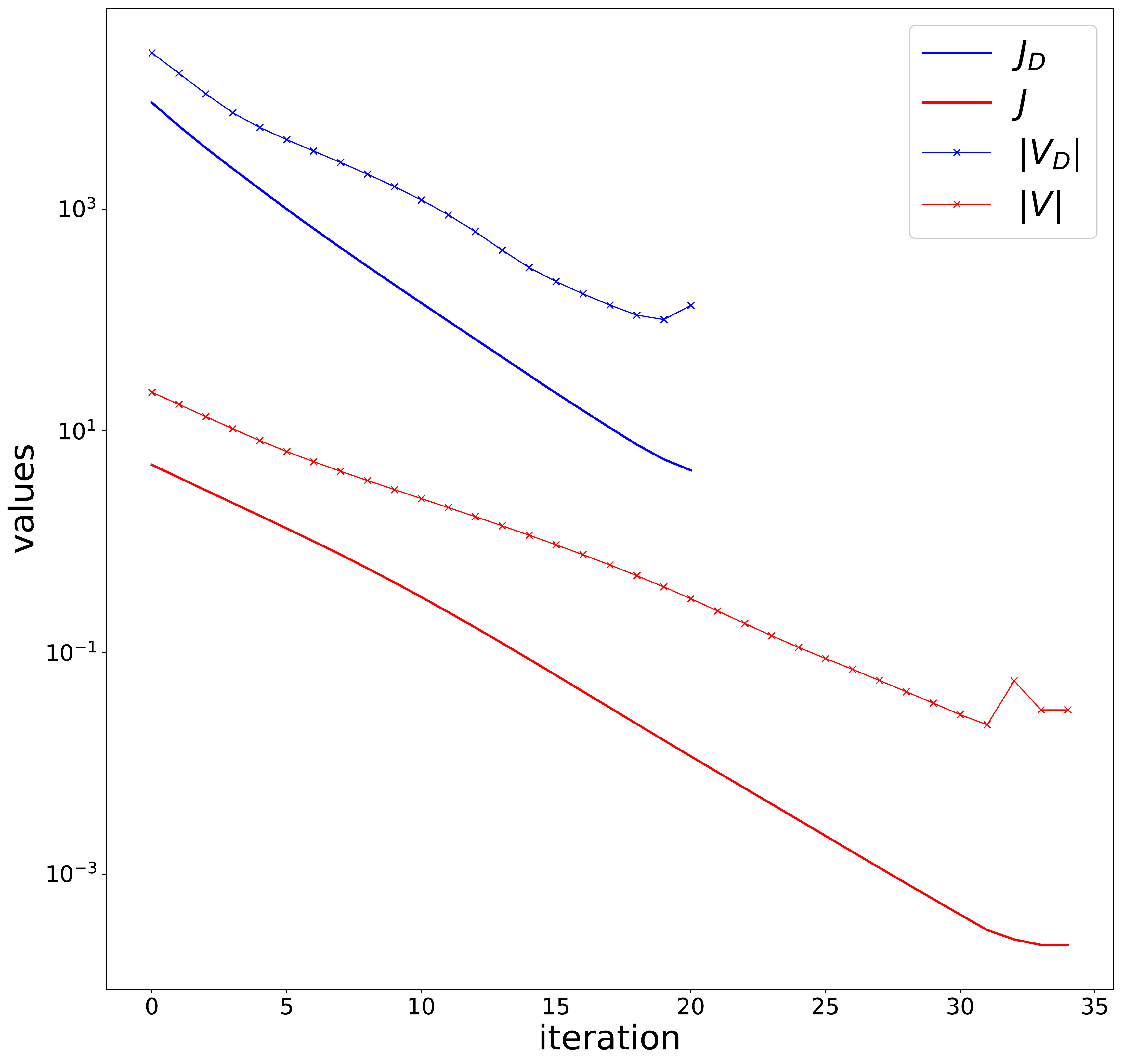}} 
\caption{Histories of cost and gradient norm values for coarse (left) and finer mesh (right)}
\label{fig:figure10}
\end{figure}
\section{Conclusions and Future Work}\label{sec:Conclusions_and_Future_Work}
In this investigation, we have developed a coupled complex boundary method in shape optimization framework to resolve a free surface problem involving the Stokes equation.
The shape gradient of the cost which was obtained naturally from the proposed formulation was delicately computed under a mild regularity assumption on the domain and without using the shape derivative of the states.
Using the shape gradient information, a Sobolev gradient-based descent scheme was formulated in order to solve the problem numerically via finite element method.
The realization and implementation of the method and scheme put forward in this article in carried out in two and three dimensions. 
Numerical results showed that the new method has some advantages over the classical approach of tracking the Dirichlet data in least-squares sense.
Moreover, it seems that the method is more accurate (in the sense that it achieves the expected optimal shape solution) compared to the classical boundary tracking method, as expected.

For future work, we propose to calculate and examine the shape Hessian of the cost functional to investigate the ill-posedness of the proposed shape optimization problem.
The said expression could then be used in a shape Newton method to numerically solve the minimization problem.
On the other hand, an application of CCBM in solving inverse obstacle problems under shape optimization settings will also be a subject of our next investigation.


%
%

\bibliographystyle{alpha} 
\bibliography{references}   

\appendix
\counterwithin{theorem}{subsection}
\section{Appendices}\label{appx} 
\renewcommand{\thesubsection}{\Alph{subsection}}
\renewcommand{\theequation}{\Alph{subsection}.\arabic{equation}}
\subsection{Lemmata proofs}\label{appxsec:proofs}
\subsubsection{Proof of Lemma \ref{lem:properties_of_the_Jacobian}}
\begin{proof}
	The desired results basically follows from the fact that $T_{t}$ is (at least) a $\mathcal{C}^{1,1}$ diffeomorphism.
	Meanwhile, we note in particular that the second statement of the lemma is already guaranteed by our assumption that $A_{t}$ is bounded as underlined in \eqref{eq:bounds_At_and_Bt}.
	Nevertheless, we briefly verify our claim as follows.
	Let $\varepsilon > 0$ be sufficiently small and $t \in \mathcal{I} = [0,\varepsilon]$.
	By reducing $\varepsilon$ if necessary, we can assume without lost of generality that $|t D\VV|_{\infty} < 1$ for $t \in \mathcal{I}$.
	This permits us to write $(DT_{t})^{-1}$ as a Neumann series:
	\[
		DT_{t}^{-1}(x) = (DT_{t})^{-1} \circ T_{t}^{-1} (x) = \sum_{j=0}^{\infty} (-t)^{j} (D\VV)^{j}(T_{t}^{-1}(x)) 
	\] 
	for each $x \in \overline{U}$.
	We can then estimate its norm as follows:
	\begin{equation}\label{eq:sup_norm_for_Jacobian_matrix}
		\abs{DT_{t}^{-1}(x)}_{\infty} 
			\leqslant \sum_{j=0}^{\infty} \abs{(-t)^{j}  (D\VV)^{j}(T_{t}^{-1}(x))}
			\leqslant \frac{1}{1 - \varepsilon \abs{D\VV}_{\infty}} =: C,
	\end{equation}
	which gives us a choice for $C>0$.
\end{proof}
\subsubsection{Proof of Lemma \ref{lem:convergence_of_functions}}\label{appxsubsec:proof_of_convergence_of_functions}
\begin{proof}
	The proof of (i) for the case $p=2$ can be found in \cite[p. 529]{DelfourZolesio2011} (see also the proof of \cite[Thm. 6.1, p. 567]{DelfourZolesio2011}).
	For the case $p>2$, the argumentation is similar.
	On the other hand, (ii) can be obtained from triangle inequality and the application of (i).
	Indeed, it is enough to show that 
	\[
		\lim_{t \to 0} \|\nabla(f \circ T_{t} - f)\|_{L^{p}(U)} = \lim_{t \to 0} \| (DT_{t})^{\top} ((\nabla f) \circ T_{t}) - \nabla f \|_{L^{p}(U)}.
	\]
	By the triangle inequality, we know that
	\begin{align*}
		&\| (DT_{t})^{\top} ((\nabla f) \circ T_{t}) - \nabla f \|_{L^{p}(U)}\\
			&\qquad\qquad \leqslant \| [ (DT_{t})^{\top} - id] ((\nabla f) \circ T_{t})  \|_{L^{p}(U)} + \| \nabla f \circ T_{t} - \nabla f \|_{L^{p}(U)}.
	\end{align*}
	For the latter summand, we have $\| \nabla f \circ T_{t} - \nabla f \|_{L^{p}(U)} \to 0$ as $t\to 0$ because of (i).
	Meanwhile, we recall that $DT_{0} = id$ and that we have $[t \mapsto A_t] \in \mathcal{C}(\mathcal{I},\mathcal{C}(\overline{U})^{d \times d})$. 
	So, in particular, the map $(DT_{t})^{\top} \to id$ holds in $\mathcal{C}(\overline{U})^{d \times d}$.
	Thus, the first summand also vanishes as $t\to 0$.
	For the proof of (iii), we may refer to \cite[Proof of Lemma 3.5]{IKP2006}.
	Finally, (iv) is a consequence of (iii) which is easily shown as follows
	\begin{align*}
		\frac{1}{t} \left( I_{t} f \circ T_{t} - f \right)
		= \frac{1}{t} \left( I_{t} - 1 \right) f^{t} + \frac{1}{t} \left( f^{t} - f \right)
		\ \ \longrightarrow \ \ f \operatorname{div} \VV  + Df \VV = \operatorname{div}(f \VV)
		\quad \text{as $t \to 0$}.
	\end{align*}
	This proves Lemma \ref{lem:convergence_of_functions}.
\end{proof}
\subsection{Computations of some identities}\label{appxsubsec:computations}
\subsubsection{Expansion of $I_{t}$}\label{appxsubsubsec:expansion}

We look at the expansion of the determinant $I_{t} := \operatorname{det}(DT_{t})$.
Note that the Jacobian of $T_{t}:= id + t \VV$, $\VV := (\theta_{1},\theta_{2},\ldots,\theta_{d})^{\top} \in \vect{\Theta}^{1}$ ($T_{t}:\Omega \to \mathbb{R}^{d}$, $d \in \mathbb{N} \setminus \{1\}$) given by $DT_{t} = (DT_{t})_{ij} =: (M_{ij})$ has entries of the form
\[
	M_{ij} = \delta_{ij} + t \frac{\partial V_{i}}{\partial x_{j}} =: \delta_{ij} + t m_{ij},
\]
where $\delta_{ij}$ denotes the Kronecker delta function.

Let $\mathcal{S}_{d}$ be the set of all permutations of $N_{d}:=\{1, \ldots, d\}$ and $\operatorname{sgn}$ be the signum of the permutation $\sigma$ of $N_{d}$ (i.e, it is equal to $+1$ or $-1$ according to whether the minimum number of transpositions (pairwise interchanges) necessary to achieve it starting from $N_{d}$ is even or odd.
Moreover, let $\mathcal{I}_{d}:=\{\sigma \in \mathcal{I}_{d} : \sigma(j) = j, j \in {N}'_{d} \subseteq N_{d} \}$ and $\iota$ be the identity permutation.
Then, by definition of determinant (see, e.g., \cite[Eq. (0.3.2.1), p. 29]{JohnsonHorn2013}), we have the following computations
\begin{align*}
	I_{t}
	&= \sum_{\sigma \in \mathcal{S}_{d}} \left( \operatorname{sgn} \sigma \prod_{i=1}^{d} M_{i\sigma(i)}\right)\\
	&= \sum_{\sigma = \iota } \prod_{i=1}^{d} \left( 1 + t \frac{\partial V_{i}}{\partial x_{i}}\right) 
		+ \sum_{\sigma \in \mathcal{I}_{d} \setminus \{\iota\} } \left( \operatorname{sgn} \sigma \prod_{i=1}^{d} M_{i\sigma(i)}\right)
		+ \sum_{\sigma \in \mathcal{S}_{d} \setminus \mathcal{I}_{d} } \left( \operatorname{sgn} \sigma \prod_{i=1}^{d} M_{i\sigma(i)}\right)\\
	&=: S_{1} + S_{2} + S_{3}.  
\end{align*}
Observe that we may write, for some function $\rho_{1}\in \mathcal{C}^{0}:=\mathcal{C}(\mathbb{R},\mathcal{C}^{0,1}(U))$, the first summand as $S_{1} = 1 + t \dive \VV + t^{2} \rho_{1}(t,\VV)$.
In addition, we can write the second sum as $t^{2} \rho_{2}(t,\VV)$, for some function $\rho_{2}\in\mathcal{C}^{0}$, since each term $S_{2}$ consists of at least two factors of $t m_{ij}$, $i \neq j$, $i,j \in N_{d}$, $d \in \mathbb{N} \setminus \{1\}$.
Meanwhile, all terms of $S_{3}$ have factors of the form $t m_{ij}$, $i\neq j$.
So, the sum $S_{3}$ can be expressed as $t^{d} \rho_{\ast}(\VV)$ which can be written as $t^{2} \rho_{3}(t,\VV)$, for some function $\rho_{3} \in \mathcal{C}^{0}$.
All together, we observe that $I_{t} = 1 + t \dive \VV + t^{2} \tilde{\rho}(t,\VV)$, for some function $\tilde{\rho} \in \mathcal{C}^{0}$.
\subsubsection{Derivative of $B_{t}|M_{t} \mathbf{n}|^{-2}$ }\label{appxsubsubsec:derivative}
Let us compute the derivative of $B_{t}|\Mt \nn|^{-2}$.
We recall that $B_{t} = I_{t} |\Mt \nn|$ and $\Mt = (DT_{t})^{-\top}$, and $B_{0} = 1$ and $|M_{0}\nn| = |id \nn| = |\nn| = 1$.
We also remember that $\frac{d}{dt}B_{t} \big\rvert_{t=0} = \dive \VV \big\rvert_{\Sigma} - ({D} \VV\nn)\cdot\nn$.

Let us consider two (column) vectors $\vect{a}:= (a_1, a_2, \ldots, a_{d})^{\top}, \vect{b}:= (b_1, b_2, \ldots, b_{d})^{\top} \in \mathbb{R}^{d}$.
Then, we have the following computations\footnote{Here, the expression $O(t)$ represents a generic remainder term.}
\begin{align*}
	\frac{d}{dt} \left( \abs{\vect{a} - t \vect{b} + O(t^{2})}^{2} \right) 
		&= \frac{d}{dt} \left( \sum_{i=1}^{d} (a_{i} - t b_{i} + O(t^{2}))^{2} \right) \\
		&= \sum_{i=1}^{d} 2 (a_{i} - t b_{i} + O(t^{2})) \cdot (- b_{i} + O(t))\\
		&= 2 \sum_{i=1}^{d} ( -a_{i} b_{i} + O(t))
		= 2 (-\vect{a} \cdot \vect{b} + O(t)). 
\end{align*} 

Now, considering the expansion (compare \eqref{eq:expansion_of_Mt})
\[
	\Mt \nn = (DT_{t})^{-\top} \nn = (id + t D\VV)^{-\top}\nn = [id - t (D\VV)^{\top} + \tilde{R}(t)]\nn,
\]
where $\tilde{R}(t)$ is a square matrix of size $d \times d$ and is of order $O(t^{2})$, and in view of \eqref{eq:regular_maps} -- knowing that $\Mt \nn$ is differentiable with respect to $t$ -- we obtain the following calculations based on the previous computation:
\[
	\frac{d}{dt} \abs{\Mt \nn}^{2} \Big|_{t=0}
		= \frac{d}{dt} \left( \abs{\nn - t (D\VV)^{\top} + \tilde{R}(t)}^{2} \right) \Big|_{t=0}
		= - 2 \nn \cdot [({D}\VV)^{\top}\nn] = - 2 ({D}\VV \nn) \cdot \nn. 	
\]
This implies that
\[
	\frac{d}{dt}\left( \frac{B_{t}}{\abs{\Mt \nn}^{2}} \right)\Big|_{t=0}
	= \frac{ \left( \dfrac{d}{dt} B_{t} \right) \abs{\Mt \nn}^{2} - B_{t} \left( \dfrac{d}{dt} \abs{\Mt \nn}^{2} \right)}{\abs{\Mt \nn}^{4}} \Big|_{t=0}
	= \dive \VV \big\rvert_{\Sigma} + ({D}\VV \nn) \cdot \nn.
\]
%
%
%
%
%
%
\subsection{Computation of the shape gradient of the cost via chain rule}
\label{appxsubsec:shape_derivatives_of_the_cost_via_chain_rule}
To validate the expression for the shape gradient, we give below the computation of the expression $g_{\Sigma}$ under a $\mathcal{C}^{2,1}$ regularity assumption on the domain, supposing in addition that $\ff \in H_{loc}^{1}(\mathbb{R}^{d})^{d}$, specifically, we assume $\ff \in H^{1}(U)^{d}$, where $U$ is a fixed convex bounded open set in $\mathbb{R}^{d}$ such that $U \supset \overline{\Omega}$.
We note that the given regularity guarantees the existence of the material and the shape derivative of the state and because of this, the shape gradient of the cost can easily be established using Hadamard's domain differentiation formula:
(see, e.g., \cite[Thm. 4.2, p. 483]{DelfourZolesio2011}), \cite[eq. (5.12), Thm. 5.2.2, p. 194]{HenrotPierre2018} or \cite[eq. (2.168), p. 113]{SokolowskiZolesio1992}):
\begin{equation}
\begin{aligned}
	\left. \left\{ \frac{{d}}{{d}t} \intOt{f(t,x)} \right\} \right|_{t=0}
		= \intO{\frac{\partial }{\partial t} f(0,x)}
			+ \intdO{ f(0,\sigma) \Vn} \label{eq:Hadamard_domain_formula}
\end{aligned}
\end{equation}
(and, of course, with the assumption that the perturbation of $\Omega$ preserves its regularity). 

The next result is a restatement of Proposition \ref{prop:the_shape_derivative_of_the_cost}, but with higher regularity assumption.
%
%
%
\begin{proposition}\label{prop:shape_derivative_with_sufficient_regularity_assumptions}
	Let $\Omega \in \mathcal{C}^{2,1}$ and $\VV \in \sfTheta^2$.
	Then, the shape derivative of $J$ at $\Omega$ along $\VV$ is given by $dJ(\Omega)[\VV] = \intS{g_{\Sigma}{\Vn}}$, where $g_{\Sigma}$ is the expression in \eqref{eq:shape_gradient}.
\end{proposition}
Before we prove the above proposition, we briefly prepare the following lemmata which will be useful in the derivation of the shape derivative.
\begin{lemma}\label{lem:shape_derivative_of_the_normal_vector}
	Let $\nn$ be the outward unit normal to $\Sigma$ and $\VV \in \sfTheta^{2}$.
	Then, the shape derivative of $\nn$ denoted by $\nn'$ is given by
	\begin{equation}\label{eq:shape_derivative_of_n}
		\nn' = - \nabla_{\Sigma}\Vn.
	\end{equation}
\end{lemma}
\begin{proof}
	See part of the proof of Proposition 5.4.14 in \cite[Chap. 5., Sec. 4.4, p. 222]{HenrotPierre2018}.
\end{proof}
%
%
\begin{lemma}
	Let $\Omega \in \sfTheta^{2}$ and $\nn$ be the outward unit normal to $\Sigma$.
	Then,
	\[
		D\tilde{\nn} \nn = (D\tilde{\nn})^{\top} \nn = (\nabla\tilde{\nn}) \nn = 0\quad \text{on} \ \Sigma,
	\]
	where $\tilde{\nn}$ is a unitary $\mathcal{C}^{1}$ extension of the vector field $\nn$ on $\Sigma$.
\end{lemma}
\begin{proof}
	Because $\Omega$ is $\mathcal{C}^{2,1}$ regular, then by Proposition 5.4.8 of \cite[p. 218]{HenrotPierre2018} (see also \cite[Lem. 16.1, p. 390]{GilbargTrudinger1998}), there exists a $\mathcal{C}^{1}$ unitary extension $\tilde{\nn}:=(\tilde{n}_{1}, \ldots, \tilde{n}_{d})^{\top}$\footnote{The same is used in Lemma \ref{lem:preparation1} and Lemma \ref{lem:preparation2}.} of $\nn$.
	So, in an open neighborhood of $\Sigma$, we have $|\tilde{\nn}|^{2} = \langle \tilde{\nn}, \tilde{\nn} \rangle = 1$.
	Thus, for each $j = 1, \ldots, d$, we have the following computation
	\[
		\dfrac{\partial}{\partial x_{j}} (|\tilde{\nn}|^{2})
			= \dfrac{\partial}{\partial x_{j}} \left( \sum_{i=1}^{d} \tilde{n}_{i}^{2} \right) 
			= 2 \sum_{i=1}^{d} \tilde{n}_{i} \dfrac{\partial \tilde{n}_{i}}{\partial x_{j}} 
			= 0, 
	\]
	or equivalently,
	\begin{equation}\label{eq:sum_one}
		\nabla (|\tilde{\nn}|^{2})
			= \nabla \left( \sum_{i=1}^{d} \tilde{n}_{i}^{2} \right) 
			= 2 \sum_{i=1}^{d} \tilde{n}_{i} \nabla \tilde{n}_{i}
			= 2 \sum_{i=1}^{d} \tilde{n}_{i} \left( \sum_{j=1}^{d} \frac{\partial \tilde{n}_{i}}{\partial x_{j}} \vect{e}_{j} \right)
			= \vect{0},
	\end{equation}
	where $\vect{e}_{j} := (\overbrace{0,\ldots,0,1}^{j},0,\ldots,0)^{\top} \in \mathbb{R}^{d}$ is the $j$th vector of the canonical basis in $\mathbb{R}^{d}$.
	Now, on the other hand, we have
	\begin{equation}\label{eq:sum_two}
	\begin{aligned}
		(D\tilde{\nn})^{\top} \tilde{\nn} 
			= \begin{pmatrix}
				\dfrac{\partial \tilde{n}_{1}}{\partial x_{1}} & \ldots & \dfrac{\partial \tilde{n}_{d}}{\partial x_{1}} \\
				\vdots & \ddots & \vdots\\
				\dfrac{\partial \tilde{n}_{1}}{\partial x_{d}} & \ldots & \dfrac{\partial \tilde{n}_{d}}{\partial x_{d}} 
			 \end{pmatrix}
			 \begin{pmatrix}
			 	\tilde{n}_{1} \\ \vdots \\ \tilde{n}_{d}
			 \end{pmatrix}
			 =  	\begin{pmatrix}
			 		\displaystyle \sum_{j=1}^{d} \dfrac{\partial \tilde{n}_{j}}{\partial x_{1}} \tilde{n}_{j} \\ \vdots \\ \displaystyle \sum_{j=1}^{d} \dfrac{\partial \tilde{n}_{j}}{\partial x_{d}} \tilde{n}_{j}
			 	\end{pmatrix}
			 = \sum_{k=1}^{d} \left( \sum_{j=1}^{d}  \dfrac{\partial \tilde{n}_{j}}{\partial x_{k}} \tilde{n}_{j} \right) \vect{e}_{k}.
	\end{aligned}
	\end{equation}
	Combining \eqref{eq:sum_one} and \eqref{eq:sum_two}, we deduce that
	\[
		\langle (D\tilde{\nn}) \nn, \nn \rangle = \langle \nn, (D\tilde{\nn})^{\top} \nn \rangle = 0\quad \text{on} \ \Gamma,
	\]
	from which we infer the conclusion.
\end{proof}
\begin{remark}
	In light of the previous lemma, and recalling the definition of the tangential Jacobian matrix from Definition \ref{def:tangential_operators}, we see that for the $\mathcal{C}^{1}$ unitary extension $\tilde{\nn}$ of $\nn$, we clearly have the identity $D\tilde{\nn}\big|_{\Sigma} = D\nn = D_{\Sigma}\nn = (D_{\Sigma}\nn)^{\top}$ (refer to \cite[Chap. 9, Sec. 5.2, eq. (5.17) -- (5.19), p. 497]{DelfourZolesio2011}, for the last equation).
\end{remark}
\begin{lemma}\label{lem:preparation1}
	Let $\Omega \in \sfTheta^{2}$ and $\nn$ be the outward unit normal to $\Sigma$.
	Then, for the solution $\uu$ of \eqref{eq:ccbm}, we have
	\[
		\frac{\partial (\uu \cdot \nn)}{\partial \nn} = \frac{\partial \uu }{\partial \nn} \cdot \nn \qquad \text{on $\Sigma$}.
	\]
\end{lemma}
\begin{proof}
	First, we note that 
	$
		\dfrac{\partial n_{j}}{\partial \nn}
			 %
			 = \nn^{\top} \nabla n_{j}
			 = \displaystyle \sum_{k=1}^{d} n_{k} \dfrac{\partial n_{j}}{\partial x_{k}}.
	$
	Moreover, from the proof of the previous lemma, we know that
	\begin{equation}\label{eq:identity0}
	\begin{aligned}
		(D\tilde{\nn})\tilde{\nn} 
			 = \sum_{k=1}^{d} \left( \sum_{j=1}^{d}  \dfrac{\partial \tilde{n}_{k}}{\partial x_{j}} \tilde{n}_{j} \right) \vect{e}_{k}
			 = \vect{0}
	\end{aligned}
	\end{equation}
	from which it can be deduced that
	\[
		 \sum_{j=1}^{d} {n}_{j} \dfrac{\partial {n}_{k}}{\partial x_{j}} = 0 \quad \text{on $\Sigma$},
	\]
	for all $k=1,\ldots,d$.
	In addition, we have the following identity
	\begin{equation}\label{eq:identity_dundotn}
		\frac{\partial \uu }{\partial \nn} \cdot \nn
		= (D\uu) {\nn} \cdot \nn
			 = \sum_{k=1}^{d} \left( \sum_{j=1}^{d}  \dfrac{\partial u_{k}}{\partial x_{j}} {n}_{j} \right) \vect{e}_{k} \cdot \sum_{i=1}^{d} n_{i} \vect{e}_{i}
			 = \sum_{k=1}^{d} \left( \sum_{j=1}^{d}  \dfrac{\partial u_{k}}{\partial x_{j}} {n}_{j} \right) n_{k},
	\end{equation}
	where the last equation follows from the fact that $\vect{e}_{i} \cdot \vect{e}_{j} = \delta_{ij}$ and $\delta_{ij}$ is the Kronecker delta function.
	Thus, we have the following computations:
	\begin{align*}
		\frac{\partial (\uu \cdot \nn)}{\partial \nn}
			&= \frac{\partial}{\partial \nn} \left(\sum_{k=1}^{d} u_{k} n_{k} \right)
			= \sum_{k=1}^{d}  \left( \frac{\partial u_{k} }{\partial \nn} n_{k} + u_{k} \frac{\partial n_{k}}{\partial \nn} \right)
			= \sum_{k=1}^{d}  \left( \sum_{j=1}^{d} n_{j} \dfrac{\partial u_{k}}{\partial x_{j}} \right) n_{k}.
	\end{align*}
	The desired identity then follows by comparing the previous equation with \eqref{eq:identity_dundotn}.
\end{proof}
\begin{lemma}\label{lem:preparation2}
	Let $\Omega \in \sfTheta^{2}$ and $\nn$ be the outward unit normal to $\Sigma$.
	Then, for the solution $\uu$ of \eqref{eq:ccbm}, we have
	\[
		\frac{\partial }{\partial \nn} \left[(\uu \cdot \nn)\nn\right] = \left( \frac{\partial \uu }{\partial \nn} \cdot \nn \right) \nn \qquad \text{on $\Sigma$}.
	\]
\end{lemma}
\begin{proof}
	Let $\tilde{\nn}$ be the $\mathcal{C}^{1}$ extension of $\nn$ as before and denote $\vect{b} = a\tilde{\nn} := \left( \sum_{i=1}^{d} u_{i} \tilde{n}_{i} \right) \tilde{\nn}$ with $\vect{b}:=(b_{1}, \ldots, b_{d}) = (a\tilde{n}_{1},\ldots,a\tilde{n}_{d})$.
	Note that 
	\begin{equation}\label{eq:identity3}
		(D\vect{b})\tilde{\nn} = \sum_{k=1}^{d} \left( \sum_{j=1}^{d}  \dfrac{\partial b_{k}}{\partial x_{j}} \tilde{n}_{j} \right) \vect{e}_{k}.
	\end{equation}
	Moreover, we have the following computations
	\begin{align*}
		\frac{\partial b_{k}}{\partial x_{j}}
			= \frac{\partial (a\tilde{n}_{k})}{\partial x_{j}}
			%
			= \left( \sum_{i=1}^{d} \left(  \frac{\partial u_{i}}{\partial x_{j}} \tilde{n}_{i} +  u_{i} \frac{\partial \tilde{n}_{i} }{\partial x_{j}} \right) \right) \tilde{n}_{k}
				+ \left( \sum_{i=1}^{d} u_{i} \tilde{n}_{i} \right) \frac{\partial \tilde{n}_{k}}{\partial x_{j}}.
	\end{align*}
	Inserting the above expression to \eqref{eq:identity3} and using Lemma \ref{lem:preparation1} (see also \eqref{eq:identity0}), we get
	\begin{align*}
	(D\vect{b})\tilde{\nn} 
	&= \sum_{k=1}^{d} \left( \sum_{j=1}^{d}  \left[ \left( \sum_{i=1}^{d} \frac{\partial u_{i}}{\partial x_{j}} \tilde{n}_{i} \right) \tilde{n}_{k}
				\right] \tilde{n}_{j} \right) \vect{e}_{k}
			+ \sum_{k=1}^{d} \left( \sum_{j=1}^{d}  \left[ \left( \sum_{i=1}^{d} u_{i} \tilde{n}_{i} \right) \frac{\partial \tilde{n}_{k}}{\partial x_{j}} \right] \tilde{n}_{j} \right) \vect{e}_{k}\\	
	&= \sum_{k=1}^{d} \left( D\uu\tilde{\nn} \cdot \nn \right) \tilde{n}_{k} \vect{e}_{k}
			+ (\uu \cdot \nn) \sum_{k=1}^{d} \left( \sum_{j=1}^{d} \frac{\partial \tilde{n}_{k}}{\partial x_{j}} \tilde{n}_{j} \right) \vect{e}_{k}\\				%
	&= \left( D\uu\tilde{\nn} \cdot \tilde{\nn} \right) \tilde{\nn}.
	\end{align*}
	Thus, on $\Sigma$, we get
	\[
		\frac{\partial }{\partial \nn} \left[(\uu \cdot \nn)\nn\right] = \left( \frac{\partial \uu }{\partial \nn} \cdot \nn \right) \nn,
	\]
	as desired.
\end{proof}
\begin{lemma}\label{lem:tangential_times_normal_vector}
	Let $\Omega$ be a sufficiently smooth domain with boundary $\Gamma:=\partial \Omega$.
	Assume that $\bphi \in \vect{W}^{1,1}(\partial \Omega)$.
	Then, the following identity holds (see, e.g., \cite[Chap. 9, Sec. 5.4, eq. (5.20), p. 497]{DelfourZolesio2011}):
	\[
		\nabla_{\Gamma} \bphi \nn = \vect{0}.	
	\]
\end{lemma}
\begin{proof}
	Assume that $\bphi \in \vect{W}^{1,1}(\Gamma)$.
	By definition $\nabla_{\Gamma} \bphi := \nabla \bphi - (\nabla \bphi\nn) \otimes \nn$, so we have
	\[
		\nabla_{\Gamma} \bphi \nn = \left[ \nabla \bphi - (\nabla \bphi\nn) \otimes \nn \right] \nn
			= \nabla \bphi \nn - \left[ (\nabla \bphi\nn) \otimes \nn \right] \nn.
	\]
	Computing $\left[ (\nabla \bphi\nn) \otimes \nn \right] \nn$, we get
	\begin{align*}
		\left[ (\nabla \bphi\nn) \otimes \nn \right] \nn
		&= \left[ \left( \sum_{j=1}^{d} \left( \sum_{i=1}^{d}  \dfrac{\partial \varphi_{i}}{\partial x_{j}} n_{i} \right) \vect{e}_{j} \right) \otimes \sum_{k=1}^{d} n_{k} \vect{e}_{k} \right] \sum_{l=1}^{d} n_{l} \vect{e}_{l}\\
		%
		%
		&= \sum_{j=1}^{d} \left( \sum_{i=1}^{d}  \dfrac{\partial \varphi_{i}}{\partial x_{j}} n_{i} \right) \vect{e}_{j} \sum_{k=1}^{d} \sum_{l=1}^{d} n_{l} n_{k} \vect{e}_{k}^{\top} \vect{e}_{l} \\
		&= \sum_{j=1}^{d} \left( \sum_{i=1}^{d}  \dfrac{\partial \varphi_{i}}{\partial x_{j}} n_{i} \right) \vect{e}_{j} \sum_{k=1}^{d} \sum_{l=1}^{d} n_{l} n_{k} \delta_{kl} \\			%
		&= \sum_{j=1}^{d} \left( \sum_{i=1}^{d}  \dfrac{\partial \varphi_{i}}{\partial x_{j}} n_{i} \right) \vect{e}_{j}\\
		&= \nabla \bphi\nn,
	\end{align*}
	from which the desired result clearly follows.
\end{proof}
%
%
%
\begin{proof}[Proof of Proposition \ref{prop:shape_derivative_with_sufficient_regularity_assumptions}]
	Let us assume that $\Omega$ is of class $\mathcal{C}^{2,1}$ and $\VV \in \sfTheta^{2}$.
	Using classical regularity theory, we have $\ur$, $\ui \in H^{3}(\Omega)^{d}$ and $p_{r}$, $\pim \in H^{2}(\Omega)$.
	Because we have sufficient regularity for $\uu$, $\Omega$, and $\VV$, then we can apply formula \eqref{eq:Hadamard_domain_formula} to obtain -- noting that $\VV= \vect{0}$ on $\Gamma$ -- the derivative
	\begin{equation}
	\label{eq:computed_first_derivative_via_Hadamard_formula}
		{{d}}J(\Omega)[\VV]  = \intO{\left( \ui \ui'  + \pim \pim' \right) } + \frac12 \intS{ \left( |\ui|^2 + |\pim|^2 \right) \Vn }=:\mathbb{I}_1+\mathbb{I}_2.
	\end{equation}

Hereafter, we proceed in four steps:
\begin{description}
	\item[\textnormal{\textit{Step 1.}}] We establish the strong form of the shape derivatives $\uu'$ and $p'$ which is characterized by the complex PDE system \eqref{eq:ccbm_shape_derivative_of_the_states}.
	\item[\textnormal{\textit{Step 2.}}] We prove the differentiability of $J(\Omega)$ in the direction of $(\delta\tilde{\uu},\delta\tilde{p}) \in X \times Q$.
	\item[\textnormal{\textit{Step 3.}}] We justify the structure of the adjoint system \eqref{eq:adjoint_system}.
	\item[\textnormal{\textit{Step 4.}}] We obtain the expression for the shape gradient via the adjoint method.
\end{description}
\medskip
\textit{Step 1.} We recall the variational equation \eqref{eq:ccbm_weak_form}. Because $\Omega$, $\uu$, $p$, and $\VV$ are regular enough, then the solution $(\uu, p) \in X \times Q$ is shape differentiable and we can differentiate \eqref{eq:ccbm_weak_form} (formally) to get
\begin{equation}\label{eq:weak_problem_derivative}
\begin{aligned}
	& \intO{\alpha \nabla {\uu'} : \nabla {\cbpsi}} + i \intS{({\uu'} \cdot \nn)({\cbpsi} \cdot \nn)} - \intO{ p' ( \nabla \cdot \cbpsi) } - \intO{ \bar{\lambda} ( \nabla \cdot {\uu'}) } \\
		&\qquad = - \intS{\alpha ( \nabla {\uu} : \nabla {\cbpsi} - p ( \nabla \cdot \cbpsi) ) \Vn}\\
		&\qquad\qquad		
		 - i \intS{({\uu} \cdot \nn')({\cbpsi} \cdot \nn)} - i \intS{({\uu} \cdot \nn)({\cbpsi} \cdot \nn')}\\
		&\qquad\qquad - i \intS{\left[ \frac{\partial}{\partial \nn}\left( ({\uu} \cdot \nn) \nn \right)  + \kappa \left( {\uu} \cdot \nn) \nn \right) \right] \cdot \cbpsi \Vn} + \intS{\ff \cdot \cbpsi \Vn },
\end{aligned}
\end{equation}
where the shape derivative $\nn'$ of the normal vector $\nn$ is given by \eqref{eq:shape_derivative_of_n}.

From the previous equation we can derive a boundary value problem for $(\uu',p')$.
Namely, choosing $\bpsi \in \vect{\mathcal{C}}^{\infty}_{0}(\Omega)^{d}$ and $\lambda \in \vect{\mathcal{C}}^{\infty}_{0}(\Omega)$\footnote{For clarity, we mention here that there is slight abuse of notations. Specifically, $\bpsi \in \vect{\mathcal{C}}^{\infty}_{0}(\Omega)^{d}$ and $\lambda \in \vect{\mathcal{C}}^{\infty}_{0}(\Omega)$ means that $\bpsi_{r}, \bpsi_{i} \in \mathcal{C}^{\infty}_{0}(\Omega)^{d}$ and $\lambda_{r}, \lambda_{i} \in \mathcal{C}^{\infty}_{0}(\Omega)$ where $\mathcal{C}^{\infty}_{0}(\Omega)^{d}$ and $\mathcal{C}^{\infty}_{0}(\Omega)$ denotes the usual space of infinitely differentiable (vector-valued and scalar-valued) functions, respectively.} reveals that (via integration by parts)
\begin{equation}\label{eq:first_equations}
- \alpha \Delta \uu' + \nabla p' = \vect{0}		\quad\text{in $\Omega$}\qquad \text{and} \qquad 
	\nabla \cdot \uu'	= \vect{0} 		\quad\text{in $\Omega$},
\end{equation}
which hold in the distributional sense.
That is, we have obtained the first equation above from
\[
\langle - \alpha \Delta \uu' + \nabla p', \cbpsi \rangle_{[\vect{\mathcal{C}}^{\infty}_{0}(\Omega)^{d}]', \vect{\mathcal{C}}^{\infty}_{0}(\Omega)^{d}}  = 
 \intO{ \left( \alpha \nabla {\uu'} : \nabla {\cbpsi} - p' ( \nabla \cdot \cbpsi) \right) } = 0,
\]
and the second one from the argument, i.e., we check that $\nabla \cdot \uu' = \vect{0} $ in $[\vect{\mathcal{C}}^{\infty}_{0}(\Omega)^{d}]'$.

Meanwhile, because $\VV \in \sfTheta^{2}$, i.e., $\VV$ vanishes on $\Gamma$, the boundary condition on $\Gamma$ satisfied by $\uu'$ easily follows.
That is, we have $\uu' = \vect{0}$ on $\Gamma$. 

Now, for the succeeding arguments, we underline here the fact that $\Omega \in \mathcal{C}^{2,1}$, and since we have the $\HH^{3}(\Omega)^{d} \times H^{2}(\Omega)$ regularity for $(\uu,p)$, we know that $\nabla \uu' \in \vect{L}^{2}(\Omega)^{d}$ and $\nabla p' \in \vect{L}^{2}(\Omega)^{d}$.

We next exhibit the boundary condition on $\Sigma$.
We choose\footnote{In fact, here, we choose a test function $\bpsi \in \HH^{2}(\Omega)^{d}$, and because $\Omega \in \mathcal{C}^{2,1}$, it follows that -- by Stein's extension theorem (see, e.g., \cite[Thm. 5.24, p. 154]{AdamsFournier2003}) -- we can construct an extension of $\bpsi$ in $\HH^{2}(\mathbb{R}^{d})^{d}$ (still denoted by $\bpsi$).} $\bphi \in \vect{\mathcal{C}}^{\infty}(\Sigma)^{d}$ and $\mu \in \vect{\mathcal{C}}^{\infty}(\Sigma)$.
Accordingly, we can find an extension $\bpsi \in \vect{\mathcal{C}}^{\infty}(\Omega)^{d}$ and $\lambda \in \vect{\mathcal{C}}^{\infty}(\Omega)$ such that $\bpsi \big|_{\Sigma} = \bphi$ and $\lambda \big|_{\Sigma} = \mu$, and $\dn{\bpsi} \big|_{\Sigma} = \vect{0}$ and $\dn{\lambda} \big|_{\Sigma} = 0$.
Applying integration by parts in the left hand side of \eqref{eq:weak_problem_derivative} and using  \eqref{eq:first_equations}, we get
\begin{align*}
	&\intS{ \left[ \alpha \dn{\uu'} + i ({\uu'} \cdot \nn) \nn - p'\nn \right] \cdot \cbpsi } \\
	&\qquad\qquad\qquad= - \intS{ \left[ \alpha \nabla {\uu} : \nabla {\cbpsi} - p ( \nabla \cdot \cbpsi) \right] \Vn}
		 + \intS{ \vect{B}_{1}[\Vn] \cdot {\cbpsi} },
\end{align*}
where
\[
	\vect{B}_{1}[\Vn] := -i [ ({\uu} \cdot \nn') \nn + ({\uu} \cdot \nn) \nn' ] - i \left[ \dn{\left( ({\uu} \cdot \nn) \nn \right)}  + \kappa ( {\uu} \cdot \nn )\nn \right] \Vn + \ff \Vn.
\]
At this juncture, let us recall the definition of tangential gradient and divergence from subsection \ref{subsec:some_identities_from_tangential_shape_calculus}, to deduce the identity
\[
	\intS{ p \nabla \cdot \cbpsi \Vn } 
	= \intS{ \left[ p \dive_{\Sigma} \cbpsi + \dn{\cbpsi} \cdot (p\nn)\right]  \Vn },
\]
and make use of the formulas (cf. equation in remark found in \cite[p. 87]{SokolowskiZolesio1992})
\begin{align*}
	(\dn{\uu} \otimes \nn) : (\dn{\cbpsi} \otimes \nn) 
	&= \dn{\uu} \cdot \dn{\cbpsi};\\
	\nabla \uu : (\dn{\cbpsi} \otimes \nn) 
	&= \dn{\uu} \cdot \dn{\cbpsi}
	= \nabla \cbpsi : (\dn{\uu} \otimes \nn),
\end{align*}
to obtain
\[
	\nabla \uu : \nabla \cbpsi = \nabla_{\Sigma} \uu : \nabla_{\Sigma} \cbpsi + \dn{\uu} \cdot \dn{\cbpsi}.
\]
Moreover, using the tangential Green's formula (see Lemma \ref{lem:tangential_formulas}) and applying integration by parts on the boundary $\Sigma$, we can write
\begin{align}
	\intS{ p \Vn \dive_{\Sigma} \cbpsi } 
		&= \intS{\left[ \kappa p \Vn \nn - \nabla_{\Sigma} (p\Vn)  \right] \cdot \cbpsi}
		=: \intS{ \vect{B}_{2}[\Vn] \cdot \cbpsi },\label{eq:p_divergence}\\
	- \intS{\alpha ( \nabla_{\Sigma} {\uu} : \nabla_{\Sigma} {\cbpsi} ) \Vn}
		&= \intS{ \cbpsi \cdot \left\{ \dive_{\Sigma}{[\alpha (\nabla_{\Sigma} \uu) \Vn]} - [\alpha \kappa  (\nabla_{\Sigma} \uu) \Vn \nn] \right\} }\nonumber\\
		&= \intS{ \cbpsi \cdot \left\{ \dive_{\Sigma}{[\alpha (\nabla_{\Sigma} \uu) \Vn]} \right\} }
		=: \intS{ \cbpsi \cdot \vect{B}_{3}[\Vn] }\nonumber,
\end{align}
where the last equality is due to the fact that $\nabla_{\Sigma} \uu\nn = \vect{0}$, see Lemma \ref{lem:tangential_times_normal_vector}.
Using these identities, together with the fact that $- \alpha \dn{\uu} + p\nn	=  i (\uu \cdot \nn)\nn$ and $\dn{\bpsi} = \vect{0}$ on $\Sigma$, we get
\begin{align*}
	& \intS{ \left( \alpha \dn{\uu'} + i ({\uu'} \cdot \nn) \nn - p'\nn \right) \cdot \cbpsi } \\
		&\qquad = - \intS{\alpha ( \nabla_{\Sigma} {\uu} : \nabla_{\Sigma} {\cbpsi} - p ( \nabla_{\Sigma} \cdot \cbpsi) ) \Vn}
		 + i \intS{ (\uu \cdot \nn) \nn \Vn \cdot \dn{\cbpsi}} + \intS{ \vect{B}_{1} \cdot {\cbpsi} }\\
	%
	%
	&\qquad= \intS{ \vect{B}[\Vn] \cdot {\cbpsi} },	
\end{align*}
where
\begin{align*}
	\vect{B}[\Vn]	&:= \vect{B}_{1}[\Vn] + \vect{B}_{2}[\Vn] + \vect{B}_{3}[\Vn]\\
	&\ = \left\{ -i [ ({\uu} \cdot \nn') \nn + ({\uu} \cdot \nn) \nn' ] - i \left[ \dn{\left( ({\uu} \cdot \nn) \nn \right)}  + \kappa ( {\uu} \cdot \nn) \nn \right] \Vn + \ff \Vn \right\} \\
	&\qquad + \left\{ \kappa p \Vn \nn - \nabla_{\Sigma} (p\Vn) \right\} 
		+ \left\{  \dive_{\Sigma}{[\alpha (\nabla_{\Sigma} \uu) \Vn]} \right\}\\
	&\ = \ff \Vn - \nabla_{\Sigma} (p\Vn) + \dive_{\Sigma}{[\alpha (\nabla_{\Sigma} \uu) \Vn]}
		+ i \left[ \left({\uu} \cdot \nabla_{\Sigma}\Vn \right) \nn + ({\uu} \cdot \nn) \nabla_{\Sigma}\Vn \right] \\
	&\qquad - \left[ i\left( \dn{\uu} \cdot \nn \right)\nn + i \kappa ( {\uu} \cdot \nn ) \nn - \kappa p \nn \right] \Vn,		
\end{align*}
where the latter equation follows from Lemma \ref{lem:preparation2} and Lemma \ref{lem:shape_derivative_of_the_normal_vector}.

At this juncture, we comment that $\vect{B}[\Vn] \in \vect{H}^{1/2}(\Sigma)^{d}$.
Indeed, since $\Omega$ is of class $\mathcal{C}^{2,1}$, the (outward) unit normal vector $\nn$ is $\mathcal{C}^{1,1}(N^{\varepsilon})$ regular and $\kappa \in \mathcal{C}^{0,1}(N^{\varepsilon}) \subset W^{1,\infty}(N^{\varepsilon}) \subset H^{1}(N^{\varepsilon}) \subset H^{1/2}(\partial\Omega)$, where $N^{\varepsilon}$ is a small neighborhood of $\partial\Omega$ (see, e.g., \cite[Sec. 7.8]{DelfourZolesio2011}).
Therefore, (using the density in $\vect{L}^{2}(\Omega)^{d}$ of the traces on $\Sigma$ of functions in $\HH^{2}(\Omega)^{d}$) we get
\begin{equation*}
\begin{aligned}
	 \alpha \dn{\uu'} + i ({\uu'} \cdot \nn) \nn - p'\nn
	 	= \vect{B}[\Vn],
		\qquad \text{on $\Sigma$}.
\end{aligned}
\end{equation*}

In the next several lines, we simplify the expression $\vect{B}[\Vn]$.
To do this, we recall some identities from Lemma \ref{lem:tangential_formulas} and the definitions of some tangential operators, specially the Laplace-Beltrami operator and its decomposition (see \eqref{eq:Laplace_Beltrami_operator_definition}).
Also, we note that, in $\Omega$, we have $-\alpha \Delta \uu + \nabla p = \ff$.
So, we get the following equations on $\Sigma$:
\begin{align*}
	\nabla_{\Sigma}(p\Vn) 
		&= \Vn \nabla_{\Sigma} p + p\nabla_{\Sigma}(\Vn),\\
	\dive_{\Sigma} (\alpha \nabla_{\Sigma} \uu \Vn)
			&= \alpha \Vn \dive_{\Sigma} (\nabla_{\Sigma} \uu) + \alpha\nabla_{\Sigma} \uu \nabla_{\Sigma}\Vn\\
			&=  \Vn (\alpha \Delta_{\Sigma} \uu) + \alpha\nabla \uu \nabla_{\Sigma}\Vn\\
			&=  - \Vn (\ff + \kappa \dn{\uu} + \dnn{\uu} - \nabla p) + \alpha\nabla \uu \nabla_{\Sigma}\Vn.
\end{align*}
Because $p\nn - \alpha \dn{\uu}	 = i (\uu \cdot \nn)\nn$ on $\Sigma$, the latter equation further implies that 
\begin{align*}
	\kappa p\nn \Vn + \left[ \ff \Vn + \dive_{\Sigma} (\alpha \nabla_{\Sigma} \uu \Vn) \right] 
		&= \kappa p\nn \Vn  + \left[ - \Vn (\kappa \dn{\uu} + \dnn{\uu} - \nabla p) + \alpha\nabla \uu \nabla_{\Sigma}\Vn \right]\\
		&= i \kappa (\uu \cdot \nn)\nn \Vn - \Vn (\dnn{\uu} - \nabla p) + \alpha\nabla \uu \nabla_{\Sigma}\Vn.
\end{align*}
These computations, together with Lemma \ref{lem:preparation2} and the identities
\[
	\nabla p - \nabla_{\Sigma} p = \dn{p}\nn
	\quad\text{and}\quad
	\nabla_{\Sigma} \Vn \cdot \uu - \Vn ( D\uu\nn\cdot\nn )
	=\!\footnote{Since $\dive \uu = 0$, then by the definition of the tangential divergence of a vector function, see Definition \ref{def:tangential_operators}, we actually have $ \dive_{\Sigma} \uu = - ( D\uu\nn\cdot\nn )$ on $\Sigma$.} \nabla_{\Sigma} \Vn \cdot \uu + \Vn \dive_{\Sigma} \uu 
	= \dive_{\Sigma} (\Vn \uu),
\]
on $\Sigma$, lead us to conclude that the expression $\vect{B}[\Vn]$ can be written equivalently in the following ways:
\begin{equation}\label{eq:shape_derivative_boundary_condition}
\begin{aligned}
	\vect{B}[\Vn] 
	%
	%
	%
	&= \ff \Vn - \nabla_{\Sigma} (p\Vn) + \dive_{\Sigma}{[\alpha (\nabla_{\Sigma} \uu) \Vn]} + i \dive_{\Sigma} (\Vn \uu)\nn
		+ i ({\uu} \cdot \nn) \nabla_{\Sigma}\Vn \\
	&\qquad - \kappa \left[ - p\nn + i ( {\uu} \cdot \nn ) \nn \right] \Vn \\	
	%
	%
	%
%
%
%
%
%
%
%
%
%
		&= \left[ \alpha\nabla \uu + ({\uu} \cdot \nn - p) id \right] \nabla_{\Sigma}\Vn
			- \left[ - \dn{p} \nn + \dnn{\uu} + i {(\dn{\uu} \cdot \nn) \nn} \right] \Vn\\
		&\qquad -i \left({\uu} \cdot \nabla_{\Sigma}\Vn \right) \nn.
\end{aligned}
\end{equation}
%
%
In summary, the shape derivative $(\uu',p')$ of $(\uu,p)$ which solves equation \eqref{eq:ccbm} is given by
\begin{equation}
\label{eq:ccbm_shape_derivative_of_the_states}
\left\{\arraycolsep=1.4pt\def\arraystretch{1.1}
\begin{array}{rcll}
	- \alpha \Delta \uu' + \nabla p'	&=& \vect{0}			&\quad\text{in $\Omega$},\\
	\nabla \cdot \uu'				&=& \vect{0} 		&\quad\text{in $\Omega$},\\
	\uu'	 					&=& \vect{0}			&\quad\text{on $\Gamma$},\\
	-p'\nn + \alpha \dn{\uu'} + i (\uu' \cdot \nn)\nn			&=& \vect{B}[\Vn]		&\quad\text{on $\Sigma$},
\end{array}
\right.
\end{equation}
where $\vect{B}[\Vn]$ is given by \eqref{eq:shape_derivative_boundary_condition}.

\medskip 

\textit{Step 2.} Because we have sufficient regularity on the unknown variables, the directional derivatives (see, e.g., \cite[Chap. 2, Sec. 2]{DelfourZolesio2011}) $J'(\Omega)\delta\tilde{\uu}$ and $J'(\Omega)\delta\tilde{p}$ exist and are easily computed as
	\begin{align*}
		J'(\Omega)\delta\tilde{\uu} 
		&=  \frac{d}{d \epsilon} \left. \left( \frac12 \intO{\left( |\ui + \epsilon \delta\tilde{\uu}|^2 + |\pim|^2 \right)} \right)\right|_{\epsilon = 0}
		= \intO{\ui \delta\tilde{\uu}},\\
		J'(\Omega)\delta\tilde{p} 
		&= \frac{d}{d \epsilon} \left. \left( \frac12 \intO{\left( |\ui|^2 + |\pim + \epsilon \delta\tilde{p} |^2 \right)} \right)\right|_{\epsilon = 0}
		= \intO{\pim \delta\tilde{p}}.
	\end{align*}

\medskip 
\textit{Step 3.} Now, to derive and justify the structure of the adjoint system \eqref{eq:adjoint_system}, we let $\vect{\bar{x}}:= (\uu,p) \in \mathcal{X}:=X \times Q$, $\vect{\xi}:=(\bpsi,\mu) \in \mathcal{X}$, $\vect{\bar{y}}:=(\vv,q) \in \mathcal{X}$, and define the operator $\mathcal{E}_{\vect{\bar{x}}}({\vect{\bar{x}}},\Omega)\vect{\xi} \in \mathscr{L}(X,X^{\ast})$ where (cf. \eqref{eq:forms_for_the_adjoint_problems})
\[
	\langle \mathcal{E}_{\vect{\bar{x}}}({\vect{\bar{x}}},\Omega)\vect{\xi} , \vect{\bar{y}}\rangle_{\mathcal{X}^{\ast}, \mathcal{X}}
	:= \tilde{\aaa}({\bpsi},{\vv}) + b(\vv,\mu) + b(\bpsi,q).
\]
The above operator is bijective if and only if for every $\bphi \in X^{\ast}$ and $\lambda \in Q^{\ast}$, there exists a unique solution $\vect{\bar{y}}:=(\vv,q) \in \mathcal{X}$ to the equation (cf. equation \eqref{eq:adjoint_system_weak_form})
\[
	\langle \mathcal{E}_{\vect{\bar{x}}}({\vect{\bar{x}}},\Omega)\vect{\xi} , \vect{\bar{y}}\rangle_{\mathcal{X}^{\ast}, \mathcal{X}}
		= (\bphi, \bpsi) + (\mu, \lambda),
\]
for all $\vect{\xi} \in \mathcal{X}$.
Similar to what we have noted in Remark \ref{rem:well_posedness_adjoint_problem}, the existence of a unique solution to the above equation -- with $\bphi = \ui \in \Vgamma$ and $\lambda = \pim \in Q$ -- can be established using similar arguments issued for the well-posedness of the state problem \eqref{eq:ccbm} (see subsection \ref{subsec:well-posedness_of_state_problem}).
So, we omit the details of the proof.

Since $(\uu,p) \in (\Vgamma \times Q) \cap (\HH^{3}(\Omega)^{d} \times H^{2}(\Omega))$, then by the previous step together with Remark \ref{rem:well_posedness_adjoint_problem}, there exists a unique adjoint state $(\vv,q) \in (\Vgamma \times Q) \cap (\HH^{3}(\Omega)^{d} \times H^{2}(\Omega))$ such that the variational equation \eqref{eq:adjoint_system_weak_form} holds.
By standard arguments as in \textit{Step 1} -- applying integration by parts and/or Green's formula --  we recover \eqref{eq:adjoint_system}.

\medskip 
\textit{Step 4.} Now, we are ready to express the shape derivative of $J$ computed in \eqref{eq:computed_first_derivative_via_Hadamard_formula} via the adjoint method -- eliminating the shape derivative of the states $\uu'$ and $p'$ appearing in $\mathbb{I}_{1}$.
To do this, let us first consider the following weak formulation of \eqref{eq:ccbm_shape_derivative_of_the_states}: find $({\uu'},p') \in \Vgamma \times Q$ such that
    \begin{equation}\label{eq:ccbm_weak_form_shape_derivative}
    \left\{\arraycolsep=1.4pt\def\arraystretch{1.1}
    \begin{array}{rcll}
    	\aaa({\uu'},{\bphi}) + b(\bphi,p')	&=& \displaystyle \intS{\vect{B}[\Vn] \cdot \cbphi},		&\quad\forall {\bphi}\in {\Vgamma},\\ [0.5em]
    				b(\uu',\lambda)	&=& 0,			&\quad\forall \lambda \in {Q}.\\ 
    \end{array}
    \right.
    \end{equation}
Following similar arguments carried out in subsection \ref{subsec:well-posedness_of_state_problem}, the existence of (unique) weak solution to the above problem is a consequence of the complex version of the Lax-Milgram lemma \cite[Thm. 1, p. 376]{DautrayLionsv21998} (see also \cite[Lem. 2.1.51, p. 40]{SauterSchwab2011}).

Now, by taking $(\bphi, \lambda) = (\vv,q) \in \Vgamma \times Q$, we get
    \[
    \left\{\arraycolsep=1.4pt\def\arraystretch{1.1}
    \begin{array}{rcl}
    	\aaa({\uu'},\vv) + b(\vv,p')	&=& \displaystyle \intS{\vect{B}[\Vn] \cdot \overline{\vv}},\\ [0.5em]
    				b(\uu',q)	&=& 0.\\ 
    \end{array}
    \right.
    \]
On the other hand, let us take $(\bpsi, \mu) = (\uu',p') \in \Vgamma \times Q$ in the variational equation \eqref{eq:adjoint_system_weak_form} of the adjoint system \eqref{eq:adjoint_system}.
This lead us to
        \[
        \left\{\arraycolsep=1.4pt\def\arraystretch{1.1}
        \begin{array}{rcll}
        	\tilde{\aaa}({\vv},\uu') + b(\uu',q)	&=& \tilde{F}(\uu'),\\ 
        				b(\vv,p')	&=& (p', \pim).
        \end{array}
        \right.
        \]
Taking the complex conjugate of both sides of the equations above, and then combining it with the previous two, yields the following identity
\[
	\intO{\left( \ui \uu'  + \pim p' \right) } = \intS{\vect{B}[\Vn] \cdot \overline{\vv}}.
\]
Comparing the imaginary parts of both sides of the above equation gives us
\[
	\mathbb{I}_{1} = \intO{\left( \ui \ui'  + \pim \pim' \right) } 
	= \Im\left\{ \intO{\left( \ui \uu'  + \pim p' \right) } \right\} 
	= \Im\left\{ \intS{\vect{B}[\Vn] \cdot \overline{\vv}} \right\}.
\]
In conclusion, the shape derivative of $J$ at $\Omega$ in the direction of the vector filed $\VV \in \sfTheta^{2}$ is given by
	\[
		{{d}}J(\Omega)[\VV]  = \Im\left\{ \intS{\vect{B}[\Vn] \cdot \overline{\vv}} \right\} + \frac12 \intS{ \left( |\ui|^2 + |\pim|^2 \right) \Vn },
	\]
as desired.	
\end{proof}
\end{document}